\documentclass[12pt,amsfonts]{amsart}
\usepackage{amssymb}
\usepackage{mathrsfs}
\usepackage[all,cmtip]{xy}
\usepackage{pstricks}

\usepackage{ulem}

\usepackage{comment}

\numberwithin{equation}{section}

\newtheorem{thm}{Theorem}[section]
\newtheorem{defn}[thm]{Definition}
\newtheorem{prop}[thm]{Proposition}
\newtheorem{res}[thm]{Result}

\newtheorem{lemma}[thm]{Lemma}
\newtheorem{cor}[thm]{Corollary}

\newtheorem{example}[thm]{Example}

\newtheorem{problem}[thm]{Problem}

\newtheorem{rem}[thm]{Remark}

\newcommand{\PGL}{{\rm PGL}}

\newcommand{\Gr}{{\rm Gr}}

\newcommand{\lt}{{\rm lt}}

\newcommand{\CD}{\xymatrix@R=1pc@C=1pc}
\newcommand{\CDR}{\xymatrix@R=1pc}
\newcommand{\CDC}{\xymatrix@C=1pc}

 \DeclareMathOperator{\Spec}{Spec}

\def\cB{{\mathcal B}}

\def\cD{{\mathcal D}}

\def\cE{{\mathcal E}}

\def\cL{{\mathcal L}}

\def\cO{{\mathcal O}}

\def\cZ{{\mathcal Z}}

\def\sF{{\mathscr F}}

\def\vs{{\mathscr I}}

\def\sO{{\mathscr O}}

\def\sR{\mathscr{R}}
\def\sV{\mathscr{V}}
\def\tsV{\widetilde{\mathscr{V}}}

\def\fG{\mathfrak{G}}

\def\2M{M}

\def\fV{\mathfrak{V}}

\def\fd{\mathfrak{d}}
\def\fe{\mathfrak{e}}

\def\fm{\mathfrak{m}}

\def\fn{\mathfrak{n}}

\def\fr{\mathfrak{r}}
\def\fs{\mathfrak{s}}
\def\ft{\mathfrak{t}}

\def\FF{{\mathbb F}}
\def\NN{{\mathbb N}}
\def\PP{{\mathbb P}}

\def\GG{{\mathbb G}}
\def\GGm{{{\mathbb G}_{\rm m}}}

\def\ZZ{{\mathbb Z}}
\def\ZZ{{\mathbb Z}}

\def\II{{\mathbb I}}

\def\AA{{\mathbb A}}

\def\QQ{{\mathbb Q}}

\def\rU{{\rm U}}
\def\Ga{{\Gamma}}
\def\tGa{{\widetilde{\Gamma}}}

\def\vt{{\vartheta}}
\def\si{{\sigma}}
\def\vs{{\varsigma}}
\def\vk{{\varkappa}}

\def\bH{{\bar H}}
\def\bL{{\bar L}}

\def\bG{{\bar G}}

\def\var{{\rm Var}}
\def\Index{{\rm Index}}
\def\invlex{{\rm invlex}}
\def\lex{{\rm lex}}

\def\zero{{=0}}
\def\one{{=1}}

\def\tW{{\widetilde W}}
\def\tX{{\widetilde X}}
\def\tY{{\widetilde Y}}
\def\tZ{{\widetilde Z}}

\def\lra{\longrightarrow}

\def\kk{{\bf k}}

\def\bx{{\bf x}}
\def\bz{{\bf z}}

\def\bm{{\bf m}}
\def\bn{{\bf n}}

\def\ua{{\underbar {\it a}}}
\def\ub{{\underbar {\it b}}}
\def\ui{{\underbar {\it i}}}

\def\uh{{\underbar {\it h}}}
\def\uk{{\underbar {\it k}}}

\def\um{{\underbar {{\mu}}}}
\def\um{{\underbar {\it m}}}
\def\uu{{\underbar {\it u}}}
\def\uv{{\underbar {\it v}}}
\def\uw{{\underbar {\it w}}}

\def\pl{{\hbox{Pl\"ucker}}}
 \def\2{{\rm I\!I}}

\def\bF{{\bar F}}
\def\-{{\setminus}}

\def\ve{{\varepsilon}}

\def\vp{{\varpi}}
\def\vr{{\varrho}}
\def\hs{{\hslash}}
\def\vi{{\varphi}}

\def\vsgn{{\rm sgn}}

\def\rk{{\rm rank \;}}
\def\ori{{\rm ori}}
\def\inc{{\rm inc}}
\def\res{{\rm res}}

\def\mn{{\rm mn}}

\def\pq{{\hbox{\scriptsize ${\rm pre}$-${\rm q}$}}}
\def\q{{\rm q}}

\def\mh{{\hbox{\scriptsize ${\rm m}$-${\rm h}$}}}

\def\sfm{{{\sF}_\um}}

\def\sfmgr{{{\sF}^{\rm rel}_{\um, \Ga}}}
\def\sfmgir{{{\sF}^{\rm irr}_{\um, \Ga}}}

\def\de{\delta}

\def\tsR{{\widetilde{\sR}}}

\def\La{\Lambda}

\def\up{\Upsilon}

\def\tV{{\widetilde V}}

\def\bW{{\underline W}}
\def\bGr{{\underline \Gr}}
\def\tGr{{\widetilde \Gr}}
\def\ud{{\underbar  d}}
\def\tt{{\tilde t}}

\begin{document}

\title{Grassmannians and Singularities}

\date{}
\author{Yi Hu}
\address{Department of Mathematics, University of Arizona, USA.}

\maketitle

\begin{abstract} 
Let $X$  be 
 an integral scheme of finite presentation over a perfect field. 
Let $q$ be a singular closed  point of $X$.
We prove that there exists an open subset $V$ of $X$ containing $q$
such that $V$ admits a resolution, that is, there exists a smooth
scheme $\tV$ and a  proper birational morphism from $\tV$ onto $V$.   
 \end{abstract}

\maketitle

\bigskip\medskip

\hskip 8.7  cm {\footnotesize \it Regularities are all alike; every}

\hskip 8.7cm {\footnotesize \it   singularity is singular in its own way.}

\bigskip

\tableofcontents

\section{Introduction}


Let $V$  be an integral scheme 
of finite presentation over a perfect  field $\kk$. 
 We say $V$ admits a resolution 
if there exists  a smooth scheme $\tV$ over $\kk$
and a proper birational morphism from $\tV$ onto $V$.

\begin{thm}\label{main:intro}
{\rm (Local Resolution, Theorem \ref{main-resolusion-0})}
Let $X$ be a scheme 
of finite presentation over a  perfect field $\kk$. 
Assume further that $X$ is integral and singular.
Take any singular closed point $q \in X$. Then, there exists an open subset
$V$ of $X$ containing the point $q$ such that $V$ admits a resolution.
\end{thm}

Mn\"ev showed  (\cite{Mnev88}) that every integral singularity type of finite type
defined over $\ZZ$ appears in some
configuration space of points on the projective plane.
 This result is called Mn\"ev's universality theorem in literature.
Lafforgue (\cite{La99} and  \cite{La03}) strengthened
 and proved the same statement scheme-theoretically. 
 Also, Lee and Vakil (\cite{LV12}) proved the similar  scheme-theoretic statement
 on incidence schemes of  points and lines on the projective plane.
Using Gelfand-MacPerson correspondence, Lafforgue's version of
Mn\"ev's universality theorem is equivalent to the statement that
every integral singularity type of finite type
defined over $\ZZ$ appears in some {\it thin Schubert cell} on the Grassmannian 
$\Gr^{3, E}$ of three-dimensional  linear subspaces
of a fixed vector space $E$ of dimension greater than 3.
Every thin Schubert cell is an open subset
of a unique closed subscheme of an affine chart of the Grassmannian.
This unique closed subscheme of that affine chart of $\Gr^{3,E}$ is called a $\Ga$-scheme in this article.

We approach Theorem \ref{main:intro} 
via a detour through Mn\"ev's universality theorem by first resolving 
all the aforementioned $\Ga$-schemes that are integral,
hence also, all the thin Schubert cells of $\Gr^{3,E}$ that are integral.

Following Lafforgue's presentation of \cite{La03}, 
suppose we have a set of vector spaces, $E_1, \cdots, E_n$ such that 
every $E_\alpha$, $1 \le \alpha \le n$,  is of dimension 1 over a field $\kk$
 (or, a free module of rank 1 over $\ZZ$),
 for some positive integer $n>1$.
We let  $$E=E_1 \oplus \ldots \oplus E_n.$$ 

 Then, the Grassmannian $\Gr^{d,E}$, defined by
$$\Gr^{d,E}=\{ \hbox{linear subspaces} \;F \hookrightarrow E \mid \dim F=d\}, $$
is a projective variety defined over $\ZZ$, for  any fixed  integer $1 \le d <n$.
(For local resolution of singularities, it suffices to focus on $\Gr^{3,E}$;
in this article,  we still consider the general Grassmannian $\Gr^{d,E}$:
see the third paragraph of \S \ref{localization}.)

We have a canonical decomposition
$$\wedge^d E=\bigoplus_{\ui =(i_1<\cdots< i_d) \in \II_{d,n}} E_{i_1}\otimes \cdots \otimes E_{i_d},$$
where $\II_{d,n}$ is the set of all sequences of $d$ distinct integers between 1 and $n$.

This gives rise to the $\pl$ embedding of the Grassmannian by
$$\Gr^{d,E} \hookrightarrow \PP(\wedge^d E)=\{(p_\ui)_{\ui \in \II_{d,n}} \in \GG_m 
\backslash (\wedge^d E \- \{0\} )\},$$
$$F \lra [\wedge^d F],$$
where $\GG_m$ is the multiplicative group.

As a closed subscheme of $\PP(\wedge^d E)$, the Grassmanian  $\Gr^{d,E}$ is defined by 
the $\pl$ ideal $I_\wp$, generated by all $\pl$ relations, 
whose typical member is expressed succinctly, in this article, as
\begin{equation}\label{eq1-intro}
F: \; \sum_{s \in S_F} \vsgn (s) p_{\uu_s} p_{\uv_s}
\end{equation}
where $S_F$ is an index set, $\uu_s, \uv_s \in \II_{d,n}$
for any $s \in S_F$, and $\vsgn (s)$ is the $\pm$ sign associated with the term $p_{\uu_s} p_{\uv_s}$
(see \eqref{pluckerEq} and \eqref{succinct-pl} for details).

Given the above $\pl$ equation, we introduce
  the projective space $\PP_F$ which comes equipped with the homogeneous coordinates
$[x_{(\uu_s,\uv_s)}]_{s\in S_F}$. 
Then, corresponding to each $\pl$ relation \eqref{eq1-intro}, there is
a linear  homogeneous equation in $\PP_F$, 
called the induced {\it  linearized $\pl$ relation}, 
\begin{equation}\label{eq2-intro}
L_F: \; \sum_{s \in S_F} \vsgn (s) x_{(\uu_s,\uv_s)}
\end{equation}
 (see Definition \ref{defn:linear-pl}). We set
 $ \La_F:=\{(\uu_s,\uv_s) \mid s \in S_F\}.$
 
As the question about resolution of singularity type is local, we can focus on an affine chart
$\rU_\um =(p_\um \ne 0)$ of the $\pl$ projective space $\PP(\wedge^d E)$ for some fixed $\um \in \II_{d,n}$. 
We can identify the coordinate ring of $\rU_\um$ with the polynomial ring 
$\kk [x_\uu]_{\uu \in \II_{d,n} \- \{\um\}}$. For any $\pl$ relation $F$, we let
$\bF$ be the de-homogenization of $F$ on the chart $\rU_\um$.
Given this chart, we then explicitly describe a set of $\pl$ relations, called
{\it $\um$-primary $\pl$ relations}, listed under a carefully chosen total order $``<_\wp"$,
$$\sfm =\{\bF_1 <_\wp \cdots <_\wp \bF_\Upsilon\},$$
with $\Upsilon= {n \choose d}-1-d(n-d)$, such that 
they together generate 
the de-homogenized ideal $I_{\wp, \um}$ of the  $\pl$ ideal $I_\wp$ on  the chart. 
Further, on the chart $\rU_\um$,  if we set $p_\um =1$
 and set $x_\uu=p_\uu$ for any $\uu \in \II_{d,n}\-\{\um\}$,
 then any  $\um$-primary relation $\bF\in \sfm$ admits the following de-homogenized expression
\begin{equation}\nonumber
\bF: \; \vsgn (s_F) x_{\uu_F} +\sum_{s \in S_F \- \{s_F\}} \vsgn (s) x_{\uu_s}x_{\uv_s},
\end{equation}
where $x_{\uu_F}$ is called the leading variable of $\bF$ whose term is called the leading term of $\bF$
and $s_F \in S_F$ is the index
for the leading term.  (See 
\eqref{equ:localized-uu} and \eqref{the-form-LF} for details.)

Next, we can introduce the natural rational map
\begin{equation}\label{this-theta-intro}
 \xymatrix{
\bar \Theta_{[\up],\Gr}: \rU_\um \cap \Gr^{d,E} \ar @{^{(}->}[r]  & \rU_\um \ar @{-->}[r]  & \prod_{\bF \in \sF_\um} \PP_F   }
\end{equation}
$$
 [x_\uu]_{\uu \in \II_{d,n}} \lra  
\prod_{\bF \in \sF_\um}   [x_\uu x_\uv]_{(\uu,\uv) \in \La_F}
$$   
where $[x_\uu]_{\uu \in \II_{d,n}} $ is the de-homogenized $\pl$ coordinates of a point 
of  $\rU_\um \cap \Gr^{d,E}$.   

We let $\sV_\um$ be the closure of the graph of the rational map $\bar\Theta_{[\up],\Gr}$. Then, 
 we obtain the following diagram 
$$ \xymatrix{
\sV_{\um} \ar[d] \ar @{^{(}->}[r]   \ar[d] \ar @{^{(}->}[r]  &
\sR_\sF:= \rU_\um  \times \prod_{\bF \in \sfm} \PP_F \ar[d] \\
\rU_\um \cap \Gr^{d,E}  \ar @{^{(}->}[r]    & \rU_\um.}
$$

The scheme $\sV_{\um}$ is singular, in general, and is birational to $\rU_\um \cap \Gr^{d,E}$.
(The reader is recommended to read \S \ref{tour} to see the purpose of introducing
the model $\sV_{\um}$.)

As the necessary and crucial steps to achieve our ultimate goal, we are to perform some specific 
sequential embedded blowups for $(\sV_\um \subset \sR_\sF)$.

For  the purpose of applying induction, employed mainly for  the purpose of proofs, 
we also introduce the following rational map. 

For any positive integer $m$, we set  $[m]:=\{1,\cdots,m\}.$

Then,  for any $k \in [\up]$, we have the rational map
\begin{equation}\label{this-theta[k]-intro}
 \xymatrix{
 \bar\Theta_{[k],\Gr}: \rU_\um \cap \Gr^{d,E} \ar @{^{(}->}[r]  & \rU_\um \ar @{-->}[r]  & 
 \prod_{i \in [k]} \PP_{F_i}   }
\end{equation}
$$
 [x_\uu]_{\uu \in \II_{d,n}} \lra  
\prod_{i \in [k]}   [x_\uu x_\uv]_{(\uu,\uv) \in \La_{F_i}}
$$ 
We let $\sV_{\sF_{[k]}}$ be the closure of the graph of the rational map $\bar\Theta_{[k],\Gr}$. Then, 
 we obtain the following diagram 
$$ \xymatrix{
\sV_{\sF_{[k]}} \ar[d] \ar @{^{(}->}[r]   \ar[d] \ar @{^{(}->}[r]  &
\sR_{\sF_{[k]}}:= \rU_\um  \times  \prod_{i \in [k]} \PP_{F_i}\ar[d] \\
\rU_\um \cap \Gr^{d,E}  \ar @{^{(}->}[r]    & \rU_\um.
}
$$
The scheme $\sV_{\sF_{[k]}}$ is birational to  $\rU_\um \cap \Gr^{d,E}$. 

Set $\sR_{\sF_{[0]}}:=\rU_\um$.  There exists a forgetful map 
$$\sR_{\sF_{[j]}} \lra \sR_{\sF_{[j-1]}},\;\;\; \hbox{ for any $j \in [\up]$}.$$

In the above notations, we have
$$\sV_\um=\sV_{\sF_{[\up]}}, \;\; \sR_{\sF}=\sR_{\sF_{[\up]}}.$$

\begin{prop}  {\rm (Proposition \ref{eq-tA-for-sV})}  The scheme $\sV_\um$, as a closed subscheme of
$\sR_\sF= \rU_\um \times  \prod_{\bF \in \sF_\um} \PP_F$,
is defined by the following relations, for all $\bF \in \sfm$,
\begin{eqnarray} 
B_{F,(s,t)}: \;\; x_{(\uu_s, \uv_s)}x_{\uu_t}x_{ \uv_t}-x_{(\uu_t, \uv_t)}x_{\uu_s}x_{ \uv_s}, \;\; \forall \;\; 
s, t \in S_F \- \{s_F\}, \label{eq-Bres-intro}\\
B_{F, (s_F,s)}: \;\; x_{(\uu_s, \uv_s)}x_{\uu_F} - x_{(\um,\uu_F)}   x_{\uu_s} x_{\uv_s}, \;\;
\forall \;\; s \in S_F \- \{s_F\},  \label{eq-B-intro} \\ 
\cB^\pq,   \;\; \;\; \;\; \;\; \;\; \;\; \;\; \;\; \;\; \;\; \;\; \;\; \;\; \label{eq-hq-intro}\\
L_F: \;\; \sum_{s \in S_F} \vsgn (s) x_{(\uu_s,\uv_s)}, 
\;\;\; \bF: \;\; \sum_{s \in S_F} \vsgn (s) x_{\uu_s}x_{\uv_s}\label{linear-pl-intro}
\end{eqnarray}
with $\bF$  expressed as 
$\vsgn (s_F) x_{\uu_F} +\sum_{s \in S_F \- \{s_F\}} \vsgn (s) x_{\uu_s}x_{\uv_s}$,
 where $\cB^\pq$ is the set of binomial equations of pre-quotient type
(see Definition \ref{defn:pre-q}).  
\end{prop}

Our construction of the desired embedded blowups on  $\sV_{\um} \subset \sR_\sF$
 is based upon the set of all binomial relations of \eqref{eq-B-intro}:
$$\cB^\mn=\{B_{F,(s_F,s)} \mid \bF \in \sfm, \; s \in S_F \- \{s_F\}\}.$$
An element $B_{F,(s_F,s)}$ of  $\cB^\mn$ is called a main binomial relation.
We  also let
$$\cB^\res=\{ B_{F,(s,t)} \mid \bF \in \sfm, \; s, t \in S_F \- \{s_F\}\}.$$
An element $B_{F,(s,t)}$ of  $\cB^\res$ is called a residual binomial relation.
The residual binomial relations or binomial relations of pre-quotient type in $\cB^\pq$
play no roles in the {\it construction} of the aforesaid embedded blowups.

To apply induction, we provide a total order on the set $S_F \- \{s_F\}$ and list it as
$$S_F \- \{s_F\}=\{s_1 < \cdots < s_{\ft_F}\}$$
where $(\ft_F+1)$ is the number of terms in the relation $F$.
This renders us  to write $B_{F,(s_F,s)}$ as $B_{(k\tau)}$ where
$F=F_k$ for some $k \in [\up]$ and $s = s_\tau$ for some $ \tau \in [\ft_{F_k}]$.

We can now synopsize  the process of the embedded blowups for  $(\sV_{\um} \subset \sR_\sF)$.

It is divided into three sequential blowups, 
named as,  $\vt$-, $\wp$-, $\eth$-blowups, listed in the order of occurrence. 
(For the specific purpose of each of the three sequential blowups, 
the reader is referred to \S \ref{tour}.)

{\it $\bullet$ On $\vt$-sets, $\vt$-centers, and $\vt$-blowups. }

For any primary $\pl$ relation $\bF_k \in \sfm$, we introduce the corresponding $\vt$-set
$\vt_{[k]}=\{x_{\uu_{F_k}}, x_{(\um,\uu_{F_k})} \}$ and the corresponding $\vt$-center 
$Z_{\vt_{[k]}} = X_{\uu_{F_k}} \cap X_{(\um,\uu_{F_k})}$ where
$X_{\uu_{F_k}} = (x_{\uu_{F_k}}=0)$ and
$ X_{(\um,\uu_{F_k})} =(x_{(\um,\uu_{F_k})} =0)$.
 We set $\tsR_{\vt_{[0]}}:=\sR_\sF$. Then, inductively,
 we let  $\tsR_{\vt_{[k]}} \to \tsR_{\vt_{[k-1]}}$ be the blowup of $\tsR_{\vt_{[k-1]}}$
 along (the proper transform of) the $\vt$-center
 $Z_{\vt_{[k]}}$ for all $k \in [\up]$. 
 This gives rise to the sequential $\vt$-blowups
\begin{equation}\label{vt-sequence-intro}
\tsR_{\vt}:=\tsR_{\vt_{[\up]}}  \lra \cdots \lra \tsR_{\vt_{[k]}} \lra \tsR_{\vt_{[k-1]}} \lra \cdots \lra \tsR_{\vt_{[0]}}.
\end{equation}
 Each morphism $\tsR_{\vt_{[k]}} \to \tsR_{\vt_{[k-1]}}$ is a smooth blowup, meaning, 
 the blowup of a smooth scheme along a smooth closed center. For any $k$, we let
 $\tsV_{\vt_{[k]}} \subset \tsR_{\vt_{[k]}} $ be the proper transform of $\sV$ in $\tsR_{\vt_{[k]}}$.
 We set $\tsV_{\vt}:=\tsV_{\vt_{[\up]}}$.
 
{\it $\bullet$ On $\wp$-sets, $\wp$-centers, and $\wp$-blowups. }

For any main binomial $B_{(k\tau)} \in \cB^\mn$, there exist a finite 
integer $\rho_{(k\tau)}$ depending on $(k\tau)$ and 
a finite integer $\si_{(k\tau)\mu}$ depending on $(k\tau)\mu$ for any $ \mu \in [\rho_{(k\tau)}]$.
We set $\tsR_{(\wp_{(11)}\fr_0)}=\tsR_\vt$.
For each $(k,\tau, \mu, h)$ as above, there exists 
a $\wp$-set $\phi_{(k\tau)\mu h}$ consisting of  two special divisors  on an inductively 
defined scheme $\tsR_{(\wp_{(k\tau)}\fr_{\mu -1})}$; its corresponding $\wp$-center
$Z_{\phi_{(k\tau)\mu h}}$ is the scheme-theoretic intersection of the two divisors.
We let $\cZ_\wp=\{Z_{\phi_{(k\tau)\mu h}} \mid k \in [\up], \tau \in [\ft_{F_k}], \mu \in [\rho_{(k\tau)}],
h \in [\si_{(k\tau)\mu}]\}$, 
totally ordered lexicographically on the indexes $(k,\tau, \mu, h)$.
We set $\tsR_{\wp_{(11)}\fr_1\fs_{0}}:=\sR_\vt$. 
Then, inductively,
 we let we $\tsR_{(\wp_{(k\tau)}\fr_\mu\fs_{h})} \to \tsR_{(\wp_{(k\tau)}\fr_\mu\fs_{h-1})}$ 
 be the blowup of $\tsR_{(\wp_{(k\tau)}\fr_\mu\fs_{h-1})}$ 
 along (the proper transform of) the $\wp$-center
 $Z_{\phi_{(k\tau)\mu h}}$. 
 This gives rise to the sequential  $\wp$-blowups
\begin{equation}\label{wp-sequence-intro}
\tsR_{\wp}  \to \cdots \to
\tsR_{(\wp_{(k\tau)}\fr_\mu\fs_{h})} \to \tsR_{(\wp_{(k\tau)}\fr_\mu\fs_{h-1})} \to \cdots \to \tsR_{\vt},
\end{equation}
where $\tsR_{\wp}:=\tsR_{(\wp_{(\up \ft_{F_{\up}})}\fr_{\rho_{(\up \ft_{F_{\up}})}}
\fs_{\si_{(\up \ft_{F_{\up}})}\rho_{(\up \ft_{F_{\up}})} })} $.
 Each morphism $\tsR_{(\wp_{(k\tau)}\fr_\mu\fs_{h})} \to \tsR_{(\wp_{(k\tau)}\fr_\mu\fs_{h-1})} $
  is a smooth blowup. For any $(k\tau)\mu h$, we let
 $\tsV_{(\wp_{(k\tau)}\fr_\mu\fs_{h})} \subset \tsR_{(\wp_{(k\tau)}\fr_\mu\fs_{h})} $
  be the proper transform of $\sV$ in $\tsR_{(\wp_{(k\tau)}\fr_\mu\fs_{h})} $.
 We set $\tsV_{\wp}:=\tsV_{(\wp_{(\up \ft_{F_{\up}})}\fr_{\rho_{(\up \ft_{F_{\up}})}}
\fs_{\si_{(\up \ft_{F_{\up}})}\rho_{(\up \ft_{F_{\up}})} })} $.

{ \it $\bullet$ On $\eth$-sets, $\eth$-centers, and $\eth$-blowups. }

For any main binomial $B_{(k\tau)} \in \cB^\mn$, there exist  a finite 
integer $\vk_{(k\tau)}$ depending on $(k\tau)$ and 
a finite integer $\vs_{(k\tau)\mu}$ depending on $(k\tau)\mu$ for any $ \mu \in [\vk_{(k\tau)}]$.
We set $\tsR_{(\eth_{(11)}\fr_0)}=\tsR_\wp$.
For each $(k,\tau, \mu, h)$ as above, there exists 
an $\eth$-set $\psi_{(k\tau)\mu h}$ consisting of  two special divisors  on an inductively 
defined scheme $\tsR_{(\eth_{(k\tau)}\fr_{\mu -1})}$; its corresponding $\eth$-center
$Z_{\psi_{(k\tau)\mu h}}$ is the scheme-theoretic intersection of the two divisors.
We let $\cZ_\eth=\{Z_{\psi_{(k\tau)\mu h}} \mid k \in [\up], \tau \in [\ft_{F_k}], \mu \in [\vk_{(k\tau)}],
h \in [\vs_{(k\tau)\mu}]\}$, 
totally ordered lexicographically on the indexes $(k,\tau, \mu, h)$.
We set $\tsR_{\eth_{(11)}\fr_1\fs_{0}}:=\sR_\wp$. Then, inductively,
 we let we $\tsR_{(\eth_{(k\tau)}\fr_\mu\fs_{h})} \to \tsR_{(\eth_{(k\tau)}\fr_\mu\fs_{h-1})}$ 
 be the blowup of $\tsR_{(\wp_{(k\tau)}\fr_\mu\fs_{h-1})}$ 
 along (the proper transform of) the $\eth$-center
 $Z_{\psi_{(k\tau)\mu h}}$. 
 This gives rise to the sequential $\eth$-blowups
\begin{equation}\label{hs-sequence-intro}
\tsR_{\eth}  \to \cdots \to
\tsR_{(\eth_{(k\tau)}\fr_\mu\fs_{h})} \to \tsR_{(\eth_{(k\tau)}\fr_\mu\fs_{h-1})} \to \cdots \to \tsR_{\eth},
\end{equation}
where $\tsR_{\eth}:=\tsR_{(\eth_{(\up \ft_{F_{\up}})}\fr_{\vk_{(\up \ft_{F_{\up}})}}
\fs_{\vs_{(\up \ft_{F_{\up}})}\rho_{(\up \ft_{F_{\up}})} })} $.
 Each morphism $\tsR_{(\eth_{(k\tau)}\fr_\mu\fs_{h})} \to \tsR_{(\eth_{(k\tau)}\fr_\mu\fs_{h-1})} $
  is a smooth blowup. For any $(k\tau)\mu h$, we let
 $\tsV_{(\eth_{(k\tau)}\fr_\mu\fs_{h})} \subset \tsR_{(\eth_{(k\tau)}\fr_\mu\fs_{h})} $
  be the proper transform of $\sV$ in $\tsR_{(\eth_{(k\tau)}\fr_\mu\fs_{h})} $.
 We set $\tsV_{\eth}:=\tsV_{(\eth_{(\up \ft_{F_{\up}})}\fr_{\vk_{(\up \ft_{F_{\up}})}}
\fs_{\vs_{(\up \ft_{F_{\up}})}\rho_{(\up \ft_{F_{\up}})}})} $.

 The schemes $\tsV_\eth \subset \tsR_\eth$ are our final destination.


To study the local structure of $\tsV_\eth \subset \tsR_\eth$, we approach 
it by induction via the sequential blowups \eqref{vt-sequence-intro},
\eqref{wp-sequence-intro}, and \eqref{hs-sequence-intro}.

 Definition \ref{general-standard-chart} introduces the covering
  standard affine charts $\fV$ for any  of the smooth schemes  $\tsR_{\vt_{[k]}}$,
  $\tsR_{(\wp_{(k\tau)}\fr_\mu\fs_h)}$,
 $\tsR_{(\eth_{(k\tau)}\fr_\mu\fs_h)}$, as above. 
 
 $(\star)$ Proposition \ref{meaning-of-var-vtk} introduces coordinate variables
 for  any standard affine chart $\fV$ of $\tsR_{\vt_{[k]}}$ and provides
 explicit geometric meaning for every coordinate variable.
 
 Proposition  \ref{eq-for-sV-vtk} provides explicit description and properties of the local defining equations of 
the scheme $\tsV_{\vt_{[k]}} \cap \fV$ on any standard affine chart $\fV$ of $\tsR_{\vt_{[k]}}$.
 
 $(\star)$ Proposition \ref{meaning-of-var-vskmuh} introduces coordinate variables
 for  any standard affine chart $\fV$ of $\tsR_{(\wp_{(k\tau)}\fr_\mu\fs_h)}$
 and provides explicit geometric meaning for every coordinate variable.
 
 Proposition  \ref{equas-vskmuh}  provides explicit description and properties
  of the local defining equations of the scheme  
$\tsV_{\wp_{(k\tau)}\fr_\mu\fs_h}  \cap \fV$ on any standard affine chart $\fV$ of 
 $\tsR_{(\wp_{(k\tau)}\fr_\mu\fs_h)}$.

 $(\star)$ Proposition \ref{meaning-of-var-p-k} introduces coordinate variables
 for  any standard affine chart $\fV$ of 
  $\tsR_{(\eth_{(k\tau)}\fr_\mu\fs_h)}$ and provides
 explicit geometric meaning for every coordinate variable.

  Proposition  \ref{equas-p-k}  provides explicit description
  and properties of the local defining equations of 
the scheme 
$\tsV_{\eth_{(k\tau)}\fr_\mu\fs_h}  \cap \fV$ on any standard affine chart $\fV$ of 
$\tsR_{(\eth_{(k\tau)}\fr_\mu\fs_h)}$. 


To summarize the progress, we depict the diagram \eqref{theDiagram} below.

Thus far, we have obtained the first two rows of the  diagram:  

 $(\star)$ In the first row: each morphism  $\tsR_{\hs} \to  \tsR_{\hs'}$ is 
$\tsR_{\vt_{[k]}} \to \tsR_{\vt_{[k-1]}}$, or
$\tsR_{(\eth_{(k\tau)}\fr_\mu\fs_{h})} \to  \tsR_{(\eth_{(k\tau)}\fr_\mu\fs_{h-1})}$,
or $\tsR_{(\wp_{(k\tau)}\fr_\mu\fs_{h})} \to  \tsR_{(\wp_{(k\tau)}\fr_\mu\fs_{h-1})}$,
and  is a smooth blowup;
each  $\sR_{\sF_{[j]}}  \to  \sR_{\sF_{[j-1]}}$  is a projection, a forgetful map.

 $(\star)$ In the second row: each morphism $\tsV_{\hs} \to  \tsV_{\hs'}$ is 
$\tsV_{\vt_{[k]}} \to \tsV_{\vt_{[k-1]}}$,
or $\tsV_{(\eth_{(k\tau)}\fr_\mu\fs_{h})} \to  \tsV_{(\eth_{(k\tau)}\fr_\mu\fs_{h-1})}$, 
or $\tsV_{(\wp_{(k\tau)}\fr_\mu\fs_{h})} \to  \tsV_{(\wp_{(k\tau)}\fr_\mu\fs_{h-1})}$;
 this morphism as well as each $\sV_{\sF_{[j]}}  \to  \sV_{\sF_{[j-1]}}$   is surjective, projective, and birational.

To explain the third and fourth rows of the  diagram,
we go back to the fixed chart $\rU_\um$. This is the affine space which comes equipped with
the coordinate variables $\var_{\rU_\um}:=\{x_\uu \mid \uu \in \II_{d,n} \- \{\um\}\}$.
Let $\Ga$ be any subset of $\var_{\rU_\um}$
 and let $Z_\Ga$ be the subscheme of $\rU_\um$ defined by the ideal generated by
 all the elements of $\Ga$ together with all the de-homogenized $\um$-primary $\pl$ relations $\bF$
 with $\bF \in \sfm$.   
This is a $\Ga$-scheme mentioned in the beginning of this introduction.
The precise relation between a given thin Schubert cell and its corresponding $\Ga$-scheme
is given in \eqref{ud=Ga}.

 Our gaol is to resolve the $\Ga$-scheme $Z_\Ga$ when it is integral and singular.

\begin{equation}\label{theDiagram}
 \xymatrix@C-=0.4cm{
  \tsR_{\eth} \ar[r] & \cdots  \ar[r] &  \tsR_{\hs} \ar[r] &  \tsR_{\hs'} \ar[r] &  \cdots  \ar[r] &   \sR_{\sF_{[j]}}  \ar[r] &  \sR_{\sF_{[j-1]}} \cdots \ar[r] &  \rU_\um \\
    \tsV_{\eth} \ar @{^{(}->} [u]  \ar[r] & \cdots  \ar[r] &  \tsV_{\hs}\ar @{^{(}->} [u]   \ar[r] &  \tsV_{\hs'} \ar @{^{(}->} [u]  \ar[r] &  \cdots    \ar[r] &   \sV_{\sF_{[j]}} \ar @{^{(}->} [u]\ar[r] &  \sV_{\sF_{[j-1]}} \cdots \ar @{^{(}->} [u]  \ar[r] &  \rU_\um \cap \Gr^{ d,E}   \ar @{^{(}->} [u]  \\
   \tZ_{\eth, \Ga} \ar @{^{(}->} [u]  \ar[r] & \cdots  \ar[r] &  \tZ_{\hs,\Ga}\ar @{^{(}->} [u]   \ar[r] &  \tZ_{\hs',\Ga} \ar @{^{(}->} [u]  \ar[r] &  \cdots    \ar[r] &   Z_{\sF_{[j]},\Ga} \ar @{^{(}->} [u]\ar[r] &  Z_{\sF_{[j-1])},\Ga} \cdots \ar @{^{(}->} [u]  \ar[r] &  Z_\Ga  \ar @{^{(}->} [u]  \\
    \tZ^\dagger_{\eth, \Ga} \ar @{^{(}->} [u]  \ar[r] & \cdots  \ar[r] &  \tZ^\dagger_{\hs,\Ga}\ar @{^{(}->} [u]   \ar[r] &  \tZ^\dagger_{\hs',\Ga} \ar @{^{(}->} [u]  \ar[r] &  \cdots    \ar[r] &   Z^\dagger_{\sF_{[j]},\Ga} \ar @{^{(}->} [u]\ar[r] &  Z^\dagger_{\sF_{[j-1])},\Ga} \cdots \ar @{^{(}->} [u]  \ar[r] &  Z_\Ga, \ar[u]_{=}       }
\end{equation}
 where all vertical uparrows are closed embeddings. 

 Let $\Ga$ be a subset $\var_{\rU_\um}$. Assume that 
 $Z_\Ga$ is integral. Then, starting from $Z_\Ga$, step by step,
  via induction within every of the sequential $\vt$-, $\wp$-, and
  $\eth$-blowups, we are able to construct
 the third and fourth rows in the  diagram  \eqref{theDiagram} such that 
 
$(\star)$ every closed subscheme in the third row, $Z_{\sF_{[j]},\Ga}$, respectively,
$\tZ_{\hs}$, admits explicit local defining equations
in any standard chart of its corresponding smooth scheme in the first row;
 
$(\star)$ every closed subscheme in the fourth row $Z^\dagger_{\sF_{[j]},\Ga}$, respectively,
$\tZ^\dagger_{\hs}$, is an irreducible component of its
corresponding subscheme $Z_{\sF_{[j]},\Ga}$, respectively, $\tZ_{\hs}$, such that  the induced 
morphism $\hbox{$Z^\dagger_{\sF_{[j]},\Ga} \lra Z_\Ga$, respectively,
$\tZ^\dagger_{\hs}\lra Z_\Ga$}$
is surjective, projective, and birational.
 
$(\star)$   the left-most   $\tZ_{\eth, \Ga} $ is smooth; so is $\tZ^\dagger_{\eth, \Ga}$,
now a connected component of $\tZ_{\eth, \Ga} $.

{\it
In this article, a scheme $X$ is smooth if it is a disjoint union of connected smooth schemes of
possibly various dimensions.}

 The closed subscheme $Z_{\sF_{[j]},\Ga}$, called an $\sF$-transform of $Z_\Ga$,
  is constructed in Lemma \ref{wp-transform-sVk-Ga};
  the closed subscheme $Z_{\vt_{[j]},\Ga}$, called a $\vt$-transform of $Z_\Ga$,
  is constructed in Lemma \ref{vt-transform-k};
 the closed subscheme $\tZ_{(\wp_{(k\tau)}\fr_\mu\fs_{h}),\Ga}$, called a $\wp$-transform of $Z_\Ga$,
 is constructed in Lemma \ref{wp-transform-ktauh};
 the closed subscheme $\tZ_{(\eth_{(k\tau)}\fr_\mu\fs_{h}),\Ga}$, called an $\eth$-transform of $Z_\Ga$,
  is constructed in Lemma \ref{vr-transform-ktauh}.

Our main theorem on the Grassmannian is

\begin{thm}\label{main2:intro} 
{\rm (Theorems \ref{main-thm} and \ref{cor:main})} 
Let $\FF$ be either $\QQ$ or a finite field with $p$ elements where
$p$ is a prime number.
Let $\Ga$ be any subset of $\var_{\rU_\um}$.
Assume that $Z_\Ga$ is integral. 
Let $\tZ_{\eth,\Ga}$ be  the $\eth$-transform of $Z_\Ga$ in $\tsV_{\eth}$.
Then,   $\tZ_{\eth,\Ga}$ is smooth over $\FF$.
In particular, the induced morphism $\tZ^\dagger_{\eth,\Ga} \to Z_\Ga$ is a resolution over $\FF$,
if $Z_\Ga$ is singular.
\end{thm}


The proof of Theorem \ref{main2:intro} 
 (Theorems \ref{main-thm} and \ref{cor:main}) is based upon the explicit  description of
 the main binomials and linearized $\pl$ defining equations of $\tZ_{\eth,\Ga}$
 (Corollary \ref{eth-transform-up})
 and detailed calculation and careful analysis on the Jacobian of these equations (\S \ref{main-statement}).


Theorem \ref{main:intro} is  obtained by applying Theorem \ref{main2:intro},
combining with Lafforgue's version of Mn\"ev's unversality theorem
(Theorems \ref{Mn-La} and \ref{Mn-La-Gr}).
 
In general, consider any fixed singular integral scheme  $X$. 
By Theorem \ref{main:intro}, $X$ can be covered by finitely many affine open subsets such that
every of these affine open subsets of  $X$ admits a resolution. 
It remains to glue finitely many  such local resolutions 
to obtain  a global one. This is being pursued.


We learned that Hironaka posted a preprint on resolution of singularities 
in positive characteristics \cite{Hironaka17}.

In spite of the current article, the author is not in a position to survey the topics of
resolution of singularities, not even very briefly.
We refer to Koll\'ar's book \cite{Kollar} for an extensive list of references on resolution of singularities.
 There have been some recent progresses since the book \cite{Kollar}:
 risking inadvertently omitting some other's works, let us just mention a few recent ones
\cite{ATW},  \cite{McG}, and \cite{Temkin}.


\medskip
The approach presented in this paper was inspired by  \cite{Hu15}.
The two articles, however, are mathematically independent. 

The author is grateful to the anonymous reviewers for their very helpful questions
and constructive suggestions, especially for pointing out
the insufficiency of the previous version. 
He thanks J\'anos Koll\'ar and Chenyang Xu for the suggestion to write 
a summary  section, \S \ref{tour}, to lead the reader a quick tour through the paper.

He thanks Laurent Lafforgue for a suggestion and sharing a general question. 
He thanks Caucher Birkar, James McKernan, and Ravi Vakil for the invitation to speak in 
seminars and for useful correspondences. He thanks
Bingyi Chen for spotting an error and  inaccuracy in the proof of Theorem 10.5.

\bigskip 

\centerline {A List of Fixed Notations Used Throughout}
\medskip

\smallskip\noindent $[n]$: the set of all  integers from 1 to $n$, $\{1, 2 \cdots ,n \}.$

\noindent 
$\II_{d,n}$: the set of all sequences of integers $\{(1 \le u_1 < \cdots < u_d \le n) \}.$

\noindent
$\PP(\wedge^d E)$: the projective space with $\pl$ coordinates 
$p_\ui, \ui \in \II_{d,n}$. 



\noindent
$I_\wp$: the ideal of $\kk[p_\ui]_{\ui \in \II_{d,n}}$ generated by all $\pl$ relations.

\noindent
$\rU_\um$: the affine chart of $\PP(\wedge^d E)$ defined by $p_\um \ne 0$ for some fixed $\um \in \II_{d,n}$.

\noindent
$\sfm$: the set of $\um$-primary $\pl$ equations.

\noindent
$L_{\sfm}$: the set of all linearized $\um$-primary $\pl$ equations.

 \noindent
$\up:= {n \choose d} -1 - d(n-d)$: the cardinality of $\sfm$;

\noindent
$\fV$:  a standard affine chart of an ambient smooth scheme; 
 
 \noindent 
$\cB^\mn$:  the set of all main binomial relations; 

 \noindent 
$\cB^\res$:  the set of all  residual binomial relations; 

\noindent 
$\cB^\pq$:  the set of all  binomial relations of pre-quotient type;

 \noindent 
$\cB^\q$:  the set of all  binomial relations of quotient type; 

 \noindent 
$\cB$:   $\cB^\mn \sqcup \cB^\res \sqcup \cB^\q$;


\noindent $\Ga$: a subset of $\rU_\um$.

\noindent $A \- a$: $A\-\{a\}$ where $A$ is a finite set and $a \in A$.

\noindent $|A|$: the cardinality of a finite set $A$.

\noindent $\kk$: a fixed perfect field.

\section{A Quick Tour: the main idea and approach}\label{tour}

{\it This section may be skipped entirely if the reader  prefers to dive into
the main text immediately. However, reading this section first is recommended.}

\medskip
$\bullet$ {\it A detour to $\Gr^{3,E}$ via Mn\"ev's universality.} 

By Mn\"ev's universality, any singularity over $\ZZ$ appears in a thin Schubert cell of 
the Grassmannian $\Gr^{3,E}$ of three dimensional linear subspaces in a vector space $E$.

Consider the $\pl$ embedding $\Gr^{3,E} \subset \PP(\wedge^3 E)$ 
with $\pl$ coordinates $p_{ijk}$.
A thin Schubert cell of $\Gr^{3,E}$ is a nonempty intersection of codimension one Schubert cells
of  $\Gr^{3,E}$; it corresponds to  a matroid $\ud$ of rank 3 on the set $[n]$.
Any Schubert divisor is defined by $p_{ijk} =0$ for some $(ijk)$. Thus, 
a thin Schubert cell  $\Gr^{3,E}_\ud$ of the matroid $\ud$
 is an open subset of the closed subscheme $\overline{Z}_\Ga$ of
$\Gr^{3,E}$ defined by $\{p_{ijk}=0 \mid p_{ijk} \in \Ga\}$ for some subset $\Ga$ of all $\pl$ variables.
The thin Schubert cell  must lie in an affine chart $(p_\um \ne 0)$ for some $\um \in \II_{3,n}$.
Thus,  $\Gr^{3,E}_\ud$  is an open subset of a closed subscheme 
of  $\Gr^{3,E} \cap (p_\um \ne 0)$ of
the following form 
$$Z_\Ga=\{ p_{ijk} =0 \mid p_{ijk} \in \Ga\} \cap \Gr^{3,E} \cap (p_\um \ne 0).$$ 
This is a closed affine subscheme of the affine chart $(p_\um \ne 0)$.
We aim to resolve $Z_\Ga$, hence also the thin Schubert cell  $\Gr^{3,E}_\ud$, when both
are integral and singular.

\medskip
 $\bullet$ {\it Minimal set of $\pl$ relations for the chart $(p_\um \ne 0)$.} 
 
 Up to permutation, we may assume that $\um=(123)$ and the chart is
$$\rU_\um:=(p_{123} \ne 0).$$
We write the de-homogenized coordinates of $\rU_\um$ as
$$\{x_{abc} \mid (abc) \in \II_{3,n} \- \{(123) \}.$$

As a closed subscheme of the affine space $\rU_\um,$
$Z_\Ga$ is defined by
$$\{ x_{ijk} =0, \;\; \bF=0 \; \mid \; x_{ijk} \in \Ga\},$$
where $\bF$ rans over all de-homogenized $\pl$ relations.
We need to pin down some explicit $\pl$ relations to form 
a minimal set of generators of $\Gr^{3,E} \cap \rU_\um$.

They are of the following forms:
\begin{eqnarray}
\bF_{(123),1uv}=x_{1uv}-x_{12u}x_{13v} + x_{13u}x_{12v}, \label{rk0-1uv}\\
\bF_{(123),2uv}=x_{2uv}-x_{12u}x_{23v} + x_{23u}x_{12v}, \label{rk0-2uv} \\
\bF_{(123),3uv}=x_{3uv}-x_{13u}x_{23v} + x_{23u}x_{13v} ,\label{rk0-3uv}\\
\bF_{(123),abc}=x_{abc}-x_{12a}x_{3bc} + x_{13a}x_{2bc} -x_{23a}x_{1bc}, \label{rk1-abc}
\end{eqnarray}
where $u < v \in [n]\-\{1,2,3\}$ and $a<b<c \in [n]\-\{1,2,3\}$.
Here, $[n]=\{1,\cdots,n\}$.

In a nutshell, we have the set
\begin{eqnarray}\label{rk0-and-1}
\sfm=\{\bF_{(123),iuv},  \; 1 \le i \le 3; \;\; \bF_{(123),abc}\}
\end{eqnarray}
Every relation of $\sfm$ is called $\um$-primary. Here, $\um=(123)$.



\medskip
$\bullet$ {\sl Nicely described equations of $\Ga$-schemes and arbitrary intersections.}

Hence, as a closed subscheme of the affine space $\rU_\um,$
$Z_\Ga$ is defined by
\begin{eqnarray}\label{normal-form}
Z_\Ga=\{ x_\uu =0, \;\; \bF_{(123),iuv},  \; 1 \le i \le 3, \;\; \bF_{(123),abc}  \mid x_\uu \in \Ga\},
\end{eqnarray}
for all $u < v \in [n]\-\{1,2,3\}$ and $a<b<c \in [n]\-\{1,2,3\}$.
The thin Schubert cell $Z_\Ga^\circ$ in $Z_\Ga$ is characterized by
$x_\uv \ne 0$ for any $x_\uv \notin \Ga$ (see Proposition \ref{to-Ga} and \eqref{ud=Ga}.)

Upon setting $x_\uu =0$ with $\uu \in \Ga$, we obtain 
the affine coordinate subspace $$\rU_{\um,\Ga} \subset \rU_\um$$ such that $Z_\Ga$,
 as a closed subscheme of
 the affine subspace $\rU_{\um,\Ga}$, is defined by 
 \begin{eqnarray}\label{reduced-normal-form}
\{\bF_{(123),iuv}|_\Ga, \; 1 \le i \le 3;  \;\;  \bF_{(123),abc}|_\Ga\},
\end{eqnarray}
where $\bF|_\Ga$ denotes the restriction of $\bF$ to the affine subspace $\rU_{\um,\Ga}$.
These are in general truncated  $\pl$ equations, some of which  may be identically zero.

 {\it One may view \eqref{normal-form} as the normal form of singularities,
and \eqref{reduced-normal-form} as the reduced normal form of singularities.}

We  {\it do not} analyze singularities of $Z_\Ga$.

But, we make some remarks.
 By the normal form of sinhularities \eqref{normal-form},
the $\Ga$-scheme $Z_\Ga$ is cut out from the affine chart $\rU_\um$ of the Grassmannian
$\Gr^{3,E}$ by the coordinate hyperplanes $x_\uu =0$
for all $x_\uu \in \Ga$. Although $Z_\Ga$ as well as the thin Schubert cell $Z_\Ga^\circ$
(see the sentence below \eqref{normal-form})
are nicely described by
$\pl$ variables and $\pl$ relations,  the intersections of these
coordinate hyperplanes with the chart $\rU_\um$ of the Grassmannian 
$\Gr^{3,E}$ are arbitrary, according to
Mn\"ev's universality. We may view $\rU_\um$ (allowing 
$\Gr^{3,E}$ to vary) as a universe that contains arbitrary singularities.
Hence, intuitively, we need to birationally change the universe $``$along these intersections" 
 so that eventually in the new universe, $``$they" re-intersect properly. 




 To achieve this, it is more workable if we can
put all the singularities in a different universe $\sV_\um$, 
birationally modified from the chart $\rU_\um \cap \Gr^{3,E}$,
  so that in the new model $\sV_\um$,  all the terms of 
the above $\pl$ relations can be separated. 
(Years had been passed before we {\it came  back} to the right approach.)

\medskip
$\bullet$ {\sl Separating the terms of $\pl$ relations.}

Motivated by \cite{Hu15}, we establish 
a local model  $\sV_\um$, birational  to the chart $\rU_\um \cap \Gr^{3,E}$,
such  that in a {\it specific} set of defining binomial equations of $\sV_\um$, {\it all the terms} of 
the above $\pl$ relations are separated.

To explain, we introduce the projective space $\PP_F$ for each and
 every $\pl$ relation $F=\sum_{s \in S_F} \vsgn (s) p_{\uu_s}p_{\uv_s}$ with
$[x_{(\uu_s,\uv_s)}]_{s \in S_F}$ as its homogeneous coordinates. 

We then let $\sV_\um$ be the closure of the graph of the rational map $\bar\Theta_{[\up],\Gr}$ of
 \eqref{this-theta-intro} in the case of $\Gr^{3,E}$.
 (This is  motivated by an analogous construction in \cite{Hu15}.)
By calculating the multi-homogeneous kernel of the homomorphism
\begin{equation}\label{vi-hom} \bar\vi: \kk[(x_\uw);(x_{(\uu,\uv)})] \lra \kk[x_\uw]
\end{equation}
$$x_{(\uu,\uv)} \to x_\uu x_\uv,$$
we determine a set of defining relations of $\sV_\um$ as a closed subscheme
of the smooth ambient scheme $$\sR_\sF:=\rU_\um \times \prod_{\bF \in \sfm} \PP_F.$$ 
These defining relations, among many others, 
include the following binomials
\begin{eqnarray}\label{mainB-tour}
x_{1uv}x_{(12u,13v)} - x_{12u}x_{13v} x_{(123,1uv)}, \; x_{1uv}x_{(13u,12v)}- x_{13u}x_{12v}x_{(123,1uv)}, \\
x_{2uv}x_{(12u,23v)} -x_{12u}x_{23v} x_{(123,2uv)}, \;  x_{2uv}x_{(23u,12v)}-x_{23u}x_{12v} x_{(123,2uv)}, \nonumber \\
x_{3uv}x_{(13u,23v)} -x_{13u}x_{23v}x_{(123,3uv)}, \;  x_{3uv}x_{(23u,13v)} -x_{23u}x_{12v}x_{(123,3uv)},\nonumber\\
 x_{abc}x_{(12a,3bc)}-x_{12a}x_{3bc} x_{(123,abc)},\;
 x_{abc}x_{(13a,2bc)}-x_{13a}x_{2bc} x_{(123,abc)},\nonumber \\
x_{abc}x_{(23a,1bc)} -x_{23a}x_{1bc}x_{(123,abc)}. \nonumber
\end{eqnarray}
We see that  the terms of all the $\um$-primary $\pl$ relations 
of \eqref{rk0-and-1} are separated into the two terms of the above binomials.

 To distinguish, we call $x_\uu$ (e.g., $x_{12u}$)
a $\vp$-variable and $X_\uu=(x_\uu=0)$  a $\vp$-divisor;
 we call $x_{(\uu,\uv)}$ (e.g., $x_{(12u,13v)}$) a $\vr$-variable and
$X_{(\uu,\uv)}=(x_{(\uu,\uv)}=0)$ a $\vr$-divisor.

The defining relations 
also include the linearized $\pl$ relations as in \eqref{eq2-intro}:
$$L_F=\sum_{s \in S_F} \vsgn (s) x_{(\uu_s,\uv_s)}, \; \; \forall \; \bF \in \sfm.$$
The set of all  linearized $\pl$ relations is denoted by $L_{\sfm}$.

There are many other {\it extra} defining relations.

All the $\Ga$-schemes $Z_\Ga$ admit  birational transforms in the singular model
$\sV_\um$. We still {\it do not} analyze the singularities of these transforms. But, we 
make a quick observation: 
when all the terms of  some of the binomials in \eqref{mainB-tour}
vanish at a point of the transform of a $\Ga$-scheme, then a singularity is likely to occur.

We immediately point out here that in the  the singular model
$\sV_\um$, for the birational transform of the thin Schubert cell $Z_\Ga^\circ$,
only for $\vp$-variables $x_\uu$, the thin cell is characterized by 
$``$either $x_\uu =0$ or $x_\uu \ne 0$ for any $x_\uu$ . But, there are numerous $\vr$-variables $x_{(\uu,\uv)}$ on the birational transform of the thin Schubert cell $Z_\Ga^\circ$ that 
are not subject to any such conditions (i.e., they can assume zero or nonzero values on the birational transform of the thin Schubert cell). 


Thus, first, we would like to $``$remove$"$ all the zero factors of all the terms of 
the binomials in \eqref{mainB-tour}.

As it turns out, through years of $``$trial and error$"$, 
$``$removing$"$ all the zero factors of all the  binomial relations of \eqref{mainB-tour}
suffices for our ultimate purpose. 

{\it The geometric intuition behind the above sufficiency is as follows.
 The equations of \eqref{mainB-tour} alone together with $L_{\sfm}$ only define
a reducible closed scheme, in general. The roles of other extra relations (to be discussed soon)
are to  pin down
its main component $\sV_\um$. 
As the process of $``$removing$"$ zero factors goes, a process of some specific blowups, 
all the boundary components are eventually blown out of existence, making  
the proper transforms of \eqref{mainB-tour} together with the linearized
$\pl$ relations  generate the ideal of the final blowup scheme $\tsV_\eth$ of $\sV_\um$,
on all charts.}

We thus call the binomial equations of  \eqref{mainB-tour} the {\it main} binomials.
The set of  main binomials is denoted $\cB^\mn$. 
The set $\cB^\mn$ is equipped with a somewhat carefully chosen total ordering 
(see \eqref{indexing-Bmn}).

The defining relations of $\sV_\um$ in $\sR_\sF$ also include many other binomials: we classify them
as {\it residual} binomials  (see Definition \ref{defn:main-res}) 
and binomials {\it of pre-quotient type} (see Definition \ref{defn:pre-q}).
The set of  residual binomials is denoted $\cB^\res$;
the set of  binomials of pre-quotient type is denoted $\cB^\pq$.

Together, the equations in the following sets
$$\cB^\mn, \; \cB^\res, \; \cB^\pq, \; L_{\sfm},\; \sfm$$
define the scheme $\sV_\um$ in the smooth ambient scheme $\sR_\sF$.
See Corollary \ref{eq-tA-for-sV}. 

When we focus on an arbitrarily fixed chart $\fV$ of $\sR_\sF$,
binomials  of pre-quotient type of $\cB^\pq$
can be further reduced to binomials {\it of quotient type} whose set
is denoted by $\cB^q_\fV$. See Definition \ref{defn:q} and Proposition \ref{equas-p-k=0}. 
(For the  reason to use the term {\it $``$of quotient type$"$}, see \cite{Hu15}.)

As mentioned in the introduction, for the purpose of inductive proofs,
we also need the rational map $\bar\Theta_{[k],\Gr}$ 
of  \eqref{this-theta[k]-intro}, 
and we let $\sV_{\sF_{[k]}}$ be the closure of the rational map of $\bar\Theta_{[k],\Gr}$, 
for all $k \in [\up]$.
In this notation, $\sV_\um=\sV_{\sF_{[\up]}}$. We let 
$$\sR_{\sF_{[k]}}:=\rU \times \prod_{i \in [k]} \PP_{F_i}.$$ 
This is a smooth scheme and contains $\sV_{\sF_{[k]}}$ as a closed subscheme.
Further, we have the natural forgetful map
\begin{equation} \label{forgetful-guide}
\sR_{\sF_{[k]}} \lra \sR_{\sF_{[k-1]}}. 
\end{equation}

\medskip
$\bullet$ {\sl  The process of $``$removing$"$ zero factors of main binomials.}

 To remove zero factors of main binomials, we either work on each   primary 
 $\pl$ relation individually (in the case of a $\vt$-blowup), or work on each 
main binomial individually (in the case of a $\wp$- or an $\eth$-blowup). 
To this end, we need to provide a total order on the set $\sfm$.

We let
$\{\bF_{(123),iuv},  \; 1 \le i \le 3\}$ go first, then followed by $\{\bF_{(123),abc}\}$.
Within $\{\bF_{(123),iuv},  \; 1 \le i \le 3\}$, we say $\bF_{(123),iuv} < \bF_{(123),ju'v'}$
if $(uv)<(u'v')$  lexicographically, or when  $(uv)=(u'v')$ , $i<j$.
Within $\{\bF_{(123),abc}\}$,  we say $\bF_{(123),abc} < \bF_{(123),a'b'c'}$
if $(abc)<(a'b'c')$  lexicographically. This ordering is compatible with that of
$\cB^\mn$.


The purpose of $``$removing$"$ zero factors is achieved through sequential blowups based
upon factors of main binomials and their proper transforms. We break
 the sequential blowups into three types, named as $\vt$-, $\wp$-, and $\eth$-blowups;
 besides $``$removing$"$ zero factors,  each serves its own corresponding purpose.

\smallskip
{\it $\star$ On $\vt$-blowups.}  

 From the main binomial equations of \eqref{mainB-tour}, we select the following closed centers
 $$\cZ_\vt: (x_{iuv}=0) \cap (x_{(123, iuv)}=0),  i \in [3]; \;
  (x_{abc}=0) \cap (x_{(123, abc)}=0), a \ne b \ne c \in [n] \- [3].$$
  We order the sets $\{(uv)\}$ and  $\{(abc)\}$ lexicographically;
  we  order  $\{(iuv)\}$, written as  $\{(i, (uv)) \mid i \in [3]\}$, reverse-lexicographically.
  We then let $\{(iuv)\}$ go before  $\{(abc)\}$.
  This way,  the set $\cZ_\vt$ is equipped with a total order induced from the above orders on the indexes.

We then blow up $\sR_\sF$ along (the proper transforms of) the centers in $\cZ_\vt$, in the above order.
This gives rise to the sequence \eqref{vt-sequence-intro} in the introduction
$$\tsR_{\vt}:=\tsR_{\vt_{[\up]}}  \lra \cdots \lra \tsR_{\vt_{[k]}} \lra \tsR_{\vt_{[k-1]}} \lra \cdots \lra \tsR_{\vt_{[0]}}.$$

For any $k \in [\up]$, we let $\tsV_{\vt_{[k]}} \subset  \tsR_{\vt_{[k]}}$ 
be the proper transform of $\sV_\um$ in $\tsR_{\vt_{[k]}}$. We then set
$\tsV_{\vt}=\tsV_{\vt_{[\up]}}$ and $ \tsR_{\vt}= \tsR_{\vt_{[\up]}}$.


Besides removing the zero factors as displayed  in  the centers of $\cZ_\vt$,
 $\vt$-blowups also make the proper transforms of the residual binomial equations become
 dependent on the proper transforms of the main binomial equations on any standard chart.
 In particular, it also leads to 
 the conclusion $\tsV_{\vt} \cap X_{\vt, (\um, \uu_k)} = \emptyset$ for all $k \in [\up]$
 where $X_{\vt, (\um, \uu_k)}$ is the proper transform of the $\vr$-divisor 
$X_{(\um, \uu_k)}=(x_{(\um, \uu_k)}=0)$  
 
 Thus, upon completing $\vt$-blowups, we can discard all the residual binomials $\cB^\res$ from
 consideration.

\smallskip
{\it $\star$ On $\wp$-blowups.}

Here, we continue the process of $``$removing$"$ zero factors
of the proper transforms of the main binomials. From now on, we focus on each main binomial individually,
starting from the first one.

The first main binomial equation of \eqref{mainB-tour} is 
$$B_{145}:  x_{145}x_{(124,135)} - x_{124}x_{135} x_{(123,145)}.$$
The proper transforms of all the variables of $B_{145}$ may assume zero value on 
$\tsV_{\vt}$ except $x_{(123,145)}$ since 
$\tsV_{\vt}\cap X_{\vt, (123,145)} = \emptyset$, where  $X_{\vt, (123,145)}$ 
is the proper transform of  $X_{(123,145)}=( x_{(123,145)}=0)$ in $\tsR_\vt$.
For each and every term of $B_{145}$, we pick a $``$zero$"$ 
 factor to form a pair, but, {\it we do not pick any 
 $\vr$-variable.} We do not pick $x_{(123,145)}$ because 
$\tsV_{\vt}\cap X_{\vt, (123,145)} = \emptyset$; we do not pick 
$x_{(124,135)}$ for a good reason.
 Such a pair is called a $\wp$-set with respect to $B_{145}$. 
 Then, there are two such pairs.
\begin{equation} \label{1st-p-sets}
\phi_1=(x_{145}, x_{124}),  \;\;
\phi_2=(x_{145}, x_{135}). \end{equation} 
(The fact that there are only two $\wp$-sets for the first equation is an accident;
for general main binomial, the number of its corresponding $\wp$-sets may be huge.)
 The common vanishing loci of the variables in $\wp$-sets give rise to  the $\wp$-centers
 \begin{equation} \label{1st-p-centers}
Z_{\phi_1}=X_{\vt, 145} \cap X_{\vt, 124} ,\;
Z_{\phi_2}=X_{\vt, 145} \cap X_{\vt, 135}, \end{equation}
where $X_{\vt, \uu}$ is the proper transform of $X_\uu$ in $\tsR_\vt$.
We can then blow up $\tsR_\vt$ along (the proper transforms) of $Z_{\phi_1}$ and $Z_{\phi_2}$.

We then move on to the next main binomial equation
$$B_{245}:  x_{245}x_{(124,235)} - x_{124}x_{235} x_{(123,245)}.$$
Note that $x_{124}$ appears in $\phi_1$, hence
the minus term of the proper transform of $B_{245}$  acquires the exceptional parameter 
$\zeta$ created by the blowup along $Z_{\phi_1}$ through the variable $x_{124}$
($x_{124}$ either turns into the exceptional parameter $\zeta$ or acquires it).   
This is an additional $``$zero$"$ factor in
the minus term of the proper transform of $B_{245}$.
Then, for each and every  term of the proper transform of $B_{245}$, we pick a 
 factor, including exceptional parameters, to form a pair, again, {\it we do not pick any 
 $\vr$-variables.} Such a pair is called a $\wp$-set with respect to $B_{245}$.  
 The common vanishing loci of the variables in  $\wp$-sets 
 give rise to  $\wp$-centers in the previously obtained blowup scheme. 
 We can blow up  that scheme along (the proper transforms) of these centers.
 
 We then move on to $B_{345}$, repeat the above, and so on. 
 
 This gives rise to the sequential blowups \eqref{wp-sequence-intro} in the introduction
 $$\tsR_{\wp}  \to \cdots \to
\tsR_{(\wp_{(k\tau)}\fr_\mu\fs_{h})} \to \tsR_{(\wp_{(k\tau)}\fr_\mu\fs_{h-1})} \to \cdots \to \tsR_{\vt}.$$
An intermediate blowup scheme in the above 
is denoted by $\tsR_{(\wp_{(k\tau)}\fr_\mu\fs_h)}$. Here $(k\tau)$ is the index of a main binomial.
As the process of $\wp$-blowups goes on, more and more exceptional 
parameters may be acquired and appear in the proper transform of the later main binomial $B_{(k\tau)}$, 
resulting more pairs of zero factors, hence more corresponding $\wp$-sets
and $\wp$-centers. 
The existence of the index $\fr_\mu$, called {\it round $\mu$}, 
is due to the need to deal with the situation when 
an  exceptional parameter with exponent greater than one is accumulated
in the proper transform of the main binomial $B_{(k\tau)}$ (such a situation does not occur for the
first few main binomials).
The index $\fs_h$, called {\it step $h$}, simply indicates the corresponding step of the blowup.

Besides removing the zero factors,  
the reason that we exclude $\vr$-variables from $\wp$-sets is to help to control 
the binomial equations of quotient type.

 \smallskip

 {\it $\star$ On $\eth$-blowups.}
 
 Here, we finalize the process of $``$removing$"$ zero factors
of the proper transforms of all the main binomials.   
Like in the $\wp$-blowups, we focus on each main binomial relation individually,
starting from the first one.  The construction is analogous to that of $\wp$-blowup.

By induction, suppose we are now considering a main binomial $B_{(k\tau)}$.
 For each and every term of the proper transform of
  the main binomial $B_{(k\tau)}$ in the previously obtained blowup scheme, we pick a possible $``$zero$"$ 
 factor to form a pair. Here, we do not exclude any variable any more. 
 Such a pair is called $\eth$-sets with respect to $B_{(k\tau)}$.  They give rise to
 $\eth$-centers with respect to $B_{(k\tau)}$.
 The set of all $\eth$-center comes equipped with a total order.
 We then  blow up the previously obtained scheme
  along (the proper transforms) of these $\eth$-centers.

 This gives rise to the final sequential blowups  \eqref{hs-sequence-intro} in the introduction
$$\tsR_{\eth}  \to \cdots \to
\tsR_{(\eth_{(k\tau)}\fr_\mu\fs_{h})} \to \tsR_{(\eth_{(k\tau)}\fr_\mu\fs_{h-1})} \to \cdots \to \tsR_{\wp}.$$
Here again, the index $(k\tau)$ indicates the main binomial $B_{(k\tau)}$;
 the existence of the index $\fr_\mu$ is due to the need to deal with
the excessive accumulation of exceptional parameters;  $\fs_h$ simply indicates 
the corresponding step of the blowup.

In the above, the constructions of $\wp$- and $\eth$-blowups are discussed in terms of coordinate variables
 of the proper transforms of the main binomials on local charts.
 In the main text, the constructions of all these  $\wp$-  and $\eth$-blowups, like $\vt$-blowups, are done globally via induction.

 \smallskip
 (From the previous discussions, one sees that the process of $\wp$- and $\eth$-blowups is highly inefficient.
To provide a concrete example for the whole process,  $\Gr(2,n)$ would miss some main points;
$\Gr(3,6)$ would be too long to include, and also, perhaps not too helpful as far as 
showing (a resolution of) a singularity is concerned.)

\medskip
$\bullet$ {\sl  $\Ga$-schemes and their $\vt$-, $\wp$-, $\eth$-transforms.}

Fix any integral $\Ga$-scheme $Z_\Ga$, 
considered as a closed subscheme of $\rU_\um \cap \Gr^{3,E}$.
Our goal is to resolve $Z_\Ga$ when it is singular.

 As in the introduction,
  we have the following instrumental diagram \eqref{theDiagram2}.

The first two rows follow from the above discussion;
we  only need to explain the third and fourth rows.

  Here,  when  $Z_{\sF_{[j-1])}}$  (resp.  $\tZ_{\hs',\Ga}$) is not contained in the corresponding
  blowup center, $Z_{\sF_{[j])}}$  (resp.  $\tZ_{\hs,\Ga}$) is, roughly, obtained
  from  the proper transform
  of $Z_{\sF_{[j-1])}}$  (resp.  $\tZ_{\hs',\Ga}$). 
  When  $Z_{\sF_{[j-1])}}$  (resp.  $\tZ_{\hs',\Ga}$) is contained in the corresponding
  blowup center, then $Z_{\sF_{[j])}}$  (resp.  $\tZ_{\hs,\Ga}$) is, roughly,
  obtained from  a canonical rational slice
  of  the total  transform of $Z_{\sF_{[j-1])}}$  (resp.  $\tZ_{\hs',\Ga}$) under the morphism 
   $ \sV_{\sF_{[j]}} \to  \sV_{\sF_{[j-1]}}$ (resp. $\tsV_\hs \to \tsV_{\hs'}$) in the second row.
   Moreover, every $Z_{\sF_{[j])}}$  (resp.  $\tZ_{\hs,\Ga}$) admits explicit defining equations 
   over any standard affine chart of the corresponding smooth scheme in the first row.
Furthermore, in every case, $Z_{\sF_{[j])}}$  (resp.  $\tZ_{\hs,\Ga}$)  contains an irreducible
component  $Z^\dagger_{\sF_{[j]},\Ga}$  (resp. 
   $\tZ^\dagger_{\hs,\Ga}$) such that it maps onto $Z_\Ga$ birationally.

\begin{equation}\label{theDiagram2}
 \xymatrix@C-=0.4cm{
  \tsR_{\eth} \ar[r] & \cdots  \ar[r] &  \tsR_{\hs} \ar[r] &  \tsR_{\hs'} \ar[r] &  \cdots  \ar[r] &   \sR_{\sF_{[j]}}  \ar[r] &  \sR_{\sF_{[j-1]}} \cdots \ar[r] &  \rU_\um \\
    \tsV_{\eth} \ar @{^{(}->} [u]  \ar[r] & \cdots  \ar[r] &  \tsV_{\hs}\ar @{^{(}->} [u]   \ar[r] &  \tsV_{\hs'} \ar @{^{(}->} [u]  \ar[r] &  \cdots    \ar[r] &   \sV_{\sF_{[j]}} \ar @{^{(}->} [u]\ar[r] &  \sV_{\sF_{[j-1]}} \cdots \ar @{^{(}->} [u]  \ar[r] &  \rU_\um \cap \Gr^{3,E}   \ar @{^{(}->} [u]  \\
   \tZ_{\eth, \Ga} \ar @{^{(}->} [u]  \ar[r] & \cdots  \ar[r] &  \tZ_{\hs,\Ga}\ar @{^{(}->} [u]   \ar[r] &  \tZ_{\hs',\Ga} \ar @{^{(}->} [u]  \ar[r] &  \cdots    \ar[r] &   Z_{\sF_{[j]},\Ga} \ar @{^{(}->} [u]\ar[r] &  Z_{\sF_{[j-1])},\Ga} \cdots \ar @{^{(}->} [u]  \ar[r] &  Z_\Ga  \ar @{^{(}->} [u]  \\
    \tZ^\dagger_{\eth, \Ga} \ar @{^{(}->} [u]  \ar[r] & \cdots  \ar[r] &  \tZ^\dagger_{\hs,\Ga}\ar @{^{(}->} [u]   \ar[r] &  \tZ^\dagger_{\hs',\Ga} \ar @{^{(}->} [u]  \ar[r] &  \cdots    \ar[r] &   Z^\dagger_{\sF_{[j]},\Ga} \ar @{^{(}->} [u]\ar[r] &  Z^\dagger_{\sF_{[j-1])},\Ga} \cdots \ar @{^{(}->} [u]  \ar[r] &  Z_\Ga.  \ar[u]_{=}       }
\end{equation}
$\bullet$ {\sl Smoothness by Jacobian of 
main binomials  and linearized $\pl$ relations.}

We are now ready to explain the smoothness of 
 $\tZ_{\eth, \Ga}$ when $Z_\Ga$ is integral.
  We first investigate the smoothness of $\tsV_\eth$ which is a special case 
 of  $\tZ_{\eth, \Ga}$ when 
 $\Ga=\emptyset$.

The question is local. So we focus on an affine chart of $\fV$ of $\tsR_\eth$.
Proposition \ref{equas-p-k} provides key properties for 
the defining equations 
$\cB^\mn_\fV, L_{\fV, \sfm},  \cB^\q_\fV$ of $\tsV_\eth \cap \fV \subset \fV$.

As envisioned, we confirm that the scheme $\tsV_\eth$
 is smooth on the chart $\fV$ by some explicit calculations
and careful analysis on the Jacobian of
{\it the main binomial relations of $\cB^\mn_\fV$ and linearized $\pl$ relations of $L_{\fV, \sfm}$.}
This implies that on the chart $\fV$, 
 the main binomial relations of $\cB^\mn_\fV$ and the linearized $\pl$ relations 
 of $L_{\fV,\sfm}$
together generate the ideal of $\tsV_\eth \cap \fV$. Thus,  as a consequence, 
the binomials of quotient type  $\cB^q_\fV$ can be discarded from consideration.

Then, the similar calculations and analysis on the Jacobian of the induced main binomial relations of
$\cB^\mn_\fV$ and the induced linearized $\pl$ relations of $L_{\fV,\sfm}$ for 
$\tZ_{\eth,\Ga}$ implies that $\tZ_{\eth,\Ga}$  is smooth as well, on all charts $\fV$.
In particular, $\tZ^\dagger_{\eth,\Ga}$, now a connected component of $\tZ_{\eth,\Ga}$,  is  smooth, too.

This implies that $\tZ^\dagger_{\eth,\Ga} \lra Z_\Ga$ is a resolution, if $Z_\Ga$ is singular.

The above are done in \S \ref{main-statement}.

\smallskip
$\bullet$ {\sl Local resolution via Mn\"ev universality.} 

Upon reviewing Lafforgue's version of Mn\"ev universality,
we can apply the resolution $\tZ^\dagger_{\eth,\Ga} \lra Z_\Ga$ to obtain a local resolution for
any singular algebraic variety $X$ defined over a prime field. 
For a singular algebraic variety $X$ over a general perfect field $\kk$, 
we spread it out and deduce that $X/\kk$ admits local resolution as well.
This is done in \S \ref{local-resolution}.

\section{Primary $\pl$ Relations and De-homogenized  $\pl$-Ideal}\label{localization}

{\it The purpose of this section is to describe a minimal set of $\pl$ relations so that they
generate the $\pl$
ideal  for a given chart.  
The approach of this article depends on these explicit relations. 
The entire section is elementary.} 

Fix a pair of positive integers $n>1$ and $1 \le d <n$.
 In this section, we focus on Grassmannians $\Gr^{d,E}$ where
 $E=E_1 \oplus \cdots \oplus E_n$ is as introduced in the introduction.
 
 For application to resolution of singularity, it suffices to consider $\Gr^{3,E}$. However, 
 we choose to work on the general case of $\Gr^{d,E}$ 
 for the following two reasons.  (1) Working on $\Gr^{3,E}$ instead of  $\Gr^{d,E}$
 saves us little space or time: if we focus on 
 \eqref{rk0-and-1} but not the general form $\sum_{s \in S_F} x_{\uu_s}x_{\uv_s}$
 in the construction of $\vt$-, $\wp$-, and $\eth$-blowups,
then  the  proofs of some key propositions would have to
 be somewhat case by case,  less conceptual, and hence may be lengthier.
 However, it is  always good to frequently use the equations 
 of \eqref{rk0-and-1} and \eqref{mainB-tour}  as examples to help to
  understand the notations and the process.
  We caution here that replying only on $\pl$ equations 
  of the form $\bF_{(123),iuv},  \; 1 \le i \le 3$ from \eqref{rk0-and-1} 
 (they correspond to $\pl$ equations of $\Gr^{2,E}$)
  might miss some crucial points. 
 (2) As a convenient benefit, the results obtained and proofs provided
 for $\Gr^{d,E}$ here  can be directly cited in \cite{Hu20}.

 All the results of this section are elementary and some might have already been known. 
 Nonetheless,  the development in the current section is  instrumental for our approach.
 Hence, some good details are necessary.

We make a convention. Let $A$ be a finite set and $a \in A$. Then, we write
$$A \- a := A\-\{a\}.$$
Also, we use $|A|$ to denote the cardinality of the set A.

\subsection{$\pl$ relations} $\ $

Fix a pair of positive integers $(n,d)$ with $n >1$ and $1 \le d <n$. 
We denote the set $\{1,\cdots, n\}$ by $[n]$.
We let $\II_{d,n}$ be the set of all sequences of distinct integers $\{1 \le u_1 < \cdots < u_d \le n \}$.
An element of $\II_{d,n}$ is frequently written as $\uu=(u_1\cdots u_d)$.
We also regard an element of $\II_{d,n}$ as a subset of $d$ distinct integers in $[n]$.
For instance, for any $\uu, \um \in \II_{d,n}$, $\uu \- \um$ takes its set-theoretic meaning.
Also, $u \in [n] \- \uu$  if and only if $u \ne u_i$ for all $1 \le i \le d$.

As in the introduction, suppose we have a set of vector spaces, $E_1, \cdots, E_n$ such that 
every $E_\alpha$, $1 \le \alpha \le n$,  is of dimension 1 over $\kk$ (or, a free module of rank 1 over $\ZZ$), 
and, we let 
$$E:=E_1 \oplus \ldots \oplus E_n.$$ 

For any fixed  integer $1 \le d <n$, the Grassmannian, defined by
$$\Gr^{d,E}=\{ F \hookrightarrow E \mid \dim F=d\}, $$
is a projective variety defined over $\ZZ$.

We have the canonical decomposition
$$\wedge^d E=\bigoplus_{\ui =(i_1,\cdots, i_d)\in \II_{d,n}} E_{i_1}\otimes \cdots \otimes E_{i_d}.$$
This gives rise to the $\pl$ embedding of the Grassmannian:
$$\Gr^{d,E} \hookrightarrow \PP(\wedge^d E)=\{(p_\ui)_{\ui \in \II_{d,n}} \in \GG_m 
\backslash (\wedge^d E \- \{0\} )\},$$
$$F \lra [\wedge^d F],$$
 where $\GGm$ is the multiplicative group.

The group $(\GGm)^n/\GG_m$, where $\GGm$
 is embedded in $(\GGm)^n$ as the diagonal, acts on $\PP(\wedge^d E)$ by
 $${\bf t} \cdot p_{\ui} = t_{i_1} \cdots t_{i_d} p_{\ui}$$
where ${\bf t} = (t_1, \cdots, t_n)$ is (a representative of) an element of $(\GGm)^n/\GG_m$
and $\ui=(i_1, \cdots, i_d)$. This action leaves $\Gr^{d,E}$ invariant.
The $(\GGm)^n/\GG_m$-action on $\Gr^{d,E}$ will only be used in \S \ref{local-resolution}.

The Grassmannian $\Gr^{d,E}$ as a closed subscheme of $\PP(\wedge^d E)$ is
defined by a set of specific quadratic relations, called $\pl$ relations. We describe them below.

For narrative convenience, we will assume that
$p_{u_1\cdots u_d}$ is defined for any sequence of
$d$ distinct integers between 1 and $n$,  not necessarily listed in 
the sequential order of natural numbers,
subject to the relation
\begin{equation}\label{signConvention}
p_{\si(u_1)\cdots \si (u_d)}=\vsgn(\si) p_{u_1\cdots u_d}
\end{equation}
for any permutation $\si$ on the set $[n]$, 
where $\vsgn(\si)$ denotes the sign of the permutation.
Furthermore, also for convenience, we set 
\begin{equation}\label{zeroConvention}
 p_{\uu} := 0,
\end{equation}
for any $\uu=(u_1\cdots u_d)$ of a set of $d$  integers  in $[n]$ if
$u_i=u_j$ for some $1 \le i \ne j \le d$.

Now, for any pair $(\uh, \uk) \in \II_{d-1,n} \times \II_{d+1,n}$ with
$$\uh=\{h_1, \cdots, h_{d-1}\}  \;\;
 \hbox{and} \;\; \uk=\{k_1, \cdots, k_{d+1}\} ,$$
we have the Pl\"ucker relation:
\begin{equation} \label{pluckerEq}
F_{\uh,\uk}= \sum_{\lambda=1}^{d+1} (-1)^{\lambda-1} p_{h_1\cdots h_{d-1} k_\lambda } p_{k_1 \cdots  \overline{k_\lambda} \cdots k_{d+1}},
\end{equation}
where  $``\overline{k_\lambda}"$ means that $k_\lambda$ is deleted from the list.

{\it To make the presentation concise,
we frequently succinctly express a  general $\pl$ relation as
\begin{equation}\label{succinct-pl}
F= \sum_{s \in S_F} \vsgn(s) p_{\uu_s}p_{\uv_s},
\end{equation}
 for some index set $S_F$, with $\uu_s, \uv_s \in \II_{d,n}$, where
$ \vsgn(s) $ is the $\pm$ sign associated with the quadratic  monomial term $p_{\uu_s}p_{\uv_s}$.
We note here that $\vsgn(s)$ depends on 
how every of ${\uu_s}$ and ${\uv_s}$ is presented, per the convention \eqref{signConvention}.
}

\begin{defn}\label{ftF}
Consider any $\pl$ relation $F=F_{\uh,\uk}$ for some pair
 $(\uh, \uk) \in \II_{d-1,n} \times \II_{d+1,n}$.
We let $\ft_{F}+1$ be the number of terms in $F$. We then define the rank of
$F$ to be $\ft_{F}-2$. We denote this number by $\rk (F)$.
\end{defn}
The integer $\ft_{F}$, as defined above, will be frequently used throughout.

\begin{example}\label{exam:(3,6)}
Consider the Grassmannian $\Gr(3,6)$. Then, the $\pl$ relation
$$F_{(16), (3456)}:  p_{163}p_{456} - p_{164}p_{356} + p_{165}p_{346}$$
is of rank zero; the $\pl$ relation
$$F_{(12), (3456)}: p_{123}p_{456} - p_{124}p_{356} + p_{125}p_{346}- p_{126}p_{345}$$
is of rank one. 
\end{example}

Let $\ZZ[p_\ui]_{\ui \in \II_{d,n}}$ be the homogeneous coordinate ring
 of the $\pl$ projective space $\PP(\wedge^d E)$ and $I_\wp \subset \ZZ[p_\ui]_{\ui \in \II_{d,n}}$ 
be the homogeneous ideal generated by all the $\pl$ relations  \eqref{pluckerEq} or
\eqref{succinct-pl}.
We call $I_\wp$ the $\pl$ ideal for the Grassmannian $\Gr^{d,E}$.
Then, the graded quotient ring $\ZZ[p_\uu]_{\uu \in \II_{d,n}}/I_\wp$
is the homogeneous coordinate ring of $\Gr^{d,E}$, called the Grassmannian algebra.

\subsection{Primary $\pl$ equations with respect to a fixed affine chart} $\ $

In this subsection, we focus on a fixed affine chart of the $\pl$ projective space
$\PP(\wedge^d E)$.

Fix any $\um \in \II_{d,n}$. In $\PP(\wedge^d E)$,
we let $$\rU_\um:=(p_\um \equiv 1)$$  stand for the open chart
 defined by $p_\um \ne 0$.  Then,
the affine space $\rU_\um$  comes equipped with the 
coordinate variables $x_\uu=p_\uu/p_\um$ for all $\uu \in \II_{d,n} \- \um$.
In practical calculations, we will simply set $p_\um =1$, whence the notation
$(p_\um \equiv 1)$ for the chart. We let
$$\rU_\um (\Gr) = \rU_\um \cap \Gr^{d, E}$$
be the corresponding induced open chart  of $\Gr^{d, E}$.

The chart $\rU_\um (\Gr)$ is canonically an affine space.
Below, we explicitly describe $$\up:={n \choose d} -1- d(n-d) $$
many specific  $\pl$ relations with respect to the chart $\rU_\um$, called the $\um$-primary 
$\pl$ relations, such that their restrictions to the chart $\rU_\um$ define
$\rU_\um (\Gr)$ as a closed subscheme of the affine space $\rU_\um$.

To this end, we write  $\um=(m_1 \cdots m_d)$. We set 
$$\II^\um_{d,n}=\{\uu \in \II_{d,n} \mid |\uu \- \um| \ge 2\} \subset \II_{d,n}$$
where $\uu$ and  $\um$ are also regarded as subsets of integers, and $|\uu \- \um|$ denotes the cardinality of $\uu \- \um$. 
In  words,  $\uu \in \II^\um_{d,n}$ if and only if
$\uu=(u_1,\cdots, u_d)$ contains at least two elements distinct from elements in  
$\um=(m_1, \cdots, m_d)$. 
It is helpful to write explicitly the set $\II_{d,n} \- \II_{d,n}^\um$:
$$\II_{d,n} \- \II_{d,n}^\um =\{\um\} \cup \{ \{u\} \cup (\um\setminus m_i) \mid \; 
\hbox{for all $ u \in [n]\- \um$ and all $1 \le i \le d$} \}, $$
where $u \notin \um$ if and only if $u \ne m_i$ for any $1 \le i \le d$.
Then, one calculates and finds
$$|\II^\um_{d,n}|=\up={n \choose d} -1- d(n-d) ,$$
where $|\II^\um_{d,n}|$ denotes the cardinality of $\II^\um_{d,n}$.

Further, let $\ua=(a_1\cdots a_k)$ be a list of some elements of $[n]$, not necessarily mutually distinct,
 for some $k<n$.
We will write $$v \ua=v(a_1\cdots a_k)=(v a_1\cdots a_k)
\;\; \hbox{and} \;\;  \ua  v=(a_1\cdots a_k)v=(a_1\cdots a_k v),$$
each is considered as a list of some elements of $[n]$, 
for any $v \in [n] \- \ua$.

Now, take any element $\uu =(u_1,\cdots, u_d) \in \II_{d,n}^\um$.
We let $u_0$ denote the smallest integer in $\uu \- \um$. 
We then set 
$$ \uh=\uu \setminus u_0 \;\; \hbox{and} \;\;  \uk=(u_0 m_1 \cdots m_d),$$
where $\uu \setminus u_0 =\uu \setminus \{u_0\}$ and $\uu$ is regarded as a set of integers.

 This gives rise
to the $\pl$ relation $F_{\uh,\uk}$, taking of the following form
\begin{equation} \label{keyTrick}
F_{\uh,\uk}=p_{(\uu \setminus u_0)u_0} p_\um -  p_{(\uu \setminus u_0) m_1} p_{u_0 (\um \setminus m_1 )}
+\cdots + (-1)^d p_{(\uu \setminus u_0) m_d} p_{u_0 (\um \setminus m_d)},
\end{equation}
where $\um \setminus m_i = \um \setminus \{m_i\}$ and $\um$ is regarded as a set of integers,
for all $i \in [d]$.

We give a new notation for this particular equation: we denote it by
\begin{equation} \label{keyTrick2}
F_{\um, \uu}=p_{(\uu \setminus u_0)u_0} p_\um + \sum_{i=1}^d (-1)^i  
p_{(\uu \setminus u_0)m_i} p_{u_0 (\um \setminus m_i )},
\end{equation} because it only depends on $\um$ and $\uu \in \II_{d,n}^\um$. 
To simplify the notation, we introduce
$$\uu^r= \uu \setminus u_0, \;\; \widehat{\um_i} = \um\setminus m_i, \;\; \hbox{for all $i \in [d]$.}$$
Then, \eqref{keyTrick2} becomes
\begin{equation} \label{keyTrick4}
F_{\um, \uu}=p_{\um}p_{\uu^r u_0} +
\sum_{i=1}^d (-1)^i p_{ \uu^rm_i } p_{u_0 \widehat{\um_i} }.
\end{equation}

We point out here that $\uu$ and $\uu^r u_0$ may differ by a permutation.

\begin{defn}\label{lt-ltvar}
We call the $\pl$ equation $F_{\um, \uu}$ of \eqref{keyTrick4} a primary $\pl$ equation
for the chart $\rU_\um=(p_\um \equiv 1)$. We also say $F_{\um,\uu}$ is $\um$-primary.
The term $p_\um p_{\uu}$ is called the leading term of $F_{\um, \uu}$.
\end{defn}

(One should not confuse  $F_{\um, \uu}$ with the expression of a general $\pl$ equation
$F_{\uh,\uk}$: we have $(\um, \uu)\in \II_{d,n}^2$ for the former and 
$(\uh,\uk) \in \II_{d-1,n} \times \II_{d+1,n}$ for the latter.)

One sees that the correspondence between $\II_{d,n}^\um$ and 
the set of $\um$-primary $\pl$ equations is a bijection.

\subsection{De-homogenized $\pl$ ideal  with respect to a fixed affine chart} $\ $

Following the previous subsection, we continue to fix an element $\um \in \II_{d,n}$ 
and  focus on the chart $\rU_\um$ of $\PP(\wedge^d E)$.

We will write  $\II_{d,n} \- \um$ for  $\II_{d,n} \- \{\um\}$.

Given any $\uu \in \II^\um_{d,n}$, by \eqref{keyTrick4}, 
it gives rise to the $\um$-primary equation 
$$ F_{\um, \uu}=p_{\um}p_{\uu^r u_0} +
\sum_{i=1}^d (-1)^i p_{\uu^r m_i } p_{ u_0 \widehat{\um_i} }.$$
 If we set $p_\um =1$ and let $x_\uw=p_\uw$, for all $\uw \in \II_{d,n} \- \um$,
 then it  becomes 
\begin{equation}\label{equ:localized-uu}
\bF_{\um, \uu}=x_{\uu^r u_0} +
\sum_{i=1}^d (-1)^i  x_{\uu^r m_i} x_{ u_0 \widehat{\um_i} }.
\end{equation}

\begin{defn}\label{localized-primary} We call the relation \eqref{equ:localized-uu} 
 the de-homogenized (or the localized)
$\um$-primary $\pl$ relation corresponding to $\uu \in \II_{d,n}^\um$.
We call the unique distinguished variable, $x_{\uu}$ 
(which may differ $x_{\uu^r u_0}$ by a sign), the leading variable
of the  de-homogenized $\pl$ relation $\bF_{\um, \uu}$. 
\end{defn}

Throughout  this paper, we often express an $\um$-primary $\pl$ equation $F$ as
\begin{equation}\label{the-form-F}
F=\sum_{s \in S_F} \vsgn (s) p_{\uu_s} p_{\uv_s}= \vsgn (s_F) p_\um p_{\uu_{s_F}}+
\sum_{s \in S_F\- s_F} \vsgn (s) p_{\uu_s} p_{\uv_s}
\end{equation}
where $s_F$ is the index for the leading term of $F$, and $S_F\- s_F:=S_F\-\{s_F\}$.
Then, upon setting $p_\um=1$ and letting $x_\uw=p_\uw$ for all $\uw \in \II_{d,n}\- \um$,
 we can write the corresponding de-homogenized  $\um$-primary $\pl$ equation $\bF$ as
\begin{equation}\label{the-form-LF}
\bF=\sum_{s \in S_F} \vsgn (s) x_{\uu_s} x_{\uv_s}= \vsgn (s_F) x_{\uu_F}+
\sum_{s \in S_F\- s_F} \vsgn (s) x_{\uu_s} x_{\uv_s}
\end{equation}
where  $x_{\uu_F}:=x_{\uu_{s_F}}$ is the leading variable of $\bF$.

\begin{defn}\label{ft-bF}
Let $F$ be an $\um$-primary $\pl$ relation and $\bF$ its 
de-homogenization with respect to the chart $\rU_\um$.
We set $\ft_{\bF}=\ft_F$  and $\rk (\bF) =\rk (F)$.
\end{defn}

For any $\uu \in \II_{d,n} \- \um$,  we let $x_\uu=p_\uu/p_\um$ for all $\uu \in \II_{d,n} \- \um$.
Then, we can identify the coordinate ring of $\rU_\um$ 
with $\kk [x_\uu]_{\uu \in \II_{d,n} \- \um}$.
We let $I_{\wp,\um}$ be the ideal of $\kk [x_\uu]_{\uu \in \II_{d,n} \- \um}$
obtained from the $\pl$ ideal $I_\wp$ be setting $p_\um=1$ and letting 
$x_\uu=p_\uu$ for all $\uu \in \II_{d,n} \- \um$.
We call  $I_{\wp,\um}$ the de-homogenized $\pl$ ideal on the chart $\rU_\um$.

\begin{defn} For any  $\uu \in \II_{d,n}^\um$, we define 
the $\um$-rank of $\uu$ (resp.  $x_\uu$) to be the rank of its corresponding
primary $\pl$ equation $F_{\um,\uu}$. 
\end{defn}

\begin{prop}\label{primary-generate} The de-homogenized $\pl$ ideal $I_{\wp,\um}$ is generated by
$$\sfm:=\{ \bF_{\um, \uu} \mid \uu \in \II_{d,n}^\um \}.$$
 Consequently, the chart $\rU_\um (\Gr)=\rU_\um \cap \Gr^{d,E}$ comes equipped with
 the set of free variables
 $$\var_{\rU_\um}:=\{x_\uu \mid \uu \in \II_{d,n} \-\{ \um\} \- \II_{d,n}^\um\}$$
 and is canonically isomorphic to the affine space with  the above variables as its coordinate variables.
\end{prop}
\begin{proof}
(This proposition is elementary; 
it serves as the initial check of an induction for  some later proposition;
we provide sufficient details for completeness.)

It suffices to observe that for any  $\uu \in \II_{d,n}^\um$, 
its corresponding de-homogenized $\pl$ 
primary $\pl$ equation $\bF_{\um, \uu}$ is equivalent to an expression of
the leading variable $x_\uu$ as a polynomial in  the free variables of
$\var_{\rU_\um}$.  For instance, one can check this by induction on the $\um$-rank, $\rk (\uu)$, 
of $\uu$, as follows.

Suppose $\rk (\uu)= 0$. Then, up to a permutation, we may write
$$\uu= (\um \setminus \{m_i m_j \}) vu$$
where $m_i, m_j \in \um$ for some $1 \le i, j, \le d$, and  $ u< v  \notin \um$. Then, we have
\begin{equation}\label{form-of-rk0}
\bF_{\uu, u}:  x_{\uu}  + (-1)^i x_{ \uu^r m_i} x_{u \widehat{m_i}}
+(-1)^j x_{ \uu^r m_j} x_{u \widehat{m_j}} ,
\end{equation}
where $\uu^r =(\um \setminus \{m_i m_j \})v$.
One sees that $x_{ \uu^r m_i}, \;  x_{u \widehat{m_i}} , \;  x_{ \uu^r m_j}$
and $x_{u \widehat{m_j}}$ belong to $\var_{\rU_\um}$. Hence, the statement holds.

Now suppose that $\rk (\uu)>0$.
Using (\ref{equ:localized-uu}),  we have
$$\bF_{\um, \uu}: \; x_{\uu^r u_0}+
\sum_{i=1}^d (-1)^i x_{\uu^r m_i } x_{u_0 \widehat{m_i} }.$$
Note that all variables $x_{u_0 \widehat{m_i} }, i \in [d]$, belong to $\var_{\rU_\um}$.
Note also that $$\rk ({\uu^r m_i }) = \rk (\uu) -1,$$ provided that
 $p_{\uu^r m_i }$ is not identically zero, that is,  it is a well-defined $\pl$ variable
 (see \eqref{zeroConvention}).
 Thus,  applying the inductive assumption, any such $x_{\uu^r m_i }$ is a polynomial in 
 the variables of $\var_{\rU_\um}$.
 Therefore,  $\bF_{\um, \uu}$, 
   is equivalent to an expression of $x_\uu$ as a polynomial in 
 the variables of $\var_{\rU_\um}$.

Let $J$ be the ideal of $\kk [x_\uu]_{\uu \in \II_{d,n}\- \um}$ generated by
 $\{ \bF_{\um, \uu} \mid \uu \in \II_{d,n}^\um \}$ and let
 $V(J)$ the subscheme of $\rU_\um$ defined by $J$.
By the above discussion,  $V(J)$ is canonically isomorphic to
 the affine space of dimension $d(n-d)$ with  
  the variables of $\var_{\rU_\um}$ as its coordinate variables.
Since $\rU_\um(\Gr) \subset V(J)$, we conclude $\rU_\um (\Gr)=V(J)$.
\end{proof}

\begin{defn}\label{basic} We call the variables in
$$\var_{\rU_\um}:=\{x_\uu \mid \uu \in \II_{d,n} \- \um \- \II_{d,n}^\um\}$$
the $\um$-basic $\pl$ variables. When $\um$ is fixed and clear from the context,
we just call them basic variables.
\end{defn}
Only non-basic $\pl$ variables correspond to $\um$-primary $\pl$ equations.

Observe that for all $\pl$ relations $F$, we have $0 \le {\rm rank} ( F) \le d-2$.
Hence, for any $0 \le r \le d-2$, we let
$$\sF^r_\um = \{\bF_{\um,\uu}  \mid {\rm rank} (F_{\um,\uu}) =r,\; \uu \in \II_{d,n}^\um\}.$$
Then, we have
$$\sfm=\bigcup_{0 \le r \le d-2} \sF^r_\um.$$

Then, one observes the following easy but useful fact.

\begin{prop}\label{leadingTerm} Fix  any $0 \le r \le d-2$
 any $\uu \in \II_{d,n}^\um$ with ${\rm rank}_\um (F_{\um, \uu})=r$.
 Then,  the leading variable 
$x_\uu$ of $\bF_{\um, \uu}$ does not appear in any relation in
$$\sF^0_\um \cup \cdots \cup \sF^{r-1}_\um \cup (\sF^r_\um \setminus \bF_{\um, \uu}).$$
\end{prop}

To close this subsection, we raise  a concrete question.
 Fix the chart $(p_\um \equiv 1)$.
 In  $\kk[x_\uu]_{\uu \in \II_{d,n}\- \um}$,
according to  Proposition \ref{primary-generate}, 
 every de-homogenized $\pl$ equation
$\bF_{\uh,\uk}$ on the chart $\rU_\um$ can be expressed
as a polynomial in the  de-homogenized primary $\pl$ relations $\bF_{\um, \uu}$ 
with $\uu \in \II_{d,n}^\um$. 
 It may be useful  in practice to find such an expression explicitly for an arbitrary $F_{\uh,\uk}$.
For example, for the case of $\Gr(2,5)$, this can be done as follows.

\begin{example}\label{pl2-5} For $\Gr(2,5)$, we have five $\pl$ relations:
$$F_1= p_{12}p_{34}-p_{13}p_{24} + p_{14}p_{23},\;
F_2= p_{12}p_{35}-p_{13}p_{25} + p_{15}p_{23}, \; $$
$$F_3= p_{12}p_{45}-p_{14}p_{25} + p_{15}p_{24},\;
F_4= p_{13}p_{45}-p_{14}p_{35} + p_{15}p_{34}, \;$$
$$F_5= p_{23}p_{45}-p_{24}p_{35} + p_{25}p_{34}. $$
On the chart $(p_{45} \equiv 1)$,
$F_3, F_4,$ and $ F_5$ are primary. 
One calculates and finds
$$p_{45} F_1 = p_{34} F_3 -p_{24} F_4 + p_{14} F_5, $$ 
$$p_{45} F_2 = p_{35} F_3 -p_{25} F_4 + p_{15} F_5.$$
In addition,  the  Jacobian of the de-homogenized $\pl$ equations of $\bF_3, \bF_4, \bF_5$ with respect to 
all the variables,
$x_{12}, x_{14}, x_{15} ,x_{13}, x_{35} , x_{34},x_{23} ,x_{24} ,x_{25}, $
is given by
$$
\left(
\begin{array}{cccccccccc}
1 & x_{25} & x_{24} & 0  & 0 & 0 & 0 & 0 & 0 \\
0 & 0 & 0 & 1  & x_{14} & x_{15}& 0 & 0 & 0 \\
0& 0 & 0 & 0  & 0 & 0 & 1 & x_{35} & x_{34} \\
\end{array}
\right).
$$
There, one sees visibly  a $(3 \times 3)$  identity minor. 
\end{example}

 



\subsection{Ordering the set of all primary $\pl$ equations}\label{order-eq} $\ $

Fix $\um \in \II_{d,n}$. We consider the set
$$\sfm =\{ \bF_{\um, \uu} \mid \uu \in \II_{d,n}^\um  \}$$
as in Proposition \ref{primary-generate}. As in the paragraph above
Proposition \ref{leadingTerm}, we have
$$\sF_\um=\bigcup_{1 \le r \le d-2} \sF_\um^r,$$ 
where $\sF^r_\um = \{\bF_{\um,\uu}  \mid {\rm rank} (\bF_{\um,\uu}) =r,\; \uu \in \II_{d,n}^\um\}$
for all $0 \le r \le d-2$.

We will provide a total order on the set $\sfm$. This ordering will be fixed and used throughout.

We first provide a partial order $\prec_\wp$ on the set $\sfm$.  

\begin{defn}\label{key-partial} Let $\bF \in \sF^i_\um$ and $\bG \in \sF^j_\um$.  Then, we say
$$\bF \prec_\wp \bG \;\; \hbox{if} \;\; i < j.$$ 
This gives rise to the partially ordered  set $(\sF_\um, \prec_\wp)$. 
\end{defn}

In what follows, we extend $\prec_\wp$ to make $\sfm$ a totally ordered set.

\begin{defn}\label{gen-order} Let $K$ be any fixed totally ordered finite set, 
with its order denoted by $<$.
Consider any two subsets $\eta \subset K$ and ${ \zeta} \subset K$ with the cardinality $n$
for some positive integer $n$.
We write $\eta=(\eta_1,\cdots,\eta_n)$ 
(respectively, of ${\zeta}=(\zeta_1,\cdots,\zeta_n)$) as an array according to the ordering of $K$.
We say $\eta <_{\lex} { \zeta}$ if the left most nonzero number in the vector $\eta-{\zeta}$ is negative,
or more explicitly,  if we can express
$$\eta=\{t_{1}< \cdots <t_{r-1} <s_r< \cdots \}$$
$${\zeta}=\{t_{1}< \cdots <t_{r-1}<t_r  < \cdots \}$$
such that $s_r< t_r$ for some integer $r \ge 1$.  We call $<_\lex$ the lexicographic order
induced by $(K, <)$.

Likewise, we say $\eta <_{\invlex} { \zeta}$
 if the right most nonzero number in the vector $\eta-{ \zeta}$ is negative,
or more explicitly, if we can express
$$\eta=\{\cdots <s_r< t_{r+1}< \cdots <t_n\}$$
$${\zeta}=\{\cdots <t_r  < t_{r+1}< \cdots <t_n\}$$
such that $s_r< t_r$ for some integer $r \ge 1$. 
We call $<_\invlex$ the reverselexicographic order
induced by $(K, <)$. 
\end{defn}

This definition can be applied to the set
$$\II_{d,n} =\{ (i_1 < i_2 < \cdots <i_d) \; \mid \; 1 \le  i_\mu \le n,\; \forall \; 1 \le \mu \le d\}$$
for all $d$ and $n$.
Thus, we have equipped the set $\II_{d,n}$ 
with both the  lexicographic ordering $``<_\lex "$ and
the reverse lexicographic ordering $``<_\invlex "$. 

\begin{defn}\label{invlex}
Consider any $\uu, \uv \in \II_{d,n} \- \um$. 

Suppose $\uu=\widehat{m_i} u$ and $\uv=\widehat{m_j} v$ are two elements of 
$\II_{d,n}\- \um \- \II_{d,n}^\um$,
for some $m_i, m_j \in \um$ and $u, v \in [n] \- \um$. We say 
$$\uu <_\wp \uv$$ if $u < v$ or when $u=v$, 
$\widehat{m_i} <_\lex \widehat{m_j}$.

Suppose $\uu$ and $\uv$ are two  elements of 
$\II_{d,n}^\um$. We say 
$$\uu <_\wp \uv$$ if one of the following three holds:
\begin{itemize}
\item ${\rm rank}_\um \; \uu < {\rm rank}_\um \; \uv$;
\item ${\rm rank}_\um \; \uu = {\rm rank}_\um \; \uv$, $\uu \- \um<_\lex \uv \- \um$;
\item ${\rm rank}_\um \; \uu ={\rm rank}_\um \; \uv$, $\uu \- \um = \uv\- \um$, and
$\um \cap \uu  <_\lex \um \cap \uv$.
\end{itemize}
\end{defn}

\begin{defn}\label{cFi-partial-order} Consider any two $\pl$ variables $x_\uu$ and $x_\uv$.
We say 
$$\hbox{$x_\uu <_\wp x_\uv$ if $\uu<_\wp \uv$.}$$

Consider any two distinct primary equations, 
$\bF_{\um,\uu}, \bF_{\um,\uv} \in \sF^i_\um$ 
of the same rank $i$ for some $0 \le i \le d-2$, with $\uu \ne \uv$. We say 
$$\bF_{\um,\uu} <_\wp \bF_{\um,\uv}\; \; \hbox{if} \;\;  \uu <_\wp \uv.$$
\end{defn} 

The above together with Definition \ref{key-partial} provides an induced total order
on the set $\sfm$.  We denote  the totally ordered set  by $(\sfm, <_\wp)$.
Hence, we can write
$$\sfm=\{\bF_1 <_\wp \cdots <_\wp \bF_\up\}.$$

In what follows, when comparing two $\pl$ variables $x_\uu$ and $x_\uv$
or  two $\um$-primary $\pl$ equations, we exclusively use $<_\wp$.
Thus, throughout, for simplicity, we will simply write $<$ for $<_\wp$.
A confusion is unlikely.





\section{A Singular Local Birational Model $\sV$ for $\Gr^{d,E}$}\label{singular-model}

{\it The purpose of this section is to establish a local model $\sV$, birational to
$\Gr^{d,E}$, such that  all
terms of all $\um$-primary $\pl$ equations can be separated in  the
defining main binomial relations of $\sV$ in a smooth ambient scheme $\sR_\sF$.}

\subsection{The construction  of $\sV \subset \sR_\sF$} $\ $

Consider the fixed affine chart $\rU_\um$ of
the $\pl$ projective space $\PP(\wedge^d E)$.
For any $\bF \in \sfm$, 
written as $F=\sum_{s \in S_F} \vsgn (s) p_{\uu_s}p_{\uv_s}$, 
we let $\PP_F$ be the projective space with 
homogeneous coordinates written as $[x_{(\uu_s, \uv_s)}]_{s \in S_F}$. For convenience, we let
\begin{equation}\label{LaF} \La_F=\{(\uu_s, \uv_s) \mid s \in S_F\}.
\end{equation}
This is an index set for  the homogeneous coordinates of the projective
space $\PP_F$.  To avoid duplication, we make a convention: 
\begin{equation}\label{uv=vu}
x_{(\uu_s, \uv_s)}=x_{(\uv_s, \uu_s)}, \; \forall \; s \in S_{F}, \; \forall \; \bF \in \sfm.
\end{equation}
If we write $(\uu_s, \uv_s)$ in the lexicographical order, i.e., we insist $\uu_s <_\lex \uv_s$,
then the ambiguity is automatically avoided. Howerer, the convention is still  useful.

\begin{defn}
We call $x_{(\uu_s, \uv_s)}$ a $\vr$-variable of $\PP_F$, or simply a $\vr$-variable.
To distinguish, we call a $\pl$ variable, $x_\uu$ with $\uu \in \II_{d,n} \- \um$, a $\vp$-variable.
\end{defn}

Fix $k \in [\up]$. We introduce the natural rational map
\begin{equation}\label{theta-k}
 \xymatrix{
\Theta_{[k]}: 
\PP(\wedge^d E) \ar @{-->}[r]  & \prod_{i \in [k]} \PP_{F_i}  } 
\end{equation}
$$
 [p_\uu]_{\uu \in \II_{d,n}} \lra  
\prod_{i \in [k]}  [p_\uu p_\uv]_{(\uu,\uv) \in \La_{F_i}}
$$   
where $[p_\uu]_{\uu \in \II_{d,n}} $ is the homogeneous $\pl$ coordinates of a point in $
\PP(\wedge^d E)$. When restricting $\Theta_{[k]}$ to $\rU_\um$,
it gives rise to 
\begin{equation}\label{bar-theta-k}
 \xymatrix{
\bar\Theta_{[k]}: \rU_\um  \ar @{-->}[r]  & \prod_{i \in [k]} \PP_{F_i}  } 
\end{equation}

We let \begin{equation}\label{tA}
 \xymatrix{
\PP_{\sF_{[k]}}  \ar @{^{(}->}[r]  & \PP(\wedge^d E) \times  \prod_{i \in [k]} \PP_{F_i}  
 }
\end{equation}
be the closure of the graph of the rational map $\Theta_{[k]}$, and
 \begin{equation}\label{bar-tA}
 \xymatrix{
\rU_{\um,\sF_{[k]}}  \ar @{^{(}->}[r]  & \sR_{\sF_{[k]}}:= \rU_\um \times  \prod_{i \in [k]} \PP_{F_i}  
 }
\end{equation}
be the closure of the graph of the rational map $\bar\Theta_{[k]}$.

\begin{defn}
Fix any $ k \in [\up]$.
We let $$R_{[k]}=\kk[p_\uu; x_{(\uv_s, \uu_s)}]_{\uu \in \II_{d,n} , s \in S_{F_i},  i \in [k]}.$$
A  polynomial $f \in R_{[k]}$ is called multi-homogeneous if it is homogenous in 
$[p_\uu]_{\uu \in \II_{d,n}}$ and is homogenous 
in $[x_{(\uv_s, \uu_s)}]_{s \in S_{F_i}}$, for every $i \in [k]$.
A multi-homogeneous polynomial $f \in R_{[k]}$
 is $\vr$-linear if it is linear in $[x_{(\uv_s, \uu_s)}]_{s \in S_{F_i}}$, 
 whenever it contains some $\vr$-variables of $\PP_{F_i}$,
 for  any $i \in [k]$.
 \end{defn}

We set $R_0:=\kk[p_\uu]_{\uu \in \II_{d,n}}$. 
Then, corresponding to 
the embedding \eqref{tA}, 
there exists a degree two homomorphism
\begin{equation}\label{vik}
\vi_{[k]}: \; R_{[k]} \lra R_0 , \;\;\; x_{(\uu_s,\uv_s)} \to p_{\uu_s} p_{\uv_s}
\end{equation} 
for all  $s \in S_{F_i}, \; i \in [k]$.    

We then let
\begin{equation}\label{bar-vik}
\bar\vi_{[k]}: \; R_{[k]} \lra R_0, \;\;\; x_{(\uu_s,\uv_s)} \to x_{\uu_s} x_{\uv_s} 
\end{equation} 
$$$$ 
for all  $s \in S_{F_i}, i \in [k]$,
be the de-homogenization of $\vi_{[k]}$ with respect to the chart $\rU_\um=(p_\um \equiv 1)$.
This corresponds to the embedding \eqref{bar-tA}.

We are mainly interested in the case when $k=\up$.
 Hence, we set 
$$R:=R_{[\up]},\;\; \vi:=\vi_{[\up]}, \;\; \bar\vi:=\bar\vi_{[\up]}.$$

We let $\ker^\mh \vi_{[k]}$ (resp. $\ker^\mh \bar\vi_{[k]}$)
denote the set of all  multi-homogeneous polynomials
 in $\ker  \vi_{[k]}$ (resp. $\ker \bar\vi_{[k]}$).


\begin{lemma}\label{defined-by-ker} 
The scheme
$\PP_{\sF_{[k]}}$, as a closed subscheme of 
 $\PP(\wedge^d E) \times  \prod_{i \in [k]} \PP_{F_i}$,
 is defined by $\ker^\mh \vi_{[k]}$.
In particular, the scheme
$\rU_{\um,\sF_{[k]}}$, as a closed subscheme of 
 $\sR_{\sF_{[k]}}= \rU_\um \times  \prod_{i \in [k]} \PP_{F_i}$,
 is defined by $\ker^\mh \bar\vi_{[k]}$. 
\end{lemma}
\begin{proof} This is immediate.
\end{proof}

We need to investigate $\ker^\mh \vi_{[k]}$.

Consider any $f \in \ker^\mh \vi_{[k]}$. We express it as the sum of its monic monomials
(monomials with constant coefficients 1)
$$f= \sum \bm_i.$$
We have $\vi_{[k]} (f)=\sum \vi_{[k]} (\bm_i)=0$ in $R_0$. Thus, the set of 
the monic monomials $\{\bm_i\}$ can be grouped into minimal groups to form partial sums
of $f$ so that {\it the images of elements of each group are 
 identical} and the image of the partial sum of each minimal group equals 0 in $R_0$.
When ch.$\kk=0$, this means each minimal group consists of a pair $(\bm_i, \bm_j)$
and its partial sum
is the difference $\bm_i -\bm_j$.
When ch.$\kk=p>0$ for some prime number $p$, this means each minimal group 
 consists of either  (1): a pair $(\bm_i, \bm_j)$ and $\bm_i -\bm_j$ is a partial sum of $f$;
or (2):   exactly $p$ elements $\bm_{i_1}, \cdots,\bm_{i_p}$
and $\bm_{i_1}+ \cdots + \bm_{i_p}$ is a partial sum of $f$.
But, the relation $\bm_{i_1}+ \cdots + \bm_{i_p}$ is always generated by 
the relations $\bm_{i_a} -\bm_{i_b}$ for all $1 \le a, b \le p$.

Thus, regardless of the characteristic of the field $\kk$, it suffices to consider binomials
$\bm -\bm' \in \ker^\mh \vi_{[k]}$.

\begin{example}\label{exam:Bq} Consider $\Gr^{3,E}$. 
Then, the following binomials belong to $\ker^\mh \vi_{[k]}$
for any fixed $k \in [\up]$.

Fix $a,b,c \in [k]$, all being distinct.
\begin{eqnarray}
x_{(12a,13b)}x_{(13a,12c)}x_{(12b,13c)} \nonumber \\
-x_{(13a,12b)}x_{(12a,13c)}x_{(13b,12c)} \label{rk0-0} 
\end{eqnarray}
Fix $a,b,c, \bar a, \bar b, \bar c \in [k]$, all being distinct.
\begin{eqnarray}
\;\; \;x_{(12a,3bc)}x_{(13a,2\bar b \bar c)}x_{(13 \bar a,2bc)} x_{(12 \bar a,3 \bar b \bar c)} \nonumber \\
-x_{(13a,2bc)}x_{(12a,3\bar b \bar c)}x_{(12 \bar a,3bc)} x_{(13 \bar a,2 \bar b \bar c)},  \label{rk1-1}
\end{eqnarray}
Fix $a,b,c, a', \bar a, \bar b, \bar c \in [k]$, all being distinct.
\begin{eqnarray}
\;\; \; x_{(12a,13a')}x_{(13a,2bc)}x_{(12a',3\bar b \bar c)}x_{(12 \bar a,3bc)} x_{(13 \bar a,2 \bar b \bar c)} 
\nonumber \\
- x_{(13a,12a')} x_{(12a,3bc)}x_{(13a',2\bar b \bar c)}x_{(13 \bar a,2bc)} x_{(12 \bar a, 3 \bar b \bar c)}.\label{rk0-1}
\end{eqnarray} 
These binomials are arranged so that one sees visibly the matching for multi-homogeneity.
\end{example}

\begin{lemma} \label{trivialB}
Fix any $i \in [k]$. We have
\begin{equation}\label{tildeBk}
p_{\uu'}p_{\uv'}x_{(\uu,\uv)} - p_\uu p_\uv x_{(\uu',\uv')} \in \ker^\mh \vi_{[k]}.
\end{equation}
where $x_{(\uu,\uv)}, x_{(\uu',\uv')}$ are any two distinct $\vr$-variables of $\PP_{F_i}$.
Likewise, we have
\begin{equation}\label{trivialBk}
x_{\uu'}x_{\uv'}x_{(\uu,\uv)} - x_\uu x_\uv x_{(\uu',\uv')} \in \ker^\mh \bar\vi_{[k]}.
\end{equation}
\end{lemma}
\begin{proof} This is trivial.
\end{proof}

\smallskip

Let $\AA^l$ (resp. $\PP^l$) be the affine (resp. projective)
space of dimension $l$ for some positive integer $l$ with
coordinate variables $(x_1,\cdots, x_l)$ (resp. with
homogeneous coordinates $[x_1,\cdots, x_l]$).
A monomial $\bf m$ is {\it square-free} if 
$x^2$ does not divide $\bf m$ for every coordinate variable $x$ in the affine space.  
A polynomial is square-free if all of its monomials are
square-free.

\smallskip

For any $\bm -\bm' \in \ker^\mh \vi_{[k]}$, we define 
$ \deg_{\vr} (\bm -\bm')$ to be the total degree of $\bm$ (equivalently, $\bm'$)
in $\vr$-variables of $R_{[k]}$.

For any $F=\sum_{s \in S_F} p_{\uu_s} p_{\uv_s}$ with $\bF \in \sfm$ 
and $s \in S_F$, we write $X_s=x_{(\uu_s,\uv_s)}$.

Recall  that we have set $\vi=\vi_{[\up]}: R=R_{[\up]} \to R_0$.

 Observe here that  for any nonzero binomial $\bm -\bm' \in \ker^\mh \vi_{[k]}$,
we automatically have $ \deg_\vr (\bm -\bm') > 0$, since $\vi_{[k]}$ restricts
to the identity on $R_0$.

\begin{lemma}\label{ker-phi-k}
Consider a binomial $ \bm -\bm' \in \ker^\mh \vi_{[k]}$ with $ \deg_\vr (\bm -\bm') > 0$. 
 We let $h$ be the maximal common factor of the two monomials $\bm$ and $\bm'$
 in $R_{[k]}$. Then, we have 
$$\hbox{ $\bm=h \prod_{i=1}^\ell \bm_i$ and $\bm'=h \prod_{i=1}^\ell \bm'_i$}$$ 
for some positive integer $\ell$ such that
 for every $i \in [\ell]$,
$\bm_i -\bm'_i \in \ker^\mh \vi_{[k]}$ and is of  the following form:
\begin{equation}\label{1st-Hq}
\vi (X_1)X_2- \vi(X'_1 )X'_2 ,
 \end{equation}
 where  every of $X_1,X_2,  X'_1,$ and $X'_2$ is a monomial of $R$ in $\vr$-variables only
 (i.e., without $\vp$-variables; here we allow $X_1=X_1'=1$)  such that 
\begin{enumerate}
 \item $X_1X_2-  X'_1 X'_2 \in \ker^\mh \vi$ and  is $\vr$-linear;
 \item $\vi(X_1X_2)$ (equivalently, $\vi(X'_1 X'_2)$) is a square-free monomial;
\item for any $\bF \in \sfm$ and $s \in S_F$,
suppose $x_{\uu_s} x_{\uv_s} $ divides $\bm$ (resp. $\bm'$), then 
 $X_s =x_{(\uu_s,\uv_s)}$ divides   $X_1$ (resp. $X_1'$) in one of the relations of \eqref{1st-Hq}. 
 \end{enumerate}
\end{lemma}
\begin{proof} 
We prove by induction on $\deg_{\vr} (\bm -\bm')$.

Suppose $\deg_{\vr}(\bm-\bm') =1$. 

Then, we can write $$\bm -\bm' = f x_{(\uu,\uv)} -g  x_{(\uu',\uv')} $$ 
for some $f, g \in R_0$,
and two $\vr$-variables of $\PP_{F_i}$, $x_{(\uu,\uv)}$ and $x_{(\uu',\uv')}$ 
for some $i \in [k]$.  
If $x_{(\uu,\uv)} =x_{(\uu',\uv')}$, then one sees that $f=g$ and $\bm -\bm'=0$.
Hence, we assume that  $x_{(\uu,\uv)} \ne x_{(\uu',\uv')}$. Then, we have
$$f p_{\uu} p_{\uv}=g  p_{\uu'} p_{\uv'}.$$
Because  $x_{(\uu,\uv)}$ and $ x_{(\uu',\uv')}$
are two distinct $\vr$-variables of $\PP_{F_i}$,  one checks from the definition that the two sets
$$\{ p_{\uu}, p_{\uv} \}, \; \{p_{\uu'}, p_{\uv'} \}$$ are disjoint. 
Consequently, 
$$  p_{\uu}p_{\uv} \mid g, \;\; p_{\uu'} p_{\uv'}  \mid f.$$
Write $$ g=g_1 p_{\uu}p_{\uv} , \;\;  f=f_1 p_{\uu'} p_{\uv'} .$$
Then we have 
$$ p_{\uu}p_{\uv} p_{\uu'} p_{\uv'}  (f_1-g_1)=0 \in R_0.$$
Hence, $f_1=g_1$. Then, we have
$$\bm -\bm' =h ( p_{\uu'} p_{\uv'}   x_{(\uu,\uv)} -  p_{\uu}p_{\uv}  x_{(\uu',\uv')}) $$ 
where $h:=f_1=g_1$.
This implies the statement, in this case.
Observe that in such a case, we have that
 $\bm -\bm'$ is generated by the relations of \eqref{tildeBk}. 

Suppose Lemma \ref{ker-phi-k} holds  for  $\deg_{\vr} < e$ for some positive integer $e >1$.

Consider $\deg_{\vr}(\bm-\bm')=e$.

By the multi-homogeneity of $(\bm -\bm')$, we can write 
\begin{equation}\label{forXs} 
\bm -\bm' = \bn X_s - \bn' X_t 
\end{equation}
such that $X_s$ and  $X_t$ 
are the $\vr$-variables of $\PP_{F_i}$  corresponding to 
some $s, t \in S_{F_i}$ for some $i \in [k]$, and $\bn, \bn' \in R_{[k]}$.

If $s=t$, then $\bm -\bm' = X_s (\bn - \bn') $. Hence, the statement follows from
the inductive assumption applied to $(\bn -\bn') \in  \ker^\mh \vi_{[k]}$
since $\deg_\vr (\bn -\bn')=e-1$.


We suppose now $s \ne t$.
Let $\bar f=\bn x_{\uu_s} x_{\uv_s} - \bn' x_{\uu_t} x_{\uv_t}$.  Then,
 $\bar f \in \ker^\mh \vi_{[k]}$. 
 
First, we suppose $\bar f=0$. 

Then, $x_{\uu_s} x_{\uv_s} \mid \bn'$
 and $x_{\uu_t} x_{\uv_t} \mid \bn$. Hence, we can write
 $$\bn'=x_{\uu_s} x_{\uv_s} \bn'_0, \; \bn=x_{\uu_t} x_{\uv_t} \bn_0.$$
And we have,
$$\bar f= \bn x_{\uu_s} x_{\uv_s} - \bn' x_{\uu_t} x_{\uv_t}
=x_{\uu_s} x_{\uv_s} x_{\uu_t} x_{\uv_t}  (\bn_0 - \bn'_0).$$
 Hence, one sees that $ \bn_0 - \bn'_0 \in \ker^\mh \vi_{[k]}$.
 If $\bn_0 - \bn'_0=0$, then $h=\bn_0=\bn'_0$ is
 the maximal common factor of $\bm$ and $\bm'$, and further,
 $$\bm -\bm' = h( x_{\uu_t} x_{\uv_t}  X_s -  x_{\uu_s} x_{\uv_s} X_t).$$
  In such a case, the statement of the lemma holds.
 
  Hence, we can assume that $0 \ne \bn_0 - \bn'_0 \in \ker^\mh \vi_{[k]}$,
  in particular, this implies that $ \deg_\vr (\bn_0 -\bn'_0) > 0$.
  Then,   we have 
  $$\bm -\bm' = \bn_0 x_{\uu_t} x_{\uv_t}  X_s - \bn'_0 x_{\uu_s} x_{\uv_s} X_t.$$
   Observe here that 
$x_{\uu_t} x_{\uv_t}  X_s -  x_{\uu_s} x_{\uv_s} X_t$ is in the form of
\eqref{1st-Hq}, verifying the conditions (1) - (3).
 Thus, in such a case, the statement of the lemma follows by applying
the inductive assumption  to $(\bn_0 -\bn'_0) \in  \ker^\mh \vi_{[k]}$
since $\deg_\vr (\bn_0 -\bn'_0)=e-1$.

 Next, we suppose $\bar f \ne 0$.  
 
 As $\deg_{\vr} \bar f <e$, by the inductive assumption,  
 we can write
$$\bn x_{\uu_s} x_{\uv_s}= h (x_{\uu_s} x_{\uv_s} \bn_s) \prod_{i=1}^\ell \bn_i,\;\
\bn' x_{\uu_t} x_{\uv_t}=h  (x_{\uu_t} x_{\uv_t}\bn_t) \prod_{j=1}^\ell \bn'_{j}$$
for some integer $\ell \ge 1$, with $\bn_0=x_{\uu_s} x_{\uv_s}\bn_s$
and $\bn_0'=(x_{\uu_t} x_{\uv_t}\bn_t)$
such that  for each $0 \le i \le \ell$, it determines (matches) a unique $0 \le i' \le \ell$
such that $(\bn_i - \bn_{i'}')$ is of the form of \eqref{1st-Hq}  and
verifies all the properties (1) - (3)  of  Lemma \ref{ker-phi-k}.
Consider $\bn_0=x_{\uu_s} x_{\uv_s} \bn_s$. It matches $\bn_{0'}'$. By
the multi-homogeneity of $\bn_0 - \bn_{0'}'$, 
\eqref{1st-Hq} and  (1) of  Lemma \ref{ker-phi-k},
 we can write $\bn_{0'}' = x_{\uu_{t'}} x_{\uv_{t'}}\bn_{t'}$
for some $t' \in S_{F_i}$ and $\bn_{t'} \in R_{[k]}$.
 Therefore, by switching $t$ with $t'$ if $t \ne t'$,
and re-run the above arguments, without loss of generality, we can assume $t'=t$
and $\bn_0=x_{\uu_s} x_{\uv_s}\bn_s$ matches
 $\bn_0'=(x_{\uu_t} x_{\uv_t}\bn_t)$. Further, by re-indexing $\{\bn_j' \mid j \in [\ell]\}$ if necessary,
we can assume that  $\bn_i$ matches $\bn_i'$ for all $1 \le i \le l$.

Now, note that we have
$\bn = h (\bn_s) \prod_{i=1}^\ell \bn_i, \; \bn' =h  (\bn_t') \prod_{i=1}^\ell \bn_i$. Hence
$$\bm = h (\bn_s X_s) \prod_{i=1}^\ell \bn_i, \;\; \bm' =h  (\bn_t' X_t) \prod_{i=1}^\ell \bn_i.$$
We let $\bm_0=\bn_s X_s$ and  $\bm_i = \bn_i$ for all $i \in [\ell]$;
 $\bm_0'=(n_t X_t)$ and  $\bm_i= \bn_i'$ for all $i \in [\ell]$.
Then, one checks directly that
  Lemma \ref{ker-phi-k} holds for $\bm -\bm'$.

This proves the lemma.
\end{proof}

\begin{defn}\label{hatB}
Let $\widehat{\cB}_{[k]}$
be the set of  all binomial relations of \eqref{1st-Hq} that verify Lemma \ref{ker-phi-k} (1) - (3);
 let $\widetilde{\cB}_{[k]}$ 
the de-homogenizations 
with respect to $(p_\um \equiv 1)$ of all binomial relations of $\widehat{\cB}_{[k]}$.
\end{defn}

\begin{cor}\label{cB-generate}
The ideal  $\ker^\mh \vi_{[k]}$ is generated by $\widehat{\cB}_{[k]}$.
Consequently, the ideal  $\ker^\mh \bar\vi_{[k]}$ is generated by $\widetilde{\cB}_{[k]}$.
\end{cor}
\begin{proof} Take any binomial
$(\bm -\bm') \in \ker^\mh \vi_{[k]}$
 with $\deg_\vr (\bm -\bm') > 0$. We express, by Lemma \ref{ker-phi-k},
$$\bm -\bm'=h (\prod_{i=1}^\ell \bm_i -  \prod_{i=1}^\ell \bm_i')$$ 
such that $\bm_i -\bm'_i \in \widehat{\cB}_{[k]}$  for all $i \in [\ell]$.
Then, we have
$$\bm -\bm'=h (\prod_{i=1}^{\ell} \bm_i - \bm_\ell'  \prod_{i=1}^{\ell-1} \bm_i
+ \bm_\ell'  \prod_{i=1}^{\ell-1} \bm_i -\prod_{i=1}^{\ell} \bm_i')$$ 
$$=h ((\bm_\ell - \bm_\ell')  \prod_{i=1}^{\ell-1} \bm_i
+ \bm_\ell'  (\prod_{i=1}^{\ell-1} \bm_i -\prod_{i=1}^{\ell-1} \bm_i')).$$ 
Thus, by a simple induction on the integer $\ell$, the corollary follows.
\end{proof}

We now let  $\bar\Theta_{[k],\Gr}$ be the restriction of  $\bar\Theta_{[k]}$
to  $\rU_\um(\Gr)=\rU_\um\cap \Gr^{d,E}$: 
\begin{equation}\label{theta-k-2}
 \xymatrix{
\bar\Theta_{[k], \Gr}: \rU_\um(\Gr)\ar @{-->}[r]  & \prod_{i \in [k]} \PP_{F_i}  } 
\end{equation}
$$
 [x_\uu]_{\uu \in \II_{d,n}} \lra  
\prod_{i \in [k]}  [x_\uu x_\uv]_{(\uu,\uv) \in \La_{F_i}}.
$$   

We let \begin{equation}\label{tA'}
 \xymatrix{
\sV_{\um, \sF_{[k]}}  \ar @{^{(}->}[r]  & \rU_\um(\Gr) \times  \prod_{i \in [k]} \PP_{F_i}  
\ar @{^{(}->}[r]  & \sR_{\sF_{[k]}}
 }
\end{equation}
be the closure of the graph of the rational map $\bar\Theta_{[k],\Gr}$.

Then, one sees that $\sV_{\um, \sF_{[k]}}$ is the proper transform 
of $\rU_\um(\Gr)$ in $\rU_{\um,\sF_{[k]}}$ under the birational morphism 
$\rU_{\um,\sF_{[k]}} \lra \rU_\um.$

Since we focus on the chart $\rU_\um$, below we write 
$\sV_{\sF_{[k]}}=\sV_{\um, \sF_{[k]}}$.

By construction, there exists the natural forgetful map
\begin{equation}\label{forgetfulMap}
\sR_{\sF_{[k]}}  \lra \sR_{\sF_{[k-1]}} \end{equation}  
 and it  induces a birational morphsim 
  \begin{equation}\label{rho-sFk} \rho_{\sF_{[k]}}: \sV_{\sF_{[k]}} \lra \sV_{\sF_{[k-1]}}.
  \end{equation}

\begin{defn}\label{defn:pre-q} 
We let $\cB_{i}$  (resp.  $\cB_{[k]}$) be the set of all binomial relations in
 \eqref{trivialBk} for any fixed $i \in [k]$ (resp. for all $i \in [k]$).
We set $\cB^\pq_{[k]}= \widetilde{\cB}_{[k]} \- \cB_{[k]}$.
An element of $\cB^\pq_{[k]}$ is called a binomial  of pre-quotient type.
\end{defn}

\begin{lemma}\label{equas-for-sVk}
 The scheme $\sV_{\sF_{[k]}}$, as a closed subscheme of
$\sR_{\sF_{[k]}}= \rU_\um \times  \prod_{i=1}^k \PP_{F_i} $,
is defined by the following relations
\begin{eqnarray}
\;\;\;\;\;\;\; B_{F_i,(s,t)}: \;\;\;  x_{(\uu_s, \uv_s)}x_{\uu_t} x_{\uv_t}- x_{(\uu_t,\uv_t)}   x_{\uu_s} x_{\uv_s},
\; \forall \;\; s, t \in S_{F_i} \- s_{F_i},    \; i \in [k], \label{eq-Bres-lek'}\\
B_{F_i,(s_{F_i},s)}: \;\; x_{(\uu_s, \uv_s)}x_{\uu_{F_i}} - x_{(\um,\uu_{F_i})}   x_{\uu_s} x_{\uv_s}, \;\;
\forall \;\; s \in S_{F_i} \- s_{F_i},  \; i \in [k], \label{eq-B-lek'}  \\
\cB^\pq_{[k]}, \;\; \;\; \;\; \;\; \;\; \;\; \;\; \;\; \;\;\; \;\; \;\; \;\; \;\; \;\; \;\; \;\; \;\; \;\; \;\; \;\; \label{eq-hq-lek}\\
L_{F_i}: \;\; \sum_{s \in S_{F_i}} \vsgn (s) x_{(\uu_s,\uv_s)},   \; i \in [k],  \label{linear-pl-lek'} \\
\bF_j: \;\; \sum_{s \in S_{F_j}} \vsgn (s) x_{\uu_s}x_{\uv_s}, \; \;   
j \in [\up] \label{linear-pl-gek}
\end{eqnarray}
where $\bF_i$ is 
expressed as $\vsgn (s_{F_i}) x_{\uu_{F_i}} +\sum_{s \in S_{F_i} \- s_{F_i}} \vsgn (s) x_{\uu_s}x_{\uv_s}$
for every $i \in [k]$.
\end{lemma}
\begin{proof} 
By  Lemma \ref{defined-by-ker} and Corollary \ref{cB-generate},
the scheme  $\rU_{\um,\sF_{[k]}}$, 
as a closed subscheme of $\sR_{\sF_{[k]}}$ 
is defined by the  relations in \eqref{eq-Bres-lek'},  \eqref{eq-B-lek'},
and \eqref{eq-hq-lek}. Thus, $\sV_{\sF_{[k]}}$, being the proper transform of $\rU_\um(\Gr)$, 
as a closed subscheme of $\sR_{\sF_{[k]}}$, 
is defined by the  relations in \eqref{eq-Bres-lek'},  \eqref{eq-B-lek'},
and \eqref{eq-hq-lek} together with the proper transforms of the 
de-homogenized $\pl$ equations $\bF =\sum_{s \in S_F} \vsgn (s) x_{\uu_s}x_{\uv_s}$ 
for all $\bF \in \sF_\um$.

It suffices to show that under the presence of \eqref{eq-Bres-lek'} and  \eqref{eq-B-lek'}, 
$L_{F_i}$ and $\bF_i$  imply each other for all $i \in [k]$.

Fix any $i \in [k]$. Take any $s \in S_{F_i}$. 
Consider the binomial relations of $\cB_i$
  \begin{equation}\label{Buv-s-1st}
  x_{(\uu, \uv)}x_{\uu_s} x_{\uv_s} - x_{\uu} x_{\uv} x_{(\uu_s,\uv_s)}, 
  \end{equation}
for all $(\uu, \uv) \in \La_{F_i}$ (cf. \eqref{LaF}). By multiplying $\vsgn (s)$ to  \eqref{Buv-s-1st}
and adding together all the resulted binomials,  
we obtain, \begin{equation}\label{Fi=Li-1st}
  x_{\uu_s} x_{\uv_s} L_{F_i} =x_{(\uu_s,\uv_s)} \bF_i \;, \; \mod (\langle \cB_i \rangle),
 \end{equation}
 where $\langle \cB_i \rangle$ is the ideal generated by the relations in $\cB_i$.
 As neither of $x_{\uu_s} x_{\uv_s}$  and $x_{(\uu_s,\uv_s)}$ belong to the ideal of $\sV_{\sF_{[k]}}$,
 and  $\sV_{\sF_{[k]}}$ is integral, we see that $L_{F_i}$ and $\bF_i$  imply each other for all $i \in [k]$.
\end{proof}

For conciseness, we set the following
$$\sV_{\um}:=\sV_{\um,\sF_{[\up]}}, \; \rU_{\um, \sF}:=\rU_{\um, \sF_{[\up]}},
\; \sR_\sF:=\sR_{\sF_{[\up]}}.$$
Then, we have the following 
diagram 
$$ \xymatrix{
\sV_{\um} \ar[d] \ar @{^{(}->}[r]  & \rU_{\um,\sF} \ar[d] \ar @{^{(}->}[r]  &
\sR_\sF= \rU_\um  \times \prod_{\bF \in \sfm} \PP_F \ar[d] \\
\rU_\um(\Gr)  \ar @{^{(}->}[r]  & \rU_\um  \ar @{=}[r]  & \rU_\um.
}
$$

 In what follows, we will sometimes write
$\sV$ for $\sV_\um$, as we will exclusively focus on the chart $\rU_\um$, throughout, unless otherwise stated.

We also set
$$\cB^\pq=\cB^\pq_{[\up]}.$$

By the case of Lemma \ref{equas-for-sVk} when $k=\up$, we have

\begin{cor}\label{eq-tA-for-sV}  The scheme $\sV_\um$, as a closed subscheme of
$\sR_\sF= \rU_\um \times  \prod_{\bF \in \sF_\um} \PP_F$,
is defined by the following relations
\begin{eqnarray} 
B_{F,(s,t)}: \;\; x_{(\uu_s, \uv_s)}x_{\uu_t}x_{ \uv_t}-x_{(\uu_t, \uv_t)}x_{\uu_s}x_{ \uv_s}, \;\; \forall \;\; 
s, t \in S_F \- s_F  \label{eq-Bres}\\
B_{F,(s_F,s)}: \;\;\;\;\; x_{(\uu_s, \uv_s)}x_{\uu_F} - x_{(\um,\uu_F)}   x_{\uu_s} x_{\uv_s}, \;\;
\forall \;\; s \in S_F \- s_F,  \label{eq-B} \\ 
\cB^\pq,   \;\; \;\; \;\; \;\; \;\; \;\; \;\; \;\; \;\; \;\; \;\; \;\; \;\; \label{eq-hq}\\
L_F: \;\; \sum_{s \in S_F} \vsgn (s) x_{(\uu_s,\uv_s)}, \label{linear-pl}\\
\bF: \;\; \sum_{s \in S_{F}} \vsgn (s) x_{\uu_s}x_{\uv_s}, \; \;   
\end{eqnarray}
for all $\bF \in \sfm$
with $\bF$ being expressed as $\vsgn (s_F) x_{\uu_F} +\sum_{s \in S_F \- s_F} \vsgn (s) x_{\uu_s}x_{\uv_s}$.
\end{cor}

\begin{defn}\label{defn:main-res} A binomial equation $B_{F,(s_F,s)}$ of \eqref{eq-B}
with $s \in S_F \- s_F$ is called a main binomial equation. We let
$$\cB^\mn_F=\{B_{F,(s_F,s)} \mid s \in S_F \- s_F  \}
\;\; \and \;\; \cB^\mn=\sqcup_{\bF \in \sfm} \cB^\mn_F.$$
A binomial equation $B_{F,(s,t)}$ of \eqref{eq-Bres}
with $s, t \in S_F \- s_F$ and $s \ne t$ is called a residual binomial equation. We let
$$\cB^\res_F=\{B_{F,(s,t )} \mid s, t \in S_F \- s_F  \}
\;\; \and \;\;  \cB^\res=\sqcup_{\bF \in \sfm} \cB^\res_F.$$
Recall that an element of  $\cB^\pq$ is called a binomial relation of pre-quotient type.
\end{defn}

\begin{defn} \label{defn:linear-pl} Given any $\bF \in \sfm$, written as 
$\bF =\sum_{s \in S_F} \vsgn (s) x_{\uu_s}x_{\uv_s}$,
the linear equation $L_F$ as in \eqref{linear-pl}
$$L_F: \;\; \sum_{s \in S_F} \vsgn (s) x_{(\uu_s,\uv_s)} $$ 
is called the linearized $\pl$ equation with respect to $\bF$ (or $F$).
\end{defn}

We observe here that
\begin{equation}\label{dim}
\dim (\prod_{\bF \in \sfm} \PP_F) = \sum_{\bF \in \sfm} |S_F\- s_F| 
= \sum_{\bF \in \sfm} |\cB_F^\mn|=|\cB^\mn|,
\end{equation}
where $|K|$ denotes the cardinality of a finite set $K$.

\subsection{$\vp$-divisors, $\vr$-divisors, and the standard charts of $\sR_\sF$} $\ $

From earlier, we have the set
$\La_F=\{(\uu_s, \uv_s) \mid s \in S_F\}.$
This is an index set for  the homogeneous coordinates of the projective
space $\PP_F$, and  is also an index set for all the variables that appear in 
the linearized $\pl$ equation $\bL_F$ of \eqref{linear-pl}. To be used later, 
 we also set  $$\La_{\sF_\um} =\sqcup_{\bF \in \sF_\um} \La_F.$$

\begin{defn} \label{vp-divisor}
Consider the scheme $\sR_\sF =\rU_\um \times  \prod_{\bF \in \sfm} \PP_F$.
Recall that the affine chart $\rU_\um$ comes equipped with the coordinate variables 
$\{x_\uu\}_{\uu \in \II_{d,n}\- \um}$.
For any $\uu \in \II_{d,n}\- \um$, we set
$$X_\uu:=(x_\uu =0) \subset \sR_\sF.$$
We call $X_\uu$ the $\pl$ divisor, in short, the $\vp$-divisor,  of $\sR_\sF$ associated with $\uu$.
We let $\cD_\vp$ be the set of all $\vp$-divisors on the scheme $\sR_\sF$.
\end{defn}

\begin{defn} \label{vr-divisor}
In addition to the $\vp$-divisors,  the scheme 
$\sR_\sF=\rU_\um \times  \prod_{\bF \in \sfm} \PP_F$ also 
comes equipped with the divisors
$$X_{(\uu, \uv)}:=(x_{(\uu, \uv)}=0)$$
for all $(\uu,\uv) \in \Lambda_{\sF_\um}$.
We call $X_{(\uu, \uv)}$ the $\vr$-divisor corresponding to $(\uu, \uv)$.  
We let $\cD_{\vr}$ be the set  of all $\vr$-divisors of $\sR_\sF$. 
\end{defn}

\begin{defn}\label{fv-k=0} Fix $k \in [\up]$.
For every $\bF_i \in \sF_\um$ with $i \in [k]$, choose and fix an arbitrary element 
$s_{F_i, o}\in S_{F_i}$.
Then, the scheme $\sR_{\sF_{[k]}}$ is covered by the affine open charts
of the form  $$\rU_\um \times \prod_{i \in [k]}
(x_{(\uu_{s_{F_i, o}}, \uv_{s_{F_i, o}})} \equiv 1)
\subset \sR_{\sF_{[k]}}= \rU_\um \times  \prod_{i \in [k]} \PP_{F_i} .$$ 
We call such an affine open subset a standard chart of
$\sR_{\sF_{[k]}}$, often denoted by $\fV$.
\end{defn}
Fix any standard chart $\fV$ as above. We let 
$$\fV'=\rU_\um \times \prod_{i \in [k-1]}
(x_{(\uu_{s_{F_i, o}}, \uv_{s_{F_i, o}})} \equiv 1)
\subset \sR_{\sF_{[k-1]}}= \rU_\um \times  \prod_{i \in [k-1]} \PP_{F_i} .$$ 
Then, this is a standard chart of $\sR_{\sF_{[k-1]}}$, uniquely determined by $\fV$.
We say $\fV$ lies over $\fV'$. In general, 
suppose $\fV''$ is a standard chart of $\sR_{\sF_{[j]}}$ with $j <k-1$. Via induction,
we say $\fV$ lies over $\fV''$ if $\fV'$ lies over $\fV''$.

Note that the standard chart $\fV$ of
$\sR_{\sF_{[k]}}$ in the above definition is uniquely indexed 
by the set 
\begin{equation}\label{index-sR}
 \La_{\sF_{[k]}}^o=\{(\uu_{s_{F_i,o}},\uv_{s_{F_i,o}}) \in \La_{F_i} \mid i \in [k] \}.
 \end{equation}
Given $ \La_{\sF_{[k]}}^o$, we let 
$$ \La_{\sF_{[k]}}^\star=(\bigcup_{i \in [k]}\La_{F_i}) \- \La_{\sF_{[k]}}^o.$$
We set $\La_\sfm^o:=\La_{\sF_{[\up]}}^o$ and $\La_\sfm^\star:=\La_{\sF_{[\up]}}^\star$.

To be cited as the initial cases of certain inductions later on,
we need the following two propositions.

\begin{prop}\label{meaning-of-var-p-k=0} Consider any standard
 chart $$\fV=\rU_\um \times \prod_{i \in [k]}
(x_{(\uu_{s_{F_i, o}}, \uv_{s_{F_i, o}})} \equiv 1)$$ of $\sR_{\sF_{[k]}}$, 
indexed by  $\La_{\sF_{[k]}}^o$ as above.
It  comes equipped with the set of free variables
$$\var_\fV=\{x_{\fV, \uw}, \; x_{\fV, (\uu,\uv)} \mid \uw \in \II_{d,n} \- \um, \; 
(\uu,\uv) \in \La_{\sF_{[k]}}^\star
\}$$ and  is isomorphic to the affine space with the variables in $\var_\fV$
as its coordinate variables.
Moreover, on the standard chart $\fV$, we have
\begin{enumerate}
\item the divisor  $X_{ \uw}\cap \fV$ is defined by $(x_{\fV,\uw}=0)$ for every 
$\uw \in \II_{d,n} \setminus \um$;
\item the divisor  $X_{(\uu,\uv)}\cap \fV$ is defined by $(x_{\fV,(\uu,\uv)}=0)$ for every 
$(\uu,\uv) \in \La_{\sF_{[k]}}^\star.$     
\end{enumerate}
\end{prop}
\begin{proof} 
Recall  that $\rU_\um=(p_\um \equiv 1)$.
Then, we let $x_{\fV, \uw}=x_\uw$ for all  $\uw \in \II_{d,n} \- \um$.
Now consider every 
$i \in [k]$.
Upon setting $x_{(\uu_{s_{F_i, 0}}, \uv_{s_{F_i, 0}})} \equiv 1$, we let
 $x_{\fV, (\uu_s,\uv_s)} = x_{(\uu_s,\uv_s)}$ be the de-homogenization of $x_{(\uu_s,\uv_s)}$
 for all  $s  \in S_{F_i} \- s_{F_i,o}$.
From here,  the statement is straightforward to check.
\end{proof}

\begin{prop}\label{equas-fV[k]}  Let the notation be as in
Propsotion \ref{meaning-of-var-p-k=0}. Then,
the scheme $\sV_{\sF_{[k]}} \cap \fV$, as a closed subscheme of $\fV$
is defined by the following relations
\begin{eqnarray}
\;\;\;\;\; x_{\fV,(\uu_s, \uv_s)}x_{\fV,\uu_t} x_{\fV,\uv_t}- x_{\fV, (\uu_t,\uv_t)}   x_{\fV,\uu_s} x_{\fV, \uv_s},
\; \forall \;\; s, t \in S_{F_i} \- s_{F_i},    \; i \in [k], \label{eq-Bres-lek'-2}\\
\;\; \;\; x_{\fV, (\uu_s, \uv_s)}x_{\fV,\uu_{F_i}} - x_{\fV,(\um,\uu_{F_i})}   
x_{\fV,\uu_s} x_{\fV,\uv_s}, \;\;
\forall \;\; s \in S_{F_i} \- s_{F_i},  \; i \in [k], \label{eq-B-lek'-2}  \\
\cB^\pq_{\fV, [k]},  \;\; \;\; \;\; \;\; \;\; \;\;\; \;\; \;\; \;\; \;\; \;\; \;\; \;\; \;\; \;\; \;\; \;\; \;\; \;\; 
\label{eq-hq-lek-2}\\
L_{\fV, F_i}: \;\; \sum_{s \in S_{F_i}} \vsgn (s) x_{\fV, (\uu_s,\uv_s)},   \; i \in [k],  \label{linear-pl-lek'-2} \\
\bF_{\fV, j}: \;\; \sum_{s \in S_{F_j}} 
\vsgn (s) x_{\fV,\uu_s}x_{\fV,\uv_s}, \; \; k < j \le \up.  \label{linear-pl-gek-2}
\end{eqnarray}
where  the equations of $\cB^\pq_{\fV, [k]}$ are the de-homogenizations of the equations
of $\cB^\pq_{[k]}$.
\end{prop}
\begin{proof} 
By applying Lemma \ref{equas-for-sVk}, 
it suffices to show that $\bF_{\fV,i}$ is redundant on the chart $\fV$ for all $i \in [k]$.
Fix any $i \in [k]$.
By \eqref{Fi=Li-1st} 
and take $s=s_{F_i,o}$ on the chart $\fV$ (cf. Definition \ref{fv-k=0}),
  we obtain, \begin{equation}\label{discard-Fk}
 x_{\fV,\uu_{s_{F_i,o}}}x_{\fV, \uv_{s_{F_i,o}}}L_{\fV, F_i} =  \bF_{\fV,i}
 \end{equation}
  because $x_{(\uu_{s_{F_k,o}}, \uv_{s_{F_k,o}})} \equiv 1$ on the chart $\fV$.
  This implies the statement.
\end{proof}

For any $f \in R$, we let $\deg_\vp f$ be the degree of $f$ considered as
a polynomial in $\vp$-variables only.

\begin{defn}\label{defn:q} Let $f \in \cB^\pq$ be a binomial relation of pre-quotient type. 
We say $f$ is a binomial relation of quotient type
if  $\deg_\vp f =0$, that is, it does not contain any $\vp$-variable.
We let $\cB^\q$ be the set of all binomial relations of quotient type.
 Fix a standard chart $\fV$ as in Definition \ref{fv-k=0}, we let $\cB^\q_{\fV, [k]}
 \subset \cB^\pq_{\fV, [k]}$ be the subset of all
the de-homogenized  binomial relations of quotient type.
\end{defn}

We write 
$$\cB^\q_{\fV}:= \cB^\q_{\fV,[\up]}, \;\; \cB^\pq_{\fV}:= \cB^\pq_{\fV,[\up]}.$$

We let $R_\vr$ be the subring of $R$ consisting of polynomials with  $\vr$-variables only.
Then, binomial relations of quotient type belong to $R_\vr$.

By Lemma \ref{ker-phi-k}, all binomials of $\cB^\q$ and $\cB^\q_\fV$ 
 are $\vr$-linear,  in particular,  
 they are square-free.

\begin{prop}\label{equas-p-k=0} 
Let the notation be as in Proposition \ref{equas-fV[k]} for $k=\up$.  
Then, the scheme $\sV\cap \fV$, as a closed subscheme of
the chart $\fV$ of $\sR_\sF$,  is defined by 
\begin{eqnarray} 
\;\;\;\;\;\;\;\;\;\;\;\;\;\; B_{ \fV,(s,t)}: \;\; x_{\fV,(\uu_s, \uv_s)}x_{\fV,\uu_t}x_{\fV, \uv_t}-x_{\fV,(\uu_t, \uv_t)}x_{\fV,\uu_s}x_{ \fV,\uv_s}, \;\; \forall \;\; s, t \in S_F \- s_F,  \label{eq-Bres-pk=0}\\
B_{ \fV,(s_F,s)}: \;\;\;\;\;\; x_{\fV, (\uu_s, \uv_s)}x_{\fV, \uu_F} - x_{\fV, (\um,\uu_F)}   x_{\fV, \uu_s} x_{\fV, \uv_s}, \;\;
\forall \;\; s \in S_F \- s_F,  \label{eq-B-k=0} \\ 
\cB^\q_\fV,  \;\;\; \;\;\; \;\; \;\; \;\; \;\; \;\; \;\; \;\; \;\; \;\; \;\; \;\; \;\; \label{eq-hq-k=0}\\
L_{\fV, F}: \;\; \sum_{s \in S_F} \vsgn (s) x_{\fV,(\uu_s,\uv_s)} \label{linear-pl-k=0}
\end{eqnarray}
for all $\bF \in \sfm$
with $\bF$ being expressed as $\vsgn (s_F) x_{\uu_F} +\sum_{s \in S_F \- s_F} \vsgn (s) x_{\uu_s}x_{\uv_s}$.
Here, we set 
$$x_{\fV,\um} \equiv 1; \;\; x_{\fV, (\uu_{s_{F, o}}, \uv_{s_{F, o}})} \equiv 1, \;\; \forall \;\; 
\bF \in \sF_\um,$$
Moreover, every binomial $B_\fV \in \cB^\q_\fV$ is linear in $\vr$-variables, in particular, square-free.
\end{prop}
\begin{proof} By Proposition \ref{equas-fV[k]} for $k=\up$,
 the scheme $\sV\cap \fV$, as a closed subscheme of
the chart $\fV$ of $\sR_\sF$,  is defined by relations in \eqref{eq-Bres-pk=0},
\eqref{eq-B-k=0} , \eqref{linear-pl-k=0},  and $\cB^\pq_\fV$.

It remains to reduce $\cB^\pq_\fV$ to $\cB^\q_\fV$.

We claim that any relation $f$ of $\cB^\pq_\fV$ is generated by relations of $\cB^\q_\fV$.

We prove it by induction on $\deg_\vp (f)$.

When  $\deg_\vp (f) =0$, the statement holds trivially.

Assume that statement holds for $\deg_\vp< e$ for some $e>0$.

Consider $\deg_\vp (f) = e$. 

By Lemma \ref{ker-phi-k}, we can write 
$$f= x_{\uu_s} x_{\uv_s} \bn_s - x_{\uu_t} x_{\uv_t} \bn_t$$
for some $s, t \in S_F$ and some $\bF \in \sfm$.
Because on the chart $\fV$, we have
$$x_{\uu_s} x_{\uv_s} - x_{\uu_{s_{F,o}}}x_{\uv_{s_{F,o}}} x_{(\uu_s, \uv_s)}, \;
x_{\uu_t} x_{\uv_t} - x_{\uu_{s_{F,o}}}x_{\uv_{s_{F,o}}} x_{(\uu_t, \uv_t)},$$
where $s_{F,o}$ is as in Definition \ref{fv-k=0} with
$x_{(\uu_{s_{F,o}},\uv_{s_{F,o}})}\equiv 1$,
we get
$$f= x_{\uu_{s_{F,o}}}x_{\uv_{s_{F,o}}}(x_{(\uu_s, \uv_s)} \bn_s - x_{(\uu_t, \uv_t)} \bn_t).$$
Observe that $(x_{(\uu_s, \uv_s)} \bn_s - x_{(\uu_t, \uv_t)} \bn_t) \in \cB^\pq_\fV$.
Since  $\deg_\vp (x_{(\uu_s, \uv_s)} \bn_s - x_{(\uu_t, \uv_t)} \bn_t) < e$, 
the statement then follows from the inductive assumption.
\end{proof}

\begin{defn}\label{all-binomials-10} We let $\cB^\mn_\fV$ (respectively, $\cB^\res_\fV$, $\cB^\q_\fV$)
be the set of all binomial relations of
\eqref{eq-B-k=0} (respectively,
\eqref{eq-Bres-pk=0}, \eqref{eq-hq-k=0}).
We call relations of $\cB^\mn_\fV$ (respectively, $\cB^\res_\fV$, $\cB^\q_\fV$)
main (respectively, residual, of quotient type)  binomial on the chart $\fV$.
We let
$$\cB=\cB^\mn \sqcup \cB^\res \sqcup \cB^\q \;\; \; \and \;\;\;
 \cB_\fV=\cB^\mn_{ \fV} \cup \cB^\res_{ \fV} \cup \cB^\q_{ \fV}.$$
We let $L_{\fV, \sfm}$ be the set of all linear equations of
 \eqref{linear-pl-k=0}. We call relations of $L_{\fV, \sfm}$ linearized 
 $\pl$ relations on the chart $\fV$.
\end{defn}


\section{$\vt$-Blowups}\label{vt-blowups} 

{\it  We begin now the process of $``$removing$"$ zero factors of  main binomials
by  sequential smooth blowups. It is divided into three subsequences.
The first are $\vt$-blowups.}
 

To start, it is useful to fix some terminology, used throughout.

\subsection{Some conventions on blowups} \label{blowupConventions} $\ $

Let $X$ be a 
scheme over the base field $\kk$.
When we blow up the scheme $X$ along the ideal (the homogeneous ideal, respectively)
$I=\langle f_0, \cdots, f_m \rangle$,  generated by some elements $f_0, \cdots, f_m$,
we will realize the blowup scheme $\widetilde X$ as the graph of the closure
of the rational map $$f: X \dashrightarrow \PP^m,$$
$$ x \to [f_0(x), \cdots, f_m(x)].$$
Then, upon fixing the generators  $f_0, \cdots, f_m$, we have a natural embedding
\begin{equation}\label{general-blowup} \xymatrix{
\widetilde{X}  \ar @{^{(}->}[r]  & X \times \PP^m.
 }
\end{equation}
We let
\begin{equation}
\pi: \widetilde{X} \lra X 
\end{equation}
be the induced blowup morphism.

We will refer to the projective space $\PP^m$ as the  {\it factor
projective space}  of the blowup corresponding to the generators  $f_0, \cdots, f_m$.
We let $[\xi_0, \cdots, \xi_m]$ be the homogeneous coordinates of
the factor projective space $\PP^m$,  corresponding  to $(f_0, \cdots, f_m)$. 

When $X$ is smooth and the center of the blowup is also smooth, then, the scheme 
$\widetilde X$, as a closed subscheme of $X \times \PP^m$, is defined by the relations
\begin{equation}
f_i \xi_j - f_j \xi_i, \;\; \hbox{for all $0 \le i \ne j \le m$}.
\end{equation}

\begin{defn}\label{general-standard-chart}
Suppose that the scheme $X$ is covered by a set $\{\fV' \}$ of open subsets, called
(standard) charts. 

Fix any $0 \le i \le m$. We let
\begin{equation} 
\fV=  (\fV' \times (\xi_i \ne 0)) \cap \widetilde{X} .
\end{equation}
We also often express this chart as
 $$\fV= (\fV' \times (\xi_i \equiv 1)) \cap \widetilde{X}.$$
It is an open subset of  $\widetilde{X}$, and will be called a standard chart of $\widetilde{X}$
lying over the (standard) chart $\fV'$ of $X$. Note that every standard chart of $\widetilde{X}$
lies over a unique (standard) chart $\fV'$ of $X$.
Clearly, $\widetilde{X}$ is covered by the finitely many  standard charts.

 In general, we let
 $$\widetilde{X}_k \lra \widetilde{X}_{k-1} \lra  \cdots  \lra \widetilde X_0:=X$$
 be a sequence of blowups such that every blowup $\widetilde{X_j} \to \widetilde{X}_{j-1}$ is
 as in \eqref{general-blowup}, $j \in [k]$.  
 
 Consider any $0 \le j < k$.  Let $\fV$ (resp. $\fV''$) be a standard chart of $\widetilde{X}_k$
 (resp. of $\widetilde{X}_j$). Let $\fV'$ be the unique standard chart $\fV'$ of $\widetilde X_{k-1}$
 such that $\fV$ lies over $\fV'$.
 Via induction, we say $\fV$ lies over $\fV''$ if $\fV'$ equals to (when $j=k-1$) or lies over $\fV''$
 (when $j < k-1$).
\end{defn}

 We keep the notation as above. Let $\widetilde{X}  \to X$ be a blowup as
in \eqref{general-blowup}; we let $\fV$ be a standard chart of $\widetilde{X}$, lying over
a unique (standard) chart $\fV'$ of $X$; let
$\pi_{\fV, \fV'}: \fV \lra \fV'$ be the induced projection.

\begin{defn}\label{general-proper-transform-of-variable}  Assume that
$\fV$ (resp. $\fV'$) is isomorphic to an affine space and comes equipped with
a set of coordinate variables $\var_\fV$ (resp. $\var_{\fV'}$).
Let $y \in \var_\fV$ (resp. $y' \in \var_{\fV'}$) be a coordinate variable of $\fV$ (resp. $\fV'$).
We say the coordinate variable $y$ is a proper transform of
the coordinate variable $y'$ if the divisor  $(y=0)$ on the chart $\fV$ is the proper transform of  
the divisor $(y'=0)$ on the chart $\fV'$.
\end{defn}

Keep the notation and assumption as in Definition \ref{general-proper-transform-of-variable}.

We assume in addition that the induced blowup morphism 
$$\pi^{-1} (\fV') \lra \fV'$$
corresponds to  the blowup of $\fV'$ along
the coordinate subspace of $\fV'$ defined by
 $$Z=\{y'_0= \cdots =y'_m =0\}$$
 with $\{y'_0, \cdots, y'_m\} \subset \var_{\fV'}$.
As earlier, we let $\PP^m$ be the corresponding factor projective space
with homogeneous coordinates $[\xi_0, \cdots, \xi_m]$, corresponding to $(y'_0, \cdots, y'_m)$.
 
 Without loss of generality, we assume that 
  the standard chart $\fV$ corresponds to $(\xi_0 \equiv 1)$, that is,
 $$\fV = (\fV' \times (\xi_0 \equiv 1)) \cap \widetilde{X}.$$
 Then, we have that $\fV$, as a closed subscheme of $\fV' \times (\xi_0 \equiv 1)$,
 is defined 
 \begin{equation}\label{general-blowup-formulas}
y'_i  - y'_0 \xi_i, \;\; \hbox{for all $i \in [m]$}.
\end{equation}

The following proposition is standard and will be applied throughout. 

\begin{prop}\label{generalmeaning-of-variables} Keep the notation and assumption as above.
In addition, we let $E$ be the exceptional divisor of the blowup $\widetilde{X}  \to X$.

Then, the standard chart $\fV$ comes equipped with a set of free variables 
$$\var_\fV=\{ \zeta, y_1, \cdots, y_m;  y:=y'  \mid y' \in \var_{\fV'} \- \{y_0' , \cdots, y_m'\} \}$$
where $\zeta:=y_0', y_i :=\xi_i, i \in [m]$, and
 is isomorphic to the affine space with the variables in $\var_\fV$
as its coordinate variables such that
\begin{enumerate}
\item $E \cap \fV =(\zeta =0)$; we call $\zeta$ the exceptional variable/parameter of $E$ on $\fV$;
\item $y_i \in \var_\fV$ is a proper transform of  $y'_i \in \var_{\fV'}$ for all  $i \in [m]$;
\item $y \in \var_\fV$ is a proper transform of  
$y' \in \var_\fV$ for all  $y' \in \var_{\fV'} \- \{y_0' , \cdots, y_m'\}$.
\end{enumerate}
\end{prop}
\begin{proof}
It is straightforward from \eqref{general-blowup-formulas}.
\end{proof}
 
 Let $\AA^l$ be the affine space of dimension $l$ for some positive integer $l$ with the set of 
coordinate variables $\var_{\AA^l}$. 
Let  $\bf m$ be a monomial in $\var_{\AA^l}$. Then, for every variable $x \in \var_{\AA^l}$,
we let $\deg_x {\bf m}$ be the degree of $x$ in $\bf m$. 
 
 \begin{defn}\label{general-proper-transforms} 
 Keep the notation and assumption as in Proposition \ref{generalmeaning-of-variables}.
In addition, we let 
$$\phi=\{y'_0, \cdots, y'_m\} \subset \var_{\fV'}.$$

 Let $B_{\fV'}=T^0_{\fV'}- T^1_{\fV'}$ be 
 a binomial with variables in $\var_{\fV'}$.
We let $$m_{\phi, T^i_{\fV'}} = \sum_{j=0}^m \deg_{y'_j} (T^i_{\fV'}), \;\; i =0, 1, $$ 
$$l_{\phi, B_{\fV'}} = \min \{m_{\phi, T^0_{\fV'}}, m_{\phi, T^1_{\fV'}}\}.$$ 

Applying \eqref{general-blowup-formulas}, we substitute $y'_i$ by $y'_0 \xi_i$, for all $i \in [m]$,
into $B_{\fV'}$ and switch $y_0'$ by $\zeta$ and $\xi_i$ by $y_i$ with $i \in [m]$ to obtain
the pullback $\pi_{\fV,\fV'}^* B_{\fV'}$
where $\pi_{\fV,\fV'}: \fV \lra \fV'$ is the induced projection.
We then let 
\begin{equation}\label{define-proper-t}
B_\fV = (\pi_{\fV,\fV'}^* B_{\fV'}) / \zeta^{l_{\phi, B_{\fV'}}}.
\end{equation}
We call $B_\fV$, a binomial in $\var_\fV$, the proper transform of $B_{\fV'}$.

In general, for any polynomial $f_{\fV'}$ in $\var_{\fV'}$ such that 
$f_{\fV'}$ does not vanish identically along 
$Z= (y'_0= \cdots= y'_m=0)$, we let
$f_\fV = \pi_{\fV,\fV'}^* f_{\fV'}$. This is the pullback, but for convenience, we also call
$f_\fV$ the proper transform of $f_{\fV'}$. (We will only apply this to linearized $\pl$ relations.)

Moreover, suppose $\zeta$ appears in $B_\fV =(\pi_{\fV,\fV'}^* B_{\fV'}) / \zeta^{l_{\psi, B_{\fV'}}}$ or 
in $f_\fV= \pi_{\fV,\fV'}^* f_{\fV'}$,  and is obtained through the substitution $y'_i$ by $y'_0 \xi_i$
(note here that $\zeta:=y'_0$ and $i$ needs not to be unique), 
then we say that the exceptional parameter $\zeta$
is acquired by $y_i'$. In general, for  sequential blowups, if $\zeta$ is acquired by $y'$ and $y'$ is 
acquired by $y''$, then we also say $\zeta$ is acquired by $y''$.
\end{defn}

\begin{lemma}\label{same-degree}
We keep the same assumption and notation as in Definition \ref{general-proper-transforms}.

We let $T_{\fV', B}$ (resp. $T_{\fV, B}$) be any fixed term of $B_{\fV'}$ (resp.
$B_\fV$).  Consider any $y \in \var_\fV \- \zeta$ and  let 
$y' \in \var_{\fV'}$ be such that $y$ is the proper transform of $y'$. 
Then,  $y^b \mid  T_{\fV, B}$ if and only if 
$y'^b \mid  T_{\fV', B}$ for all integers $b \ge 0$.
\end{lemma}
\begin{proof}
This is clear from \eqref{define-proper-t}.
\end{proof}

\begin{defn}\label{general-termination} 
We keep the same assumption and notation as in Definition \ref{general-proper-transforms}.

Consider an arbitrary binomial  $B_{\fV'}$
(resp. $B_\fV$) with variables in $\var_{\fV'}$ (resp. $\var_\fV$).
Let $\bz' \in \fV'$  (resp. $\bz \in \fV$)  be any fixed closed point of the chart. 
We say $B_{\fV'}$ (resp. $B_\fV$)  terminates at $\bz'$ (resp. $\bz$)
 if (at least) one of the monomial terms of $B_{\fV'}$
(resp. $B_\fV$), say, $T_{\fV', B}$
(resp. $T_{\fV,B}$),
does not vanish at  $\bz'$ (resp. $\bz$).  In such a case, we also say 
$T_{\fV', B}$ (resp. $T_{\fV,B}$) terminates  at $\bz'$ (resp. $\bz$).
 \end{defn}

\subsection{Main binomial equations: revisited} $\ $


Recall that we have  chosen and fix the total order $``<"$ on $\sfm$ and we have listed it as
$$\sfm=\{\bF_1 < \cdots < \bF_\Upsilon\}.$$
Now,  fix any $k \in [\up]$ and consider $F_k$.
We express $F_k=\sum_{s \in S_{F_k}} \vsgn (s) p_{\uu_s} p_{\uv_s}$.
 Its corresponding linearized $\pl$ equation can
be expressed as $\sum_{s \in S_{F_k}} \vsgn(s) x_{(\uu_s,\uv_s)}$, denoted by $L_{F_k}$. 
We let $s_{F_k} \in S_{F_k}$ be the index for the leading term of $F_k$, written as 
$ \vsgn (s_{F_k}) p_\um p_{\uu_{F_k}}$.
Correspondingly, the leading term of the de-homogenization $\bF_k$ of $F_k$
 (resp.  the linearized $\pl$ equation $L_{F_k}$)
 is defined to be $\vsgn (s_{F_k})x_{\uu_{F_k}}$ (resp. $\vsgn (s_{F_k}) x_{(\um,\uu_{F_k})}$).
We then choose and fix an arbitrary order on the index set $S_{F_k}$ such that $s_{F_k}$ is the smallest element.  For preciseness, we choose and use the following order:
for any two $s,  t \in S_{F_k} \- s_{F_k}$ with $s \ne t$,  we say $s <t$ if 
$(\uu_s,\uv_s) <_\lex (\uu_t,\uv_t)$ (see Definition \ref{gen-order}). As in Definition \ref{ftF},
we let $(\ft_{F_k}+1)$ be the number of terms in ${F_k}$. Then, we can  
list $S_{F_k}$ as
$$S_{F_k}=\{s_{F_k}< s_1 < \cdots < s_{\ft_{F_k}}\}.$$
By Corollary \ref{eq-tA-for-sV}, the scheme $\sV$ as a closed subscheme of
$\sR_\sF= \rU_\um \times  \prod_{\bF \in \sF_\um} \PP_F $
is defined by the following relations
\begin{eqnarray}
  \cB^\res, \;\;  \cB^\pq,   \\ 
\;\;\;\;\;\;\;\;\; B_{(k\tau)}: \; x_{(\uu_{s_\tau}, \uv_{s_\tau})}x_{\uu_{F_k}} - x_{(\um,\uu_{F_k})}   x_{\uu_{s_\tau}} x_{\uv_{s_\tau}}, \;\;
\forall \;\; {s_\tau} \in S_{F_k} \- s_{F_k},  \; 1 \le \tau \le \ft_{F_k}, \label{eq-B-ktau'} \\ 
L_{F_k}: \;\; \sum_{s \in S_{F_k}} \vsgn (s) x_{(\uu_s,\uv_s)}, \label{linear-pl-ktau'} \\
\bF_k=\sum_{s \in S_{F_k}} \vsgn (s) x_{\uu_s} x_{\uv_s}
\end{eqnarray}
for all $k \in [\up].$

\begin{defn}\label{pm-term} Given any binomial equation
$B_{(k\tau)}$ as in \eqref{eq-B-ktau'},
we let $T^+_{(k\tau)}= \; x_{(\uu_{s_\tau}, \uv_{s_\tau})}x_{\uu_{F_k}}$, called the plus-term of $B_{(k\tau)}$,
and $T^-_{(k\tau)}= x_{(\um,\uu_{F_k})}   x_{\uu_{s_\tau}} x_{\uv_{s_\tau}},$ 
called the minus-term of $B_{(k\tau)}$.
Then, we have $B_{(k\tau)}=T^+_{(k\tau)}-T^-_{(k\tau)}$.
\end{defn}

We do not name any term of a binomial of $\cB^\res \cup \cB^\pq$ a plus-term or a minus-term since
the two terms of such a  binomial  are indistinguishable.  

In addition, we let $\cB^\mn_{F_k}=\{B_{(k\tau)} \mid  \tau \in [\ft_{F_k}]  \}$ for any $k \in [\up]$.
Then, we have
 $$\cB^\mn=\bigsqcup_{k\in [\up]} \cB^\mn_{F_k}=\{B_{(k\tau)} \mid k \in [\up], \;  \tau \in [\ft_{F_k}]  \}.$$ 
We let 
\begin{equation}\label{indexing-Bmn}
\Index_{\cB^\mn}=\{ (k\tau) \mid k \in [\up], \; \tau \in [\ft_{F_k}] \}
\end{equation}
be the index set of $\cB^\mn$. 
Then, the set $\cB^\mn$ comes equipped with a total order $``<"$ induced by
the lexicographic order on $\Index_{\cB^\mn}$, that is, 
$$B_{(k\tau)} < B_{(k'\tau')} \iff (k,\tau) <_\lex (k',\tau').$$

\subsection{$\vt$-centers and $\vt$-blowups}\label{vr-centers} $\ $

{\it Besides serving as a part of   the process of $``removing"$   zero factors of
 the main binomial relations, the reason to perform $\vt$-blowups first 
 is  to eliminate all residual binomial relations
by making them dependent on the main binomial relations.  
}

Recall that the scheme $\tsR_{\vt_{[0]}}:=\sR_\sF$ comes equipped with two kinds of divisors:
$\vp$-divisors $X_\uw$ for all $\uw \in \II_{d,n}\- \um$
(Definition \ref{vp-divisor}) and
$\vr$-divisors $X_{(\uu,\uv)}$ for all $(\uu,\uv) \in \La_\sfm$  (Definition \ref{vr-divisor}).

\begin{defn}\label{defn:vr-centers}
Fix any $\uu \in \II_{d,n}^\um$. 
We let
$$\vt_\uu=(X_\uu, X_{(\um,\uu)}).$$
We call it the  $\vt$-set with respect to $\uu$.
We then call the scheme-theoretic intersection
 $$Z_{\vt_\uu}=X_\uu \cap X_{(\um,\uu)}$$
the $\vt$-center with respect to $\uu$.
\end{defn}

We let
$$\Theta=\{\vt_\uu \mid \uu \in \II_{d,n}^\um\},\;\;\;
Z_\Theta=\{Z_{\vt_\uu} \mid \uu \in \II_{d,n}^\um\}.$$
We let $\Theta$, respectively,  $Z_\Theta$,
inherit the total order from $\II_{d,n}^\um$. 
Thus,  if we write
$$\II_{d,n}^\um=\{\uu_1 < \cdots <\uu_{\up}\}$$
and also write $\vt_{\uu_k}=\vt_{[k]}$, $Z_{\vt_{\uu_k}}=Z_{\vt_{[k]}}$, then, we can express
$$Z_\Theta=\{Z_{\vt_{[1]}} < \cdots < Z_{\vt_{[\up]}} \}.$$

We then blow up $\sR_\sF$ along 
$Z_{\vt_{[k]}},\; k \in [\up]$, in the above order.  More precisely,
we start by setting $\tsR_{\vt_{[0]}}:=\sR_\sF$.  Suppose 
$\tsR_{\vt_{[k-1]}}$ has been constructed for some $k \in [\up]$. We then let
$$\tsR_{\vt_{[k]}} \lra \tsR_{\vt_{[k-1]}}$$
be the blowup of $\tsR_{\vt_{[k-1]}}$ along the proper transform of $Z_{\vt_{[k]}}$,
and we call it  the $\vt$-blowup in ($\vt_{[k]}$).

 The above gives rise to the following sequential $\vt$-blowups
\begin{equation}\label{vt-sequence}
\tsR_{\vt_{[\up]}} \to \cdots \to \tsR_{\vt_{[1]}} \to \tsR_{\vt_{[0]}}:=\sR_\sF,
\end{equation}

Every blowup $\tsR_{\vt_{[j]}} \lra \tsR_{\vt_{[j-1]}}$
comes equipped with an exceptional divisor, denoted by $E_{\vt_{[j]}}$.
Fix $k \in [\up]$. For any $j < k$, we let $E_{\vt_{[k]},j }$ be the proper transform
of $E_{\vt_{[j]}}$ in $\tsR_{\vt_{[k]}}$.  For notational consistency, we set
$E_{\vt_{[k]}}=E_{\vt_{[k]},k }$. We call the divisors $E_{\vt_{[k]},j }$, $j \le k$, the exceptional divisors
on $\tsR_{\vt_{[k]}}$.  For every 
$\uw \in \II_{d,n} \setminus \um$, we let
$X_{\vt_{[k]}, \uw}$ be the proper transform of $X_\uw$ in $\tsR_{\vt_{[k]}}$,
still called $\vp$-divisor;
for every $(\uu,\uv) \in \La_\sfm$,
we let  $X_{\vt_{[k]}, (\uu,\uv)}\cap \fV$ be 
 the proper transform of $X_{(\uu,\uv)}$ in $\tsR_{\vt_{[k]}}$,
still called $\vr$-divisor.

\subsection{Properties of $\vt$-blowups}\label{prop-vt-blowups} $\ $

By Definition \ref{general-standard-chart}, the scheme $\tsR_{\vt_{[k]}}$
is covered by a set of standard charts.

\begin{prop}\label{meaning-of-var-vtk}
Consider any standard chart $\fV$ of $\tsR_{\vt_{[k]}}$, 
 lying over a unique chart $ \fV_{[0]}$ of $\tsR_{\vt_{[0]}}=\sR_\sF$.
 We suppose that the chart $ \fV_{[0]}$ is indexed by
$\La_\sfm^o=\{(\uv_{s_{F,o}},\uv_{s_{F,o}}) \mid \bF \in \sfm \}$
(cf. \eqref{index-sR}).
As earlier, we have $\La_\sfm^\star=\La_\sfm \- \La_\sfm^o$.

Then, the standard chart $\fV$ comes equipped with 
$$\hbox{a subset}\;\; \fe_\fV  \subset \II_{d,n} \- \um \;\;
 \hbox{and a subset} \;\; \fd_\fV  \subset \La_{\sfm}^\star$$
such that every exceptional divisor 
of $\tsR_{\vt_{[k]}}$
with $E_{\vt_{[k]}} \cap \fV \ne \emptyset$ is 
either labeled by a unique element $\uw \in \fe_\fV$
or labeled by a unique element $(\uu,\uv) \in \fd_\fV$. 
We let $E_{\vt_{[k]}, \uw}$ be the unique exceptional divisor 
on the chart $\fV$ labeled by $\uw \in \fe_\fV$; we call it an $\vp$-exceptional divisor.
We let $E_{\vt_{[k]}, (\uu,\uv)}$ be the unique exceptional divisor 
on the chart $\fV$ labeled by $(\uu,\uv) \in \fd_\fV$;  we call it an $\vr$-exceptional divisor.
(We note here that being $\vp$-exceptional or $\vr$-exceptional is strictly relative to the given
standard chart.)

Further, the standard chart $\fV$  comes equipped with the set of free variables
\begin{equation}\label{variables-vtk} 
\var_{\fV}:=\left\{ \begin{array}{ccccccc}
\ve_{\fV, \uw} , \;\; \de_{\fV, (\uu,\uv) }\\
x_{\fV, \uw} , \;\; x_{\fV, (\uu,\uv)}
\end{array}
  \; \Bigg| \;
\begin{array}{ccccc}
 \uw \in  \fe_\fV,  \;\; (\uu,\uv)  \in \fd_\fV  \\ 
\uw \in  \II_{d,n} \- \um \- \fe_\fV,  \;\; (\uu, \uv) \in \La_\sfm^\star \-  \fd_\fV  \\
\end{array} \right \},
\end{equation}
such that it is canonically  isomorphic to the affine space with the variables in
\eqref{variables-vtk} as its coordinate variables. Moreover, on the standard chart $\fV$, we have
\begin{enumerate}
\item the divisor  $X_{\vt_{[k]}, \uw}\cap \fV$ 
 is defined by $(x_{\fV,\uw}=0)$ for every 
$\uw \in \II_{d,n} \setminus \um \- \fe_\fV$;
\item the divisor  $X_{\vt_{[k]}, (\uu,\uv)}\cap \fV$ is defined by $(x_{\fV,(\uu,\uv)}=0)$ for every 
$(\uu,\uv) \in \La^\star_\sfm\- \fd_\fV$;
\item the divisor  $X_{\vt_{[k]}, \uw}\cap \fV$ does not intersect the chart for all $\uw \in \fe_\fV$;
\item the divisor  $X_{\vt_{[k]}, (\uu, \uv)}$ does not intersect the chart for all $(\uu, \uv) \in \fd_\fV$;
\item the $\vp$-exceptional divisor 
$E_{\vt_{[k]}, \uw} \;\! \cap  \fV$  labeled by an element $\uw \in \fe_\fV$
is define by  $(\ve_{\fV,  \uw}=0)$ for all $ \uw \in \fe_\fV$;
\item the $\vr$-exceptional divisor 
$E_{\vt_{[k]},  (\uu, \uv)}\cap \fV$ labeled by  an element $(\uu, \uv) \in \fd_\fV$
is define by  $(\de_{\fV,  (\uu, \uv)}=0)$ for all $ (\uu, \uv) \in \fd_\fV$;
\item  any of the remaining exceptional divisor of $\tsR_{\vt_{[k]}}$
other than those that are labelled by some  some $\uw \in \fe_\fV$ or $(\uu,\uv) \in \fd_\fV$ 
 does not intersect the chart.
\end{enumerate}
\end{prop}
\begin{proof}
When $k=0$, we have $\tsR_{\vt_{[0]}}=\sR_\sF$. 
In this case, we set $$ \fe_\fV = \fd_\fV =\emptyset.$$
Then, the statement follows from 
Proposition \ref{meaning-of-var-p-k=0} with $k=\up$.

We now suppose that the statement holds for $\tsR_{\vt_{[k-1]}}$
for some $k \in [\up]$. 

We consider $\tsR_{\vt_{[k]}}$.

As in the statement, we let $\fV$ be a standard chart 
of $\tsR_{\vt_{[k]}}$, lying over a (necessarily unique) 
standard chart  $\fV'$ of $\tsR_{\vt_{[k-1]}}$.

If $(\um,\uu_k) \in  \La_{\sF_{[k]}}^o$
(cf. \eqref{index-sR}), 
then $\fV'$
does not intersect the proper transform of the blowup center $Z_{\vt_k}$ and
$\fV \to \fV'$ is an isomorphism. In this case, we let $\var_\fV=\var_{\fV'}$,
$ \fe_{\fV'}=\fe_\fV, $ and $\fd_{\fV} =\fd_{\fV'}.$ Then, the statements on $\fV'$
carry over to $\fV$.

In what follows, we assume $(\um,\uu_k) \notin  \La_{\sF_{[k]}}^o$.

Consider the embedding $$\tsR_{\vt_{[k]}} \lra  \tsR_{\vt_{[k-1]}} \times \PP_{\vt_{[k]}}$$
where $\PP_{\vt_{[k]}}$ is the factor projective space with homogeneous coordinates $[\xi_0,\xi_1]$
corresponding to $(X_{\uu_k}, X_{(\um,\uu_k)})$.
We let $E_{\vt_{[k]}}$ be the exceptional divisor created by 
the blowup $\tsR_{\vt_{[k]}} \to  \tsR_{\vt_{[k-1]}}$.

First, we consider the case when 
$$\fV = \tsR_{\vt_{[k]}} \cap (\fV' \times (\xi_0 \equiv 1).$$
We let $Z'_{\vt_k}$ be the proper transform of the $\vt$-center $Z_{\vt_k}$
in  $\tsR_{\vt_{[k-1]}}$. Then,
in this case, on the chart $\fV'$, we have
 $$Z'_{\vt_k} \cap \fV' = \{ x_{\fV', \uu_k} = x_{\fV', (\um,\uu_k)}=0\}$$
 where   $x_{\fV', \uu_k}$ (resp. $x_{\fV', (\um,\uu_k)}$) is the proper transform of $x_\uu$ (resp.
 $x_{ (\um,\uu_k)}$)
 on the chart $\fV'$.
 Then, $\fV$ as a closed subset of  $\fV' \times (\xi_0 \equiv 1)$ is defined by
 $$x_{\fV', (\um,\uu_k)} = x_{\fV', \uu_k} \xi_1.$$
 We let 
 $$\fe_\fV= \uu_k \sqcup \fe_{\fV'},\; \fd_\fV=  \fd_{\fV'},\;\; \hbox{and}$$
 $$\ve_{\fV, \uu_k}=x_{\fV', \uu_k}, \;  x_{\fV, (\um,\uu_k)}=\xi_1; \;
 y_\fV = y_{\fV'}, \; \forall \; y_{\fV'} \in \var_{\fV'}\- \{x_{\fV', \uu_k}, x_{\fV', (\um,\uu_k)}\}.$$
 Observe that  
 $E_{\vt_{[k]}} \cap \fV = (\ve_{\fV, \uu_k}=0)$ and
 $x_{\fV, (\um,\uu_k)}=\xi_1$ is the proper transform of  $x_{\fV', (\um,\uu_k)}$.
 By the inductive assumption on the chart $\fV'$, one verifies directly that
 (1) - (7) of the proposition hold (cf. Proposition \ref{generalmeaning-of-variables}).
 
Next, we consider the case when 
$$\fV = \tsR_{\vt_{[k]}} \cap (\fV' \times (\xi_1 \equiv 1).$$
 Then, $\fV$ as a closed subset of  $\fV' \times (\xi_1 \equiv 1)$ is defined by
 $$ x_{\fV', \uu_k} =x_{\fV', (\um,\uu_k)} \xi_0.$$
 We let 
 $$\fe_\fV=  \fe_{\fV'},\; \fd_\fV=  \{(\um,\uu_k)\} \sqcup \fd_{\fV'},\;\; \hbox{and}$$
 $$\de_{\fV,  (\um,\uu_k)}=x_{\fV',  (\um,\uu_k)}, \;  x_{\fV, \uu_k}=\xi_0; \;
 y_\fV = y_{\fV'}, \; \forall \; y_{\fV'} \in \var_{\fV'}\- \{x_{\fV', \uu_k}, x_{\fV', (\um,\uu_k)}\}.$$
 Observe that  
 $E_{\vt_{[k]}} \cap \fV = (\de_{\fV,  (\um,\uu_k)}=0)$ and
 $x_{\fV, \uu_k}=\xi_0$ is the proper transform of  $x_{\fV', \uu_k}$.
 By the inductive assumption on the chart $\fV'$, like in the above case,
 one checks directly that
 (1) - (7) of the proposition hold.

 This proves the proposition.
 \end{proof}

Observe here that $x_{\fV, \uu}$ with $\uu \in \fe_{\fV}$ and
 $x_{\fV, (\uu,\uv)}$ with  $(\uu,\uv) \in \de_{\fV}$ are not variables in $\var_\fV$.
For notational convenience, to be used throughout, we make a convention:  
\begin{equation}\label{conv:=1}
\hbox{$\bullet$ $x_{\fV, \uu} = 1$ if  $\uu \in \fe_{\fV}$; 
\;\; $\bullet$ $x_{\fV, (\uu,\uv)} = 1$ if  $(\uu,\uv) \in \fd_{\fV}$.}
\end{equation}

\smallskip
For any $k \in [\up]$, the $\vt$-blowup in ($\vt_{[k]}$)
gives rise to \begin{equation}\label{tsV-vt-k} \xymatrix{
\tsV_{\vt_{[k]}} \ar[d] \ar @{^{(}->}[r]  &\tsR_{\vt_{[k]}} \ar[d] \\
\sV \ar @{^{(}->}[r]  & \sR_\sF,
}
\end{equation}
where $\tsV_{\vt_{[k]}}$ is the proper transform of $\sV$ 
in  $\tsR_{\vt_{[k]}}$. 

Alternatively, we can set $\tsV_{\vt_{[0]}}:=\sV_\sF$.  Suppose 
$\tsV_{\vt_{[k-1]}}$ has been constructed for some $k \in [\up]$. We then let
 $\tsV_{\vt_{[k]}} \subset \tsR_{\vt_{[k]}}$ be the proper transform of $\tsV_{\vt_{[k-1]}}$.

\begin{defn} Fix any standard chart $\fV$ of $\tsR_{\vt_{[k]}}$ lying over
a unique standard chart $\fV'$ of $\tsR_{\vt_{[k-1]}}$ for any $k \in [\up]$. 
When $k=0$, we let $B_\fV$ and $L_{\fV, F}$ be as in 
 Proposition \ref{equas-p-k=0} for any
$B \in \cB^\mn \cup \cB^\res  \cup \cB^\q$ and $\bF \in \sfm$. Consider any fixed general $k \in [\up]$.
Suppose $B_{\fV'}$ and $L_{\fV', F}$ have been constructed over $\fV'$.
Applying Definition \ref{general-proper-transforms}, we obtain the proper transforms on the chart $\fV$
$$B_{\fV}, \;\; \forall \; B \in \cB^\mn \cup \cB^\res  \cup \cB^\q; \;\; L_{\fV, F}, \;\;
\forall \; \bF \in \sfm.$$ 
\end{defn}

We need the following notations.

Fix any $k \in [\up]$.
We let $\cB^\mn_{< k}$ (resp. $\cB^\res_{< k}$ or $L_{ <k}$)
be the set of all main (resp. residual or linear $\pl$) relations corresponding to
$F<F_k$. Similarly, we let $\cB^\mn_{ > k}$ (resp. $\cB^\res_{ > k}$, $\cL_{ >k}$)
be the set of all main (resp. residual or linear $\pl$) relations corresponding to
$F>F_k$.  Likewise, replacing $<$ by $\le$ or $>$ by $\ge$,
we can introduce 
$\cB^\mn_{ \le k}$, $\cB^\res_{ \le k}$, and $L_{ \le k}$ or
$\cB^\mn_{ \ge k}$,  $\cB^\res_{ \ge k}$, and $\cL_{\ge k}$.
Then, upon restricting the above to a fixed standard chart $\fV$, we obtain
$\cB^\mn_{\fV, < k}$, $\cB^\res_{\fV, < k}$, $L_{\fV, <k}$, etc..

Recall from the above proof, we have 
$$\tsR_{\vt_{[k]}} \subset \tsR_{\vt_{[k-1]}} \times \PP_{\vt_k}$$
where $\PP_{\vt_k}$ be the factor projective space of the blowup
$\tsR_{\vt_{[k]}} \lra \tsR_{\vt_{[k-1]}}$. 
We write $\PP_{\vt_k}=\PP_{[\xi_0,\xi_1]}$ such that 
$[\xi_0,\xi_1]$ corresponds to $(X_{\uu_k}, X_{(\um,\uu_k)})$.

\begin{defn}  {\rm (cf. Definition \ref{general-standard-chart})}
Let $\fV'$ be any standard chart on $\tsR_{\vt_{[k-1]}}$. Then, 
we call $$\fV=\tsR_{\vt_{[k]}} \cap (\fV' \times (\xi_0 \equiv 1))$$
a $\vp$-standard chart of $\tsR_{\vt_{[k]}}$; we call $$\fV=\tsR_{\vt_{[k]}} \cap (\fV' \times (\xi_1 \equiv 1))$$
a $\vr$-standard chart of $\tsR_{\vt_{[k]}}$.
\end{defn}

\begin{prop}\label{eq-for-sV-vtk}
We keep the notation and assumptions in Proposition \ref{meaning-of-var-vtk}.

Suppose $(\um,\uu_k) \in \La_{\sF_{[k]}}^o$ or
 $\fV$ is a $\vr$-standard chart.  Then, we have
that the scheme $\tsV_{\vt_{[k]}}  \cap \fV$, as a closed subscheme of
the chart $\fV$ of $\tsR_{\vt_{[k]}} $,  is defined by 
\begin{eqnarray} 
\cB^\q_\fV,  \;\; \cB^\mn_{\fV, < k} ,\;\; \cL_{\fV, <k} , \;\;\;\;\;\;\;\;\;\;  \\
B_{ \fV,(s_{F_k},s)}: \;\;\;\;\;\; x_{\fV, (\uu_s, \uv_s)}  x_{\fV, \uu_k} 
  - \tilde{x}_{\fV, \uu_s} \tilde{x}_{\fV, \uv_s}, \;\;
\forall \;\; s \in S_{F_k} \- s_{F_k},  \label{eq-B-vt-lek=0} \\ 
L_{\fV, F_k}: \;\; \vsgn (s_F) \de_{\fV, (\um,\uu_k)} +
\sum_{s \in S_F \- s_F} \vsgn (s) x_{\fV,(\uu_s,\uv_s)},   \label{linear-pl-vtk=0} \\
\cB^\mn_{\fV, > k},\;\; \cB^\res_{\fV, > k}, \;\; \cL_{\fV, >k},\;\;\; \;\;\; \;\; \;\; \;\; \;\; \;\; \;\; \;\; \;\; \;\; \;\; \;\; \;\; \label{eq-hq-vtk=0}
\end{eqnarray}
where $\tilde{x}_{\fV, \uu_s}$ and $ \tilde{x}_{\fV, \uv_s}$ are some monomials in $\var_\fV$.

 Suppose  $(\um,\uu_k) \notin \La_{\sF_{[k]}}^o$ and $\fV$ is a $\vp$-standard chart.
Then, we have
that the scheme $\tsV_{\vt_{[k]}}  \cap \fV$, as a closed subscheme of
the chart $\fV$ of $\tsR_{\vt_{[k]}} $,  is defined by 
\begin{eqnarray} 
\cB^\q_\fV,  \;\; \cB^\mn_{\fV, < k} ,\;\; \cL_{\fV, <k}, \;\;\;\;\;\;\;\;\;\;  \\
B_{ \fV,(s_{F_k},s)}: \;\;\;\;\;\; x_{\fV, (\uu_s, \uv_s)} - x_{\fV, (\um,\uu_k)}  
 \tilde{x}_{\fV, \uu_s} \tilde{x}_{\fV, \uv_s}, \;\;
\forall \;\; s \in S_{F_k} \- s_{F_k},  \label{eq-B-vt-lek=0-00} \\ 
L_{\fV, F_k}: \;\; \vsgn (s_F) \ve_{\fV, \uu_k} x_{\fV,(\um,\uu_k)}+
\sum_{s \in S_F \- s_F} \vsgn (s) x_{\fV,(\uu_s,\uv_s)},   \label{linear-pl-vtk=0-00} \\
\cB^\mn_{\fV, > k},\;\; \cB^\res_{\fV, > k}, \;\; \cL_{\fV, >k},\;\;\; \;\;\; \;\; \;\; \;\; \;\; \;\; \;\; \;\; \;\; \;\; \;\; \;\; \;\; \label{eq-hq-vtk=0-00}
\end{eqnarray}
where $\tilde{x}_{\fV, \uu_s}$ and $ \tilde{x}_{\fV, \uv_s}$ are some monomials in $\var_\fV$.


Moreover, for any binomial 
$B \in \cB^\mn \sqcup \cB^\res_{>k}$, $B_\fV$ is $\vr$-linear and square-free.

Furthermore, consider  an arbitrary binomial $B \in \cB^\q$ and its proper transform $B_{\fV}$ on the chart 
$\fV$. Let $T_{\fV, B}$ be any fixed term of $B_\fV$.
Then, $T_{\fV, B}$  is $\vr$-linear and admits
at most one $\vt$-exceptional parameter in the form of 
$\delta_{(\um, \uu)}$ for some $(\um, \uu) \in \fd_\fV$ 
or $\ve_\uu x_{(\um, \uu)}$ for some $\uu \in \fe_\fV$.
In particular, it is  square-free.
\end{prop}
\begin{proof} We follow the notation as in the proof of 
Proposition \ref{meaning-of-var-vtk}.

When $k=0$,   we have $(\tsV_{\vt_{[0]}} \subset \tsR_{\vt_{[0]}})=  (\sV_\sF \subset \sR_\sF)$. 
Then, the statement follows from
Proposition \ref{equas-p-k=0}.

Suppose that the statement holds for $(\tsV_{\vt_{[k-1]}} \subset \tsR_{\vt_{[k-1]}})$
for some $k \in [\up]$. 

We now consider $(\tsV_{\vt_{[k]}} \subset \tsR_{\vt_{[k]}})$.

As in  the proof of 
Proposition \ref{meaning-of-var-vtk}, we let $\fV$ be a standard chart 
of $\tsR_{\vt_{[k]}}$ lying over a (necessarily unique) 
standard chart  $\fV'$ of $\tsR_{\vt_{[k-1]}}$.
Also, $\fV$ lies over a unique standard chart $\fV_{[0]}$ of 
of $\tsR_{\vt_{[0]}}$. We let $\pi_{\fV, \fV_{[0]}}: \fV \to \fV_{[0]}$ be the induced projection.
 
 To prove the statement about the defining equations of $\tsV_{\vt_{[k]}} \cap \fV$ in $\fV$, 
 by applying the inductive assumption to  $\fV'$,
 it suffices to prove that the proper transform
 of any residual binomial of $F_k$ depends on the main binomials
 on the chart $\fV$.
 
 For that purpose,  we take any two $s, t \in S_{F_k} \- s_{F_k}$ and consider
 the residual binomial  $B_{F_k,(s,t)}$ (cf. \eqref{eq-Bres}).
 
Suppose $(\um,\uu_k) \in \La_{\sF_{[k]}}^o$, hence $x_{\fV, (\um,\uu_k)}\equiv 1$ 
 on the chart $\fV$. In this case, the blowup along (the proper transform of) 
 $Z_{\vt_{[k]}}$ does not affect the chart $\fV'$ of $\tsR_{\vt_{[k-1]}}$. 
 Likewise, suppose $\fV$ is a $\vr$-standard chart. Then, $(\um,\uu_k) \in \fd_\fV$, hence
  $x_{\fV, (\um,\uu_k)}= 1$ 
 on the chart $\fV$ by \eqref{conv:=1}.
In any case, one calculates and finds that we have the following two main binomials 
 $$B_{ \fV,(s_{F_k},s)}: \;\;\;\;\;\; x_{\fV, (\uu_s, \uv_s)}  x_{\fV, \uu_k}   - 
 \tilde x_{\fV, \uu_s} \tilde x_{\fV, \uv_s},$$
 $$B_{ \fV,(s_{F_k},t)}: \;\;\;\;\;\; x_{\fV, (\uu_t, \uv_t)}  x_{\fV, \uu_k}   - \tilde x_{\fV, \uu_t} \tilde x_{\fV, \uv_t},$$
 where  $\tilde x_{\fV, \uw}=\pi_{\fV, \fV_{[0]}}^* x_{\fV_{[0]},\uw}$ denoted the pullback
 for any $\uw \in \II_{d,n} \- \um$.
  Similarly, one calculates and finds  that 
 we have 
$$ B_{ \fV,(s,t)}: \;\; x_{\fV,(\uu_s, \uv_s)}\tilde x_{\fV,\uu_t} \tilde x_{\fV, \uv_t}-
x_{\fV,(\uu_t, \uv_t)} \tilde x_{\fV,\uu_s} \tilde x_{ \fV,\uv_s}.$$ 
Then, one verifies directly that we have
$$ B_{ \fV,(s,t)}=x_{\fV, (\uu_s, \uv_s)} B_{ \fV,(s_{F_k},t)} -x_{\fV, (\uu_t, \uv_t)} B_{ \fV,(s_{F_k},s)}.$$
This proves  the statement about the defining equations of $\tsV_{\vt_{[k]}} \cap \fV$ in $\fV$.

Moreover,  consider any $B \in \cB^\mn$ with respect to $F_j$.
Observe that  $x_{(\um, \uu_k)}$ uniquely appears in the main binomials 
with respect to $F_k$;  $x_{\uu_k}$ only appears in the main binomials 
with respect to $F_k$ and   the minus terms of certain main binomials 
of $F_{j}$ with $j>k$. It follows that $B_\fV$ is $\vr$-linear and square-free.
 
 Likewise, consider any $B \in \cB^\res$ with respect to $F_j$ with $j >k$.
It is of the form
$$ B_{(s,t)}: \;\; x_{(\uu_s, \uv_s)}x_{\uu_t}  x_{\uv_t}-
x_{(\uu_t, \uv_t)} x_{\uu_s}  x_{ \uv_s}$$ 
for some $s \ne t \in S_{F_j}$.
Observe here that $B$ does not contain any $\vr$-variable of the form $x_{(\um, \uu)}$ and
the $\vp$-variables in $B$ are identical to
those of the minus terms of the corresponding main binomials.
Hence, the same line of the proof above for  main binomials
 implies that $B_\fV$ is $\vr$-linear and square-free.

Further, consider any $B \in \cB^q$. If $B_{\fV'}$ does not contain $x_{\fV',(\um, \uu_k)}$
or $(\um,\uu_k) \in \La_{\sF_{[k]}}^o$,
then the form of $B_{\fV'}$ remains unchanged (except for the meanings of its variables).
Suppose next that $B_{\fV'}$  contains $x_{\fV',(\um, \uu_k)}$ and $(\um,\uu_k) \notin \La_{\sF_{[k]}}^o$.
Note that the proper transform of the $\vt$-center 
$\vt_{[k]}$ on the chart $\fV'$ equals to $(x_{\fV', \uu_k}, x_{\fV', (\um, \uu_k)})$. Thus,
from the chart $\fV'$ to the $\vr$-standard chart $\fV$, we have that  $x_{\fV',(\um, \uu_k)}$ becomes 
$\de_{\fV,(\um, \uu_k)}$ in $B_\fV$. 
 By Lemma \ref{ker-phi-k} (2), applied to the variable $p_\um$ (before 
 de-homogenization), we see that any fixed term $T_B$ of $B$ contains
 at most one $\vr$-variables of the form $x_{(\um, \uu)}$ with
 $\uu \in \II_{d,n}^\um$. Hence, one sees that 
the last statement on  $B_\fV$ holds, in this case. 

Thus,  this proves the  statement of the proposition when 
$(\um,\uu_k) \in  \La_{\sF_{[k]}}^o$ or when $\fV$ is a $\vr$-standard chart.

Next, we consider the case when $(\um,\uu_k) \notin \La_{\sF_{[k]}}^o$
and  $\fV$ is a $\vp$-standard chart.

 Again, to prove the statement about the defining equations of $\tsV_{\vt_{[k]}} \cap \fV$ in $\fV$, 
  it suffices to prove that the proper transform
 of any residual binomial of $F_k$ depends on the main binomials on the chart $\fV$.
 
 To show this, we again take any two $s, t \in S_{F_k} \- s_{F_k}$.
 
 On the chart $\fV$, we have the following two the main binomials 
 $$B_{ \fV,(s_{F_k},s)}: \;\;\;\;\;\; x_{\fV, (\uu_s, \uv_s)}    - x_{\fV, (\um,\uu_k)}
 \tilde x_{\fV, \uu_s} \tilde x_{\fV, \uv_s},$$
 $$B_{ \fV,(s_{F_k},t)}: \;\;\;\;\;\; x_{\fV, (\uu_t, \uv_t)}    - x_{\fV, (\um,\uu_k)}
 \tilde x_{\fV, \uu_t} \tilde x_{\fV, \uv_t}.$$
We also have  the following residual binomial
$$ B_{ \fV,(s,t)}: \;\; x_{\fV,(\uu_s, \uv_s)}\tilde x_{\fV,\uu_t} \tilde x_{\fV, \uv_t}-
x_{\fV,(\uu_t, \uv_t)} \tilde x_{\fV,\uu_s} \tilde x_{ \fV,\uv_s}.$$ 
 Then, we have
 $$B_{ \fV,(s,t)}=  \tilde x_{\fV, \uu_t} \tilde x_{\fV, \uv_t} B_{ \fV,(s_{F_k},s)}
 -\tilde x_{\fV, \uu_s} \tilde x_{\fV, \uv_s} B_{ \fV,(s_{F_k},t)}.$$
 Thus, the  statement of the proposition about the equations of  $\tsV_{\vt_{[k]}}  \cap \fV$ follows.
 
 Next, consider any $B \in \cB^\mn$. The fact
 that $B_\fV$ is $\vr$-linear and 
 square-free follows from the same line of proof in the previous case.
 
Finally, consider any $B \in \cB^q$. If $B_{\fV'}$ does not contain $x_{\fV',(\um, \uu_k)}$,
then the form of $B_{\fV'}$ remains unchanged.
Suppose next that $B_{\fV'}$  contains $x_{\fV',(\um, \uu_k)}$.
Again, the proper transform of the $\vt$-center 
$\vt_{[k]}$ on the chart $\fV'$ equals to $(x_{\fV', \uu_k}, x_{\fV', (\um, \uu_k)})$. Hence,
from the chart $\fV'$ to the $\vp$-standard chart $\fV$, we have that 
 $x_{\fV',(\um, \uu_k)}$ turns into $\ve_{\fV, \uu_k} x_{\fV,(\um, \uu_k)}$  in $B_\fV$.
 Then, again, by applying Lemma \ref{ker-phi-k} (2), applied to the variable $p_\um$ (before 
 de-homogenization), we have that
 any fixed term $T_B$ of $B$ contains
 at most one $\vr$-variables of the form $x_{(\um, \uu)}$ with
 $\uu \in \II_{d,n}^\um$.   Hence, one sees that 
the last statement on  $B_\fV$ holds. 

This completes the proof of the proposition.
 \end{proof}

We need the final case of $\vt$-blowups.

We set $\tsR_{\vt}:=\tsR_{\vt_{[\up]}}$. $\tsV_{\vt}:=\tsV_{\vt_{[\up]}}$.

\begin{cor}\label{eq-for-sV-vr} 
Let the notation be as in Proposition \ref{equas-fV[k]} for $k=\up$.  
Then, the scheme $\tsV_{\vt} \cap \fV$, as a closed subscheme of
the chart $\fV$ of $\tsR_{\vt}=\tsR_{\vt_{[\up]}}$,  is defined by 
\begin{eqnarray} 
\cB^\q_\fV,  \;\; \cB^\mn_{\fV} ,\;\; L_{\fV, \sfm}.
\end{eqnarray}
Further, for any binomial $B_\fV \in \cB^\mn_\fV \cup \cB^\q_\fV$, it is  $\vr$-linear and square-free.
\end{cor}

\begin{cor}\label{no-(um,uu)} Let $X_{\vt, (\um,\uu_k)}$ be the proper transform of
$X_{(\um,\uu_k)}$ in $\tsR_{\vt}$. Then 
$$\tsV_{\vt} \cap X_{\vt, (\um,\uu_k)} =\emptyset,\;\; \forall \; k \in [\up].$$
Further, on any chart $\fV$ of $\tsR_\vt$, 
 either  $x_{(\um,\uu_k)} =1$ or
the variable $x_{\fV, (\um,\uu_k)}$ exists in $\var_\fV$ and is invertible
along $\tsV_\vt \cap \fV$ for all $k \in [\up]$.
\end{cor}
\begin{proof} Fix any standard chart $\fV$. 

If $\fV$ lies over the chart $(x_{(\um,\uu_k)} \equiv 1)$ of $\sR_\sF$, 
that is, $(\um,\uu_k) \in  \La_{\sF_{[k]}}^o$,
then the statement follows from the definition. 

If $\fV$ lies over a $\vr$-standard chart of $\tsR_{\vt_{[k]}}$, 
then the fact that $\tsV_{\vt} \cap X_{\vt, (\um,\uu_k)} =\emptyset$ follows from 
Proposition \ref{meaning-of-var-vtk}  (4);  $x_{(\um,\uu_k)} = 1$ by the convention \eqref{conv:=1}.

Suppose $\fV$ lies over a $\vp$-standard chart of $\tsR_{\vt_{[k]}}$. Then, in this case,
we have the following main binomial relation 
\begin{equation}\label{invertible-umuu} B_{ \fV,(s_{F_k}, s_{F, o})}: \;\;\;\;\;\; 1   - 
x_{\fV, (\um,\uu_k)}x_{\fV, \uu_{s_{F, o}}} x_{\fV, \uv_{s_{F, o}}}
\end{equation}
because $x_{\fV, (\uu_{s_{F,o}},\uu_{s_{F,o}})} \equiv 1$ 
with $(\uu_{s_{F,o}},\uu_{s_{F,o}}) \in  \La_{\sF_{[k]}}^o$
and $x_{\fV, \uu_k} = 1$ by \eqref{conv:=1}.  
This  implies that $x_{\fV, (\um,\uu_k)}$ is nowhere vanishing along
  $\tsV_\vt \cap \fV$.
\end{proof}

\section{${\wp}$-Blowups}\label{vs-blowups} 

{\it 
Besides serving as the necessary part of   the process of $``removing"$   zero factors of
 the main binomial relations, another purpose of ${\wp}$-blowups is  to help to control 
 the proper transforms of the binomial relations of quotient type.}

\subsection{The initial setup: $({\wp}_{(11)}\fr_0)$}    $\ $

Our initial scheme in $({\wp}_{(11)}\fr_0)$ is $\tsR_{({\wp}_{(11)}\fr_0)}:=\tsR_{\vt}$.

On the scheme $\tsR_{\vt}$, we have  three kinds of divisors:

 $\bullet$   the proper transforms $X_{\vt, \uw}$ of $\vp$-divisors $X_\uw$ for all $\uw \in \II_{d,n}\- \um$;

$\bullet$  the proper transforms $X_{\vt, (\uu,\uv)}$
of $\vr$-divisors $X_{(\uu,\uv)}$ for all $(\uu,\uv) \in \La_\sfm$; 

 $\bullet$    the proper transforms
 $E_{\vt, [k]} \subset \tsR_{\vt}$
 of the $\vt$-exceptional divisors $E_{\vt_{[k]}} \subset \tsR_{\vt_{[k]}}$ for all $k \in [\up]$.

The set of all  $\vt$-exceptional divisors of $\tsR_{({\wp}_{(11)}\fr_0)}:=\tsR_{\vt}$
can be expressed as
$$\cE_{(\wp_{(11)}\fr_0)}=\{E_{(\wp_{(11)}\fr_0), (11)0 h} 
=E_{\vt, [h]} \mid h \in [\si_{(11)0}]\}, \;\; \hbox{where $\si_{(11)0}:=\up$.} $$ 

As a set of the initial data, we need to introduce the instrumental notion: $``$association$"$
with multiplicity  as follows. 

\begin{defn} Consider any main binomial relation $B \in \cB^\mn$ written as
$$B=B_{ \fV,(s_{F_k},s)}=T^+_B-T^-_B: \;\;\;\;\;\; x_{\fV, (\uu_s, \uv_s)}  x_{\fV, \uu_k}   - 
x_{(\um, \uu_k)} x_{\fV, \uu_s} x_{\fV, \uv_s}$$
for some $k \in [\up]$ and $s \in S_{F_k}$, where $\uu_k=\uu_{F_k}$.
Consider any $\vp$-divisor, $\vr$-divisor, or  $\vt$-exceptional divisor $Y$ on $\tsR_{\vt}$.

Let  $Y=X_{\vt,\uu}$ be any $\vp$ divisor for some $\uu \in \II_{d,n}$. 
We set
$$m_{Y, T^+_B}=\left\{
\begin{array}{rcl}
1, & \;\; \hbox{if  $\uu=\uu_k$} \\
0 ,& \;\; \hbox{otherwise.}\\ 
\end{array} \right. $$ 
$$m_{Y, T^-_B}=\left\{
\begin{array}{rcl}
1, & \;\; \hbox{if  $\uu=\uu_s$ or $\uu=\uv_s$,} \\
0 ,& \;\; \hbox{otherwise.} \\ 
\end{array} \right. $$ 

Let  $Y=X_{\vt, (\uu,\uv)}$ be any $\vr$ divisor. We set
$$m_{Y, T^+_B}=\left\{
\begin{array}{rcl}
1, & \;\; \hbox{if  $(\uu,\uv)=(\uu_s,\uv_s)$} \\
0 ,& \;\; \hbox{otherwise.}\\ 
\end{array} \right. $$

Due to Corollary \ref{no-(um,uu)}, we do not  associate
$X_{(\um,\uu_k)}$ with $T^-_B$.
Hence, we set
$$m_{Y, T^-_B}=0.$$

Let  $Y=E_{\vt, [j]}$ be any $\vt$-exceptional divisor for some $j \in [\up]$. If $k=j$,
we set  $$m_{Y, T^+_B}=m_{Y, T^-_B}=0.$$
Suppose now $k \ne j$. We  set
$$m_{Y, T^+_B}=0, $$ 
$$m_{Y, T^-_B}=m_{X_{\vt, \uu_j}, T^-_B}. $$ 

We call the number $m_{Y, T^\pm_B}$ the multiplicity of $Y$ associated with the term $T^\pm_B$.
We say $Y$ is associated with $T^\pm_B$ if $m_{Y, T^\pm_B}$ is positive. 
We do not say $Y$ is associated with $T^\pm_B$
if the multiplicity $m_{Y, T^\pm_B}$ is zero.
\end{defn}

\begin{defn} Consider any linearized $\pl$ relation
$$L_F=\sum_{s \in S_F} \vsgn (s) x_{(\uu_s, \uv_s)}.$$
for some $F\in \sfm$.  Fix any $s \in S_F$.
Consider any $\vp$-divisor, $\vr$-divisor, or  $\vt$-exceptional divisor $Y$ on $\tsR_{\vt}$.

Let  $Y=X_{\vt,\uw}$ be any $\vp$ divisor for some $\uw \in \II_{d,n}$. We set
$m_{Y, s}=0.$

Let  $Y=X_{\vt, (\uu,\uv)}$ be any $\vr$ divisor. We set
$$m_{Y, s}=\left\{
\begin{array}{rcl}
1, & \;\; \hbox{if  $(\uu,\uv)=(\uu_s,\uv_s)$} \\
0 ,& \;\; \hbox{otherwise.}\\ 
\end{array} \right. $$

Let  $Y=E_{\vt, [k]}$ be any $\vt$-exceptional divisor for some $k \in [\up]$. We let
$$m_{Y, s}=m_{X_{\vt, (\um, \uu_k)}, s} \;. $$

We call the number $m_{Y, s}$ the multiplicity of $Y$ associated with $s \in S_F$.
We say $Y$ is associated with $s$ if $m_{Y,s}$ is positive. We do not say $Y$ is associated with $s$
if the multiplicity $m_{Y,s}$ is zero.
\end{defn}

\subsection{${\wp}$-centers and ${\wp}$-blowups in $({\wp}_{(k\tau)}\fr_\mu)$}\label{vs-centers} $\ $

We proceed by applying  induction on 
the set $$\{(k\tau)\mu \mid k \in [\up], \tau \in [\ft_{F_k}],
\mu \in [\rho_{(k\tau)}]\},$$ ordered lexicographically on $(k,\tau, \mu)$, 
where $\rho_{(k\tau)}$ is a  to-be-defined 
finite positive integer depending on $(k\tau) \in \Index_{\cB^\mn}$
(cf. \eqref{indexing-Bmn}).


The initial case is $\wp_{(11)}\fr_0$ and the initial scheme is $\sR_{({\wp}_{(11)}\fr_{0})}:=\tsR_{\vt}$.

We suppose that the following package in $({\wp}_{(k\tau)}\fr_{\mu-1})$ has been introduced
for some  integer $\mu \in [\rho_{(k\tau)}]$, where $1 \le \rho_{(k\tau)} \le \infty$ is 
 an integer depending on $(k\tau) \in \Index_{\cB^\mn}$. (It will be proved to be finite.)
 Here, to reconcile notations, we make the convention:
$$({\wp}_{(k\tau)}\fr_0):=({\wp}_{(k(\tau -1))}\fr_{\rho_{(k(\tau-1))}}), \; \forall \;\; 1 \le k \le \up, \;
2 \le \tau \le \ft_{F_k},$$
$$({\wp}_{(k1)}\fr_0):=({\wp}_{((k-1)\ft_{F_{k-1}})}\fr_{\rho_{((k-1)\ft_{F_{k-1}})}}), \; \forall \;\; 2 \le k \le \up, $$
provided that $\rho_{(k(\tau-1))}$ and $\rho_{((k-1)\ft_{F_{k-1}})}$
are (proved to be) finite.

\medskip\noindent
$\bullet$  {\sl The inductive assumption.} 
{\it The scheme $\tsR_{({\wp}_{(k\tau)}\fr_{\mu-1})}$ has been constructed;
it comes equipped with the set of $\vp$-divisors, 
$$\cD_{({\wp}_{(k\tau)}\fr_{\mu-1}),\vp}: \;\; X_{({\wp}_{(k\tau)}\fr_{\mu-1}), \uw}, \;\; \uw \in \II_{d,n} \- \um,$$
the set of $\vr$-divisors,
$$\cD_{({\wp}_{(k\tau)}\fr_{\mu-1}), \vr}: \;\; X_{({\wp}_{(k\tau)}\fr_{\mu-1}), (\uu,\uv)}, \;\; (\uu, \uv) \in \La_\sfm,$$
 and the set of the exceptional divisors, 
 $$\cE_{({\wp}_{(k\tau)}\fr_{\mu-1})}: \;\; E_{({\wp}_{(k\tau)}\fr_{\mu-1}), (k'\tau') \mu' h'},\;\; 
 \hbox{$(11)0 \le (k'\tau') \mu'  \le (k\tau)(\mu -1), \; h' \in [\si_{(k'\tau') \mu'}]$} $$
for some finite positive integer $\si_{(k'\tau')\mu'}$ depending on $(k'\tau')\mu'$.

We let $$\cD_{({\wp}_{(k\tau)}\fr_{\mu-1})}=\cD_{({\wp}_{(k\tau)}\fr_{\mu-1}),\vp}
 \sqcup \cD_{({\wp}_{(k\tau)}\fr_{\mu-1}),\vr}
\sqcup \cE_{({\wp}_{(k\tau)}\fr_{\mu-1})}$$ be the set of all  the aforelisted divisors. 

Fix  any $Y \in \cD_{({\wp}_{(k\tau)}\fr_{\mu-1})}$.
 Consider any $B \in \cB^\q \cup \cB^\mn$ and let $T_B$ be any fixed  term of $B$.
Then, we have that $Y$ is associated with $T_B$ with the multiplicity  $m_{Y,T_B}$, a nonnegative integer. 
In what follows, we say $Y$ is associated with $T_B$  if $m_{Y,T_B}>0$; 
we do not say $Y$ is associated with $T_B$  if $m_{Y,T_B}=0$.

Likewise, for any term of $T_s=\vsgn (s) x_{(\uu_s,\uv_s)}$ of $L_F
=\sum_{s \in S_F} \vsgn (s) x_{(\uu_s,\uv_s)}$, $Y$ 
is associated with $T_s$ with the multiplicity  $m_{Y,s}$, a nonnegative integer. 
We say $Y$ is associated with $T_s$  if $m_{Y,s}>0$; 
we do not say $Y$ is associated with $T_s$  if $m_{Y,s}=0$.}

\smallskip
We are now  to construct the scheme $\tsR_{({\wp}_{(k\tau)}\fr_\mu)}$.
 The process consists of a finite steps of blowups; the scheme $\tsR_{({\wp}_{(k\tau)}\fr_\mu)}$
is the one obtained in the final step. 




As before, fix any $k \in [\up]$,  we write $\cB^\mn_{F_k}=\{B_{(k\tau)} \mid \tau \in [\ft_{F_k}]\}.$
For every $B_{(k\tau)} \in \cB_{F_k}^\mn$, we have the expression  
$$B_{(k\tau)}=T_{(k\tau)}^+ -T_{(k\tau)}^- =x_{(\uu_s,\uv_s)}x_{\uu_k} -
x_{\uu_s}x_{\uv_s}x_{(\um,\uu_k)} $$  where $ s\in S_{F_k} \- s_{F_k}$ corresponds to $\tau$
and
$x_{\uu_k}$ is the leading variable of $\bF_k$ for some $\uu_k \in \II_{d,n}^\um$.

\begin{defn}\label{wp-sets-kmu} 
A pre-${\wp}$-set $\phi$ in $({\wp}_{(k\tau)}\fr_\mu)$, written as
$$\phi=\{Y^+, \; Y^- \},$$
 consists of exactly two divisors 
 of the scheme $\tsR_{({\wp}_{(k\tau)}\fr_{\mu-1})}$
such that  
$$Y^+ \ne X_{({\wp}_{(k\tau)}\fr_{\mu-1}), (\uu_s,\uv_s)}, \;
Y^- \ne X_{({\wp}_{(k\tau)}\fr_{\mu-1}), (\um,\uu_k)},$$
and $Y^\pm$ is  associated with $T_{(k\tau)}^\pm$. 
  
  Given the above pre-${\wp}$-set $\phi$, we let 
$$Z_\phi = Y^+ \cap Y^-$$
 be the scheme-theoretic intersection.
The pre-$\wp$-set $\phi$ (resp. $Z_\phi $) is called a $\wp$-set (resp. $\wp$-center) in $({\wp}_{(k\tau)}\fr_\mu)$ if
 $$Z_\phi \cap \tsV_{({\wp}_{(k\tau)}\fr_{\mu-1})} \ne \emptyset.$$
\end{defn}


As there are only finitely many $\vp$-, $\vr$-, and exceptional divisors
on the scheme $\tsR_{({\wp}_{(k\tau)}\fr_{\mu-1})}$, that is,
the set $\cD_{({\wp}_{(k\tau)}\fr_{\mu-1})}$ is finite,
one sees that there are only finitely many ${\wp}$-sets in $({\wp}_{(k\tau)}\fr_\mu)$. 
We let $\Phi_{{\wp}_{(k\tau)}\fr_\mu}$ be the finite set of all ${\wp}$-sets
in $({\wp}_{(k\tau)}\fr_\mu)$; we let $\cZ_{{\wp}_{(k\tau)}\fr_\mu}$ be the finite set of all 
corresponding ${\wp}$-centers in $({\wp}_{(k\tau)}\fr_\mu)$.
We need  a total ordering on the set
$\Phi_{{\wp}_{(k\tau)}\fr_\mu}$ to produce a canonical sequence of blowups.

\begin{defn}\label{order-Om-kmu} Consider the set 
$\cD_{({\wp}_{(k\tau)}\fr_{\mu-1})}=\cD_{({\wp}_{(k\tau)}\fr_{\mu-1}),\vp} \sqcup \cD_{({\wp}_{(k\tau)}\fr_{\mu-1}),\vr}
\sqcup \cE_{({\wp}_{(k\tau)}\fr_{\mu-1})}$. 
We introduce a total order $``<_\flat"$ on the set as follows. 
  Consider any two distinct elements, $Y_1, Y_2 \in  \cD_{({\wp}_{(k\tau)}\fr_{\mu-1})}$. 
 \begin{enumerate} 
  \item  Suppose $Y_1 \in \cE_{({\wp}_{(k\tau)}\fr_{\mu-1})}$ and $Y_2  \notin  \cE_{({\wp}_{(k\tau)}\fr_{\mu-1})}$. Then, $Y_1 <_\flat Y_2$.
  \item  Suppose $Y_1 \in \cD_{({\wp}_{(k\tau)}\fr_{\mu-1}),\vr}$ and $Y_2  \in \cD_{({\wp}_{(k\tau)}\fr_{\mu-1}),\vp}$.
   Then, $Y_1 <_\flat Y_2$.
   \item Suppose $Y_1, Y_2 \in \cE_{({\wp}_{(k\tau)}\fr_{\mu-1})}$. 
   Then, we can write $Y_i=E_{({\wp}_{(k\tau)}\fr_{\mu-1}), (k_i\tau_i)\mu_ih_i}$ for some
   $(k_i\tau_i)\mu_i\le (k\tau)(\mu-1)$ and some $1 \le h_i\le \si_{(k_i\tau_i)\mu_i}$, with $i=1,2$.
   Then, we say $Y_1 <_\flat Y_2$  if $(k_1,\tau_1, \mu_1, h_1)<_\invlex (k_2,\tau_2, \mu_2, h_2)$.
  \item  Suppose $Y_1, Y_2 \in \cD_{({\wp}_{(k\tau)}\fr_{\mu-1}),\vp}$. Write $Y_1= X_{({\wp}_{(k\tau)}\fr_{\mu-1}), \uu}$ and
   $Y_2=X_{({\wp}_{(k\tau)}\fr_{\mu-1}), \uv}$.
  Then, $Y_1 <_\flat Y_2$ if  $\uu <_\lex \uv$.
  \item Suppose $Y_1, Y_2 \in \cD_{({\wp}_{(k\tau)}\fr_{\mu-1}),\vr}$. Write $Y_i= X_{({\wp}_{(k\tau)}\fr_{\mu-1}),(\uu_i,\uv_i)}$ 
  with $i=1,2$. 
  Then, $Y_1 <_\flat Y_2$ if  $(\uu_1,\uv_1) <_\lex (\uu_2,\uv_2)$.
 \end{enumerate}
\end{defn}

Using the order  $``<_\flat"$ of Definition \ref{order-Om-kmu}, 
and applying Definition \ref{gen-order}, 
one sees that  the set $\Phi_{{\wp}_{(k\tau)}\fr_\mu}$ now
comes equipped with  the total order $``<:=<_{\flat, \lex}"$ which is the lexicographic order
induced by $``<=<_\flat"$. 
Thus, we can list $\Phi_{{\wp}_{(k\tau)}\fr_\mu}$ as
$$\Phi_{ {\wp}_{(k\tau)}\fr_\mu}=\{\phi_{(k\tau)\mu 1} < \cdots < \phi_{(k\tau)\mu \si_{(k\tau)\mu}}\}$$
for some finite positive integer $\si_{(k\tau)\mu}$ depending on $(k\tau)\mu$. We then let the set 
$\cZ_{\wp_{(k\tau)}\fr_\mu}$ of the corresponding $\wp$-centers
inherit the total order from that of $\Phi_{{\wp}_{(k\tau)}\fr_\mu}$. 
Then, we can express
$$\cZ_{\wp_{(k\tau)}\fr_\mu}=\{Z_{\phi_{(k\tau)\mu1}} < \cdots < Z_{\phi_{(k\tau)\mu\si_{(k\tau)\mu}}}\}.$$


We let $\tsR_{({\wp}_{(k\tau)}\fr_\mu \fs_1)} \lra \tsR_{({\wp}_{(k\tau)}\fr_{\mu-1})}$ 
be the blowup of $\tsR_{({\wp}_{(k\tau)}\fr_{\mu-1})}$
along the {$\wp$-}center $Z_{\phi_{(k\tau)\mu1}}$. 
Inductively, we assume that 
$\tsR_{({\wp}_{(k\tau)}\fr_\mu\fs_{(h-1)})}$ has been constructed for some 
$ h \in [\si_{(k\tau)\mu}]$.
We then let  $$\tsR_{({\wp}_{(k\tau)}\fr_\mu\fs_h)} \lra \tsR_{({\wp}_{(k\tau)}\fr_\mu\fs_{h-1})}$$ be the blowup of
$ \tsR_{({\wp}_{(k\tau)}\fr_\mu\fs_{h-1})}$ along (the proper transform of) the  $\wp$-center 
$Z_{\phi_{(k\tau)\mu h}}$.
Here, to reconcile notation, we set 
$$\tsR_{({\wp}_{(k\tau)}\fr_\mu\fs_{0})}:=\tsR_{({\wp}_{(k\tau)}\fr_{\mu-1})}
:=\tsR_{({\wp}_{(k\tau)}\fr_{\mu-1}\fs_{\si_{(k\tau)(\mu-1)}})}.$$

All of these can be summarized as the sequence 
$$\tsR_{({\wp}_{(k\tau)}\fr_\mu)}:= \tsR_{({\wp}_{(k\tau)}\fr_\mu\fs_{\si_{(k\tau)\mu}})}
 \lra \cdots \lra \tsR_{({\wp}_{(k\tau)}\fr_\mu\fs_1)} \lra \tsR_{({\wp}_{(k\tau)}\fr_{\mu-1})}.$$

 Given $h \in [\si_{(k\tau)\mu}]$,
consider the induced morphism
$\tsR_{({\wp}_{(k\tau)}\fr_\mu\fs_h)} \lra 
\tsR_{({\wp}_{(k\tau)}\fr_{\mu-1})}$. 

$\bullet$ We let $X_{({\wp}_{(k\tau)}\fr_\mu\fs_h), \uw}$ be the proper transform
of $X_{(\wp_{(k\tau)}\fr_{\mu-1}), \uw}$ in $\tsR_{({\wp}_{(k\tau)}\fr_\mu\fs_h)}$,
for all $\uw \in \II_{d,n} \- \um$. These are still called $\vp$-divisors.
We denote the set of all $\vp$-divisors
 on  $\tsR_{({\wp}_{(k\tau)}\fr_\mu\fs_h)}$ by  $\cD_{({\wp}_{(k\tau)}\fr_\mu\fs_h),\vp}$.
 
$\bullet$ We let $X_{({\wp}_{(k\tau)}\fr_\mu\fs_h), (\uu, \uv)}$ be the proper transform
  of the $\vr$-divisor $X_{({\wp}_{(k\tau)}\fr_{\mu-1}), (\uu, \uv)}$ in $\tsR_{({\wp}_{(k\tau)}\fr_\mu\fs_h)}$,
  for all $(\uu, \uv) \in \La_\sfm$. These are still called $\vr$-divisors.
We denote the set of all $\vr$-divisors
 on  $\tsR_{({\wp}_{(k\tau)}\fr_\mu\fs_h)}$ by  $\cD_{({\wp}_{(k\tau)}\fr_\mu\fs_h),\vr}$.
 
 $\bullet$ We let
 $ E_{({\wp}_{(k\tau)}\fr_\mu\fs_h),  (k'\tau')\mu' h'}$ be the proper transform 
 of $E_{({\wp}_{(k\tau)}\fr_{\mu-1}),  (k'\tau')\mu' h'}$ in $\tsR_{({\wp}_{(k\tau)}\fr_\mu \fs_h)}$,
 for all $k'\in [\up],  \tau' \in [\ft_{F_{k'}}]$, $0 \le \mu' \le \rho_{(k'\tau')}$ 
 and $h' \in [\si_{(k'\tau')\mu'}]$ with $(k'\tau')\mu' \le (k\tau)(\mu-1)$.
 (Here, we allow $(k'\tau')\mu'= (11)0$. 
 Recall that $E_{({\wp}_{(k\tau)}\fr_{\mu-1}), (11)0 h}$ with $h \in [\si_{(11)0}]=[\up]$ correspond
 to $\vt$-divisors.). These are named as  $\wp$-exceptional or simply
 exceptional divisors.  We denote the set of these exceptional divisors on  
  $\tsR_{({\wp}_{(k\tau)}\fr_\mu\fs_h)}$ by $\cE_{({\wp}_{(k\tau)}\fr_\mu \fs_h),{\rm old}}$.
 
  We let  $$\bar\cD_{({\wp}_{(k\tau)}\fr_\mu\fs_h)}= \cD_{({\wp}_{(k\tau)}\fr_\mu\fs_h),\vp}
  \sqcup \cD_{({\wp}_{(k\tau)}\fr_\mu\fs_h),\vr} \sqcup \cE_{({\wp}_{(k\tau)}\fr_\mu\fs_h),{\rm old}}$$ 
  be the set of  all of the aforementioned divisors on $\tsR_{({\wp}_{(k\tau)}\fr_\mu \fs_h)}$.
 
  
As each divisor in $\bar\cD_{({\wp}_{(k\tau)}\fr_\mu\fs_h)}$ is the proper transform of a unique 
divisor in $\cD_{({\wp}_{(k\tau)}\fr_{\mu-1})}$. The total order on $\cD_{({\wp}_{(k\tau)}\fr_{\mu-1})}$
carries over to provide an induced total order
on the set $\bar\cD_{({\wp}_{(k\tau)}\fr_\mu\fs_h)}$.

 In addition to the proper transforms of the divisors from $\tsR_{({\wp}_{(k\tau)}\fr_{\mu-1})}$, there are
the following {\it new} exceptional divisors.

For any $ h \in [\si_{(k\tau)\mu}]$, we let $E_{({\wp}_{(k\tau)}\fr_\mu\fs_h)}$ be the exceptional divisor of 
the blowup $\tsR_{({\wp}_{(k\tau)}\fr_\mu\fs_h)} \lra \tsR_{({\wp}_{(k\tau)}\fr_\mu\fs_{h-1})}$.
 For any $1 \le h'< h \le \si_{(k\tau)\mu}$,
we let $E_{({\wp}_{(k\tau)}\fr_\mu\fs_h), (k\tau)\mu h'}$ 
be the proper transform in $\tsR_{({\wp}_{(k\tau)}\fr_\mu \fs_h)} $
of the exceptional divisor $E_{({\wp}_{(k\tau)}\fr_\mu \fs_{h'})}$. 
To reconcile notation, we also set 
$E_{({\wp}_{(k\tau)}\fr_\mu\fs_h), (k\tau)\mu h}:=E_{({\wp}_{(k\tau)}\fr_\mu \fs_h)}$.
We  set $$\cE_{({\wp}_{(k\tau)}\fr_\mu \fs_h),{\rm new}}=\{ E_{({\wp}_{(k\tau)}\fr_\mu\fs_h), 
(k\tau)\mu h'} \mid 1 \le h' \le h \le \si_{(k\tau)\mu} \} .$$
We then order the exceptional divisors of $\cE_{({\wp}_{(k\tau)}\fr_\mu\fs_h),{\rm new}}$
in the reverse order of occurrence, that is, 
$E_{({\wp}_{(k\tau)}\fr_\mu\fs_h), (k\tau)\mu h''} \le E_{({\wp}_{(k\tau)}\fr_\mu\fs_h), (k\tau)\mu h'}$
 if $h'' \ge h'$.
 
 We then let 
 $$\cD_{({\wp}_{(k\tau)}\fr_\mu \fs_h)} =\bar\cD_{({\wp}_{(k\tau)}\fr_\mu \fs_h)} 
\sqcup  \cE_{({\wp}_{(k\tau)}\fr_\mu \fs_h),{\rm new}}.$$
For any $Y_1 \in \cE_{({\wp}_{(k\tau)}\fr_\mu \fs_h),{\rm new}}$ and any
 $Y_2 \in \bar\cD_{({\wp}_{(k\tau)})\fr_\mu\fs_h}$,
we say $Y_1 <_\flat Y_2$.  As each of $\bar\cD_{({\wp}_{(k\tau)}\fr_\mu \fs_h)}$ 
and $\cE_{({\wp}_{(k\tau)}\fr_\mu \fs_h),{\rm new}}$ is a totally ordered set, 
this endows $\cD_{({\wp}_{(k\tau)}\fr_\mu \fs_h)}$
with a total order $``<_\flat"$.

Finally, we set
 $\tsR_{({\wp}_{(k\tau)})\fr_\mu}:= \tsR_{({\wp}_{(k\tau)}\fr_\mu\fs_{\si_{(k\tau)\mu}})}$, and let
 $$\cD_{({\wp}_{(k\tau)}\fr_\mu),\vp}=\cD_{({\wp}_{(k\tau)}\fr_\mu\fs_{\si_{(k\tau)\mu}}),\vp}, \;
 \cD_{({\wp}_{(k\tau)}\fr_\mu),\vr}=\cD_{({\wp}_{(k\tau)}\fr_\mu\fs_{\si_{(k\tau)\mu}}),\vr},$$
 $$\cE_{{\wp}_{(k\tau)}\fr_\mu}=\cE_{({\wp}_{(k\tau)}\fr_\mu\fs_{\si_{(k\tau)\mu}}), {\rm old}}  
 \sqcup \cE_{({\wp}_{(k\tau)}\fr_\mu\fs_{\si_{(k\tau)\mu}}), {\rm new}} .$$
 This can be  summarized as 
$$\cD_{({\wp}_{(k\tau)}\fr_\mu)}:=\cD_{({\wp}_{(k\tau)}\fr_\mu),\vp} \sqcup \cD_{({\wp}_{(k\tau)}\fr_\mu),\vr}
\sqcup \cE_{({\wp}_{(k\tau)}\fr_\mu)}.$$
 This way,  we have equipped 
  the scheme  $\tsR_{({\wp}_{(k\tau)}\fr_\mu)}$ 
  with the set $\cD_{({\wp}_{(k\tau)}\fr_\mu),\vp}$ of $\vp$-divisors, 
 the set $\cD_{({\wp}_{(k\tau)}\fr_\mu),\vr}$ of $\vp$-divisors, and the set 
 $\cE_{{\wp}_{(k\tau)}\fr_\mu}$ of exceptional divisors. 


Now, we are ready to introduce the notion of $``$association$"$ in  $({\wp}_{(k\tau)}\fr_\mu)$,
as required to carry on the process of induction. 

We do it inductively  on  the set $ [ \si_{(k\tau)\mu}]$. 



\begin{defn} \label{association-vs} 
Fix any $B \in \cB^\q \cup \cB^\mn$. We let $T_B$ be any fixed term of the binomial $B$.
Meanwhile, we also consider any $\bF \in \sfm$ and let $T_s$ be the term of $L_F$ corresponding to
any fixed $s \in S_F$.

We assume that the notion of $``$association$"$ in  
$({\wp}_{(k\tau)}\fr_\mu\fs_{h-1})$ has been introduced. That is, for every divisor 
$Y' \in \cD_{({\wp}_{(k\tau)}\fr_\mu\fs_{h-1})}$, the multiplicities
$m_{Y', T_B}$ and $m_{Y',s}$ have been defined.

Consider an arbitrary divisor $Y \in \cD_{({\wp}_{(k\tau)}\fr_\mu\fs_h)}$.



First, suppose $Y \ne E_{({\wp}_{(k\tau)}\fr_\mu\fs_h)}$.
Then, it is the proper transform of a (unique) divisor $Y' \in \cD_{({\wp}_{(k\tau)}\fr_\mu\fs_{h-1})}$.
In this case, we set 
$$m_{Y, T_B}=m_{Y', T_B},\;\;\; m_{Y,s}=m_{Y',s}.$$

Next, we consider the exceptional $Y=E_{({\wp}_{(k\tau)}\fr_\mu\fs_h)}$.

We let $\phi= \phi_{(k\tau)\mu h}$. 
We have that 
$$\phi=\{ Y^+,  Y^-  \} \subset \cD_{(\wp_{(k\tau)}\fr_{\mu-1})}.$$

For any $B \in \cB^\mn \cup \cB^\q$, we write $B=T_B^0-T_B^1$. We let
$$m_{\phi, T_B^i}=m_{Y^+, T_B^i}+ m_{Y^-, T_B^i}, \; i=0,1,$$ 
 $$l_{\phi, B}=\min \{m_{\phi, T_B^0}, m_{\phi, T_B^1}\}.$$ 
 (For instance, by definition, $l_{\phi, B}>0$ when $B=B_{(k\tau)}$. In general, it can be zero.)
 Then, we let
  $$m_{E_{({\wp}_{(k\tau)}\fr_\mu\fs_h)},T_B^i}= m_{\phi, T_B^i}-  l_{\phi, T_B^i}.$$

Likewise, 
for $s \in S_F$ with $F\in \sfm$,
we let
$$m_{E_{(\wp_{(k\tau)}\fr_\mu\fs_h)},s}=m_{Y^+, s} +  m_{Y^-, s}.$$

We say $Y$ is associated with $T_B$ (resp. $T_s$)  if its multiplicity $m_{Y, T_B}$ 
(resp. $m_{Y, s}$) is positive.
We do not say $Y$ is associated with $T_B$ (resp. $T_s$)   
if its multiplicity $m_{Y,T_B}$ (resp. $m_{Y, s}$) equals to zero.
\end{defn}

When $h = \si_{(k\tau)\mu}$, we obtain all the desired data on 
$\tsR_{({\wp}_{(k\tau)}\fr_\mu)}=\tsR_{({\wp}_{(k\tau)}\fr_\mu)\fs_{\si_{(k\tau)\mu}}}$.

Now, with all the aforedescribed data equipped for $\tsR_{({\wp}_{(k\tau)}\fr_\mu)}$, 
we obtain our inductive package in 
 $({\wp}_{(k\tau)}\fr_\mu)$. This allows us to introduce the ${\wp}$-sets 
 $\Phi_{\wp_{(k\tau)}\fr_{\mu+1}}$
 and ${\wp}$-centers $\cZ_{\wp_{(k\tau)}\fr_{\mu+1}}$
 in  $({\wp}_{(k\tau)}\fr_{\mu+1})$ as in Definition \ref{wp-sets-kmu}, 
 endow total orders on $\Phi_{\wp_{(k\tau)}\fr_{\mu+1}}$
 and  $\cZ_{\wp_{(k\tau)}\fr_{\mu+1}}$ as in the paragraph immediately following
 Definition \ref{order-Om-kmu}, 
 and then  advance to the next 
round of  the $\wp$-blowups.  Here, 
 to reconcile notations, we set
 $$({\wp}_{(k\tau)}\fr_{\rho_{(k\tau)+1}}):=({\wp}_{((k(\tau +1))}\fr_{1}), \;\; 1 \le \tau <\ft_{F_k};$$
$$({\wp}_{(k\ft_{F_k})}\fr_{\rho_{(k\ft_{F_k})+1}}):=({\wp}_{((k+1)1)}\fr_{1}), \;\;  1 \le k < \up,$$
provided that $\rho_{(k\tau)}$ and $\rho_{(k\ft_{F_k})}$ are (proved to be) finite.

Given any  $({\wp}_{(k\tau)}\fr_\mu \fs_h)$, the ${\wp}$-blowup in (${\wp}_{(k\tau)}\fr_\mu\fs_h$)
gives rise to \begin{equation}\label{tsv-ktauh} \xymatrix{
\tsV_{({\wp}_{(k\tau)}\fr_\mu\fs_h)} \ar[d] \ar @{^{(}->}[r]  & \tsR_{({\wp}_{(k\tau)}\fr_\mu\fs_h)} \ar[d] \\
\sV \ar @{^{(}->}[r]  & \sR_\sF,
}
\end{equation}
where $\tsV_{({\wp}_{(k\tau)}\fr_\mu\fs_h)} $ is the proper transform of $\sV$ 
in  $\tsR_{({\wp}_{(k\tau)}\fr_\mu\fs_h)}$. 

We let $\tsV_{({\wp}_{(k\tau)}\fr_\mu)}=\tsV_{({\wp}_{(k\tau)}\fr_\mu\fs_{\si_{(k\tau)\mu}})}$.

\begin{defn}\label{defn:rhoktau}
Fix any $k \in [\up], \tau \in [\ft_{F_k}]$. Suppose there exists a finite integer $\mu$ such that
for any pre-$\wp$-set $\phi$  
in $({\wp}_{(k\tau)}\fr_{\mu +1})$ (cf. Definition \ref{wp-sets-kmu}), 
we have 
$$Z_\phi  \cap \tsV_{({\wp}_{(k\tau)}\fr_\mu)} = \emptyset.$$
Then, we let $\rho_{(k\tau)}$ be the smallest integer such that the above holds.
Otherwise, we let $\rho_{(k\tau)}=\infty$.
\end{defn}
{\it It will be shown soon 
that $\rho_{(k\tau)}$ is finite for all  $k \in [\up]$,
$\tau \in [\ft_{F_k}]$.}

For later use, we let
\begin{equation}\label{indexing-Phi}
\Phi=\{ \phi_{(k\tau)\mu h} \mid k \in [\up], \tau \in [\ft_{F_k}],
1 \le \mu \le \rho_{(k\tau)},
 h \in [\si_{(k\tau)\mu}]\},  \end{equation}
$$\Index_\Phi=\{ (k\tau)\mu h \mid k \in [\up], \tau \in [\ft_{F_k}], 1 \le \mu \le \rho_{(k\tau)},
 h \in [\si_{(k\tau)\mu}]\}. $$


Upon proving that $\rho_{(k\tau)}$ is finite  for all $(k\tau) \in \Index_{\cB^\mn}$,
 we can summarize the process
of ${\wp}$-blowups as a single sequence of blowup morphisms:
\begin{equation}\label{grand-sequence-wp}
\tsR_{\wp} \lra \cdots \lra 
\tsR_{({\wp}_{(k\tau)}\fr_\mu\fs_h)} \lra \tsR_{({\wp}_{(k\tau)}\fr_\mu\fs_{h-1})}
\lra \cdots \lra \tsR_{({\wp}_{(11)}\fr_0)}:=\tsR_{\vt},\end{equation}
where $\tsR_{\wp}:=\tsR_{({\wp}_{(\up \ft_{\up})}\fr_{\rho_{\up \ft_{\up}}})} 
:=\tsR_{({\wp}_{(\up \ft_{\up})}\fr_{\rho_{\up \ft_{\up}}}\fs_{\si_{(\up \ft_\up)\rho_{\up \ft_\up}}})} $
is the  blowup scheme reached in  the final step 
$(\wp_{(\up \ft_{\up})}\fr_{\rho_{\up \ft_{\up}}}\fs_{\si_{(\up \ft_\up)\rho_{\up \ft_\up}}})$ of all $\wp$-blowups.

Further, the end of all ${\wp}$-blowups gives rise to the following induced diagram 
\begin{equation}\label{tsv-final} \xymatrix{
\tsV_\wp \ar[d] \ar @{^{(}->}[r]  & \tsR_\wp \ar[d] \\
\sV \ar @{^{(}->}[r]  & \sR_\sF,
}
\end{equation}
where $\tsV_\wp$ is the proper transform of $\sV$  in  $\tsR_\wp$.

\subsection{Properties of $\wp$-blowups}\label{subs:prop-vs-blowups} $\ $

Recall that we have set $\tsR_{({\wp}_{(11)}\fr_1\fs_{0})}:=\tsR_{\vt}:=\tsR_{\vt_{[\up]}}$.

Now, fix and consider any $(k\tau)\mu h \in \Index_\Phi$ (cf. \eqref{indexing-Phi}).

\begin{prop}\label{meaning-of-var-vskmuh} 
Suppose that  the scheme $\tsR_{({\wp}_{(k\tau)}\fr_\mu\fs_h)}$ has been constructed,
covered by a finite set of open subsets, called standard charts (see Definition \ref{general-standard-chart}).

Consider any standard chart $\fV$ of $\tsR_{({\wp}_{(k\tau)}\fr_\mu\fs_h)}$, 
 lying over a unique chart $ \fV_{[0]}$ of $\sR_\sF$.
 We suppose that the chart $ \fV_{[0]}$ is indexed by
$\La_\sfm^o=\{(\uv_{s_{F,o}},\uv_{s_{F,o}}) \mid \bF \in \sfm \}.$  As earlier, we have $\La_\sfm^\star=\La_\sfm \- \La_\sfm^o$.

Then, the chart $\fV$ comes equipped with 
$$\hbox{a subset}\;\; \fe_\fV  \subset \II_{d,n} \- \um \;\;
 \hbox{and a subset} \;\; \fd_\fV  \subset \La_{\sfm}^\star$$
such that every exceptional divisor $E$ 
(i.e., not a $\vp$- nor a $\vr$-divisor)  of $\tsR_{({\wp}_{(k\tau)}\fr_\mu\fs_h)}$
with $E \cap \fV \ne \emptyset$ is 
either labeled by a unique element $\uw \in \fe_\fV$
or labeled by a unique element $(\uu,\uv) \in \fd_\fV$. 
We let $E_{({\wp}_{(k\tau)}\fr_\mu\fs_h), \uw}$ be the unique exceptional divisor 
on the chart $\fV$ labeled by $\uw \in \fe_\fV$; we call it an $\vp$-exceptional divisor.
We let $E_{({\wp}_{(k\tau)}\fr_\mu\fs_h), (\uu,\uv)}$ be the unique exceptional divisor 
on the chart $\fV$ labeled by $(\uu,\uv) \in \fd_\fV$;  we call it an $\vr$-exceptional divisor.
(We note here that being $\vp$-exceptional or $\vr$-exceptional is strictly relative to the given
standard chart.)

Further, the chart $\fV$ comes equipped with 
 a set of free variables
\begin{equation}\label{variables-p-vskmu} 
\var_{\fV}:=\left\{ \begin{array}{ccccccc}
\ve_{\fV, \uw} , \;\; \de_{\fV, (\uu,\uv) }\\
x_{\fV, \uw} , \;\; x_{\fV, (\uu,\uv)}
\end{array}
  \; \Bigg| \;
\begin{array}{ccccc}
 \uw \in  \fe_\fV,  \;\; (\uu,\uv)  \in \fd_\fV  \\ 
\uw \in  \II_{d,n} \- \um \- \fe_\fV,  \;\; (\uu, \uv) \in \La_\sfm^\star \-  \fd_\fV  \\
\end{array} \right \},
\end{equation}
such that it is canonically  isomorphic to the affine space with the variables in
\eqref{variables-p-vskmu} as its coordinate variables. Moreover, on the standard chart $\fV$, we have
\begin{enumerate}
\item the divisor  $X_{({\wp}_{(k\tau)}\fr_\mu\fs_h), \uw}\cap \fV$ 
 is defined by $(x_{\fV,\uw}=0)$ for every 
$\uw \in \II_{d,n} \setminus \um \- \fe_\fV$;
\item the divisor  $X_{({\wp}_{(k\tau)}\fr_\mu\fs_h), (\uu,\uv)}\cap \fV$ is defined by $(x_{\fV,(\uu,\uv)}=0)$ for every 
$(\uu,\uv) \in \La^\star_\sfm\- \fd_\fV$;
\item the divisor  $X_{({\wp}_{(k\tau)}\fr_\mu\fs_h), \uw}$ does not intersect the chart for all $\uw \in \fe_\fV$;
\item the divisor  $X_{({\wp}_{(k\tau)}\fr_\mu\fs_h), (\uu, \uv)}$ does not intersect the chart for all $ (\uu, \uv) \in \fd_\fV$;
\item the $\vp$-exceptional divisor 
$E_{({\wp}_{(k\tau)}\fr_\mu\fs_h), \uw} \;\! \cap  \fV$  labeled by an element $\uw \in \fe_\fV$
is define by  $(\ve_{\fV,  \uw}=0)$ 
for all $ \uw \in \fe_\fV$;
\item the $\vr$-exceptional divisor 
$E_{({\wp}_{(k\tau)}\fr_\mu\fs_h),  (\uu, \uv)}\cap \fV$ labeled by  an element $(\uu, \uv) \in \fd_\fV$
is define by  $(\de_{\fV,  (\uu, \uv)}=0)$ 
for all $ (\uu, \uv) \in \fd_\fV$;
\item any of the remaining exceptional divisor
of $\tsR_{({\wp}_{(k\tau)}\fr_\mu\fs_h)}$
other than those that are labelled by some  some $\uw \in \fe_\fV$ or $(\uu,\uv) \in \fd_\fV$ 
 does not intersect the chart.
\end{enumerate}
\end{prop}
\begin{proof} 
 We prove by induction on $(k\tau)\mu h \in \{(11)10 \} \sqcup \Index_\Phi$.

For the initial case, the scheme is $\tsR_{({\wp}_{(11)}\fr_1\fs_0)}=\tsR_{\vt}=\tsR_{\vt_{[\up]}}$. 
Then, this proposition is the same as Proposition \ref{meaning-of-var-vtk} with $k=\up$.
Thus, it holds.

We suppose that the statement holds over $\tsR_{({\wp}_{(k\tau)}\fr_\mu\fs_{h-1})}$
with $(k\tau)\mu h \in  \Index_\Phi$.

We now consider $\tsR_{({\wp}_{(k\tau)}\fr_\mu\fs_h)}$. We have the embedding
$$\xymatrix{
\tsR_{({\wp}_{(k\tau)}\fr_\mu\fs_h)}\ar @{^{(}->}[r]  &
 \tsR_{({\wp}_{(k\tau)}\fr_\mu\fs_{h-1})} \times \PP_{\phi_{(k\tau)\mu h}},  }
$$ where $\PP_{\phi_{(k\tau)\mu h}}$ is the factor projective space. 
We let $\phi'_{(k\tau)\mu h}=\{Y'_0, Y'_1\}$ 
be the proper transforms in $\tsR_{({\wp}_{(k\tau)}\fr_\mu\fs_{h-1})}$
of the two  divisors of the ${\wp}$-set $\phi_{(k\tau)\mu h}=\{Y^+,Y^-\}$
with $Y^\pm$ being associated with $T^\pm_{(k\tau)}$.
In addition, we let $[\xi_0, \xi_1]$ 
 be the  homogenous coordinates of $\PP_{\phi_{(k\tau)\mu h}}$
 corresponding to $\{Y'_0, Y'_1\}$.

Then, the given standard chart $\fV$ of $\tsR_{({\wp}_{(k\tau)}\fr_\mu\fs_h)}$ lies over a
unique standard chart $\fV'$ of  $\tsR_{({\wp}_{(k\tau)}\fr_\mu\fs_{h-1})}$ such that 
$\fV=(\fV' \times (\xi_i \equiv 1)) \cap \tsR_{({\wp}_{(k\tau)}\fr_\mu\fs_h)},$ for $i=0$ or 1. 

By assumption, the chart $\fV'$ comes equipped with
a subset $\fe_{\fV'} \subset \II_{d,n}\- \um$, a subset  $\fd_{\fV'} \subset \La^\star_\sfm$,
and admits a set of coordinate variables
\begin{equation}\label{variables-p-(k-1)} 
\var_{\fV'}:=\left\{ \begin{array}{ccccccc}
\ve_{\fV', \uw} , \;\; \de_{\fV', (\uu,\uv) }\\
x_{\fV', \uw} , \;\; x_{\fV', (\uu,\uv)}
\end{array}
  \; \Bigg| \;
\begin{array}{ccccc}
 \uw \in  \fe_{\fV'},  \;\; (\uu,\uv)  \in \fd_{\fV'}  \\ 
\uw \in  \II_{d,n} \- \um \- \fe_{\fV'},  \;\; (\uu, \uv) \in \La_\sfm^\star \-  \fd_{\fV'}  \\
\end{array} \right \},
\end{equation}
verifying the properties (1)-(7) as in the proposition.

We now prove the statements for the chart $\fV$ of $\tsR_{({\wp}_{(k\tau)}\fr_\mu\fs_h)}$.

First, we suppose that the proper transform $Z'_{\psi_{(k\tau) \mu h}}$ 
in $\tsR_{({\wp}_{(k\tau)}\fr_\mu\fs_{h-1})}$
of the $\wp$-center $Z_{\psi_{(k\tau) \mu h}}$  does not
meet the chart $\fV'$.  
Then, we let $\fV$ inherit all the data from those of $\fV'$, that is,
we set $\fe_{\fV}=\fe_{\fV'}$, $\fd_{\fV}=\fd_{\fV'}$, and $\var_\fV= \var_{\fV'}$:
changing the subindex $``\ \fV' \ "$ for all the variables in $\var_{\fV'}$
to $``\ \fV \ "$.
As the $\wp$-blowup  along the proper transform of $Z_{({\wp}_{(k\tau)}\fr_\mu\fs_h)}$
does not affect the chart $\fV'$, one sees that 
the statements of the proposition hold for $\fV$.

Next, we suppose that 
$Z'_{\psi_{(k\tau) \mu h}}$
meets the chart $\fV'$ along a nonempty closed subset. 
On the chart $\fV'$, by the inductive assumption, we can suppose
\begin{equation}\label{YY01-wp}
Y'_0 \cap \fV' =(y'_0 =0), \; Y'_1 \cap \fV' =(y'_1 =0), 
 \;\; \hbox{for some $y'_0, y'_1 \in \var_{\fV'}$.} 
 \end{equation} 
Then, the chart  $\fV=(\fV' \times (\xi_i \equiv 1)) \cap \tsR_{({\wp}_{(k\tau)}\fr_\mu\fs_h)}$ 
of the scheme $\tsR_{({\wp}_{(k\tau)}\fr_\mu\fs_{h})}$, as a closed subscheme of  
$\fV' \times (\xi_i \equiv 1),$ is defined by
\begin{equation}\label{proof:sf-fv-vskmu}
y'_j = y'_i \xi_j, \;  \hbox{ with $j  \in \{0, 1\} \- i$}.             
\end{equation}

There are three possibilities for
$Y'_i \cap \fV'$ according to the assumption on the chart $\fV'$.
Based on every of such possibilities, 
we set 
\begin{equation}\label{proof:de-fv-vskmu}
\left\{ 
\begin{array}{lccr}
\fe_{\fV}=\fe_{\fV'} \sqcup \uw, \; \fd_{\fV}= \fd_{\fV'},
&  \hbox{if} \; y'_i=x_{\fV', \uw}\; \hbox{for some} \;\; \uw \in \II_{d,n}\-\um \- \fe_{\fV'}\\
\fe_{\fV}=\fe_{\fV'}, \; \fd_{\fV}= \fd_{\fV'} , &\hbox{if} \;    y'_i=\ve_{\fV', \uw} \; \hbox{for some} \;\; \uw \in  \fe_{\fV'}\\
\fd_{\fV}=\fd_{\fV'}, \; \fe_{\fV}= \fe_{\fV'}, & \hbox{if} \; 
y'_i=\de_{\fV', (\uu,\uv)}\; \hbox{for some}  \;\; (\uu,\uv) \in  \fd_{\fV'}  .
\end{array} \right.
\end{equation}
Accordingly, we introduce 
\begin{equation}\label{proof:new-ex-fv-vskmu}
\left\{ 
\begin{array}{lccr}
\ve_{\fV, \uw}=y'_i, \; x_{\fV, \uw} =1,
&  \hbox{if} \; y'_i=x_{\fV', \uw}\; \hbox{for some} \;\; \uw \in \fe_\fV \- \fe_{\fV'}\\
\ve_{\fV, \uw}=y'_i  , &\hbox{if} \;    y'_i=\ve_{\fV', \uw} \; \hbox{for some} \;\; \uw \in  \fe_{\fV'}=\fe_{\fV}\\
\de_{\fV, (\uu,\uv)}=y'_i, & \hbox{if} \; 
y'_i=\de_{\fV', (\uu,\uv)}\; \hbox{for some}  \;\; (\uu,\uv) \in  \fd_{\fV'}=\fd_\fV  .
\end{array} \right.
\end{equation}

To introduce the set $\var_\fV$, for $j  \in \{0, 1\} \- i$,
we then set
\begin{equation}\label{proof:var-xi-fv-vskmu}
\left\{ 
\begin{array}{lcr}
x_{\fV,\ua}=\xi_j, &   \hbox{if $y'_j=x_{\fV', \ua} $}\\
\ve_{\fV, \ua}=\xi_j, &  \hbox{if $y'_j= \ve_{\fV', \ua}$} \\
\de_{\fV, (\ua, \ub)}= \xi_j, & \;\;\;\; \hbox{if $y'_j= \de_{\fV', (\ua, \ub)}$}.
\end{array} \right.
\end{equation}
Thus, we have introduced $y'_i,  \xi_j \in \var_\fV$ where $y'_i$ (respectively,  $\xi_j$)
is endowed with its new name as in \eqref{proof:new-ex-fv-vskmu}, respectively, in 
\eqref{proof:var-xi-fv-vskmu}.
Next, we define the set 
$\var_\fV \- \{y'_i, \xi_j\}$ 
to consist of the following variables:
\begin{equation}\label{proof:var-fv-vskmu}
\left\{ 
\begin{array}{lccr}
x_{\fV,\uw}=x_{\fV', \uw}, &  \forall \;\; \uw \in \II_{d,n}\-\um \- \fe_\fV \;\; \hbox{and $x_{\fV', \uw} \ne y'_j$}\\
x_{\fV, (\uu, \uv)}=x_{\fV', (\uu, \uv)}, & \;\; \forall \;\; (\uu, \uv) \in \La_\sfm^\star \- \fd_{\fV}  \\
\ve_{\fV, \uw}= \ve_{\fV', \uw}, & \forall \;\; \uw \in \fe_\fV \;\; \hbox{and $\ve_{\fV', \uw} \ne y'_j$},  y'_i \\
\de_{\fV, (\uu, \uv)}= \de_{\fV', (\uu, \uv)}, &  \;\; \forall \;\; (\uu, \uv) \in \fd_{\fV} \;\; \hbox{and $\de_{\fV', (\uu, \uv)} \ne y'_j$}, y'_i.
\end{array} \right.
\end{equation}

Substituting \eqref{proof:sf-fv-vskmu}, one checks that the chart $\fV$ is  
 isomorphic to the affine space with the variables in  \eqref{proof:new-ex-fv-vskmu},
 \eqref{proof:var-xi-fv-vskmu}, and \eqref{proof:var-fv-vskmu},  
   as its coordinate variables (cf. Proposition \ref{generalmeaning-of-variables}). 
   Putting all together, the above matches
   description of $\var_\fV$ in  \eqref{variables-p-vskmu},
 
 Now, it remains to verity (1)-(7) of the proposition on the chart $\fV$.

First, consider  the unique new exceptional divisor $E_{({\wp}_{(k\tau)}\fr_\mu\fs_h)}$
created by the blowup
$$\tsR_{({\wp}_{(k\tau)}\fr_\mu\fs_h)} \lra \tsR_{({\wp}_{(k\tau)}\fr_\mu\fs_{h-1})}. $$ 
Then, we have
$$ E_{({\wp}_{(k\tau)}\fr_\mu\fs_h)} \cap \fV = (y'_i=0)$$
where $y'_i$ is renamed as in \eqref{proof:new-ex-fv-vskmu}
according to the three possibilities of its form in $\var_{\fV'}$.
This way, the new exceptional divisor $E_{({\wp}_{(k\tau)}\fr_\mu\fs_h)}$ is labelled on the chart $\fV$.
In any case of the three situations, we have that the proper transform in $\tsR_{({\wp}_{(k\tau)}\fr_\mu\fs_h)}$ 
of $Y'_i$ does not meet the chart $\fV$, and
if $Y'_i$ is a $\vp$-exceptional divisor labeled by some element of $\fe_{\fV'}$
(resp, $\vr$-exceptional divisor labeled by some element of $\fd_{\fV'}$),
 then, on the chart $\fV$, its proper transform is no longer labelled by
any element of $\II_{d,n} \-\um$ (resp. $\La_\sfm^\star$).
 This verifies any of (3)-(7) whenever the statement therein involves the newly created exceptional divisor 
 $E_{({\wp}_{(k\tau)}\fr_\mu\fs_h)}$ and/or its corresponding exceptional variable $y'_i$ 
 (which is renamed in \eqref{proof:new-ex-fv-vskmu} to match that of 
 \eqref{variables-p-vskmu} in the statement of
 the proposition).
 Observe that the statements (1) and (2) are not related to the new exceptional divisor 
 $E_{({\wp}_{(k\tau)}\fr_\mu\fs_h)}$ or its corresponding exceptional variable $y'_i$ (again, renamed in \eqref{proof:new-ex-fv-vskmu}).
  
For any of the remaining $\vp$-, $\vr$-, and exceptional divisors on $\tsR_{({\wp}_{(k\tau)}\fr_\mu\fs_h)}$, 
it is the proper transform of a unique corresponding
 $\vp$-, $\vr$-, and exceptional divisor on $\tsR_{({\wp}_{(k\tau)}\fr_\mu\fs_{h-1})}$. Hence,
 applying   the inductive assumption on $\fV'$, one verifies directly that
every of the properties (1)-(7) of the proposition is satisfied whenever it applies to such a divisor
and/or its corresponding local variable in $\var_\fV$.
 
 This completes the proof.
\end{proof}

\subsection{Proper transforms of 
defining equations in $({\wp}_{(k\tau)}\fr_\mu\fs_h)$} $\ $

  Consider any fixed $B \in \cB^\mn  \cup \cB^\q$ and $\bF \in \sfm$.
Suppose $B_{\fV'}$ and $L_{\fV', F}$ have been constructed over $\fV'$.
Applying Definition \ref{general-proper-transforms}, we obtain the proper transforms on the chart $\fV$
$$B_{\fV}, \; B \in \cB^\mn  \cup \cB^\q; \;\; L_{\fV, F}, \; \bF \in \sfm.$$

\begin{defn}\label{termninatingB-wp} 
{\rm  (cf. Definition \ref{general-termination})} 
Consider any main binomial relation $B \in \cB^\mn$. 
 We say $B$ terminates on a given standard chart $\fV$
 of the scheme $\tsR_{({\wp}_{(k\tau)}\fr_\mu\fs_h)}$ 
 if $B_{ \fV}$ terminates at all closed points of $ \tsV_{({\wp}_{(k\tau)}\fr_\mu\fs_h)} \cap \fV$.
 We say $B$ terminates on  $\tsR_{({\wp}_{(k\tau)}\fr_\mu\fs_h)}$ if it terminates on
 all standard charts $\fV$ of  $\tsR_{({\wp}_{(k\tau)}\fr_\mu\fs_h)}$.
 \end{defn}

In what follows, for any $B =T^+_B - T^-_B \in \cB^\mn$, we express
$B_\fV= T^+_{\fV, B} - T^-_{\fV, B}$. If $B=B_{(k\tau)}$ for some $k \in [\up]$ and $\tau \in [\ft_{F_k}]$,
we also write $B_\fV= T^+_{\fV, (k\tau)} - T^-_{\fV, (k\tau)}$.

Below, we follow the notations of Proposition \ref{meaning-of-var-vskmuh} as well as those in its proof. 

In particular, we have that $\fV$ is a standard chart of $\tsR_{({\wp}_{(k\tau)}\fr_\mu\fs_h)}$,
 lying over a standard chart $\fV'$ of $\tsR_{({\wp}_{(k\tau)}\fr_\mu\fs_{h-1})}$. We have that
$\phi'_{(k\tau)\mu h}=\{Y'_0, \;Y'_1\}$  is the proper transforms of $\phi_{(k\tau)\mu h}=\{Y^+, Y^-\}$
 in $\tsR_{({\wp}_{(k\tau)}\fr_\mu\fs_{h-1})}$ with $Y^\pm$ being associated with $T^\pm_{(k\tau)}$.
 Likewise, $Z'_{\phi_{(k\tau)\mu h}}$ is the proper transforms of the ${\wp}$-center $Z_{\phi_{(k\tau)\mu h}}$
   in $\tsR_{({\wp}_{(k\tau)}\fr_\mu\fs_{h-1})}$.
 Also, assuming that $Z'_{\phi_{(k\tau)\mu h}} \cap \fV'  \ne \emptyset$,
then, as in \eqref{YY01-wp}, we have
 $$Y'_0 \cap \fV' =(y'_0 =0), \; Y'_1 \cap \fV' =(y'_1 =0), 
 \;\; \hbox{with $y'_0, y'_1 \in \var_{\fV'}$} $$
Further, we have $\PP_{\phi_{(k\tau)\mu h}}=\PP_{[\xi_0,\xi_1]}$ with 
 the homogeneous coordinates  $[\xi_0,\xi_1]$ corresponding to $(y'_0,y'_1)$.

\begin{prop}\label{equas-vskmuh} 
Let the notation be as in Proposition \ref{meaning-of-var-vskmuh} and be as in above.

Let $\fV$ be any standard chart of $\tsR_{(\wp_{(k\tau)}\fr_\mu\fs_h)}$. Then, the scheme 
$\sV_{({\wp}_{(k\tau)}\fr_\mu\fs_h)}\cap \fV$, as a closed subscheme of the chart $\fV$ 
 is defined by $$\cB_\fV^\mn, \; \cB_\fV^\q, \; L_{\fV, \sfm}.$$

Suppose that $Z'_{\phi_{(k\tau)\mu h}} \cap \fV'  \ne \emptyset$ where
$Z'_{\phi_{(k\tau)\mu h}}$ is the proper transform 
  of  the $\wp$-center $Z_{\phi_{(k\tau)\mu h}}$  in $\tsR_{(\wp_{(k\tau)}\fr_\mu\fs_{h-1})}$.
Further, we let $\zeta=\zeta_{\fV, (k\tau)\mu h}$ be the exceptional parameter in $\var_\fV$ such that
$$E_{({\wp}_{(k\tau)}\fr_\mu\fs_h)} \cap \fV = (\zeta=0).$$

Then, we have that the following hold.
\begin{enumerate}
\item Suppose $\fV= (\fV' \times (\xi_0 \equiv 1)) \cap \tsR_{(\eth_{(k\tau)}\fr_\mu\fs_h)}$.
We let $y_1 \in \var_\fV$  be the proper transform of $y_1'$. 
Then, we have
\begin{itemize}
\item[(1a)]  The plus-term 
$T^+_{\fV, (k\tau)}$ is square-free and $\deg (T^+_{\fV, (k\tau)}) =\deg (T^+_{\fV', (k\tau)})-1$.
We let $\deg_{y_1'} T^-_{\fV', (k\tau)}=b$ for some
integer $b$,  positive by definition, then we have $\deg_{\zeta} T^-_{\fV, (k\tau)}=b-1$. Consequently,
either $T^-_{\fV, (k\tau)}$ is linear in $y_1$ or else $\zeta \mid T^-_{\fV,  (k\tau)}$.
\item[(1b)] Let $B \in  \cB^\mn_{F_k}$ 
with $B > B_{(k\tau)}$. Then,
$T^+_{\fV, B}$ is square-free. 
Suppose $y_1 \mid T^-_{\fV, B}$,
 then either $T^-_{\fV, B}$ is linear in $y_1$ or $\zeta \mid T^-_{\fV, B}$.
\item[(1c)] Let $B \in  \cB^\mn_{F_k}$ 
 with $B < B_{(k\tau)}$
or  $B \in   \cB^\mn \- \cB^\mn_{F_k}$. Then,  $T^+_{\fV, B}$ is square-free.
$y_1 \nmid T^+_{\fV, B}$. If $y_1 \mid T^-_{\fV, B}$, then $\zeta \mid T^-_{\fV, B}$.
\item[(1d)] Consider any $B \in \cB^\q$. Then, $B_\fV$ is square-free.
\end{itemize} 
\item Suppose $\fV= (\fV' \times (\xi_1 \equiv 1)) \cap \tsR_{(\eth_{(k\tau)}\fr_\mu\fs_h)}$.
We let $y_0  \in \var_\fV$  be the proper transform of $y_0'$. 
 Then, we have
\begin{itemize}
\item[(2a)]  $T^+_{\fV, (k\tau)}$ is square-free.
 $y_0 \nmid T^-_{\fV, (k\tau)}$.
 $ \deg (T^-_{\fV, (k\tau)}) =\deg (T^-_{\fV', (k\tau)})-1$.
 We let $\deg_{y_1'} T^-_{\fV', (k\tau)}=b$ for some 
integer $b$,  positive by definition, then we have $\deg_{\zeta} T^-_{\fV, (k\tau)}=b-1$. 
Note here that $y_1' \in \var_{\fV'}$ becomes $\zeta \in \var_\fV$.
\item[(2b)] Let $B \in   \cB^\mn_{F_k}$ 
with $B > B_{(k\tau)}$. Then,
$T^+_{\fV, B}$ is square-free.  
Further,  $y_0 \nmid T^-_{\fV, B}$.
\item[(2c)] Let $B \in  \cB^\mn_{F_k}$ 
with $B < B_{(k\tau)}$
or  $B \in  \cB^\mn \- \cB^\mn_{F_k}$. Then,  $T^+_{\fV, B}$ is square-free.
$y_0 \nmid T^+_{\fV, B}$. If $y_0 \mid T^-_{\fV, B}$, then $\zeta \mid T^-_{\fV, B}$.
\item[(2d)] Consider any $B \in \cB^\q$. Then, $B_\fV$ is square-free.
\end{itemize}

Consequently, we have
\item  $\rho_{(k\tau)}< \infty$.
\item Moreover,  consider 
$\tsR_{(\wp_{(k\tau)}\fr_{\rho_{(k\tau)}})}=\tsR_{(\wp_{(k\tau)}\fr_{\rho_{(k\tau)}} \fs_{\si_{(k\tau)\rho_{(k\tau)}}})}$.
Let $\fV$ be an arbitrary standard chart  of $\tsR_{(\wp_{(k\tau)}\fr_{\rho_{(k\tau)}})}$. 
For every $B=B_{(k'\tau')} \in \cB^\mn$ 
 for some $(k'\tau') \in \Index_{\cB^\mn}$ with $(k'\tau') \le (k\tau)$,
 we write   $B_{(k'\tau')}=T_{B}^+ -T_{B}^- =x_{(\uu_t,\uv_t)}x_{\uu_{k'}} -
x_{\uu_t}x_{\uv_t}x_{(\um,\uu_{k'})} $ where $ t \in S_{F_{k'}} \- s_{F_{k'}}$ corresponds
to $\tau'$ and
$x_{\uu_{k'}}$ is the leading variable of $\bF_{k'}$ for some $\uu_{k'} \in \II_{d,n}^\um$.
Then,  $x_{\fV, (\uu_t,\uv_t)} \mid T^+_{\fV, B}$ and $T^+_{\fV, B}/x_{\fV, (\uu_t,\uv_t)}$
is invertible along $\fV \cap  \tsV_{(\wp_{(k\tau)}\fr_{\rho_{(k\tau)}})}$. 
(Recall the convention: $x_{\fV, (\uu,\uv)} = 1$ if  $(\uu,\uv) \in \fd_{\fV}$.) \end{enumerate}      
\end{prop}
\begin{proof} 
We continue to follow  the notation in the proof of Proposition \ref{meaning-of-var-vskmuh}.

We prove the  proposition by applying induction on
$(k\tau) \mu h  \in ((11)1 0) \sqcup \Index_\Phi$.


The initial case is $(11)1 0$ with $\tsR_{({\wp}_{(11)}\fr_1 \fs_0)}=\tsR_\vt$. In this case, 
the  statement about defining equations
 of  $\tsR_{(\eth_{(11)}\fr_1 \fs_0)} \cap \fV$ follows from
Proposition \ref{eq-for-sV-vtk} with $k=\up$; the remainder statements (1) - (4) are void.

Assume that the proposition holds  for $({\wp}_{(k\tau)}\fr_\mu \fs_{h-1})$
with $(k\tau)\mu h  \in \Index_\Phi$.

Consider $({\wp}_{(k\tau)}\fr_\mu \fs_h)$. 

Consider any standard chart $\fV$ of $\tsR_{({\wp}_{(k\tau)}\fr_\mu \fs_{h})}$,
lying over a  standard chart of $\fV'$ of  $\tsR_{{\wp}_{(k\tau)}\fr_\mu \fs_{h-1})}$.
By assumption, all the desired statements of the proposition hold over the chart $\fV'$.

To begin with, we suppose that the proper transform $Z'_{\phi_{(k\tau)\mu h}}$ 
of the ${\wp}$-center $Z_{\phi_{{\wp}_{(k\tau)}\mu h}}$  in $\tsR_{({\wp}_{(k\tau)}\fr_\mu\fs_{h-1})}$
 does not meet the chart $\fV'$.  
Then, by the proof of Proposition \ref{meaning-of-var-vskmuh},
  we have that $\fV$ retains all the data from those of $\fV'$.
As the ${\wp}$-blowup  along the proper transform of $Z_{\phi_{(k\tau)\mu h}}$
does not affect the chart $\fV'$,  we have that
the statement of the proposition about the defining equations 
 follows immediately from the inductive assumption.
 The statements (1) and (2) are void.

In what follows, we suppose that the proper transform $Z'_{\phi_{(k\tau)\mu h}}$ 
meets the chart $\fV'$ along a nonempty closed subset. 

The  statement of the proposition on the defining equations of $\tsV_{(\wp_{(k\tau)}\fr_\mu \fs_{h})} \cap \fV$
follows straightforwardly from the inductive assumption.

 {\it Proof of (1).  } 


(1a).  We may express 
$$B_{(k\tau)}= x_{(\uu_{s_\tau},\uv_{s_\tau})} x_{\uu_k} -x_{\uu_{s_\tau}} x_{\uv_{s_\tau}} x_{(\um,\uu_k)}$$
where  $x_{\uu_k}$ is the leading variable of $\bF_k$, and $s_\tau \in S_{F_k} \- s_{F_k}$ corresponds
to $\tau \in [\ft_{F_k}]$. 
Observe that none of $\vr$-divisors appear in any  $\wp$-set.
Observe in addition that $x_{\uu_k}$ does not appear in any 
$B \in \cB^\mn_{F_j}$ with $j <k$.
Thus, the fact that the plus-term 
$T^+_{\fV, (k\tau)}$ is square-free is immediate if $\tau=1$.             
For a general $\tau \in [\ft_{F_k}]$,  it follows from the inductive assumption on $T^+_{\fV', (k\tau)}$.
The remainder statements follow from straightforward calculations. We omit the obvious details.

(1b). 
We can write $B=B_{(k\tau')}$ with $\tau' \in [\ft_{F_k}]$ and $\tau' > \tau$.
We may express 
$$B=B_{(k\tau')}= x_{(\uu_{s_{\tau'}},\uv_{s_{\tau'}})} x_{\uu_k} -
x_{\uu_{s_{\tau'}}} x_{\uv_{s_{\tau'}}} x_{(\um,\uu_k)}.$$
We have $T^+_B=x_{(\uu_{s_{\tau'}},\uv_{s_{\tau'}})} x_{\uu_k}$ and it
retains this form prior to the $\wp$-blowups with respect to binomials of $\cB_{F_k}$.
(Recall here the convention: $x_{\fV, \uu} = 1$ if  $\uu \in \fe_{\fV}$;
 $x_{\fV, (\uu,\uv)} = 1$ if  $(\uu,\uv) \in \fd_{\fV}$.)
 Note that the $\vr$-variable $x_{(\uu_{s_{\tau'}},\uv_{s_{\tau'}})}$
 does not appear in any $\vt$-set.
 Thus, as none of $\vr$-divisors appear in any  $\wp$-set, we see that
  all possible exceptional variables from the earlier blowups appearing  in $T^+_{\fV',B}$
are acquired (cf. Definition \ref{general-proper-transforms})
through the {\it unique} $\vp$-variable $x_{\uu_k}$ of $T^+_B$.
Note further that $x_{\uu_k}$ does not appear in any 
$B \in \cB^\mn_{F_j}$ with $j <k$.
 From here, one sees directly  that $T^+_{\fV, (k\tau')}$ is square-free.
Now, suppose $y_1 \mid T^-_{\fV, B}$. We can assume $\deg_{y'_1}  (T^-_{\fV', B})=b$ for some 
integer $b>0$. 
Since $B \in \cB_{F_k}$ and $B > B_{(k\tau)}$, then by the above discussion about $T^+_B$,
 we must have $y_0' \mid  T^+_{\fV, B}$.
Thus, $\deg_{y_1}  (T^-_{\fV, B})=b$
and $\deg_{\zeta}  (T^-_{\fV, B})=b-1$.
Hence, in such a case,
either $T^-_{\fV, B}$ is linear in $y_1$, when $b=1$, 
 or $\zeta \mid T^-_{\fV, B}$ when $b>1$.

(1c). Consider $B \in  \cB^\mn_{F_k}$ with $B < B_{(k\tau)}$
or  $B \in   \cB^\mn \- \cB^\mn_{F_k}$.
We can write $B=B_{(k'\tau')}$ for some $(k'\tau') \ne  (k\tau)$ 
and we may suppose $\tau'$ corresponds to some index 
$t \in S_{F_{k'}}$.  We have $T^+_B= x_{(\uu_t, \uv_t)} x_{\uu_{k'}}$.
Assume $(k'\tau')< (k\tau)$. Then,  by the fact that the variable 
$x_{(\uu_t, \uv_t)}$ uniquely appears in $T^+_B$ among all main binomial equations of $\cB^\mn$
and Proposition \ref{equas-vskmuh}  (4) in the case of $(\wp_{(k'\tau')}\fr_{\rho_{(k'\tau')}})$ 
(which holds by the inductive assumption since  $(k'\tau') <(k\tau)$), we 
conclude that neither of the variables $\{y'_0, y'_1\}$
appears in $T^+_{\fV',B}$. Since $T^+_{\fV',B}$ is square-free by the inductive assumption,
this implies all three of the statements of (1c) when $(k'\tau') < (k\tau)$.
Assume $k' > k$. We have $T^+_B= x_{(\uu_t, \uv_t)} x_{\uu_{k'}}$.
Because $x_{(\uu_t, \uv_t)}$ uniquely appears in $T^+_B$ among all main binomial equations and
the leading variable $x_{\uu_{k'}}$ of $\bF_{k'}$ does not appear in any $\bF_i$ with $i <k'$,
 we again conclude that neither of the variables $\{y'_0, y'_1\}$ 
appears in $T^+_{\fV',B}$.
Hence, again, all three of the statements of (1c) hold as well when $k'>k$.

(1d).  Consider any $B \in \cB^\q$ and let $T_B$ be any fixed term of $B$.
 We let $\fV_\vt$ be the unique standard chart of $\tsR_\vt$ such that $\fV$ lies over $\fV_\vt$.
If $T_B$ does not contain any $\vr$-variable of the form $x_{(\um, \uu)}$ with $\uu \in \II^\um_{d,n}$, 
then we immediately have that $T_{\fV, B}$ is square-free because the  $\vr$-variables of $T_B$
never appear in any $\vt$-set or $\wp$-set.  Now assume that 
$x_{(\um, \uu)} \mid T_B$ for some $\uu \in \II^\um_{d,n}$. 
 By  Lemma \ref{ker-phi-k} (2), applied to the variable $p_\um$,
 such a $\vr$-variable is unique.  By the last statement of Proposition \ref{eq-for-sV-vtk} with $k=\up$  
 about $\cB^\q$, either  $\delta_{\fV_\vt, (\um, \uu)} \mid T_{\fV_\vt,B}$ or 
 $\ve_{\fV_\vt, \uu}x_{\fV_\vt, (\um, \uu)} \mid T_{\fV_\vt,B}$. 
 We can write $\uu=\uu_{F_j}$ for some $j \in [\up]$.
Suppose $j <k$. Then, $y_0'$ can not appear in $T_{\fV', B}$.
Suppose $j=k$. Then, $y_1'$ can not appear in $T_{\fV', B}$.
Suppose $j>k$. Then, neither of $\{y'_0, y'_1\}$ can appear in $T_{\fV', B}$.
In any case, one sees that at most one variable  of $\{y'_0, y'_1\}$ 
may appear in $T_{\fV', B}$.  By the inductive assumption, $T_{\fV', B}$ is square-free. 
Hence, $T_{\fV, B}$ is square-free.  

{\it Proof of (2).  }  

(2a).  The proof of the fact that the plus-term 
$T^+_{\fV, (k\tau)}$ is square-free is totally analogous to the corresponding part of (1a). 
The remainder statements follow from straightforward calculations. We omit the obvious details.

(2b). The fact that $T^+_{\fV, (k\tau)}$ is square-free, again, follows straightforwardly, 
and its proof is analogous to the corresponding proof of (1b).
 Since $B \in \cB_{F_k}$ and $B > B_{(k\tau)}$, we must have $y_0' \mid  T^+_{\fV, B}$,
 thus $y_0' \nmid  T^-_{\fV, B}$.
 Hence,  $y_0 \nmid T^-_{\fV, B}$.

(2c).  This follows from the same line of arguments as in (1c).

(2d). This follows from the same line of arguments as in (1d).

{\it Proof of (3).  }  

First, we observe the following.

By the construction of $\wp$-centers, we always have
$x_{\fV, (\uu_{s_\tau},\uv_{s_\tau})} \mid T^+_{\fV, (k\tau)}$.

The number of variables on any standard chart is a constant.

If (1) occurs, we have (1a): $\deg (T^+_{\fV, (k\tau)}) =\deg (T^+_{\fV', (k\tau)})-1$.

If (2) occurs, we have  (2a): $\deg (T^-_{\fV, (k\tau)}) =\deg (T^-_{\fV', (k\tau)})-1$.

Hence, if (1) keeps occurring, then,
as $\mu$ increases, after  finitely many rounds, we will eventually have 
$T^+_{\fV, (k\tau)}/x_{\fV, (\uu_{s_\tau},\uv_{s_\tau})}$ is invertible over
$\fV \cap \tsV_{({\wp}_{(k\tau)}\fr_\mu)}$, for all $\fV$.
If (2) keeps happening,
as $\mu$ increases, after  finitely many rounds, 
 using Corollary \ref{no-(um,uu)}, we see that $T^-_{\fV, (k\tau)}$ must terminate at all points
 of $\fV \cap \tsV_{({\wp}_{(k\tau)}\fr_\mu)}$,
  hence,  $B_{\fV, (k\tau)}$ terminates over $\fV$,
 for all $\fV$.

The above implies that the process of $\wp$-blowups in ($\wp_{(k\tau)}$)
must terminate after finitely many rounds.
That is, $\rho_{(k\tau)} < \infty$.

{\it Proof of (4).  }  

When $B=B_{(k'\tau')} \in \cB^\mn$ with $(k'\tau') <{(k\tau)}$, the  statement
follows from the corresponding statement on $\tsR_{\wp_{(k'\tau')}\fr_{\rho_{(k'\tau')}}}$ by
the inductive assumption in $({\wp}_{(k' \tau')}\fr_{\rho_{(k'\tau')}})$.

Now consider $B=B_{(k\tau)}$. As earlier, we express
$B_{(k\tau)}= x_{(\uu_{s},\uv_{s})} x_{\uu_k} -x_{\uu_{s}} x_{\uv_{s}} x_{(\um,\uu_k)}$ with
$s=s_\tau \in S_{F_k} \- s_{F_k}$ corresponding to $\tau$.
Consider any standard chart $\fV$ of the {\it final} scheme $\tsR_{\wp_{(k\tau)}\fr_{\rho_{(k\tau)}}}$
in ($\wp_{(k\tau)})$.
Suppose that the standard chart $\fV$ is the result after  (1) occurs. Then,
$B_\fV$ must be of the  form
$$x_{\fV, (\uu_s,\uv_s)} - T^-_{\fV, (k\tau)}.$$
Suppose that $\fV$ is the result after  (2) occurs. Then, 
$B_\fV$ must be  of the  form
$$T^+_{\fV, (k\tau)} -x_{\fV, (\um, \uu_k)}.$$
(Recall here the convention: $x_{\fV, (\uu,\uv)} = 1$ if  $(\uu,\uv) \in \fd_{\fV}$.)
The above implies the statements, by applying Corollary \ref{no-(um,uu)} to the last case.

Therefore, by induction on 
$(k\tau)\mu h \in \{(11)10\} \sqcup \Index_\Phi$,   Proposition \ref{equas-vskmuh} is proved.  
\end{proof}

\section{$\eth$-Blowups}\label{barh-blowups} 

{\it In  this section, we finalize the process of $``removing"$ all possible zero factors of
all the main binomial relations, 
building upon the already performed $\vt$- and ${\wp}$-blowups.}


\subsection{The initial setup: $(\eth_{(11)}\fr_0)$} $\ $


Our initial scheme in $(\eth_{(11)}\fr_0)$ is $\tsR_{(\eth_{11}\fr_0)}:=\tsR_{\wp}$.
Recall that the scheme $\tsR_\wp$ comes equipped with three kinds of divisors:


 $\bullet$   the proper transforms $X_{\wp, \uw}$ of $\vp$-divisors $X_\uw$ for all $\uw \in \II_{d,n}\- \um$;

$\bullet$  the proper transforms $X_{\wp, (\uu,\uv)}$
of $\vr$-divisors $X_{(\uu,\uv)}$ for all $(\uu,\uv) \in \La_\sfm$; 

  $\bullet$   the proper transforms $E_{\wp, (11)0 h}$ of $\vt$-exceptional divisors 
 $E_{\vt, [h]} \subset \tsR_\vt$ with $h \in [\up]$; the proper transforms $E_{\wp, (k\tau)\mu h}$
 of the $\wp$-exceptional divisors $E_{(k\tau)\mu h} \subset \tsR_{(\wp_{(k\tau)}\fr_\mu \fs_h)}$ 
with $(k\tau)\mu h \in  \Index_\Phi$.

  The set of exceptional divisors on $\sR_{\wp}$
can be written as
 $$\cE_{(\eth_{11}\fr_0)}=\{E_{\wp, (k\tau)\mu i} \mid
 (k\tau)\mu i \in \{(11)0i \mid i \in [\up]\} \sqcup \Index_\Phi\}.$$
 This is an ordered set with the cardinality $\vs_{(11)0} := \up + |\Phi|$. Thus, 
 for notational consistency later on, using the order of the set $\cE_{(\eth_{11}\fr_0)}$,
 we can  express it as
  $$\cE_{(\eth_{11}\fr_0)}=\{E_{(\eth_{11}\fr_0), (11)0 h}  \mid h \in [\vs_{(11)0}]\} .$$

\subsection{ $\eth$-centers and
 $\eth$-blowups in $(\eth_{(k\tau)}\fr_\mu)$}\label{notional-system} $\ $

We suppose that the following package in $({\eth}_{(k\tau)}\fr_{\mu-1})$ has been introduced
for some  integer $\mu \in [\vk_{(k\tau)}]$, where
$\vk_{(k\tau)}$ is a  to-be-defined finite positive integer depending on $(k\tau) \in \Index_{\cB^\mn}$.
 Here, to reconcile notations, we make the convention:
$$({\eth}_{(k\tau)}\fr_0):=({\eth}_{(k(\tau -1))}\fr_{\vk_{(k(\tau-1)}}), \; \forall \;\; 1 \le k \le \up, \;
2 \le \tau \le \ft_{F_k},$$
$$({\eth}_{(k1)}\fr_0):=({\eth}_{((k-1)\ft_{F_{k-1}})}\fr_{\vk_{((k-1)\ft_{F_{k-1}})}}), \; \forall \;\; 2 \le k \le \up, $$
provided that $\vk_{(k(\tau-1)}$ and $\vk_{((k-1)\ft_{F_{k-1}})}$
are (proved to be) finite.


\smallskip\noindent
$\bullet$  {\sl The inductive assumption.} 
{\it The scheme $\tsR_{(\eth_{(k\tau)}\fr_{\mu-1})}$ has been constructed;
it comes equipped with the set of $\vp$-divisors, 
$$\cD_{(\eth_{(k\tau)}\fr_{\mu-1}),\vp}: \;\; X_{(\eth_{(k\tau)}\fr_{\mu-1}), \uw}, \;\; \uw \in \II_{d,n} \- \um,$$
the set of $\vr$-divisors,
$$\cD_{(\eth_{(k\tau)}\fr_{\mu-1}), \vr}: \;\; X_{(\eth_{(k\tau)}\fr_{\mu-1}), (\uu,\uv)}, \;\; (\uu, \uv) \in \La_\sfm,$$
 and the set of the exceptional $\vr$-divisors, 
 $$\cE_{(\eth_{(k\tau)}\fr_{\mu-1})}: \;\; E_{(\eth_{(k\tau)}\fr_{\mu-1}), 
(k'\tau')\mu' h'},\; \hbox{$(11)0 \le (k'\tau')\mu'\le (k\tau)(\mu-1), \; h' \in  [\vs_{(k'\tau')\mu'}]$}$$
for some finite positive integer $\vs_{(k'\tau')\mu'}$ depending on $(k'\tau')\mu'$.


We let $$\cD_{(\eth_{(k\tau)}\fr_{\mu-1})}=\cD_{(\eth_{(k\tau)}\fr_{\mu-1}),\vp}
 \sqcup \cD_{(\eth_{(k\tau)}\fr_{\mu-1}),\vr}
\sqcup \cE_{(\eth_{(k\tau)}\fr_{\mu-1})}$$ be the set of all  the aforelisted divisors. 

Fix  any $Y \in \cD_{(\eth_{(k\tau)}\fr_{\mu-1})}$.
 Consider any $B \in \cB^\q \cup \cB^\mn$ and let $T_B$ be any fixed  term of $B$.
Then, we have that $Y$ is associated with $T_B$ with the multiplicity  $m_{Y,T_B}$, a nonnegative integer. 
In what follows, we say $Y$ is associated with $T_B$  if $m_{Y,T_B}>0$; 
we do not say $Y$ is associated with $T_B$  if $m_{Y,T_B}=0$.

 Likewise, for any term of $T_s=\vsgn (s) x_{(\uu_s,\uv_s)}$ of $L_F
=\sum_{s \in S_F} \vsgn (s) x_{(\uu_s,\uv_s)}$, $Y$ is either associated with $T_s$ with
a positive multiplicity $m_{Y,s}$ or not associated with $T_s$ in which case we set $m_{Y,s}=0$.}

We are to construct the scheme $\tsR_{(\eth_{(k\tau)})\mu}$ in the next round $\fr_\mu$. 
This  round  consists of  finitely many steps and the scheme $\tsR_{(\eth_{(k\tau)})\mu}$
is the one obtained in the final step. 


As before, fix any $k \in [\up]$,  we write $\cB^\mn_{F_k}=\{B_{(k\tau)} \mid \tau \in [\ft_{F_k}]\}.$
For every main binomial $B_{(k\tau)}$ of $\cB_{F_k}^\mn$, we have the expression  
$$B_{(k\tau)}=T_{(k\tau)}^+ -T_{(k\tau)}^- =x_{(\uu_s,\uv_s)}x_{\uu_k} -
x_{\uu_s}x_{\uv_s}x_{(\um,\uu_k)} $$  where $ s\in S_{F_k} \- s_{F_k}$ corresponds to $\tau$
and
$x_{\uu_k}$ is the leading variable of $\bF_k$. 

\begin{defn}\label{vr-sets-ktau} Consider the above main binomial relation $B_{(k\tau)} \in \cB^\mn_{F_k}$.
A pre-$\eth$-set  in $(\eth_{(k\tau)}\fr_\mu)$
$$\psi=\{Y^+=X_{(\eth_{(k\tau)}\fr_{\mu-1}), (\uu_s,\uv_s)}, Y^-\}$$
consists of the $\vr$-divisor $Y^+=X_{(\eth_{(k\tau)}\fr_{\mu-1}), (\uu_s,\uv_s)}$ and 
a $\vp$- or exceptional divisor $Y^-$  on the scheme $\tsR_{(\eth_{(k\tau)}\fr_{\mu-1})}$
such that $Y^-$ is associated with $T_{(k\tau)}^-$. 
(Here,  $Y^+$ is automatically associated with $T_{(k\tau)}^+$.
Also observe that $X_{(\eth_{(k\tau)}\fr_{\mu-1}), (\um,\uu_k)}$ is the only $\vr$-divisor
associated with $T^-_{(k\tau)}$, but, we do not use it due to Corollary \ref{no-(um,uu)}.) Given
a pre-$\eth$-set $\psi=\{Y^+, Y^-\}$,
we let $Z_\psi$ be the scheme-theoretic intersection
$$Z_\psi = Y^+ \cap Y^-.$$
The pre-$\eth$-set $\psi$ (resp. $Z_\psi$)  is called an $\eth$-set
(resp. an $\eth$-center)   in $(\eth_{(k\tau)}\fr_\mu)$ if 
$$Z_\psi  \cap \tsV_{(\eth_{(k\tau)}\fr_{\mu-1})} \ne \emptyset.$$
\end{defn}

As there are only finitely many $\vp$-, $\vr$-, and exceptional divisors
on the scheme $\tsR_{(\eth_{(k\tau)}\fr_{\mu-1})}$, that is,
the set $\cD_{(\eth_{(k\tau)}\fr_{\mu-1})}$ is finite,
one sees that there are only finitely many
$\eth$-sets in $(\eth_{(k\tau)}\fr_\mu)$.  We let $\Psi_{(k\tau)\mu}$ be the finite set of all $\eth$-sets
in $(\eth_{(k\tau)}\fr_\mu)$. 
We let $\cZ_{\eth_{(k\tau)}\fr_\mu}$ be the finite set of the corresponding 
$\eth$-centers in $(\eth_{(k\tau)}\fr_\mu)$. 
We need a total ordering on the finite set
$\Psi_{(k\tau)\mu}$ to proceed.

\begin{defn}\label{order-Psi-ktau} Consider the set 
$\cD_{(\eth_{(k\tau)}\fr_{\mu-1})}=\cD_{(\eth_{(k\tau)}\fr_{\mu-1}),\vp} \sqcup \cD_{(\eth_{(k\tau)}\fr_{\mu-1}),\vr}
\sqcup \cE_{(\eth_{(k\tau)}\fr_{\mu-1})}$. 
We introduce a total order $``<_\flat"$ on the set as follows. 
  Consider any two distinct elements, $Y_1, Y_2 \in  \cD_{(\eth_{(k\tau)}\fr_{\mu-1})}$. 
 \begin{enumerate} 
  \item  Suppose $Y_1 \in \cE_{(\eth_{(k\tau)}\fr_{\mu-1})}$ and $Y_2  \notin  \cE_{(\eth_{(k\tau)}\fr_{\mu-1})}$. Then, $Y_1 <_\flat Y_2$.
  \item  Suppose $Y_1 \in \cD_{(\eth_{(k\tau)}\fr_{\mu-1}),\vr}$ and $Y_2  \in \cD_{(\eth_{(k\tau)}\fr_{\mu-1}),\vp}$.
   Then, $Y_1 <_\flat Y_2$.
   \item Suppose $Y_1, Y_2 \in \cE_{(\eth_{(k\tau)}\fr_{\mu-1})}$. 
   Then, we can write $Y_i=E_{(\eth_{(k\tau)}\fr_{\mu-1}), (k'_i\tau_j)\mu_jh'_i}$ for some
   $(k_j\tau_j)\mu_j\le (k\tau)(\mu-1)$ and some $1 \le h'_i\le \vs_{(k_j\tau_j)\mu_j}$, with $i=1,2$.
   Then, we say $Y_1 <_\flat Y_2$  if $(k_1',\tau'_1, \mu_1', h_1')<_\invlex (k_2',\tau'_2, \mu_2', h_2')$.
  \item  Suppose $Y_1, Y_2 \in \cD_{(\eth_{(k\tau)}\fr_{\mu-1}),\vp}$. Write $Y_1= X_{(\eth_{(k\tau)}\fr_{\mu-1}), \uu}$ and
   $Y_2=X_{(\eth_{(k\tau)}\fr_{\mu-1}), \uv}$.
  Then, $Y_1 <_\flat Y_2$ if  $\uu <_\lex \uv$.
  \item Suppose $Y_1, Y_2 \in \cD_{(\eth_{(k\tau)}\fr_{\mu-1}),\vr}$. Write $Y_i= 
  X_{(\eth_{(k\tau)}\fr_{\mu-1}),(\uu_i,\uv_i)}$ 
  with $i=1,2$. 
  Then, $Y_1 <_\flat Y_2$ if  $(\uu_1,\uv_1) <_\lex (\uu_2,\uv_2)$.
 \end{enumerate}
\end{defn}

Using the order  $``<_\flat"$ of Definition \ref{order-Psi-ktau}, 
and applying Definition \ref{gen-order}, 
one sees that  the set $\Psi_{(k\tau)}$ now
comes equipped with  the total order $``<:=<_{\flat, \lex}"$ which is the lexicographic order
induced by $``<_\flat"$. 
Thus, we can list $\Psi_{(k\tau)\mu}$ as
$$\Psi_{(k\tau)\mu}=\{\psi_{(k\tau)\mu 1} < \cdots < \psi_{(k\tau)\mu\vs_{(k\tau)\mu}}\}$$
for some finite positive integer $\vs_{(k\tau)\mu}$ depending on $(k\tau)\mu$. We then let the set 
$\cZ_{\eth_{(k\tau)}\fr_\mu}$ of the corresponding $\eth$-centers
inherit the total order from that of $\Psi_{(k\tau)\mu}$. 
Then, we can express
$$\cZ_{\eth_{(k\tau)}\fr_\mu}=\{Z_{\psi_{(k\tau)\mu1}} < \cdots < Z_{\psi_{(k\tau)\mu\vs_{(k\tau)\mu}}}\}.$$


We let $\tsR_{(\eth_{(k\tau)}\fr_\mu \fs_1)} \lra \tsR_{(\eth_{(k\tau)}\fr_{\mu-1})}$ 
be the blowup of $\tsR_{(\eth_{(k\tau)}\fr_{\mu-1})}$
along the $\eth$-center $Z_{\psi_{(k\tau)\mu1}}$. 
Inductively, we assume that 
$\tsR_{(\eth_{(k\tau)}\fr_\mu\fs_{(h-1)})}$ has been constructed for some 
$ h \in [\vs_{(k\tau)\mu}]$.
We then let  $$\tsR_{(\eth_{(k\tau)}\fr_\mu\fs_h)} \lra \tsR_{(\eth_{(k\tau)}\fr_\mu\fs_{h-1})}$$ be the blowup of
$ \tsR_{(\eth_{(k\tau)}\fr_\mu\fs_{h-1})}$ along (the proper transform of) the $\eth$-center 
$Z_{\psi_{(k\tau)\mu h}}$. Here, to reconcile notation, we set $\tsR_{(\eth_{(k\tau)}\fr_\mu\fs_{0})}:=\tsR_{(\eth_{(k\tau)}\fr_{\mu-1})}.$

All of these can be summarized as the sequence 
$$\tsR_{(\eth_{(k\tau)}\fr_\mu)}:= \tsR_{(\eth_{(k\tau)}\fr_\mu\fs_{\vs_{(k\tau)\mu}})}
 \lra \cdots \lra \tsR_{(\eth_{(k\tau)}\fr_\mu\fs_1)} \lra \tsR_{(\eth_{(k\tau)}\fr_\mu \fs_{0})}:
 =\tsR_{(\eth_{(k\tau)}\fr_{\mu-1})}.$$

 For any $h \in [\vs_{(k\tau)\mu}]$,
consider the induced  morphism
$\tsR_{(\eth_{(k\tau)}\fr_\mu\fs_h)} \lra \tsR_{(\eth_{(k\tau)}\fr_{\mu-1})}$. 

$\bullet$ We let $X_{(\eth_{(k\tau)}\fr_\mu\fs_h), \uw}$ be the proper transform
of $X_{(\eth_{(k\tau)}\fr_{\mu-1}), \uw}$ in $\tsR_{(\eth_{(k\tau)}\fr_\mu\fs_h)}$,
for all $\uw \in \II_{d,n} \- \um$. These are still called $\vp$-divisors.
We denote the set of all $\vp$-divisors
 on  $\tsR_{(\eth_{(k\tau)}\fr_\mu\fs_h)}$ by  $\cD_{(\eth_{(k\tau)}\fr_\mu\fs_h),\vp}$.
 
$\bullet$ We let $X_{(\eth_{(k\tau)}\fr_\mu\fs_h), (\uu, \uv)}$ be the proper transform
  of the $\vr$-divisor $X_{(\eth_{(k\tau)}\fr_{\mu-1}), (\uu, \uv)}$ in $\tsR_{(\eth_{(k\tau)}\fr_\mu\fs_h)}$,
  for all $(\uu, \uv) \in \La_\sfm$. These are still called $\vr$-divisors.
We denote the set of all $\vr$-divisors
 on  $\tsR_{(\eth_{(k\tau)}\fr_\mu\fs_h)}$ by  $\cD_{(\eth_{(k\tau)}\fr_\mu\fs_h),\vr}$.
 
 $\bullet$ We let
 $ E_{(\eth_{(k\tau)}\fr_\mu\fs_h),  (k'\tau')\mu' h'}$ be the proper transform 
 of $E_{(\eth_{(k\tau)}\fr_{\mu-1}), (k'\tau')\mu' h'}$ in $\tsR_{(\eth_{(k\tau)}\fr_\mu \fs_h)}$,
 for all $(11)0 \le (k'\tau')\mu' \le (k\tau)(\mu-1)$ and $h' \in [\vs_{(k'\tau')\mu'}]$.
  (Here, $E_{(\eth_{(k\tau)}\fr_{\mu-1}), (11)0 h}$ with $h \in [\vs_{(11)0}]$ correspond
 to $\wp$-divisors, where $\vs_{(11)0}=\up +|\Phi|$.). These are named as $\eth$-exceptional or simply
 exceptional divisors.  We denote the set of these exceptional divisors on  
  $\tsR_{(\eth_{(k\tau)}\fr_\mu\fs_h)}$ by $\cE_{(\eth_{(k\tau)}\fr_\mu \fs_h),{\rm old}}$.
 
  We let 
  $$\bar\cD_{(\eth_{(k\tau)}\fr_\mu\fs_h)}= \cD_{(\eth_{(k\tau)}\fr_\mu\fs_h),\vp}
  \sqcup \cD_{(\eth_{(k\tau)}\fr_\mu\fs_h),\vr} \sqcup \cE_{(\eth_{(k\tau)}\fr_\mu\fs_h),{\rm old}}$$ 
  be the set of  all of the aforementioned divisors on $\tsR_{(\eth_{(k\tau)}\fr_\mu \fs_h)}$.
As each divisor in $\bar\cD_{(\eth_{(k\tau)}\fr_\mu\fs_h)}$ is the proper transform of a unique 
divisor in $\cD_{(\eth_{(k\tau)}\fr_{\mu-1})}$. The total order on $\cD_{(\eth_{(k\tau)}\fr_{\mu-1})}$
carries over to provide an induced total order
on the set $\bar\cD_{(\eth_{(k\tau)}\fr_\mu\fs_h)}$.

 In addition to the proper transforms of the divisors from $\tsR_{(\eth_{(k\tau)}\fr_{\mu-1})}$, there are
the following {\it new} exceptional divisors.

For any $ h \in [\vs_{(k\tau)\mu}]$, we let $E_{(\eth_{(k\tau)}\fr_\mu\fs_h)}$ be the exceptional divisor of 
the blowup $\tsR_{(\eth_{(k\tau)}\fr_\mu\fs_h)} \lra \tsR_{(\eth_{(k\tau)}\fr_\mu\fs_{h-1})}$.
 For any $1 \le h'< h \le \vs_{(k\tau)\mu}$,
we let $E_{(\eth_{(k\tau)}\fr_\mu\fs_h), (k\tau)\mu h'}$ 
be the proper transform in $\tsR_{(\eth_{(k\tau)}\fr_\mu \fs_h)} $
of the exceptional divisor $E_{(\eth_{(k\tau)}\fr_\mu \fs_{h'})}$. 
To reconcile notation, we also set 
$E_{(\eth_{(k\tau)}\fr_\mu\fs_h), (k\tau)\mu h}:=E_{(\eth_{(k\tau)}\fr_\mu \fs_h)}$.
We  set $$\cE_{(\eth_{(k\tau)}\fr_\mu \fs_h),{\rm new}}=\{ E_{(\eth_{(k\tau)}\fr_\mu\fs_h), 
(k\tau)\mu h'} \mid 1 \le h' \le h \le \vs_{(k\tau)\mu} \} .$$
We then order the exceptional divisors of $\cE_{(\eth_{(k\tau)}\fr_\mu\fs_h),{\rm new}}$
in the reverse order of occurrence, that is, 
$E_{(\eth_{(k\tau)}\fr_\mu\fs_h), (k\tau)\mu h''} \le E_{(\eth_{(k\tau)}\fr_\mu\fs_h), (k\tau)\mu h'}$
 if $h'' \ge h$.
 
 We then let 
 $$\cD_{(\eth_{(k\tau)}\fr_\mu \fs_h)} =\bar\cD_{(\eth_{(k\tau)}\fr_\mu \fs_h)} 
\sqcup  \cE_{(\eth_{(k\tau)}\fr_\mu \fs_h),{\rm new}}.$$
For any $Y_1 \in \cE_{(\eth_{(k\tau)}\fr_\mu \fs_h),{\rm new}}$ and any
 $Y_2 \in \bar\cD_{(\eth_{(k\tau)})\fr_\mu\fs_h}$,
we say $Y_1 <_\flat Y_2$.  As each of $\cE_{(\eth_{(k\tau)}),{\rm new}}$ 
and $\bar\cD_{(\eth_{(k\tau)})\fs_h}$ is a totally ordered set, this endows $\cD_{(\eth_{(k\tau)}\fr_\mu \fs_h)}$
with a total order $``<_\flat"$.

Finally, we set
 $\tsR_{(\eth_{(k\tau)})\fr_\mu}:= \tsR_{(\eth_{(k\tau)}\fr_\mu\fs_{\vs_{(k\tau)\mu}})}$, and let
 $$\cD_{(\eth_{(k\tau)}\fr_\mu),\vp}=\cD_{(\eth_{(k\tau)}\fr_\mu\fs_{\vs_{(k\tau)\mu}}),\vp}, \;
 \cD_{(\eth_{(k\tau)}\fr_\mu),\vr}=\cD_{(\eth_{(k\tau)}\fr_\mu\fs_{\vs_{(k\tau)\mu}}),\vr},
 $$
 $$\cE_{\eth_{(k\tau)}\fr_\mu}=\cE_{(\eth_{(k\tau)}\fr_\mu\fs_{\vs_{(k\tau)\mu}}), {\rm old}}  
 \sqcup \cE_{(\eth_{(k\tau)}\fr_\mu\fs_{\vs_{(k\tau)\mu}}), {\rm new}} $$
 This can be  summarized as 
$$\cD_{(\eth_{(k\tau)}\fr_\mu)}=\cD_{(\eth_{(k\tau)}\fr_\mu),\vp} \sqcup \cD_{(\eth_{(k\tau)}\fr_\mu),\vr}
\sqcup \cE_{(\eth_{(k\tau)}\fr_\mu)}.$$
 This way,  we have equipped 
  the scheme  $\tsR_{(\eth_{(k\tau)}\fr_\mu)}$ in $(\eth_{(k\tau)}\fr_\mu)$ with
 the set $\cD_{(\eth_{(k\tau)}\fr_\mu),\vp}$ of $\vp$-divisors, 
 the set $\cD_{(\eth_{(k\tau)}\fr_\mu),\vr}$ of $\vp$-divisors, and the set 
 $\cE_{\eth_{(k\tau)}\fr_\mu}$ of exceptional divisors, as required  by induction.


Now, we are ready to introduce the notion of $``$association$"$ in  $(\eth_{(k\tau)}\fr_\mu)$,
as required to carry on the process of induction. 

We do it inductively  on the set $[\vs_{(k\tau)\mu}]$



\begin{defn} \label{association} 
Fix any $B \in \cB^\q \cup \cB^\mn$. We let $T_B$ be any fixed term of the binomial $B$.
Meanwhile, we also fix any $\bF \in \sfm$ and let $T_s$ be the fixed term of $L_F$ corresponding to
some $s \in S_F$.

We assume that the notion of $``$association$"$ in  
$(\eth_{(k\tau)}\fr_\mu\fs_{h-1})$ has been introduced. That is, for every divisor 
$Y' \in \cD_{(\eth_{(k\tau)}\fr_\mu\fs_{h-1})}$, the multiplicities
$m_{Y', T_B}$ and $m_{Y',s}$ have been defined.

Consider an arbitrary divisor $Y \in \cD_{(\eth_{(k\tau)}\fr_\mu\fs_h)}$.

First, suppose $Y \ne E_{(\eth_{(k\tau)}\fr_\mu\fs_h)}$.
Then, it is the proper transform of a (unique) divisor $Y' \in \cD_{(\eth_{(k\tau)}\fr_\mu\fs_{h-1})}$.
In this case, we set 
$$m_{Y, T_B}=m_{Y', T_B},\;\;\; m_{Y,s}=m_{Y',s}.$$

Next, we consider the exceptional $Y=E_{(\eth_{(k\tau)}\fr_\mu\fs_h)}$.

We let $\psi= \psi_{(k\tau)\mu h}$. 
We have that $$\psi=\{ Y^+,  Y^- \} \subset \cD_{(\eth_{(k\tau)}\fr_{\mu-1})}.$$

For any $B \in \cB^\q \cup \cB^\mn$, we write $B=T_B^0-T_B^1$. We let
$$ m_{\psi, T_B^i}=m_{Y^+, T_B^i}+ m_{Y^-, T_B^i}, \; i=0,1, \;\; \hbox{and}$$
 $$l_{\psi, B}=\min \{m_{\psi, T_B^0}, m_{\psi, T_B^1}\}.$$ 
 (For instance, by definition, $l_{\psi, B}>0$, when $B=B_{(k\tau)}$.)
 Then, we let
  $$m_{E_{(\eth_{(k\tau)}\fr_\mu\fs_h)},T_B^i}= m_{\psi, T_B^i}- l_{\psi, T_B^i}.$$

Likewise, 
for $s \in S_F$ with $\bF \in \sfm$,
we let
$$m_{E_{(\eth_{(k\tau)}\fr_\mu\fs_h)},s}= m_{Y^+, s}+ m_{Y^-,s}.$$

We say $Y$ is associated with $T_B$ (resp $T_s$) if its multiplicity $m_{Y, T_B}$ 
(resp. $m_{Y,s}$)  is positive.
We do not say $Y$ is associated with $T_B$ (resp $T_s$)  
if its multiplicity $m_{Y,T_B}$ (resp. $m_{Y,s}$)  equals to zero.
\end{defn}

When $h = \vs_{(k\tau)\mu}$, we obtain all the desired data on 
$\tsR_{(\eth_{(k\tau)}\fr_\mu)}=\tsR_{(\eth_{(k\tau)}\fr_\mu)\fs_{\vs_{(k\tau)\mu}}}$.

Now, with all the aforedescribed data equipped for $\tsR_{(\eth_{(k\tau)}\fr_\mu)}$, 
we obtain our inductive package in $(\eth_{(k\tau)}\fr_\mu)$. 
This allows us to introduce the $\eth$-sets 
 $\Psi_{\eth_{(k\tau)}\fr_{\mu+1}}$
 and $\eth$-centers $\cZ_{\eth_{(k\tau)}\fr_{\mu+1}}$
 in  $(\eth_{(k\tau)}\fr_{\mu+1})$  as in Definition \ref{vr-sets-ktau}, 
 endow total orders on $\Psi_{\eth_{(k\tau)}\fr_{\mu+1}}$
 and  $\cZ_{\eth_{(k\tau)}\fr_{\mu+1}}$ as in the paragraph immediately following
 Definition \ref{order-Psi-ktau}, 
 and then  advance to the next 
round of  the $\wp$-blowups.  Here, 
 to reconcile notations, we set
  $$({\eth}_{(k\tau)}\fr_{\vk_{(k\tau)+1}}):=({\eth}_{((k(\tau +1))}\fr_{1}), \;\; 1 \le \tau <\ft_{F_k};$$
$$({\eth}_{(k\ft_{F_k})}\fr_{\vk_{(k\ft_{F_k})+1}}):=({\eth}_{((k+1)1)}\fr_{1}), \;\;  1 \le k < \up,$$
provided that $\vk_{(k\tau)}$ and $\vk_{(k\ft_{F_k})}$ are (proved to be) finite.

For any $(k\tau)\mu h\in \Index_\Psi$, the $\eth$-blowup in ($\eth_{(k\tau)}\fr_\mu\fs_h$)
gives rise to \begin{equation}\label{eth-ktauh} \xymatrix{
\tsV_{(\eth_{(k\tau)}\fr_\mu\fs_h)} \ar[d] \ar @{^{(}->}[r]  & \tsR_{(\eth_{(k\tau)}\fr_\mu\fs_h)} \ar[d] \\
\sV \ar @{^{(}->}[r]  & \sR_\sF,  }
\end{equation}
where $\tsV_{(\eth_{(k\tau)}\fr_\mu\fs_h)} $ is the proper transform of $\sV$ 
in  $\tsR_{(\eth_{(k\tau)}\fr_\mu\fs_h)} $.

We let $\tsV_{(\eth_{(k\tau)}\fr_\mu)}=\tsV_{(\eth_{(k\tau)}\fr_\mu\fs_{\vs_{(k\tau)\mu}})}$.

\begin{defn}\label{defn:vk-ktau}
Fix any $k \in [\up], \tau \in [\ft_{F_k}]$. Suppose there exists a finite integer $\mu$ such that
for any  pre-$\eth$-set $\psi$ 
in $(\eth_{(k\tau)}\fr_{\mu +1})$ 
(cf. Definition \ref{vr-sets-ktau}),
we have 
$$Z_\psi \cap \tsV_{(\eth_{(k\tau)}\fr_\mu)}= \emptyset.$$
Then, we let $\vk_{(k\tau)}$ be the smallest integer such that the above holds.
Otherwise, we let $\vk_{(k\tau)}=\infty$.
\end{defn}
{\it It will be shown soon 
that $\vk_{(k\tau)}$ is finite for all  $k \in [\up], \tau \in [\ft_{F_k}]$.}

For later use, we let
\begin{eqnarray}\label{indexing-Psi} 
\Psi=\{ \psi_{(k\tau)\mu h} \mid (k\tau)\in \Index_{\cB^\mn}, 1 \le \mu \le \vk_{(k\tau)},
 h \in [\vs_{(k\tau)\mu}]\},  \\ 
\Index_\Psi=\{(k\tau)\mu h \mid (k\tau)\in \Index_{\cB^\mn}, 1 \le \mu \le \vk_{(k\tau)},
 h \in [\vs_{(k\tau)\mu}]\}. \nonumber  
\end{eqnarray}

Upon proving that $\vk_{(k\tau)}$ is finite  for all for all $(k\tau) \in \Index_{\cB^\mn}$ later on,
 we can summarize the process
of $\eth$-blowups as a single sequence of blowup morphisms:
\begin{equation}\label{grand-sequence-vr}
\tsR_\eth \lra \cdots \lra 
\tsR_{(\eth_{(k\tau)}\fr_\mu\fs_h)} \lra \tsR_{(\eth_{(k\tau)}\fr_\mu\fs_{h-1})}
\lra \cdots \lra \tsR_{(\eth_{(11)}\fr_0)}:=\tsR_{\wp},\end{equation}
where $\tsR_\eth:=\tsR_{(\eth_{(\up \ft_{\up})}\fr_{\vk_{\up \ft_{\up}}})} 
:=\tsR_{(\eth_{(\up \ft_{\up})}\fr_{\vk_{\up \ft_{\up}}}\fs_{\vs_{(\up \ft_\up)\vk_{\up \ft_\up}}})} $
is the  blowup scheme reached in  the final step 
$(\eth_{(\up \ft_{\up})}\fr_{\vk_{\up \ft_{\up}}}\fs_{\vs_{(\up \ft_\up)\vk_{\up \ft_\up}}})$ of all $\eth$-blowups.

The end of all $\eth$-blowups gives rise to the following induced diagram 
\begin{equation}\label{eth-final} \xymatrix{
\tsV_\eth \ar[d] \ar @{^{(}->}[r]  & \tsR_\eth \ar[d] \\
\sV \ar @{^{(}->}[r]  & \sR_\sF,
}
\end{equation}
where $\tsV_\eth$ is the proper transform of $\sV$  in  $\tsR_\eth$.



\subsection{Properties of $\eth$-blowups}\label{subs:prop-wp-blowups} $\ $

To begin with, recall that we have set
$\tsR_{(\eth_{11}\fr_1\fs_{0})}:=\tsR_{\wp}.$

Fix and consider any $(k\tau)\mu h \in \Index_\Psi$ (see \eqref{indexing-Psi}).

\begin{prop}\label{meaning-of-var-p-k} 
Suppose that  the scheme $ \tsR_{(\eth_{(k\tau)}\fr_\mu\fs_{h})}$ has been constructed.

Consider any standard chart $\fV$ of $\tsR_{(\eth_{(k\tau)}\fr_\mu\fs_h)}$, 
 lying over a unique chart $\fV_{[0]}$ of 
 $\sR_\sF$.
 We suppose that the chart $ \fV_{[0]}$ is indexed by
$\La_\sfm^o=\{(\uv_{s_{F,o}},\uv_{s_{F,o}}) \mid \bF \in \sfm \}.$  
As earlier, we have $\La_\sfm^\star=\La_\sfm \- \La_\sfm^o$.

Then, the chart $\fV$ comes equipped with 
$$\hbox{a subset}\;\; \fe_\fV  \subset \II_{d,n} \- \um \;\;
 \hbox{and a subset} \;\; \fd_\fV  \subset \La_{\sfm}^\star$$
such that every exceptional divisor $E$ 
(i.e., not a $\vp$- nor a $\vr$-divisor)  of $\tsR_{(\eth_{(k\tau)}\fr_\mu\fs_h)}$
with $E \cap \fV \ne \emptyset$ is 
either labeled by a unique element $\uw \in \fe_\fV$
or labeled by a unique element $(\uu,\uv) \in \fd_\fV$. 
We let $E_{(\eth_{(k\tau)}\fr_\mu\fs_h), \uw}$ be the unique exceptional divisor 
on the chart $\fV$ labeled by $\uw \in \fe_\fV$; we call it an $\vp$-exceptional divisor.
We let $E_{(\eth_{(k\tau)}\fr_\mu\fs_h), (\uu,\uv)}$ be the unique exceptional divisor 
on the chart $\fV$ labeled by $(\uu,\uv) \in \fd_\fV$;  we call it an $\vr$-exceptional divisor.
(We note here that being $\vp$-exceptional or $\vr$-exceptional is strictly relative to the given
standard chart.)

Further, the chart $\fV$ admits 
 a set of free variables
\begin{equation}\label{variables-p-k} 
\var_{\fV}:=\left\{ \begin{array}{ccccccc}
\ve_{\fV, \uw} , \;\; \de_{\fV, (\uu,\uv) }\\
x_{\fV, \uw} , \;\; x_{\fV, (\uu,\uv)}
\end{array}
  \; \Bigg| \;
\begin{array}{ccccc}
 \uw \in  \fe_\fV,  \;\; (\uu,\uv)  \in \fd_\fV  \\ 
\uw \in  \II_{d,n} \- \um \- \fe_\fV,  \;\; (\uu, \uv) \in \La_\sfm^\star \-  \fd_\fV  \\
\end{array} \right \},
\end{equation}
such that it is canonically  isomorphic to the affine space with the variables in
\eqref{variables-p-k} as its coordinate variables. Moreover, on the standard chart $\fV$, we have
\begin{enumerate}
\item the divisor  $X_{(\eth_{(k\tau)}\fr_\mu\fs_h), \uw}\cap \fV$ 
 is defined by $(x_{\fV,\uw}=0)$ for every 
$\uw \in \II_{d,n} \setminus \um \- \fe_\fV$;
\item the divisor  $X_{(\eth_{(k\tau)}\fr_\mu\fs_h), (\uu,\uv)}\cap \fV$ is defined by $(x_{\fV,(\uu,\uv)}=0)$ for every 
$(\uu,\uv) \in \La^\star_\sfm\- \fd_\fV$;
\item the divisor  $X_{(\eth_{(k\tau)}\fr_\mu\fs_h), \uw}\cap \fV$ does not intersect the chart for all $\uw \in \fe_\fV$;
\item the divisor  $X_{(\eth_{(k\tau)}\fr_\mu\fs_h), (\uu, \uv)}$ does not intersect the chart for all $ (\uu, \uv) \in \fd_\fV$;
\item the $\vp$-exceptional divisor 
$E_{(\eth_{(k\tau)}\fr_\mu\fs_h), \uw} \;\! \cap  \fV$  labeled by an element $\uw \in \fe_\fV$
is define by  $(\ve_{\fV,  \uw}=0)$ 
for all $ \uw \in \fe_\fV$;
\item the $\vr$-exceptional divisor 
$E_{(\eth_{(k\tau)}\fr_\mu\fs_h),  (\uu, \uv)}\cap \fV$ labeled by  an element $(\uu, \uv) \in \fd_\fV$
is define by  $(\de_{\fV,  (\uu, \uv)}=0)$ 
for all $ (\uu, \uv) \in \fd_\fV$;
\item any of the remaining exceptional divisor
of $\tsR_{(\eth_{(k\tau)}\fr_\mu\fs_h)}$
other than those that are labelled by some  some $\uw \in \fe_\fV$ or $(\uu,\uv) \in \fd_\fV$ 
 does not intersect the chart.
\end{enumerate}
\end{prop}
\begin{proof} 
 We prove by induction on $(k\tau)\mu h \in \{(11)10 \} \sqcup \Index_\Psi$.
(The proof is basically parallel to that of Proposition \ref{meaning-of-var-vskmuh}.
We provide details for completeness.)

For the initial case, the scheme is $\tsR_{(\eth_{11}\fr_1\fs_0)}:=\sR_{{\wp}}$, 
the  scheme obtained in the final step of ${\wp}$-blowups.
Then, this proposition is the same as Proposition \ref{meaning-of-var-vskmuh} for $\sR_{{\wp}}$.
Thus, it holds.

We suppose that the statement holds over $\tsR_{(\eth_{(k\tau)}\fr_\mu\fs_{h-1})}$
with $(k\tau)\mu h \in  \Index_\Psi$.

We now consider $\tsR_{(\eth_{(k\tau)}\fr_\mu\fs_h)}$. We have the embedding
$$\xymatrix{
\tsR_{(\eth_{(k\tau)}\fr_\mu\fs_h)}\ar @{^{(}->}[r]  &
 \tsR_{(\eth_{(k\tau)}\fr_\mu\fs_{h-1})} \times \PP_{\psi_{(k\tau)\mu h}},  }
$$
where $\PP_{\psi_{(k\tau)\mu h}}$ ($\cong \PP^1$)
is the factor projective space. The $\eth$-set $\psi_{(k\tau)\mu h}$
consists of two divisors $\{Y^+, Y^-\}$ on 
$\tsR_{(\eth_{(k\tau)}\fr_{\mu-1})}=\tsR_{(\eth_{(k\tau)}\fr_{\mu-1}\fs_{\vs_{(k\tau)(\mu-1)}})}$. We let 
$Y'_0$ and $Y'_1$ be the proper transforms in $\tsR_{(\eth_{(k\tau)}\fr_\mu\fs_{h-1})}$
of the two divisors in the $\eth$-set $\psi_{(k\tau)\mu h}=\{Y^+, Y^-\}$.
In addition, we let $[\xi_0,\xi_1]$ be the homogenous coordinates of $\PP_{\psi_{(k\tau)\mu h}}$,
corresponding to $(Y_0',Y_1')$.

Fix any standard chart $\fV'$ of  $\tsR_{(\eth_{(k\tau)}\fr_\mu\fs_{h-1})}$, 
let $i$ be any of 0 or 1
and set $(\xi_i \equiv 1)$. Then, we obtain 
a standard  chart, $\fV=(\fV' \times (\xi_i \equiv 1)) \cap \tsR_{(\eth_{(k\tau)}\fr_\mu\fs_h)}$,
of $\tsR_{(\eth_{(k\tau)}\fr_\mu\fs_h)}$. 
Any standard chart of $\tsR_{(\eth_{(k\tau)}\fr_\mu\fs_h)}$ is such a form. As $\fV$ is assumed to lie over
$ \fV_{[0]}$, so does $\fV'$. 
By assumption, the chart $\fV'$ comes equipped with
a subset $\fe_{\fV'} \subset \II_{d,n}\- \um$, a subset  $\fd_{\fV'} \subset \La^\star_\sfm$,
and admits a set of coordinate variables
\begin{equation}\label{variables-p-(k-1)-eth} 
\var_{\fV'}:=\left\{ \begin{array}{ccccccc}
\ve_{\fV', \uw} , \;\; \de_{\fV', (\uu,\uv) }\\
x_{\fV', \uw} , \;\; x_{\fV', (\uu,\uv)}
\end{array}
  \; \Bigg| \;
\begin{array}{ccccc}
 \uw \in  \fe_{\fV'},  \;\; (\uu,\uv)  \in \fd_{\fV'}  \\ 
\uw \in  \II_{d,n} \- \um \- \fe_{\fV'},  \;\; (\uu, \uv) \in \La_\sfm^\star \-  \fd_{\fV'}  \\
\end{array} \right \},
\end{equation}
verifying the properties (1)-(7) as in the proposition.

We now prove the statements for the chart $\fV$ of $\tsR_{(\eth_{(k\tau)}\fr_\mu\fs_h)}$.

First, we suppose that the proper transform $Z'_{\psi_{(k\tau)\mu h}}$ 
in $\tsR_{(\eth_{(k\tau)}\fr_\mu\fs_{h-1})}$
of the $\eth$-center $Z_{\psi_{(k\tau)\mu h}}$  does not
meet the chart $\fV'$.  
Then, we let $\fV$ inherit all the data from those of $\fV'$, that is,
we set $\fe_{\fV}=\fe_{\fV'}$, $\fd_{\fV}=\fd_{\fV'}$, and $\var_\fV= \var_{\fV'}$:
changing the subindex $``\ \fV' \ "$ for all the variables in $\var_{\fV'}$
to $``\ \fV \ "$.
As the $\eth$-blowup  along the proper transform $Z'_{\psi_{(k\tau)\mu h}}$
does not affect the chart $\fV'$, one sees that 
the statements of the proposition hold for $\fV$.

Next, we suppose that the proper transform $Z'_{\psi_{(k\tau)\mu h}}$ 
meets the chart $\fV'$ along a nonempty closed subset. 
That is, we have $Y'_0 \cap Y'_1 \cap \fV' \ne \emptyset$.
On the chart $\fV'$, by the inductive assumption, we can suppose
$$Y'_j \cap \fV' =(y'_j =0),\;\;
\hbox{ for some $y'_j \in \var_{\fV'}$ with $j=0,1$.} $$
Fix and consider the chart $(\xi_i \equiv 1)$ for some $i=0$ or 1. 
We let $j \in \{0,1\} \- i$. 
Then, the chart  $\fV=(\fV' \times (\xi_i \equiv 1)) \cap \tsR_{(\eth_{(k\tau)}\fr_\mu\fs_h)}$ 
of the scheme $\tsR_{(\eth_{(k\tau)}\fr_\mu\fs_{h})}$, as a closed subscheme of  
$\fV' \times (\xi_i \equiv 1),$ is defined by
\begin{equation}\label{proof:sf-fv-k}
y'_{j}= y'_i \xi_{j}.
\end{equation}

There are 4 possibilities for
$Y'_i \cap \fV'$ according to the assumption on the chart $\fV'$.
Based on every of such possibilities, 
we set 
\begin{equation}\label{proof:de-fv-k}
\left\{ 
\begin{array}{lccr}
\fe_{\fV}=\fe_{\fV'} \sqcup \{\uw\}, \; \fd_{\fV}= \fd_{\fV'},
&  \hbox{if} \; y'_i=x_{\fV', \uw}\; \hbox{for some} \;\; \uw \in \II_{d,n}\-\um \- \fe_{\fV'}\\
\fe_{\fV}= \fe_{\fV'}, \fd_{\fV}=\fd_{\fV'} \sqcup \{(\uu, \uv)\},  & \hbox{if} \; 
y'_i=x_{\fV', (\uu,\uv)} \; \hbox{with}  \ (\uu,\uv) \in \La_\sfm^\star \- \fd_{\fV'} \\
\fe_{\fV}=\fe_{\fV'}, \; \fd_{\fV}= \fd_{\fV'} , &\hbox{if} \;    y'_i=\ve_{\fV', \uw} \; \hbox{for some} \;\; \uw \in  \fe_{\fV'}\\
\fd_{\fV}=\fd_{\fV'}, \; \fe_{\fV}= \fe_{\fV'}, & \hbox{if} \; 
y'_i=\de_{\fV', (\uu,\uv)}\; \hbox{for some}  \;\; (\uu,\uv) \in  \fd_{\fV'}  .
\end{array} \right.
\end{equation}
Accordingly, we introduce 
\begin{equation}\label{proof:new-ex-fv-k}
\left\{ 
\begin{array}{lccr}
\ve_{\fV, \uw}=y'_i, \; x_{\fV, \uw} =1,
&  \hbox{if} \; y'_i=x_{\fV', \uw}\; \hbox{for some} \;\; \uw \in \fe_\fV \- \fe_{\fV'}\\
\de_{\fV, (\uu,\uv)}=y'_i, \; x_{\fV, (\uu,\uv)} =1, & \hbox{if} \; 
y'_i=x_{\fV', (\uu,\uv)} \; \hbox{for some}  \;\; (\uu,\uv) \in \fd_\fV \- \fd_{\fV'} \\
\ve_{\fV, \uw}=y'_i  , &\hbox{if} \;    y'_i=\ve_{\fV', \uw} \; \hbox{for some} \;\; \uw \in  \fe_{\fV'}=\fe_{\fV}\\
\de_{\fV, (\uu,\uv)}=y'_i, & \hbox{if} \; 
y'_i=\de_{\fV', (\uu,\uv)}\; \hbox{for some}  \;\; (\uu,\uv) \in  \fd_{\fV'}=\fd_\fV  .
\end{array} \right.
\end{equation}

To introduce the set $\var_\fV$, we then set
\begin{equation}\label{proof:var-xi-fv-k}
\left\{ 
\begin{array}{lcr}
x_{\fV,\ua}=\xi_{j}, &   \hbox{if $y'_{j}=x_{\fV', \ua} $}\\
x_{\fV, (\ua, \ub)}=\xi_{j}, & \;\;\; \hbox{if $y'_{j}=x_{\fV', (\ua, \ub)}$}\\
\ve_{\fV, \ua}=\xi_{j}, &  \hbox{if $y'_{j}= \ve_{\fV', \ua}$} \\
\de_{\fV, (\ua, \ub)}= \xi_{j}, & \;\;\;\; \hbox{if $y'_{j}= \de_{\fV', (\ua, \ub)}$}.
\end{array} \right.
\end{equation}
Thus, we have introduced $y'_i,  \xi_{j} \in \var_\fV$ where $y'_i$, respectively,  $\xi_{j}$, 
is endowed with its new name as in \eqref{proof:new-ex-fv-k}, respectively, in 
\eqref{proof:var-xi-fv-k}.
Next, we define the set $\var_\fV \- y'_i \- \xi_{j}$ to consist of the following variables:
\begin{equation}\label{proof:var-fv-k}
\left\{ 
\begin{array}{lccr}
x_{\fV,\uw}=x_{\fV', \uw}, &  \forall \;\; \uw \in \II_{d,n}\-\um \- \fe_\fV \;\; \hbox{and $x_{\fV', \uw} \ne y'_{j}$}\\
x_{\fV, (\uu, \uv)}=x_{\fV', (\uu, \uv)}, & \;\; \forall \;\; (\uu, \uv) \in \La_\sfm^\star \- \fd_{\fV} \;\; \hbox{and
 $x_{\fV', (\uu, \uv)} \ne y'_{j}$}\\
\ve_{\fV, \uw}= \ve_{\fV', \uw}, & \forall \;\; \uw \in \fe_\fV \;\; \hbox{and $\ve_{\fV', \uw} \ne y'_{j}$},  y'_i \\
\de_{\fV, (\uu, \uv)}= \de_{\fV', (\uu, \uv)}, &  \;\; \forall \;\; (\uu, \uv) \in \fd_{\fV} \;\; \hbox{and $\de_{\fV', (\uu, \uv)} \ne y'_{j}$}, y'_i.
\end{array} \right.
\end{equation}

Substituting \eqref{proof:sf-fv-k}, one checks that the chart $\fV$ is  
 isomorphic to the affine space with the variables in  \eqref{proof:new-ex-fv-k},
 \eqref{proof:var-xi-fv-k}, and \eqref{proof:var-fv-k},  
   as its coordinate variables. Putting all together, the above matches
   description of $\var_\fV$ in  \eqref{variables-p-k},
 
 Now, it remains to verity (1)-(7) of the proposition on the chart $\fV$.

First, consider  the unique new exceptional divisor $E_{(\eth_{(k\tau)}\fr_\mu\fs_h)}$
created by the blowup
$$\tsR_{(\eth_{(k\tau)}\fr_\mu\fs_h)} \lra \tsR_{(\eth_{(k\tau)}\fr_\mu\fs_{h-1})}. $$ 
Then, we have
$$ E_{(\eth_{(k\tau)}\fr_\mu\fs_h)} \cap \fV = (y'_i=0)$$
where $y'_i$ is renamed as in \eqref{proof:new-ex-fv-k}
according to the four possibilities of its form in $\var_{\fV'}$.
This way, the new exceptional divisor $E_{(\eth_{(k\tau)}\fr_\mu\fs_h)}$ is labelled on the chart $\fV$.
In any case of the four situations, we have that the proper transform in $\tsR_{(\eth_{(k\tau)}\fr_\mu\fs_h)}$ 
of $Y'_i$ does not meet the chart $\fV$, and
if $Y'_i$ is a $\vp$-exceptional divisor labeled by some element of $\fe_{\fV'}$
(resp, $\vr$-exceptional divisor labeled by some element of $\fd_{\fV'}$),
 then, on the chart $\fV$, its proper transform is no longer labelled by
any element of $\II_{d,n} \-\um$ (resp. $\La_\sfm^\star$).
 This verifies any of (3)-(7) whenever the statement therein involves the newly created exceptional divisor 
 $E_{(\eth_{(k\tau)}\fr_\mu\fs_h)}$ and/or its corresponding exceptional variable $y'_i$ 
 (which is renamed in \eqref{proof:new-ex-fv-k} to match that of \eqref{variables-p-k} in the statement of
 the proposition).
 Observe that the statements (1) and (2) are not related to the new exceptional divisor 
 $E_{(\eth_{(k\tau)}\fr_\mu\fs_h)}$ or its corresponding exceptional variable $y'_i$ (again, renamed in \eqref{proof:new-ex-fv-k}).
  
For any of the remaining $\vp$-, $\vr$-, and exceptional divisors on $\tsR_{(\eth_{(k\tau)}\fr_\mu\fs_h)}$, 
it is the proper transform of a unique corresponding
 $\vp$-, $\vr$-, and exceptional divisor on $\tsR_{(\eth_{(k\tau)}\fr_\mu\fs_{h-1})}$. Hence,
 applying   the inductive assumption on $\fV'$, one verifies directly that
every of the properties (1)-(7) of the proposition is satisfied whenever it applies to such a divisor
and/or its corresponding local variable in $\var_\fV$.

 This completes the proof.
\end{proof}

\subsection{Combining all $\vt$-blowups, $\wp$-blowups, and $\eth$-blowups}\label{combine-all} $\ $

 For the convenience of narration, we combine all
 $\vt$-blowups, $\wp$-blowups, and $\eth$-blowups
 into a single grand sequence.

\begin{defn}\label{defn:game}
We introduce
\begin{equation}\label{game} 
\Game:=\left\{ \begin{array}{ccccccc}
\hs \end{array}
  \; \Bigg| \;
\begin{array}{ccccc}
\hbox{$\hs$ equals to  $\vt_{[k]}$ with $k \in [\up]$, or 
$(\wp_{(k'\tau')}\fr_{\mu'}\fs_{h'})$ with} \\ 
\hbox{ $(k'\tau')\mu' h' \in \Index_\Phi$,
or $(\eth_{(k\tau)}\fr_\mu\fs_h)$  with  $(k\tau)\mu h \in \Index_\Psi$} \\
\end{array} \right \}.
\end{equation}
The set $\Game$ comes equipped with a partial order induced from
the total orders of the sets $[\up]$, $\Index_\Phi$, and $\Index_\Psi$.
We then let 
$$\vt_{[k]} < (\wp_{(k'\tau')}\fr_{\mu'}\fs_{h'}) <(\eth_{(k\tau)}\fr_\mu\fs_h)$$ for all
$\vt_{[k]}$ with $k \in [\up]$,  $(\wp_{(k'\tau')}\fr_{\mu'}\fs_{h'})$ with $(k'\tau')\mu' h' \in \Index_\Phi$,
and $(\eth_{(k\tau)}\fr_\mu\fs_h)$ with $(k\tau)\mu h \in \Index_\Psi$.
This way, $\Game$ is endowed with a total order compatible with
the orders of $\vt$-, $\wp$-, and $\eth$-blowups.
\end{defn}

\begin{defn}\label{termninatingB} {\rm (cf. Definition \ref{general-termination})}
Consider an arbitrary standard chart of $\fV$ of  $\tsR_{\hs}$ with $\hs \in \Game$.
We let $\tsV_\hs \subset \tsR_\hs$ be the proper transform of $\sV$ in $\tsR_\hs$.

Let $\bz \in \fV$ be any fixed closed point of the chart. 
Consider any main binomial relation $B \in \cB^\mn$. 
 We say that $B^\mn_{\fV}$ terminates at a closed point $\bz \in \fV$ if (at least) one of its two terms does not
 vanish at $\bz$.
 
 We say $B^\mn_{ \fV}$ terminates on the chart $\fV$ if it terminates at all closed points of 
 $\tsV_\hs \cap \fV $.
 We say $B^\mn$ terminates on the scheme $\tsR_{\hs}$ 
 if $B^\mn_{ \fV}$ terminates over all standard charts $\fV$ of  $\tsR_{\hs}$.
 \end{defn}
 In a similar vein, for all $\bF \in \sfm$, the linearized $\pl$ equation $L_{\fV,F}$  
automatically terminates on the chart $\fV$ because 
$L_{\fV,F}$  is the pullback of $L_{\fV_{[0]},F}$, where $\fV_{[0]}$ is the unique standard chart 
of $\sR_\sF$ such that $\fV$ lies over $\fV_{[0]}$, and
at any closed point  of $\fV_{[0]}$, at least one of the terms of $L_{\fV_{[0]},F}$ does not vanish. 

\begin{lemma}\label{termi-carry-over}
Consider any two indexes $\hs', \hs \in \Game$ with $\hs' < \hs$.
Let $\fV$ be a standard chart of $\tsR_{\hs}$, lying  over a unique standard chart
$\fV'$  of $\tsR_{\hs'}$. 
 Fix and consider any main binomial relation $B \in \cB^\mn$.
Assume that  $B^\mn_{ \fV'}$ terminates on $\fV'$. 
Then, $B^\mn_{ \fV}$ terminates on $\fV$. 
\end{lemma}
\begin{proof} Under the assumption, we have
 $$B^\mn_{ \fV}=\pi^*_{\fV,\fV'} B^\mn_{ \fV'}=\pi^*_{\fV,\fV'} T^+_{\fV'} - \pi^*_{\fV,\fV'} T^-_{\fV'},$$
where $ \pi_{\fV,\fV'}: \fV \lra \fV'$ is the (induced) projection.
 Then, one sees that the statement follows immediately.
\end{proof}

Fix and consider any of $\vt$-blowups, $\wp$-blowups, and $\eth$-blowups, that is, we consider
$\tsR_{\vt_{[k]}} \to \tsR_{\vt_{[k-1]}}$ for some $k \in [\up]$, or
 $\tsR_{(\wp_{(k\tau)}\fr_\mu\fs_h)} \to \tsR_{(\wp_{(k\tau)}\fr_\mu\fs_{h-1})}$
 for some $(k\tau)\mu h\in \Index_\Phi$, or
  $\tsR_{(\eth_{(k\tau)}\fr_\mu\fs_h)} \to \tsR_{(\eth_{(k\tau)}\fr_\mu\fs_{h-1})}$
 for some $(k\tau)\mu h \in \Index_\Psi$. 
 We let $$\tsR_{\hs} \to \tsR_{\hs'}$$ denote the above fixed blowup.

We let $E_\hs$ be the exceptional divisor created by the above fixed blowup.

Consider an arbitrary standard chart $\fV$ of $\tsR_\hs$, lying over 
a (unique) standard chart $\fV'$ of $\tsR_{\hs'}$.
Suppose the chart $\fV'$ intersects the center of the blowup $\tsR_{\hs} \to \tsR_{\hs'}$
along a nonempty closed subset. Then,
on the chart $\fV'$, the corresponding $\vt$- or $\wp$- or $\eth$-center,  corresponds
to two variables 
$$\{y'_0, \; y'_1\} \subset \var_{\fV'}.$$
We let   $\PP^1_{[\xi_0,\xi_1]}$ be the corresponding factor projective space with 
$[\xi_0,\xi_1]$ corresponding to $(y'_0,  y'_1)$.
We suppose that the chart $\fV$ corresponds to $(\xi_i \equiv 1)$ for one of $i \in \{0, 1\}$.
We let $j \in \{0, 1\} \- i$ and let $y_j \in \var_\fV$ be the proper transform of $y'_j$.
We let $\zeta_\fV \in \var_\fV$  be the exceptional parameter 
of $E_\hs$ on the chart $\fV$ (it corresponds to $y'_i$).


We point out here that in the above, if $\hs ={\vt_{[k]}}$,  then it corresponds
 to the fixed $F_k$; if $\hs=(\wp_{(k\tau)}\fr_\mu\fs_h)$ or 
 $(\eth_{(k\tau)}\fr_\mu\fs_h)$, then it corresponds to the fixed $B_{(k\tau)}$.

\begin{defn} \label{termi-v0}  Consider $B_{(k\tau)} \in \cB^\mn$.
(Here, if $\hs ={\vt_{[k]}}$, then we consider any $\tau \in [\ft_{F_k}]$.)
Suppose $B_{(k\tau)}$ does not terminate on the chart $\fV'$,
but terminates on the chart $\fV$. Then, we call 
the proper transform $y_j \in \var_\fV$ of $y'_j \in \var_{\fV'}$ the terminating central variable for  the binomial equation $B_{(k\tau)}$ on the chart $\fV$.  We call $\zeta_\fV$  
the terminating exceptional parameter on the chart $\fV$.

Moreover, for any standard chart $\widetilde \fV$ of some fixed $\vt$- or $\wp$- or $\eth$-blowup scheme,
 lying over the above  chart $\fV$, we also say $B_{\widetilde \fV, (k\tau)}$ terminates on the chart
  $\widetilde \fV$;  we also call 
 the proper transform $y_{\widetilde \fV, j} \in \var_{\widetilde \fV}$ of $y_j$
  the terminating central variable for  the binomial equation $B_{(k\tau)}$
  on the chart $\widetilde \fV$. 
  
  Likewise, 
  we call the proper transform $\zeta_{\widetilde \fV} \in \var_{\widetilde \fV}$ of $\zeta_\fV \in \var_\fV$
  the terminating exceptional parameter for  the binomial equation $B_{(k\tau)}$
  on the chart $\widetilde \fV$.
  
  In particular, the above applies to the standard charts of the final scheme $\tsR_\eth$.
\end{defn}

\subsection{Proper transforms of defining 
 equations in $(\eth_{(k\tau)}\fr_\mu\fs_h)$} $\ $

We keep the notations of Proposition \ref{meaning-of-var-p-k} as well as those in its proof.

In particular, we have that $\fV$ is a standard chart of $\tsR_{(\eth_{(k\tau)}\fr_\mu\fs_h)}$, lying over
 a standard chart $\fV'$ of $\tsR_{(\eth_{(k\tau)}\fr_\mu\fs_{h-1})}$.
We suppose that $B_{(k\tau)}$ corresponds to $s \in S_{F_k}\-s_{F_k}$. We then let
 $$\psi'_{(k\tau)\mu h}=\{Y'_0=X_{(\eth_{(k\tau)}\fr_\mu\fs_{h-1}), (\uu_s,\uv_s)}, \;\; Y'_1\}$$
  be the proper transforms of $\eth$-set $\psi_{(k\tau)\mu h}$
  in $\tsR_{(\eth_{(k\tau)}\fr_\mu\fs_{h-1})}$;
we let  $Z'_{\psi_{(k\tau)\mu h}}$ be proper transform 
  of  the $\eth$-center $Z_{\psi_{(k\tau)\mu h}}$  in $\tsR_{(\eth_{(k\tau)}\fr_\mu\fs_{h-1})}$.
In addition, we assume that $Z'_{\psi_{(k\tau)\mu h}} \cap \fV'  \ne \emptyset$.
 For convenience, we let $y_0'=x_{\fV', (\uu_s,\uv_s)}$.
 We then have
 \begin{eqnarray}\label{vars-wpktaumuh}
Y'_0 \cap \fV' =(y_0'=0)   \;\;\;\;\; \;\;\;\;\;\;\;\;\;\; \\
  Y'_1 \cap \fV' =(y_1'=0), \; \hbox{for some $y_1'  \in \var_{\fV'}$}. \nonumber
  \end{eqnarray}
 Note here that the first identity holds
 according to Proposition \ref{meaning-of-var-p-k} (2). 
Furthermore,  we let $[\xi_0,\xi_1]$ be the homogenous coordinates of 
the factor projective space $\PP_{\psi_{(k\tau)\mu h}}$
 corresponding to $(y'_0,y'_1)$.

  Consider any fixed $B \in \cB^\mn  \cup \cB^\q$ and $\bF \in \sfm$.
Suppose $B_{\fV'}$ and $L_{\fV', F}$ have been constructed over $\fV'$.
Applying Definition \ref{general-proper-transforms}, we obtain the proper transforms on the chart $\fV$
$$B_{\fV}, \; B \in \cB^\mn  \cup \cB^\q; \;\; L_{\fV, F}, \; \bF \in \sfm.$$

In what follows, for any $B =T^+_B - T^-_B \in \cB^\mn$, we express
$B_\fV= T^+_{\fV, B} - T^-_{\fV, B}$. If $B=B_{(k\tau)}$ for some $k \in [\up]$ and $\tau \in [\ft_{F_k}]$,
we also write $B_\fV= T^+_{\fV, (k\tau)} - T^-_{\fV, (k\tau)}$.


\begin{prop}\label{equas-p-k} 
Let the notation be as in Proposition \ref{meaning-of-var-p-k} and be as in above.

Let $\fV$ be any standard chart of $\tsR_{(\eth_{(k\tau)}\fr_\mu\fs_h)}$. Then, the scheme 
$\tsV_{(\eth_{(k\tau)}\fr_\mu\fs_h)}\cap \fV$, as a closed subscheme of the chart $\fV$,
is defined by $$\cB_\fV^\mn, \; \cB_\fV^\q, \; L_{\fV, \sfm}.$$
 
Suppose 
that $Z'_{\psi_{(k\tau)\mu h}} \cap \fV'  \ne \emptyset$ where
$Z'_{\psi_{(k\tau)\mu h}}$ is the proper transform 
  of  the $\eth$-center $Z_{\psi_{(k\tau)\mu h}}$  in $\tsR_{(\eth_{(k\tau)}\fr_\mu\fs_{h-1})}$.
Further, we let $\zeta=\zeta_{\fV, (k\tau)\mu h}$ be the exceptional parameter in $\var_\fV$ such that
$$E_{(\eth_{(k\tau)}\fr_\mu\fs_h)} \cap \fV = (\zeta=0).$$

Then, the following hold.

\begin{enumerate}
\item Suppose $\fV= (\fV' \times (\xi_0 \equiv 1)) \cap \tsR_{(\eth_{(k\tau)}\fr_\mu\fs_h)}$.
We let $y_1 \in \var_\fV$  be the proper transform of $y_1' \in \var_{\fV'}$. 
Then, we have
\begin{itemize}
\item[(1a)]  The plus-term 
$T^+_{\fV, (k\tau)}$ terminates on the chart $\fV$. $y_1 \nmid T^+_{\fV, (k\tau)}$.
Suppose $\deg_{y_1'} T^-_{\fV', (k\tau)}=b$ for some 
integer $b$, positive by definition, then we have $\deg_{\zeta} T^-_{\fV, (k\tau)}=b-1$. Consequently,
either $T^-_{\fV, (k\tau)}$ is linear in $y_1$ or else $\zeta \mid T^-_{\fV,  (k\tau)}$.
(Note here that $y_0' \in \var_{\fV'}$ becomes $\zeta \in \var_\fV$.)
\item[(1b)] Let $B \in  \cB^\mn \- \{ B_{(k\tau)}\}$. 
Then, $y_1 \nmid T^+_{\fV, B}$; if $y_1 \mid T^-_{\fV, B}$, then $\zeta \mid T^-_{\fV, B}$.
\item[(1c)] Consider any fixed term $T_B$ of any given $B \in \cB^\q$. 
Suppose $y_1 \mid T_{\fV, B}$, then either $T_{\fV, B}$ is linear in $y_1$ or 
$\zeta \mid T_{\fV, B}$.
\end{itemize}
\item Suppose $\fV= (\fV' \times (\xi_1 \equiv 1)) \cap \tsR_{(\eth_{(k\tau)}\fr_\mu\fs_h)}$.
We let $y_0 =x_{\fV, (\uu_s,\uv_s)} \in \var_\fV$  be the proper transform 
of $y_0'=x_{\fV', (\uu_s,\uv_s)}$. 
 Then, we have
\begin{itemize}
\item[(2a)]   $T^+_{\fV, (k\tau)}$ is linear in $y_0$. $y_0 \nmid T^-_{\fV, (k\tau)}$.
$\deg (T^-_{\fV, (k\tau)}) =\deg (T^-_{\fV', (k\tau)})-1$.
Indeed, if we let $\deg_{y_1'} T^-_{\fV', (k\tau)}=b$ for some 
integer $b$, positive by definition, then we have $\deg_{\zeta} T^-_{\fV, (k\tau)}=b-1$. 
Note here that $y_1' \in \var_{\fV'}$ becomes the exceptional variable $\zeta \in \var_\fV$.  
\item[(2b)] Let $B \in   \cB^\mn \- \{ B_{(k\tau)}\}$. 
Then, $y_0 \nmid T^\pm_{\fV, B}$.
\item[(2c)] Consider any fixed term $T_B$ of any given $B \in \cB^\q$. 
Suppose $y_0 \mid T_{\fV, B}$, then  $T_{\fV, B}$ is linear in $y_0=x_{\fV, (\uu_s,\uv_s)}$.
\end{itemize}
Consequently, we have
\item  $\vk_{(k\tau)}< \infty$;
\item   Moreover, consider $\tsR_{(\eth_{(k\tau)}\fr_{\vk_{(k\tau)}})}=
\tsR_{(\eth_{(k\tau)}\fr_{\vk_{(k\tau)}}\fs_{\vs_{(k\tau)\vk_{(k\tau)}}})}$.
For every $B \in \cB^\mn$ 
with $B \le B_{(k\tau)}$, we have that $B$ 
terminates on $\tsR_{(\eth_{(k\tau)}\fr_{\vk_{(k\tau)}})}$.
\end{enumerate}
\end{prop}
\begin{proof} 
We continue to follow  the notation in the proof of Proposition \ref{meaning-of-var-p-k}. 
(In many ways, the proofs are similar to those of Proposition \ref{equas-vskmuh}.)

We prove the proposition 
 by applying induction on
$(k\tau) \mu h  \in ((11)1 0) \sqcup \Index_\Psi$. 

The initial case is $((11)1 0)$ with $\tsR_{(\eth_{(11)}\fr_1 \fs_0)}=\tsR_{\wp}$. In this case, 
the  statement about defining equations
 of  $\tsR_{(\eth_{(11)}\fr_1 \fs_0)} \cap \fV$ follows from 
 Proposition \ref{equas-vskmuh} for $\tsR_{\wp}$;  the remainder statements (1) - (4) are void.

Assume that the proposition holds  for $(\eth_{(k\tau)}\fr_\mu \fs_{h-1})$
with $(k\tau) \mu h  \in  \Index_\Psi$.

Consider $(\eth_{(k\tau)}\fr_\mu \fs_h)$. 
Consider a standard chart $\fV$ of $\tsR_{(\eth_{(k\tau)}\fr_\mu \fs_{h})}$,
lying over a unique standard chart of $\fV'$ over  $\tsR_{(\eth_{(k\tau)}\fr_\mu \fs_{h-1})}$.
By assumption, all the desired statements of the proposition hold over $\fV'$.

To begin with, we suppose that the proper transform $Z'_{\psi_{(k\tau)\mu h}}$ 
of the $\eth$-center $Z_{\psi_{(k\tau)\mu h}}$  in $\tsR_{(\eth_{(k\tau)}\fr_\mu\fs_{h-1})}$ does not
meet the chart $\fV'$.  
Then, by the proof of Proposition \ref{meaning-of-var-p-k},
  we have that $\fV$ retains all the data from those of $\fV'$.
As the $\eth$-blowup  along the proper transform of $Z_{\psi_{(k\tau)\mu h}}$
does not affect the chart $\fV'$,  we have that
the statement of the proposition on the defining equations
 follows immediately from the inductive assumption.
The statements (1) and (2) are void.

In what follows, we suppose that the proper transform $Z'_{\psi_{(k\tau)\mu h}}$ 
meets the chart $\fV'$ along a nonempty closed subset. 

The  statement of the proposition on the defining equations of $\tsV_{(\eth_{(k\tau)}\fr_\mu \fs_{h})} \cap \fV$
follows straightforwardly from the inductive assumption.

 {\it Proof of (1).  }

(1a).  The fact that $T^+_{\fV, (k\tau)}$ terminates on the chart $\fV$
 follows from the form of the $\eth$-set $\psi_{(k\tau)\mu h}$ on the chart $\fV'$
(see \eqref{vars-wpktaumuh})  and Proposition \ref{equas-vskmuh} (4).
The remainder follows from a direct check.

(1b). Consider any $B=B_{(k'\tau')} \in   \cB^\mn$ with $(k'\tau') \ne (k\tau)$.
  We can express $T^+_B= x_{(\uu_t, \uv_t)} x_{\uu_{k'}}$ for some $t \in S_{F_{k'}}$ and $t \ne s$
  (if $k'=k$).
Suppose $(k'\tau') < (k\tau)$, by Proposition \ref{equas-p-k} (4) on
$\tsR_{(\eth_{(k'\tau')}\fr_{\vk_{(k'\tau')}})}$, which holds by the inductive assumption, 
$T^+_{\fV',B}$ terminates on $\fV'$. Hence,  we have $y_0', y_1' \nmid T^+_{\fV',B}$.
Suppose $(k'\tau') > (k\tau)$,  by Proposition \ref{equas-vskmuh} (4), 
 we   we have $y_0', y_1' \nmid T^+_{\fV',B}$.
 In either case,  this implies
the statements of (1b).  

(1c).  Consider any $B \in \cB^\q$. 
 Fix and consider any term  $T_{\fV', B}$ of $B_{\fV'}$. We assume $y_1 \mid T_{\fV, B}$.
If none of the variables $\{y_0', y_1'  \}$ of \eqref{vars-wpktaumuh} appears in the other term $T'_{\fV',B}$
of $B_{\fV'}$, then the statement  follows immediately. 
Suppose one of the variables of \eqref{vars-wpktaumuh} appears in $T'_{\fV',B}$.
Note that $y_0'=x_{\fV', (\uu_s,\uu_s)}$, a $\vr$-variable,
 is {\it linear} in $B_{\fV'}$ (see Proposition \ref{equas-p-k=0}), and it
becomes the exceptional variable $\zeta \in \var_\fV$. We may write $y_1'=\zeta y_1$.
As $\deg_{y_1'} T^-_{\fV', B}=b$ for some integer $b >0$, one sees that $\zeta^{b-1} y_1^b  \mid T_{\fV, B}$.
This implies the statement.

{\it Proof of (2).  }

 In this case, $y_1'=\zeta$, that is, $y_1'$ is renamed as  the exceptional parameter in $\var_\fV$, and
 we have $y_0=x_{\fV, (\uu_s,\uv_s)}$ is the proper transform of $y_0'$.

(2a) The fact that $T^+_{\fV, (k\tau)}$ is linear in $y_0=x_{\fV, (\uu_s,\uv_s)}$ follows immediately
from the expression $T^+_{(k\tau)} = x_{(\uu_s,\uv_s)} x_{\uu_k}$.
The remainder follows from the form of the $\eth$-set $\psi_{(k\tau)\mu h}$ on the chart $\fV'$
(see \eqref{vars-wpktaumuh})  and a straightforward calculation.

(2b) follows immediately because $x_{(\uu_s, \uv_s)}$ uniquely appears in 
$B_{(k\tau)}$ among all main binomial equations  of $\cB^\mn$.

(2c) follows immediately because 
$B_\fV$ is linear in $\vr$-variables
 (see Proposition \ref{equas-p-k=0}).

{\it Proof of (3).  } (The proof is analogous to that of Proposition \ref{equas-vskmuh} (3).)

If (1) occurs, then by (1a), $T^+_{\fV, (k\tau)}$  terminates, thus $B_{\fV, (k\tau)}$ terminates.

Assuming  (1) does not occur and (2) keeps happening, then, by (2a),
as $\mu$ increases, after  finitely many rounds,
 $T^-_{\fV, (k\tau)}$ must terminate, hence, $B_{\fV, (k\tau)}$ terminates,
 for any standard chart $\fV$.

This implies that  the $\eth$-blowups in $(\eth_{(k\tau)})$ must terminate after finitely many rounds.
That is, $\vk_{(k\tau)} < \infty$.

{\it Proof of (4).  }

Suppose $B \in \cB^\mn$ with $B < B_{(k\tau)}$.
We write $B=B_{(k'\tau')}$ for some $(k'\tau') \in \Index_{\cB^\mn}$ with $(k'\tau') < (k\tau)$.
Then, $B$ terminates on  $\tsR_{(\eth_{(k\tau)}\fr_{\vk_{(k\tau)}})}$ by applying
Proposition \ref{equas-p-k} (4) in $(\eth_{(k'\tau')})$ which holds by the inductive assumption, 
and Lemma \ref{termi-carry-over}.

For $B=B_{(k\tau)}$, it  follows immediately from the proof of (3).

Therefore, by induction on 
$(k\tau) \mu h \in ((11)10) \sqcup \Index_\Psi$,   Proposition \ref{equas-p-k} is proved.
\end{proof}


The following is the special case of Proposition \ref{equas-p-k}.

\begin{cor}\label{all-terminate} 
 The binomial $B$ 
 terminates on $\tsR_\eth$  for all $B \in \cB^\mn$. 
\end{cor}

\begin{cor}\label{e2} 
Let the notation be as in Definition \ref{termi-v0}.

Consider $B=B_{(k\tau)} \in \cB^\mn$ with $(k\tau) \in \Index_{\cB^\mn}$.
 Suppose $B_{\fV, (k\tau)}$ terminates on a standard chart $\fV$ of $\tsR_\hs$ for some
 $\hs \in \Game$. 
Suppose the minus term $T^-_{\fV, (k\tau)}$ of $B_{\fV, (k\tau)}$ is divisible by 
the terminating exceptional parameter $\zeta_\fV$.
Then, we have $$E_\hs \cap \tsV_\hs \cap \fV=\emptyset.$$ In particular, suppose
$E_\hs \cap \tsV_\hs \cap \fV \ne \emptyset$, then,
   the binomial equation $B_{\fV, (k\tau)}$ is linear 
   in the (unique) terminating central variable $y \in \var_\fV$.
\end{cor}
\begin{proof}
First, observe that  we have $\zeta (\bz) =0$  for any $\bz \in E_\hs \cap \fV$.
Hence, if $\zeta_\fV \mid T^-_{\fV, (k\tau)}$, we must have
$E_\hs \cap \tsV_\hs \cap \fV=\emptyset$, because $B_{\fV, (k\tau)}$ terminates.

Suppose now $E_\hs \cap \tsV_\hs \cap \fV \ne \emptyset$.

If $\fV$ is a standard chart of $\vt$-blowup scheme $\tsR_{\vt_{[k]}}$ for some $k \in [\up]$, then
the statements follow immediately from  Proposition \ref{eq-for-sV-vtk}. 

If $\fV$ is a standard chart of $\wp$-blowup scheme $\tsR_{\wp_{(k\tau)}\fr_\mu\fs_h}$
 for some $(k\tau)\mu h\in \Index_\Phi$, then
the statements follow immediately from  
Proposition \ref{equas-vskmuh} (1a) and (2a). 

If $\fV$ is a standard chart of $\eth$-blowup scheme $\tsR_{\eth_{(k\tau)}\fr_\mu\fs_h}$
 for some $(k\tau)\mu h\in \Index_\Psi$, then
the statements follow immediately from  
Proposition \ref{equas-p-k} (1a) and (2a). 
\end{proof}

\section{$\Ga$-schemes and Their 
Transforms}\label{Gamma-schemes}

\subsection{$\Ga$-schemes} $\ $

Here, we return to the initial affine chart $\rU_\um \subset \PP(\wedge^d E)$.

\begin{defn}\label{ZGa} 
Let $\Ga$ be an arbitrary subset of $\var_{\rU_\um}=\{x_\uu \mid \uu \in \II_{d,n} \-\um \}$. 
We let $I_\Ga$ be the ideal of $\kk[x_\uu]_{\uu \in \II_{d,n}\-\um}$ generated by all the elements 
$x_{\uu}$ in $\Ga$, 
 and, 
 $$I_{\wp,\Ga}=\langle x_\uu, \; \bF \mid x_\uu \in \Ga, \; \bF \in \sfm \rangle$$
  be the ideal of $\kk[x_\uu]_{\uu \in \II_{d,n}\-\um}$ generated by $I_\Ga$ together with 
all the de-homogenized $\um$-primary $\pl$ relations of $\Gr^{d,E}$. We let
 $Z_\Ga$ $(\subset \Gr^{d,E} \cap \rU_\um)$ be the closed subscheme of 
 the affine space $\rU_\um$ defined by the ideal $I_{\wp,\Ga}$.
The subscheme $Z_\Ga$ is called  the $\Ga$-scheme of $\rU_\um$. 
Note that  $Z_\Ga \ne \emptyset$ since $0 \in Z_\Ga$.
\end{defn}

(Thus, a $\Ga$-scheme is  an intersection of certain Schubert divisors with the chart $\rU_\um$.
 But, in this article, we do not investigate $\Ga$-schemes in any {\it Schubert} way.)
 

Take $\Gamma =\emptyset$. Then, $I_{\wp,\emptyset}$ is the ideal generated by 
all the de-homogenized $\um$-primary $\pl$ relations.
Thus,  $Z_\emptyset =\rU_\um \cap \Gr^{d,E}$.  

Let $\Ga$ be any fixed subset of $\var_{\rU_\um}$.
We let  $\rU_{\um,\Ga}$ be the coordinate  subspace of $\rU_\um$ defined by $I_\Ga$.
That is,
$$\rU_{\um,\Ga}=\{(x_{\uu} =0)_{ x_\uu \in \Ga}\} \subset \rU_{\um}.$$
This is a coordinate subspace of dimension 
${n \choose d}-1 - |\Ga|$ where $|\Ga|$ is the cardinality of $\Ga$. Then,
$Z_\Ga$ is the scheme-theoretic intersection of $\Gr^{d,E}$ with
the coordinate  subspace $\rU_{\um, \Ga}$.
For any $\um$-primary $\pl$ equation $\bF \in \sfm$, we let $\bF|_\Ga$ be the induced 
polynomial obtained from  the de-homogeneous polynomial $\bF$
by setting $x_{\uu} =0$  for all $x_\uu \in \Gamma$. 
Then, $\bF|_\Ga$ becomes a polynomial on the affine subspace $\rU_{\um,\Ga}$.
We point out that $\bF|_\Ga$ can be identically zero on $\rU_{\um,\Ga}$.

\begin{defn}\label{rel-irrel}
 Let $\Ga$ be any fixed subset of $\var_{\rU_\um}$.
 Let ($\bF$) $F$ be any fixed (de-homogenized)  $\um$-primary $\pl$ relation. 
 We say ($\bF$) $F$  is $\Ga$-irrelevant if  every term of 
 $\bF$ belongs to  the ideal $I_\Ga$.   Otherwise, we say ($\bF$) $F$ is $\Ga$-relevant.
 We let $\sfmgr$ be the set of all $\Ga$-relevant de-homogenized $\um$-primary $\pl$ relations.
 We let $\sfmgir$ be the set of all $\Ga$-irrelevant  de-homogenized $\um$-primary $\pl$ relations.
\end{defn}

If  $\bF$ is $\Ga$-irrelevant,  then $\bF|_\Ga$ is identically zero along $\rU_{\um,\Ga}$. 
Indeed,  $\bF$ is $\Ga$-irrelevant if and only if every term of $\bF$ contains a member of $\Ga$.
 The sufficiency direction is clear. To see the necessary direction,
 we suppose a term $x_\uu x_\uv \in I_\Ga$, then as $I_\Ga$ is prime 
 (the coordinate  subspace $\rU_{\um, \Ga}$ is integral), we have
 $x_\uu$ or  $x_\uv \in \Ga$.

\subsection{$\sF$-transforms of  $\Ga$-schemes in $\sV_{\sF_{[k]}}$}
\label{subsection:wp-transform-sfk}   $\ $





In what follows, we keep  notation of Proposition \ref{equas-fV[k]}. 

Recall that for any $\bF \in \sfm$,
$\La_F=\{(\uu_s, \uv_s) \mid s \in S_F\}.$


\begin{lemma}\label{wp-transform-sVk-Ga} 
Fix any  subset $\Ga$ of $\rU_\um$.  Assume that $Z_\Ga$ is integral. 

Consider $F_k \in \sfm$ for any fixed $k \in [\up]$.

Then, we have the following:
\begin{itemize}
\item there exists a closed subscheme $Z_{\sF_{  [k]},\Ga}$ of $\sV_{\sF_{[k]}}$
with an induced morphism  $Z_{\sF_{  [k]},\Ga}
 \to Z_\Ga$;
\item $Z_{\sF_{  [k]},\Ga}$ comes equipped with an irreducible component  
$Z^\dagger_{\sF_{  [k]},\Ga}$ with the induced morphism 
$Z^\dagger_{\sF_{  [k]},\Ga} 
 \to Z_\Ga$;
 \item  for any standard chart $\fV$ of $\sR_{\sF_{[k]}}$ such that
$Z_{\sF_{[k]},\Ga} \cap \fV \ne \emptyset$, there are two subsets, possibly empty,
$$ \tGa^\zero_{\fV} \; \subset \;  \var_\fV, \;\;\;
\tGa^\one_{\fV} \; \subset \;  \var_\fV.$$ 
\end{itemize}

Further,  consider any given standard chart $\fV$ of $\sR_{\sF_{[k]}}$ with
$Z_{\sF_{[k]},\Ga} \cap \fV \ne \emptyset$. Then,
 there exists a subset $\sF^\star_{[k]; \fV,\Ga} \subset \sfm$ such that the following hold.
\begin{enumerate} 
\item 
The scheme $Z_{\sF_{[k]},\Ga} \cap \fV$, as a closed subscheme of the chart $\fV$,
is defined by the following relations
\begin{eqnarray} 
\;\;\;\;\; y , \; \; \; y \in  \tGa^\zero_\fV , \label{Ga-rel-sVk=0}\\
\;\;\;  y -1, \; \; \; y \in  \tGa^\one_\fV,  \nonumber \\
\cB^\pq_{\fV, [k]},  \;\;\; \;\;\; \;\; \;\; \;\; \;\; \;\; \;\; \;\; \;\; \;\; \;\; \;\; \;\; \label{eqs-sVk} \\
B_{\fV,(s,t)}: \;\;\;  x_{\fV, (\uu_s, \uv_s)}x_{\fV,\uu_t} x_{\fV,\uv_t} - x_{\fV, (\uu_t,\uv_t)}  
 x_{\fV,\uu_s} x_{\fV,\uv_s},  \;\; s, t \in S_{F_i} 
  \;  i \in [k], \nonumber\\
L_{\fV, F_i}: \;\; \sum_{s \in S_{F_i}} \vsgn (s) x_{\fV, (\uu_s,\uv_s)}, \; \;  
\bF_i \in   \sF^\star_{[k]; \fV,\Ga}, \; i \in [k], \nonumber \\
\bF_{\fV,j}: \;\; \sum_{s \in S_{F_j}} \vsgn (s) x_{\fV, \uu_s}x_{\fV,\uv_s}, \; \; k < j \le \up.  \nonumber
\end{eqnarray} 
 Further, we take $\tGa^\zero_\fV \subset \var_\fV$
 to be the maximal subset (under inclusion)
among all those subsets that satisfy the above.
\item The induced morphism  $Z^\dagger_{\sF_{  [k]},\Ga}  
\to Z_\Ga$ is birational. 
\item Fix any variable $y=x_{\fV, \uu}$ or $y=x_{\fV, (\uu,\uv)} \in \var_\fV$,
$Z^\dagger_{\sF_{[k]},\Ga} \cap \fV \subset (y=0)$ 
if and only if $Z_{\sF_{[k]},\Ga} \cap \fV \subset (y=0)$.
(We remark here that this property is not used within this lemma, but 
will be used as the initial case of Lemma \ref{vt-transform-k}.)
\end{enumerate}
\end{lemma}
\begin{proof}
We prove the statement by induction on $k$ with $k \in \{0\} \cup [\up]$.

 When $k=0$,  we have $\sR_{\sF_{[0]}}:=\rU_\um$, $\sV_{\sF_{[0]}}:=\rU_\um \cap \Gr^{d,E}$.
 There exists a unique chart $\fV=\rU_\um$. In this case,  we set 
 $$Z_{\sF_{[0]},\Ga}=Z^\dagger_{\sF_{[0]},\Ga}:=Z_\Ga$$
 Further, we let
 $$\tGa^\zero_\fV=\Ga,\;\; \tGa^\one_\fV=\emptyset.$$
Then,  the statement holds trivially.

Inductively, we suppose that Lemma \ref{wp-transform-sVk-Ga}
 holds for $\sV_{\sF_{[k-1]}} \subset \sR_{\sF_{[k-1]}} $.

 We now consider $\sV_{\sF_{[k]}} \subset \sR_{\sF_{[k]}}$.
 
Recall from  \eqref{rho-sFk},   we have the natural birational morphsim
$$\rho_{\sF_{[k]}}: \sV_{\sF_{[k]}} \lra \sV_{\sF_{[k-1]}},$$
induced from the forgetful map  $\sR_{\sF_{[k]}} \lra \sR_{\sF_{[k-1]}}$.

First, we suppose $F_k$ is $\Ga$-relevant. 


In this case, we set 
\begin{equation}\label{construction-ZkLa}\La^\zero_{F_k, \Ga} 
:=\{ x_{(\uu,\uv)}\in \La_{F_k} \mid \hbox{$x_\uu$ or $x_\uv \in \Ga$}\}, \;\;\; \La^\one_{F_k, \Ga}:=\emptyset.
\end{equation}
(Here, recall the convention of \eqref{uv=vu}: $x_{(\uu,\uv)}= x_{(\uv,\uu)}$.) 

We then let $\rho_{\sF_{[k]}}^{-1}(Z_{\sF_{ [k-1]},\Ga})$ be the scheme-theoretic pre-image and define
$Z_{\sF_{[ k]},\Ga}$ to be
  the scheme-theoretic intersection
\begin{equation}\label{construction-ZkGa}
Z_{\sF_{[ k]},\Ga}=\rho_{\sF_{[k]}}^{-1}(Z_{\sF_{ [k-1]},\Ga}) \cap (x_{(\uu,\uv)} =0 \mid (\uu,\uv) \in  \La^\zero_{F_k, \Ga}),
\end{equation}

Next, because $F_k$ is $\Ga$-relevant
 and $Z^\dagger_{\sF_{ [k-1]},\Ga}$ is birational to $Z_\Ga$,
one checks that  $Z^\dagger_{\sF_{ [k-1]},\Ga}$ 
is not contained in the exceptional locus of
the birational morphism $\rho_{\sF_{[k]}}$. 
Thus, there exists  a Zariski open subset $Z^{\dagger\circ}_{\sF_{ [k-1]},\Ga}$
 of $Z^\dagger_{\sF_{ [k-1]},\Ga}$ such that 
 $$\rho_{\sF_{[k]}}^{-1}(Z^{\dagger\circ}_{\sF_{ [k-1]},\Ga})  \lra Z^{\dagger\circ}_{\sF_{ [k-1]},\Ga}$$
is an isomorphism.

We claim 
\begin{equation}\label{inclusion-dagger}
 \rho_{\sF_{[k]}}^{-1}(Z^{\dagger\circ}_{\sF_{ [k-1]},\Ga})   \subset Z_{\sF_{[ k]},\Ga}=
\rho_{\sF_{[k]}}^{-1}(Z_{\sF_{ [k-1]},\Ga}) \cap (x_{(\uu,\uv)}=0 \mid (\uu,\uv) \in  \La^\zero_{F_k, \Ga}).
\end{equation}
To see this, note that since $\bF_k$ is $\Ga$-relevant, 
there exists a term $x_{\uu_s}x_{\uv_s}$ of $\bF_k$ 
for some $s \in S_{F_k}$ such that it does not vanish
generically along $Z^{\dagger}_{\sF_{ [k-1]},\Ga}$ (which is birational to $Z_\Ga$).
Then, we consider the binomial relation of $\sV_{\sF_{[k]}}$ in $\sR_{\sF_{[k]}}$
  \begin{equation}\label{Buv-s}
  x_{(\uu, \uv)}x_{\uu_s} x_{\uv_s} - x_{\uu} x_{\uv} x_{(\uu_s,\uv_s)},
  \end{equation}
  for any $(\uu, \uv) \in \La_{F_k}$.
 It follows that $x_{(\uu, \uv)}$  vanishes identically 
along $\rho_{\sF_{[k]}}^{-1}(Z^{\dagger\circ}_{\sF_{ [k-1]},\Ga}) \cong Z^{\dagger\circ}_{\sF_{ [k-1]},\Ga}$ 
 if $x_\uu$ or $x_\uv \in \Ga$.  Hence, \eqref{inclusion-dagger} holds.

 We then let $Z^\dagger_{\sF_{ [k]},\Ga}$ be the closure of 
 $\rho_{[k]}^{-1}(Z^{\dagger\circ}_{\sF_{ [k-1]},\Ga})$ in $Z_{\sF_{ [k]},\Ga}$.
  Since $Z^\dagger_{\sF_{ [k]},\Ga}$  is closed in $Z_{\sF_{ [k]},\Ga}$
 and contains the Zariski  open subset $\rho_{[k]}^{-1}(Z^{\dagger\circ}_{\sF_{ [k-1]},\Ga})$
 of $Z_{\sF_{[k]},\Ga}$, it is an irreducible component of  $Z_{\sF_{ [k]},\Ga}$.
   

Further, consider any standard chart $\fV$ of $\sR_{\sF_{[k]}}$,
lying over a unique standard chart $\fV'$ of $\sR_{\sF_{[k-1]}}$,
such that $Z_{\sF_{[k]},\Ga}\cap \fV \ne \emptyset$.
We set 
\begin{equation}\label{zero-one-fV-sFk}
\tGa^\zero_\fV= \tGa^\zero_{\fV'} \sqcup 
\{ x_{(\uu,\uv)}\in \La_{F_k} \mid \hbox{$x_\uu$ or $x_\uv \in \Ga$}\}; \;\;\;\;
\tGa^\one_\fV=\tGa^\one_{\fV'}.
\end{equation}

We are now ready to 
prove Lemma \ref{wp-transform-sVk-Ga} (1), (2) and (3) in the case of $\sR_{\sF_{[k]}}$.

(1).  Note that scheme-theoretically, we have
$$\rho_{\sF_{[k]}}^{-1}(Z_{\sF_{ [k-1]},\Ga}) \cap \fV=
 \pi_{\sF_{[k]}, \sF_{[k-1]}}^{-1}(Z_{\sF_{ [k-1]},\Ga}) \cap \sV_{\sF_{[k]}} \cap \fV$$
 where $\pi_{\sF_{[k]},  \sF_{[k-1]}}: \sR_{\sF_{[k]}} \to \sR_{\sF_{[k-1]}}$ is the projection.
 We can apply Lemma \ref{wp-transform-sVk-Ga} (1)  in the case of $\sR_{\sF_{[k-1]}}$
 to  $Z_{\sF_{ [k-1]},\Ga}$ and $\pi_{\sF_{[k]}, \sF_{[k-1]}}^{-1}(Z_{\sF_{ [k-1]},\Ga})$, 
 apply Proposition \ref{equas-fV[k]}  to $\sV_{\sF_{[k]}} \cap \fV$, and 
 use the construction \eqref{construction-ZkGa} of $Z_{\sF_{[ k]},\Ga}$
 (cf. \eqref{construction-ZkLa} and \eqref{zero-one-fV-sFk}), 
 we then obtain that $Z_{\sF_{[ k]},\Ga} \cap \fV$, as a closed subscheme of $\fV$, is defined by 
$$y, \;\; y \in  \tGa^\zero_\fV;\;\;
 y -1, \;\; y \in  \tGa^\one_\fV; \;\;\; \cB^\pq_{[k]}; $$
$$  B_{\fV, (s,t)},\;\; s, t \in S_{F_i} \; \hbox{ with all $i \in [k]$}$$
 $$  L_{\fV,F_i}, \;  \hbox{ with $\bF_i \in \sF^\star_{[k-1]; \fV',\Ga}$}; \; L_{\fV, F_k};\;
   \bF_{\fV,j}, \; k <j \le \up.$$ 
 Now, we let $\sF^\star_{[k]; \fV,\Ga}=\sF^\star_{[k-1];  \fV',\Ga} \cup \{\bF_k\}$.  Then,
the above implies Lemma \ref{wp-transform-sVk-Ga} (1) in the case of $\sR_{\sF_{[k]}}$.

(2). By construction, we have that the composition
$\tZ^\dagger_{\sF_{[k]},\Ga} \to \tZ^\dagger_{\sF_{[k-1]},\Ga} \to Z_\Ga$
is birational. 
This proves Lemma \ref{wp-transform-sVk-Ga} (2) in the case of $\sR_{\sF_{[k]}}$.

  (3). It suffices to prove that if $Z^\dagger_{\sF_{[ k]},\Ga}\cap \fV \subset (y=0)$,
 then $Z_{\sF_{[ k]},\Ga} \cap \fV \subset (y=0).$
 
 If $y=x_{\fV,\uu}$ ($=x_\uu$, cf. the proof of Proposition \ref{meaning-of-var-p-k=0}), 
 then 
$x_\uu \in \Ga$ because $Z^\dagger_{\sF_{[ k]},\Ga}$ is birational to $Z_\Ga$.
Therefore, $Z_{\sF_{[ k]},\Ga} \cap \fV \subset (y=0)$ by \eqref{Ga-rel-sVk=0},
which holds by (the just proved) Lemma \ref{wp-transform-sVk-Ga} (1) for  $\sR_{\sF_{[k]}}$.

Now assume $y=x_{\fV, (\uu, \uv)}$.
Here, $x_{\fV, (\uu, \uv)}$ is the de-homogenization of $x_{(\uu,\uv)}$
(cf. the proof of Proposition \ref{meaning-of-var-p-k=0}). Below,
upon setting $x_{(\uu_{s_{F_i,o}}, \uv_{s_{F_i,o}})} \equiv 1$ for all $i \in [k]$ (cf. Definition \ref{fv-k=0}),
 we can write $x_{\fV, (\uu, \uv)}=x_{\fV', (\uu, \uv)}=x_{(\uu,\uv)}$.

 Suppose $(\uu, \uv) \in \La_{F_i}$ with $i \in [k-1]$.
By taking the images of $Z^\dagger_{\sF_{[ k]},\Ga}\cap \fV \subset (y=0)$
under $\rho_{\sF_{[k]}}$, we obtain 
$Z^\dagger_{\sF_{[ k-1]},\Ga} \cap \fV' \subset (x_{(\uu, \uv)}=0)$.
Hence, we have
$Z_{\sF_{[k-1]},\Ga} \cap \fV' \subset (x_{(\uu, \uv)}=0)$ by 
Lemma \ref{wp-transform-sVk-Ga} (3)  for 
$ \sR_{\sF_{[k-1]}}$.
Therefore,  $x_{(\uu, \uv)} \in \tGa^\zero_{\fV'}$ by the maximality of the subset $ \tGa^\zero_{\fV'}$.
Then, by  \eqref{zero-one-fV-sFk},
$Z_{\sF_{[ k]},\Ga} \cap \fV \subset   (x_{(\uu, \uv)}=0)$.

Now suppose $(\uu, \uv) \in \La_{F_k}$.
 Consider the relations  $$x_{\fV,\uu} x_{\fV,\uv} -x_{\fV, (\uu, \uv)}x_{\uu_{s_{F_k,o}}}x_{\uv_{s_{F_k,o}}} .$$ 
 Here, we have used $x_{(\uu_{s_{F_k,o}}, \uv_{s_{F_k,o}})} \equiv 1$. 
  Then, we have  $x_{\fV,\uu} x_{\fV,\uv}$ vanishes identically along $Z^\dagger_{\sF_{[ k]},\Ga}$,
  hence,  so does
  one of $x_{\fV, \uu}$ and $x_{\fV, \uv}$, that is, $x_\uu$ or $x_\uv \in \Ga$,
  since $Z^\dagger_{\sF_{[ k]},\Ga}$ (birational to $Z_\Ga$) is integral. In either case, it implies that 
  $Z_{\sF_{[ k]},\Ga} \subset (x_{(\uu, \uv)}=0)$ by \eqref{construction-ZkLa} and  
  \eqref{construction-ZkGa}. 




This proves the lemma when $F_k$ is $\Ga$-relevant.

\smallskip

 Next, we suppose $F_k$ is $\Ga$-irrelevant. 
 
Take an arbitrary standard chart $\fV$ of $\sR_{\sF_{[k]}}$ lying over
a unique standard chart $\fV'$ of $\sR_{\sF_{[k-1]}}$ such that
$ \rho_{\sF_{[k]}}^{-1}(Z_{\sF_{  [k-1]},\Ga}) \cap \fV \ne \emptyset$ 
(equivalently, $Z_{\sF_{[k-1]}} \cap \fV' \ne \emptyset$).
Then, we have that
$$( \rho_{\sF_{[k]}}^{-1}(Z_{\sF_{  [k-1]},\Ga}) \cap \fV) / (Z_{\sF_{[k-1]},\Ga}\cap \fV')$$ 
is defined by the set of
equations of $L_{\fV, F_k}$ and $\cB^\pq_{\fV, [k]}$, all regarded as
relations in $\vr$-variables of $F_k$. Thus, all these relations are  
 linear in $\vr$-variables of $F_k$, according to Lemma \ref{ker-phi-k}.
 Putting them together, we call $\{L_{\fV, F_k}, \cB^\pq_{\fV, [k]}\}$ a linear system 
  in $\vr$-variables of $F_k$ on the chart $\fV/\fV'$.
 
 We can let 
 $\La^{\rm det}_{F_k,  \Ga}$ be the subset of $\La_{F_k}$ such that 
  the minor corresponding to  variables
 $$\{x_{\fV,(\uu,\uv)} \mid (\uu,\uv) \in \La^{\rm det}_{F_k, \Ga}\}$$ 
 achieves the maximal rank of the linear system $ \{L_{\fV, F_k}, \cB^\pq_{\fV, [k]}|\}$,
  regarded as
relations in $\vr$-variables of $F_k$,
 at any point of some fixed Zariski open subset $Z_{\sF_{[k-1]},\Ga}^{\dagger\circ}$ of 
 $Z^\dagger_{\sF_{[k-1]},\Ga}$. 
 By shrinking if necessary, we may assume that
 $Z^{\dagger\circ}_{\sF_{[k-1]},\Ga}$ is contained in the intersection of all
 standard charts $\fV'$ with $Z_{\sF_{[k-1]}} \cap \fV' \ne \emptyset$.

Then, we let  
 \begin{equation}\label{Fk-irrf}
 \La^\one_{F_k,\Ga}= \La_{F_k} \- \La^{\rm det}_{F_k,\Ga}.
  \end{equation}
 We then set and plug 
 \begin{equation}\label{a-ne-0}
 x_{\fV, (\uu,\uv)} = 1, \; \forall \; (\uu,\uv)  \in  \La^\one_{F_k, \Ga} 
 \end{equation}
   into the linear system $\{L_{\fV, F_k}, \cB^\pq_{\fV, [k]}\}$
   to obtain an induced linear system of full rank 
   over $Z_{\sF_{[k-1]},\Ga}^{\dagger\circ}$. This induced linear system can be solved
  over the Zariski open subset $Z^{\dagger\circ}_{\sF_{[k-1]},\Ga}$ such that
  all variables $$\{x_{\fV,(\uu,\uv)} \mid (\uu,\uv) \in \La^{\rm det}_{F_k, \Ga}\}$$ 
 are explicitly  determined by the coefficients of the induced linear system.
    
     We then  let 
   \begin{equation}\label{det=0}
   \La^\zero_{F_k, \Ga} \subset \{ (\uu,\uv) \in \La^{\rm det}_{F_k,\Ga}\}
   \end{equation}
   be the subset consisting of all $(\uu,\uv) \in \La^{\rm det}_{F_k ,\Ga}$ 
   such that $x_{(\uu,\uv)} \equiv 0$ over
   $Z^{\dagger\circ}_{\sF_{[k-1]},\Ga}$.
    Observe here that we immediately obtain that  for any $(\uu,\uv) \in \La_{F_k}$,
\begin{equation}\label{who-in-dagger}
\hbox{$x_{(\uu,\uv)}$  vanishes identically over  $Z^{\dagger\circ}_{\sF_{[k-1]},\Ga}$
   if and only if  $(\uu,\uv) \in \La^\zero_{F_k,\Ga}$.}
   \end{equation}

We let $Z_{\sF_{[ k]},\Ga}$ be
  the scheme-theoretic intersection
   \begin{equation}\label{ZkGa-irr}
   \rho_{\sF_{[k]}}^{-1}(Z_{\sF_{ [k-1]},\Ga}) \cap ( x_{(\uu,\uv)}=0 , \; (\uu,\uv) \in  \La^\zero_{F_k,\Ga}; \;\;
  x_{(\uu,\uv)} = 1,  \; (\uu,\uv) \in  \La^\one_{F_k,\Ga}) . \end{equation}

 Further,  for the above standard chart $\fV$ of $\sR_{\sF_{[k]}}$, lying over the standard chart $\fV'$ 
  of $\sR_{\sF_{[k-1]}}$ with $\tZ_{\sF_{[k]},\Ga}\cap \fV \ne \emptyset$,  we set 
\begin{eqnarray}\label{zero-one-fV-sFk-irr}
\tGa^\zero_\fV= \tGa^\zero_{\fV'} \sqcup    \{ x_{\fV, (\uu,\uv)} \mid 
(\uu, \uv) \in \La^\zero_{F_k,\Ga}\}; \\ 
\tGa^\one_\fV=\tGa^\one_{\fV'} \sqcup   \{ x_{\fV, (\uu,\uv)} \mid 
(\uu, \uv) \in \La^\one_{F_k,\Ga}\}.
\end{eqnarray}

We are now ready to 
prove Lemma \ref{wp-transform-sVk-Ga} (1), (2) and (3) in the case of $\sR_{\sF_{[k]}}$.

  Similar to the proof of Lemma \ref{wp-transform-sVk-Ga} (1) for the previous case when
  $\bF_k$ is $\Ga$-relevant, 
  by Lemma \ref{wp-transform-sVk-Ga} (1)  in the case of $\sR_{\sF_{[k-1]}}$
  applied to $Z_{\sF_{ [k-1]},\Ga}$ and $\rho_{\sF_{[k]}}^{-1}(Z_{\sF_{ [k-1]},\Ga})$,
 applying Proposition \ref{equas-fV[k]} to $\sV_{F_{[k]}} \cap \fV$, and using 
 \eqref{ZkGa-irr} and \eqref{zero-one-fV-sFk-irr}, 
 we obtain that $Z_{\sF_{[ k]},\Ga} \cap \fV$, as a closed subscheme of $\fV$, is defined by 
$$y, \;\; y \in  \tGa^\zero_\fV;\;\;
 y -1 \;\; y \in  \tGa^\one_\fV; \;\; \cB^\pq_{[k]}; $$
$$  B_{\fV, (s,t)},\;\; s, t \in S_{F_i} \; \hbox{ with all $i \in [k]$}$$
$$ L_{\fV,F_i}, \; \hbox{ with $\bF_i \in \sF^\star_{[k-1]; \fV',\Ga} $}, \;L_{\fV, F_k}; \; 
 \bF_{\fV,j}, \; k <j \le \up.$$ 
 Then, we say 
 $$\left\{ \begin{array}{rcl}
\bF_k \in \sF^\star_{[k]; \fV,\Ga}  \; ,  & \;\;  \hbox{if $\La_{F_k}
 \ne \La^\zero_{F_k,\Ga} \cup \La^\one_{F_k,\Ga}$;}\\
\bF_k \notin \sF^\star_{[k]; \um,\Ga}\; , & \;\; \hbox{if $\La_{F_k} 
= \La^\zero_{F_k,\Ga} \cup \La^\one_{F_k,\Ga}$.}\\ 
\end{array} \right. $$ 
Put it equivalently, 
$\bF_k \in \sF^\star_{[k]; \fV,\Ga}  $ if 
upon setting
  $y=0$ for all $y \in  \tGa^\zero_\fV$ and
 $y =1$ for all $y \in  \tGa^\one_\fV$, and plugging them into $L_{\fV, F_k}$, we have that $L_{\fV, F_k}$
 still contains a nontrivial $\vr$-variable of $\var_\fV$;
 $\bF_k \notin \sF^\star_{[k]; \fV,\Ga}  $, otherwise.   We then set
 $\sF^\star_{[k]; \fV,\Ga} =\sF^\star_{[k-1];\fV',\Ga}  \cup \{\bF_k\}$ if $\bF_k \in \sF^\star_{[k]; \fV,\Ga} $;
 $\sF^\star_{[k]; \fV,\Ga} =\sF^\star_{[k-1];\fV',\Ga} $ if $\bF_k \notin \sF^\star_{[k];\fV,\Ga} $.
 Then, the above implies that Lemma \ref{wp-transform-sVk-Ga} (1) holds  on $\sR_{\sF_{[k]}}$.



  Next, by construction, the induced morphism
  $$ \rho_{\sF_{[k]}}^{-1} (Z^{\dagger\circ}_{\sF_{[k]},\Ga}) 
  \cap (x_{(\uu,\uv)}=0 , \; (\uu,\uv) \in  \La^\zero_{F_k,\Ga}; \; x_{(\uu,\uv)} = 1, \; (\uu,\uv) \in  \La^\one_{F_k,\Ga}) 
   \lra Z^{\dagger\circ}_{\sF_{[k-1]}}$$
  is an isomorphism. We let 
$Z^\dagger_{\sF_{ [k]},\Ga}$ be the closure of 
$$ \rho_{\sF_{[k]}}^{-1} (Z^{\dagger\circ}_{\sF_{[k]},\Ga}) 
  \cap (x_{(\uu,\uv)}=0 , \; (\uu,\uv) \in  \La^\zero_{F_k,\Ga}; \; x_{(\uu,\uv)} =1, \; (\uu,\uv) \in  \La^\one_{F_k,\Ga}) $$ in $Z_{\sF_{[ k]},\Ga}$.
  Then, it is closed in $Z_{\sF_{[ k]},\Ga}$ and contains 
  an open subset of $Z_{\sF_{[ k]},\Ga}$, hence, is an irreducible
  component of $Z_{\sF_{[ k]},\Ga}$.  It follows that the composition
  $$Z^\dagger_{\sF_{ [k]},\Ga}\to Z^{\dagger}_{\sF_{[k-1]}} \to Z_\Ga$$  is birational.
  This proves    Lemma \ref{wp-transform-sVk-Ga} (2) on $\sR_{\sF_{[k]}}$.
  

Finally,  we are to prove  Lemma \ref{wp-transform-sVk-Ga} (3) on $\sR_{\sF_{[k]}}$.
Suppose $Z^\dagger_{\sF_{ [k]},\Ga} \cap \fV \subset (y=0)$ for some $y \in \var_\fV$.
If $y=x_{\fV, \uu}$ or $y=x_{\fV, (\uu, \uv)}$ with  $(\uu, \uv) \in \La_{F_i}$ with $i \in [k-1]$,
then the identical  proof in the previous case carries over here without changes.
We now suppose $Z^\dagger_{\sF_{ [k]},\Ga}\cap \fV \subset (x_{\fV,(\uu,\uv)}=0)$
with $(\uu, \uv) \in \La_{F_k}$, then by \eqref{who-in-dagger},
$(\uu,\uv) \in \La^\zero_{F_k,\Ga}$. Thus,  by \eqref{ZkGa-irr}, 
$Z_{\sF_{ [k]},\Ga} \subset (x_{(\uu,\uv)}=0)$.
This proves    Lemma \ref{wp-transform-sVk-Ga} (3) on $\sR_{\sF_{[k]}}$.

By induction, Lemma \ref{wp-transform-sVk-Ga}  is proved.
\end{proof}

\begin{defn}
We call $Z_{\sF_{[k]},\Ga}$ the $\sF$-transform of $Z_\Ga$
in $\sV_{\sF_{[k]}}$ for any $k\in [\up]$.
\end{defn}


 We keep the notation of Lemma \ref{wp-transform-sVk-Ga}.
We set
\begin{equation}\label{Lrel-up}
\sF^\star_{\fV,\Ga}=\sF^\star_{[\up];\fV, \Ga},\;
 L_{\fV, \sF^\star_{[k]; \fV,\Ga} }=\{L_{\fV, F} \mid \bF \in  \sF^\star_{[k]; \fV,\Ga} \},\;
L_{\fV, { \sF^\star_{\um,\Ga}}}=L_{\fV, \sF^\star_{[\up]; \fV, \Ga}}. 
\end{equation}




\subsection{$\vt$-transforms of  $\Ga$-schemes in  $\tsV_{\vt_{[k]}}$}
\label{subsection:wp-transform-vtk}  $\ $

We now construct the $\vt$-transform of $Z_\Ga$ in $\tsV_{\vt_{[k]}}
\subset \tsR_{\vt_{[k]}}$. 

\begin{lemma}\label{vt-transform-k} 
 Fix any subset $\Ga$ of $\rU_\um$.  Assume that $Z_\Ga$ is integral. 

Fix any $k \in [\up]$.

Then, we have the following:
\begin{itemize}
\item  there exists a closed subscheme $\tZ_{\vt_{[k]},\Ga}$ of
$\tsV_{\vt_{[k]}}$ with an induced morphism  
$\tZ_{\vt_{[k]},\Ga} 
\to Z_\Ga$;
\item   $\tZ_{\vt_{[k]},\Ga}$ comes equipped with an irreducible component  
$\tZ^\dagger_{\vt_{[k]},\Ga}$ with the induced morphism 
$\tZ^\dagger_{\vt_{[k]},\Ga}  
\to Z_\Ga$;
\item  for any standard chart $\fV$ of $\tsR_{\vt_{[k]}}$ such that
$\tZ_{\vt_{[k]},\Ga} \cap \fV \ne \emptyset$, there are two subsets, possibly empty,
$$ \tGa^\zero_{\fV} \; \subset \;  \var_\fV, \;\;\;
\tGa^\one_{\fV} \; \subset \;  \var_\fV.$$ 
\end{itemize}

Further,  consider any given standard chart $\fV$ of $\tsR_{\vt_{[k]}}$ with
$\tZ_{\vt_{[k]},\Ga} \cap \fV \ne \emptyset$. Then,  the following hold:
\begin{enumerate}
\item the scheme $\tZ_{\vt_{[k]},\Ga} \cap \fV$,
 as a closed subscheme of the chart $\fV$,
is defined by the following relations
\begin{eqnarray} 
\;\;\;\;\; y , \; \; \; y \in  \tGa^\zero_\fV , \label{Ga-rel-wp-ktauh-00}\\
\;\;\;  y -1, \; \; \; y \in  \tGa^\one_\fV,  \nonumber \\
\cB_\fV^\mn, \; \cB^\res_{\fV, >k}, \;  \cB^\q_\fV, \; L_{\fV,\sF^\star_{\um,\Ga}}, \nonumber
\end{eqnarray}
where $L_{\fV,\sF^\star_{\um,\Ga}}$ is as in \eqref{Lrel-up};
 further, we take $\tGa^\zero_\fV \subset \var_\fV$
 to be the maximal subset (under inclusion)
among all those subsets that satisfy the above;
\item the induced morphism $\tZ^\dagger_{ \vt_{[k]},\Ga} \to Z_\Ga$ is birational;
\item for any variable $y \in \var_\fV$, $\tZ^\dagger_{\vt_{[k]},\Ga} \cap \fV \subset (y=0)$ if and only 
if  $\tZ_{\vt_{[k]},\Ga} \cap \fV \subset (y=0)$. Consequently, 
 $\tZ^\dagger_{\vt_{[k]},\Ga} \cap \fV \subset \tZ_{\vt_{[k+1]}} \cap \fV$ if and only 
if  $\tZ_{\vt_{[k]},\Ga} \cap \fV \subset \tZ_{\vt_{[k+1]}}\cap \fV$ 
where $\tZ_{\vt_{[k]}}$ is the proper transform of 
 the $\vt$-center $Z_{\vt_{[k+1]}}$ in  $\tsR_{\vt_{[k]}}$. 
\end{enumerate}
\end{lemma}  
\begin{proof} We prove by induction on 
$k \in \{0\} \cup [\up]$. 


The initial case is $k=0$. In this case, we have
 $$\tsR_{\vt_{[0]}}:=\sR_{\sF}, \;\;
 \tsV_{\vt_{[0]}}:=\sV_{\sF}, \;\;
  \tZ_{\vt_{[0]},\Ga}:=Z_{\sF_{ [\up]},\Ga}, \;\;
  \tZ^\dagger_{\vt_{[0]},\Ga}:=Z^\dagger_{\sF_{ [\up]},\Ga}.$$
Then, in this case,  Lemma \ref{vt-transform-k}
 is Lemma  \ref{wp-transform-sVk-Ga} for $k=\up$.

We now suppose that Lemma \ref{vt-transform-k} holds
over $\tsR_{\vt_{[k-1]}}$ for some $k\in[\up]$.

We then consider the case of  $\tsR_{\vt_{[k]}}$.

Suppose that $\tZ_{\vt_{[k-1]}, \Ga}$,  or equivalently $\tZ^\dagger_{\vt_{[k-1]},\Ga}$,
by Lemma \ref{vt-transform-k} (3) in the case of $\tsR_{\vt_{[k-1]}}$,
 is not contained in 
$Z'_{\vt_{[k]}}$ where $Z'_{\vt_{[k]}}$ is the proper transform in $\tsR_{\vt_{[k-1]}}$
 of the $\vt$-center $Z_{\vt_{[k]}}$ (of $\tsR_{\vt_{[0]}}$).
We then let $\tZ_{\vt_{[k]},\Ga}$ (respectively, $\tZ^\dagger_{\vt_{[k]},\Ga}$)
 be the proper transform of $\tZ_{\vt_{[k-1]},\Ga}$ (respectively,
  $\tZ^\dagger_{\vt_{[k-1]},\Ga}$) in $\tsV_{\vt_{[k]}}$.
As $\tZ^\dagger_{\vt_{[k]},\Ga}$ is closed in $\tZ_{\vt_{[k]},\Ga}$
and contains a Zariski open subset of $\tZ_{\vt_{[k]},\Ga}$, it is an irreducible
component of $\tZ_{\vt_{[k]},\Ga}$.

Further, consider any standard chart $\fV$ of $\tsR_{\vt_{[k]}}$,
lying over a unique standard chart $\fV'$ of $\tsR_{\vt_{[k-1]}}$,
such that $\tZ_{\vt_{[k]},\Ga}\cap \fV \ne \emptyset$.
We set 
$$\tGa^\zero_\fV=\{y_\fV \mid y_\fV  \hbox{ is the proper transform of some $y_{\fV'} \in \tGa^\zero_{\fV'}$}\};$$
$$\tGa^\one_\fV=\{y_\fV \mid y_\fV  \hbox{ is the proper transform of some $y_{\fV'} \in \tGa^\one_{\fV'}$}\}.$$

We now prove Lemma \ref{vt-transform-k} (1), (2) and (3) in the case of $\tsR_{\vt_{[k]}}$.

 We can apply
 Lemma \ref{vt-transform-k} (1) in the case of $\tsR_{\vt_{[k-1]}}$
 to $\tZ_{\vt_{[k-1]},\Ga}$ to obtain the defining equations of $\tZ_{\vt_{[k-1]},\Ga}\cap \fV'$
as stated in the lemma; we note here that these equations include $\cB^\res_{\fV', \ge k}$.
We then  take the proper transforms of these equations in $\fV'$ to obtain the corresponding
equations in $\fV$,  and then apply (the proof of) 
 Proposition \ref{eq-for-sV-vtk} to reduce $\cB^\res_{\fV, \ge k}$ to $\cB^\res_{\fV, > k}$.
 Because $\tZ_{\vt_{[k]}\Ga}$  is the proper transform of $\tZ_{\vt_{[k-1]},\Ga}$,
this implies Lemma \ref{vt-transform-k} (1)  in the case of $\tsR_{\vt_{[k]}}$.

By construction, we have that the composition
$\tZ^\dagger_{\vt_{[k]},\Ga} \to \tZ^\dagger_{\vt_{[k-1]},\Ga} \to Z_\Ga$
is birational. 
This proves
Lemma \ref{vt-transform-k} (2) in the case of $\tsR_{\vt_{[k]}}$.

To show Lemma \ref{vt-transform-k} (3) in $\tsR_{\vt_{[k]}}$, 
we fix any $y \in \var_\fV$. It suffices to show that
if $\tZ^\dagger_{\vt_{[k]},\Ga}\cap \fV
\subset (y=0)$, then $\tZ_{\vt_{[k]},\Ga}\cap \fV
\subset (y=0)$.  By construction, $y \ne \zeta_{\fV,\vt_{[k]}}$, the exceptional variable
in $\var_\fV$ corresponding to the $\vt$-center $Z_{\vt_{[k]}}$. Hence, $y$ is the proper transform
of some $y' \in \var_{\fV'}$. Then, by taking the images of $\tZ^\dagger_{\vt_{[k]},\Ga}\cap \fV
\subset (y=0)$ under the morphism $\rho_{\vt_{[k]}}: \tsV_{\vt_{[k]}} \to \tsV_{\vt_{[k-1]}}$ 
(which is  induced from the blowup morphism $\pi_{\vt_{[k]}}: \tsR_{\vt_{[k]}} \to \tsR_{\vt_{[k-1]}}$), 
we obtain
$\tZ^\dagger_{\vt_{[k-1]},\Ga}\cap \fV'
\subset (y'=0)$,  hence, $\tZ_{\vt_{[k-1]},\Ga}\cap \fV'
\subset (y'=0)$ by the inductive assumption.
 Then, as $\tZ_{\vt_{[k]},\Ga}$ is the proper transform of $\tZ_{\vt_{[k-1]},\Ga}$,
 we obtain $\tZ_{\vt_{[k]},\Ga}\cap \fV \subset (y=0)$.
 
 The last statement Lemma \ref{vt-transform-k} (3)  follows from the above because
 $\tZ_{\vt_{[k+1]}}\cap \fV=(y_0=y_1=0)$ for some $y_0, y_1 \in \var_\fV$.

\smallskip

We now suppose that $\tZ_{\vt_{[k-1]},\Ga}$, or equivalently $\tZ^\dagger_{\vt_{[k-1]},\Ga}$,
 by  Lemma \ref{vt-transform-k} (3) in $\tsR_{\vt_{[k-1]}}$,
 is contained in the proper transform $Z'_{\vt_{[k]}}$ 
 of the $\vt$-center $Z_{\vt_{[k]}}$. 

Consider  any standard chart $\fV$ of $\tsR_{\vt_{[k]}}$,
lying over a unique standard chart $\fV'$ of $\tsR_{\vt_{[k-1]}}$,
such that $\tZ_{\vt_{[k]},\Ga}\cap \fV \ne \emptyset$.

We let $\vt'_{[k]}$
 be the proper transform in  the chart $\fV'$ of the $\vt$-set ${\vt_{[k]}}$.
Then, $\vt'_{[k]}$ consists of two  variables 
 $$\vt'_{[k]}=\{y'_0,y'_1\} \subset \var_{\fV'}.$$

 We let $\PP^1_{[\xi_0,\xi_1]}$ be the factor projective space for the
$\vt$-blowup $\tsR_{\vt_{[k]}} \to \tsR_{\vt_{[k-1]}}$ with $[\xi_0,\xi_1]$ corresponding to $(y'_0,y'_1)$.
Without loss of generality,  we can assume that the open chart $\fV$ is given by 
  $$(\fV' \times (\xi_0 \equiv 1) ) \cap \tsR_{\vt_{[k]}} \subset \fV' \times \PP^1_{[\xi_0,\xi_1]}.$$

 We let $\zeta_\fV:=\zeta_{\fV, \vt_{[k]}} \in \var_\fV$ be such that 
$E_{\vt_{[k]}} \cap \fV=(\zeta_\fV =0)$ where $E_{\vt_{[k]}}$ is the exceptional divisor
of the blowup $\tsR_{\vt_{[k]}} \to \tsR_{\vt_{[k-1]}}$. 
Note here that according to the proof of Proposition \ref{meaning-of-var-vtk},
the variable $y'_0$ 
corresponds to (or turns into) the exceptional $\zeta_\fV$ on the chart $\fV$.
 We then let $y_1 (=\xi_1) \in \var_\fV$ be 
 the proper transform of $y'_1 \in \var_{\fV'}$ on the chart $\fV$.

  
In addition, we  observe that 
 $$\vt'_{[k]}=\{y'_0,  y'_1\} \subset \tGa^\zero_{\fV'}$$ 
because $\tZ_{\vt_{[k]},\Ga}$ is contained in the proper transform
$Z'_{\vt_{[k]}}$ of the $\vt$-center $Z_{\vt_{[k]}}$.

We set, 
\begin{equation}\label{Gazero-ktah-contained-in-bar-vtk} 
\overline{\Ga}^\zero_\fV= \{\zeta_\fV, \; y_\fV \mid y_\fV  
\hbox{ is the proper transform of some $ y_{\fV'} \in \tGa^\zero_{\fV'} \- \vt'_{[k]}$}\},
\end{equation}
\begin{equation}\label{Gaone-ktah-contained-in-bar-vtk}
 \overline{\Ga}^\one_{\fV}=\{\ y_\fV \mid y_\fV  
\hbox{ is the proper transform of some $y_{\fV'} \in \tGa^\one_{\fV'}$} \}.
\end{equation}

Consider the scheme-theoretic pre-image  
$\rho_{\vt_{[k]}}^{-1}(\tZ_{\vt_{[k-1]},\Ga})$
where $\rho_{\vt_{[k]}}: \tsV_{\vt_{[k]}} \to \tsV_{\vt_{[k-1]}}$ is  induced 
from the blowup morphism $\pi_{\vt_{[k]}}: \tsR_{\vt_{[k]}} \to \tsR_{\vt_{[k-1]}}$.

Note that  scheme-theoretically, we have,
$$\rho_{\vt_{[k]}}^{-1}(\tZ_{\vt_{[k-1]},\Ga})  \cap \fV=
\pi_{\vt_{[k]}}^{-1}(\tZ_{\vt_{[k-1]},\Ga}) \cap \tsV_{\vt_{[k]}} \cap \fV .$$
Applying Lemma \ref{vt-transform-k} (1) in $\tsR_{\vt_{[k-1]}}$ to $\tZ_{\vt_{[k-1]},\Ga}$ 
and $\rho_{\vt_{[k]}}^{-1}(\tZ_{\vt_{[k-1]},\Ga})$, 
and applying Proposition \ref{eq-for-sV-vtk} to 
$\tsV_{\vt_{[k]}} \cap \fV$, we obtain that 
$\rho_{\vt_{[k]}}^{-1}(\tZ_{\vt_{[k-1]},\Ga})  \cap \fV$,
as a closed subscheme of $\fV$,  is defined by 
\begin{equation}\label{vt-pre-image-defined-by}
y_\fV \in \overline{\Ga}^\zero_\fV; \;\;  y_\fV-1, \; y_\fV \in \overline{\Ga}^\one_{\fV};\;\;
\cB^\mn_\fV; \;\; \cB^\res_{\fV, >k}; \;\; \cB^\q_{\fV}; \;\; L_{\fV,\sF^\star_{\um,\Ga}}.
\end{equation} 
(Observe here that $\zeta_\fV \in  \overline{\Ga}^\zero_\fV$.)

Thus, by setting $y_\fV=0$ for all $y_\fV \in  \overline{\Ga}^\zero_\fV$ and
 $y_\fV=1$ for all $y_\fV \in \overline{\Ga}^\one_{\fV}$ in 
 $\cB^\mn_\fV, \cB^\res_{\fV, >k}, \cB^\q_{\fV}, L_{\fV,\sF^\star_{\um,\Ga}}$ of  the above,
 we obtain 
 \begin{equation}\label{vt-lin-xi-pre}
 \tilde{\cB}^\mn_\fV, \tilde\cB^\res_{\fV,>k}, \tilde\cB^\q_{\fV}, \tilde{L}_{\fV,\sF^\star_{\um,\Ga}}.
 \end{equation}
 Note that  for any $\bF \in  \sF^\star_{\fV,\Ga}$
 (cf.  \eqref{Lrel-up}), if $L_{\fV, F}$ contains $y_1$,
 then it contains $\zeta_\fV$, hence $\tilde L_{\fV, F}$ does not contain $y_1$.
 We keep those equations of \eqref{vt-lin-xi-pre} that  contain the variable
 $y_1$ and obtain
 \begin{equation}\label{lin-xi-vtk}
 \hat{\cB}^\mn_\fV, \hat\cB^\res_{\fV,>k}, \hat\cB^\q_{\fV}, 
 \end{equation}
 viewed as a  system of equations in $y_1$.
 By Proposition \ref{eq-for-sV-vtk} (the last two statements), 
one sees that \eqref{lin-xi-vtk} is  a  {\it linear} system of equations in $y_1$.
Furthermore, we have that 
 $$(\rho_{\vt_{[k]}}^{-1}(\tZ_{\vt_{[k-1]},\Ga})\cap \fV)/(\tZ_{\vt_{[k-1]},\Ga}\cap \fV')$$
 is defined by the linear system \eqref{lin-xi-vtk}.

There are the following two cases for \eqref{lin-xi-vtk}:

$(\star a)$  the rank of the linear system  $\eqref{lin-xi-vtk}$ equals one 
over general points of  $\tZ^\dagger_{\vt_{[k-1]},\Ga}$.

$(\star b)$  the rank of the linear system  $\eqref{lin-xi-vtk}$ equals zero
at general points of  $\tZ^\dagger_{\vt_{[k-1]},\Ga}$, hence at all points
of $\tZ^\dagger_{\vt_{[k-1]},\Ga}$.

\smallskip\noindent
{\it Proof of Lemma \ref{vt-transform-k} for $\tsR_{\vt_{[k]}}$ under the condition $(\star a)$.}
\smallskip

By the condition $(\star a)$,
 there exists a Zariski open subset $\tZ^{\dagger\circ}_{\vt_{[k-1]},\Ga}$ 
of $\tZ^\dagger_{\vt_{[k-1]},\Ga}$ such that the rank of the linear system
 \eqref{lin-xi-vtk} equals one at any point of  $\tZ^{\dagger\circ}_{\vt_{[k-1]},\Ga}$. 
By solving $y_1$ from
 the linear system \eqref{lin-xi-vtk} over $\tZ^{\dagger\circ}_{\vt_{[k-1]},\Ga}$, we obtain 
that the induced morphism
$$
\rho_{\vt_{[k]}}^{-1}(\tZ^{\dagger\circ}_{\vt_{[k-1]},\Ga}) 
 \lra \tZ^\circ_{\vt_{[k-1]},\Ga}$$ is an isomorphism. 

First, we suppose  $y_1$ is identically zero along 
$\rho_{\vt_{[k]}}^{-1}(\tZ^{\dagger\circ}_{\vt_{[k-1]},\Ga})$.
We then set,  
\begin{equation}\label{Gazero-ktah-contained-in-a-vtk} 
\tGa^\zero_\fV=\{y_1 \} \cup \overline{\Ga}^\zero_\fV
\end{equation}
where $\overline{\Ga}^\zero_\fV$ is as in \eqref{Gazero-ktah-contained-in-bar-vtk}.
In this case, we let
\begin{equation}\label{ZktahGa-with-xi-vtk}
\tZ_{\vt_{[k]},\Ga}=\rho_{\vt_{[k]}}^{-1}(\tZ_{\vt_{[k-1]},\Ga})\cap D_{y_1}
\end{equation}
scheme-theoretically, where $D_{y_1}$ is the closure of $(y_1=0)$ in $\tsR_{\vt_{[k]}}$.
We remark here that $D_{y_1}$ does not depend on the choice of the chart $\fV$.

Next, suppose  $y_1$ is not identically zero along 
$\rho_{\vt_{[k]}}^{-1}(\tZ^{\dagger\circ}_{\vt_{[k-1]},\Ga})$.
We then set,  
\begin{equation}\label{Gazero-ktah-contained-in-a-vtk'} 
\tGa^\zero_\fV= \overline{\Ga}^\zero_\fV
\end{equation}
where $\overline{\Ga}^\zero_\fV$ is as in \eqref{Gazero-ktah-contained-in-bar-vtk}.
In this case, we let
\begin{equation}\label{ZktahGa-without-xi-vtk}
\tZ_{\vt_{[k]},\Ga}=\rho_{\vt_{[k]}}^{-1}(\tZ_{\vt_{[k-1]},\Ga}).
\end{equation}
We always set (under the condition $(\star a)$)
\begin{equation}\label{Gaone-ktah-contained-in-a-vtk} \tGa^\one_{\fV}= \overline{\Ga}^\one_{\fV}
\end{equation}
where $\overline{\Ga}^\one_{\fV}$ is as in  \eqref{Gaone-ktah-contained-in-bar-vtk}.

In each case, by construction, we have 
$$\rho_{\vt_{[k]}}^{-1}(\tZ^{\dagger\circ}_{\vt_{[k-1]},\Ga})  
\subset \tZ_{\vt_{[k]}\Ga},$$ and we let $\tZ^\dagger_{\vt_{[k]},\Ga}$ be the closure of 
$\rho_{(\vt_{[k]}}^{-1}(\tZ^{\dagger\circ}_{\vt_{[k-1]},\Ga}) $ in $\tZ_{\vt_{[k]},\Ga}$.
It is an irreducible component of $\tZ_{\vt_{[k]},\Ga}$ because
$\tZ^\dagger_{\vt_{[k]},\Ga}$ is closed in $\tZ_{\vt_{[k]},\Ga}$ and contains
the Zariski open subset $\rho_{\vt_{[k]}}^{-1}(\tZ^{\dagger\circ}_{\vt_{[k-1]},\Ga}) $
of $\tZ_{\vt_{[k]},\Ga}$.
Then, we obtain that the composition
 $$\tZ^\dagger_{\vt_{[k]},\Ga} \lra \tZ^\dagger_{\vt_{[k-1]},\Ga}
 \lra Z_\Ga$$ is birational.
 This proves Lemma \ref{vt-transform-k} (2)  over $\tsR_{\vt_{[k]}}$.


In each case of the above (i.e., \eqref{Gazero-ktah-contained-in-a-vtk} and
\eqref{Gazero-ktah-contained-in-a-vtk'}), 
  by the paragraph of \eqref{vt-pre-image-defined-by},
one sees that $\tZ_{\vt_{[k]},\Ga}\cap \fV$,
 as a closed subscheme of $\fV$, is defined by the equations as stated in the Lemma.  
 This proves Lemma \ref{vt-transform-k} (1)  over $\tsR_{\vt_{[k]}}$.

It remains to prove Lemma \ref{vt-transform-k} (3) over $\tsR_{\vt_{[k]}}$.

Fix any $y \in \var_\fV$, it suffices to show that
if $\tZ^\dagger_{\vt_{[k]},\Ga}\cap \fV
\subset (y=0)$, then $\tZ_{(\vt_{[k]},\Ga}\cap \fV
\subset (y=0)$.  
If $y \ne \zeta_\fV,  y_1$, then $y$ is the proper transform of some variable $y' \in \var_{\fV'}$.
Hence, by taking the images under $\rho_{\vt_{[k]}}$, we have $\tZ^\dagger_{\vt_{[k-1]},\Ga}\cap \fV'
\subset (y'=0)$; by Lemma \ref{vt-transform-k} (3) in $\tsR_{\vt_{[k-1]}}$, we obtain
$\tZ_{\vt_{[k-1]},\Ga}\cap \fV'
\subset (y'=0)$,  thus $y' \in \tGa^\zero_{\fV'}$ by the maximality of
the subset $\tGa^\zero_{\fV'}$. Therefore, 
$\tZ_{\vt_{[k]},\Ga}\cap \fV \subset ( y=0)$,  
by (the already-proved) Lemma \ref{vt-transform-k} (1) for $\tsR_{\vt_{[k]}}$
(cf. \eqref{Gazero-ktah-contained-in-bar-vtk} and \eqref{Gazero-ktah-contained-in-a-vtk} 
or \eqref{Gazero-ktah-contained-in-a-vtk'}). 
Next, suppose $y = y_1$ (if it occurs). Then, by construction, 
$\tZ_{\vt_{[k]},\Ga}\cap \fV \subset (y=0)$. 
Finally, we let $y =\zeta_\fV$. 
 Again, by construction, $\tZ_{\vt_{[k]},\Ga}\cap \fV \subset ( \zeta_\fV=0)$. 

As earlier, the last statement Lemma \ref{vt-transform-k} (3)  follows from the above.

\smallskip\noindent
{\it  Proof of Lemma \ref{vt-transform-k} over $\tsR_{\vt_{[k]}}$ under the condition $(\star b)$.}
\smallskip

Under the condition $(\star b)$, we have that $$ 
\rho_{\vt_{[k]}}^{-1}(\tZ^\dagger_{\vt_{[k-1]},\Ga})
 \lra \tZ^\dagger_{\vt_{[k-1]},\Ga}$$
can be canonically identified with the trivial $\PP_{[\xi_0,\xi_1]}$-bundle:
$$\rho_{\vt_{[k]}}^{-1}(\tZ^\dagger_{(\vt_{[k-1]},\Ga}) 
= \tZ^\dagger_{\vt_{[k-1]},\Ga}
\times \PP_{[\xi_0,\xi_1]}.$$
In this case, we define 
$$\tZ_{\vt_{[k]},\Ga}= \rho_{\vt_{[k]}}^{-1}(\tZ_{\vt_{[k-1]},\Ga})\cap ((\xi_0,\xi_1)= (1,1)),$$
 $$\tZ^\dagger_{\vt_{[k]},\Ga}= 
 \rho_{\vt_{[k]}}^{-1}(\tZ^\dagger_{\vt_{[k-1]},\Ga})
 \cap ((\xi_0,\xi_1)= (1,1)),$$
 both scheme-theoretically.
 The induced morphism $\tZ^\dagger_{\vt_{[k]},\Ga} \lra \tZ^\dagger_{\vt_{[k-1]},\Ga}$
 is an isomorphism. 
 Again, one sees that $\tZ^\dagger_{\vt_{[k]},\Ga}$ is 
 an irreducible component of $\tZ_{\vt_{[k]},\Ga}$. Therefore,
 $$\tZ^\dagger_{\vt_{[k]},\Ga} \lra
 \tZ^\dagger_{\vt_{[k-1]},\Ga} \lra  Z_\Ga$$ is  birational.
 This proves Lemma \ref{wp-transform-ktauh} (2) over $\tsR_{\vt_{[k]}}$.

Further, under the condition $(\star b)$, we set 
 \begin{eqnarray}
 \tGa^\zero_\fV= \overline{\Ga}^\zero_\fV,  \;\;\;\;
\tGa^\one_\fV=\{y_1\}\cup \overline{\Ga}^\one_{\fV}. \nonumber \end{eqnarray}
Then, again, by the paragraph of \eqref{vt-pre-image-defined-by},
one sees that $\tZ_{\vt_{[k]},\Ga}\cap \fV$,
 as a closed subscheme of $\fV$, is defined by the equations as stated in the Lemma. 
 This proves Lemma \ref{vt-transform-k} (1) over $\tsR_{\vt_{[k]}}$.

 It remains to prove Lemma \ref{vt-transform-k} (3) in over $\tsR_{\vt_{[k]}}$.

Fix any $y \in \var_\fV$, it suffices to show that
if $\tZ^\dagger_{\vt_{[k]}, \Ga}\cap \fV
\subset (y=0)$, then $\tZ_{\vt_{[k]},\Ga}\cap \fV
\subset (y=0)$.  By construction, $y \ne y_1$.
Then, the corresponding proof of Lemma \ref{vt-transform-k} (3) for $\tsR_{\vt_{[k]}}$
under the condition $(\star a)$ goes through here without change.
 The last statement Lemma \ref{vt-transform-k} (3)  follows from the above.
This proves Lemma \ref{vt-transform-k} (3) in $\tsR_{\vt_{[k]}}$
under the condition $(\star b)$.

This completes the proof of Lemma \ref{vt-transform-k}.
\end{proof}

 \begin{defn}
 We call $\tZ_{\vt_{[k]},\Ga}$ the $\vt$-transform of $Z_\Ga$ 
 in $\tsV_{\vt_{[k]}}$ for any $k \in [\up]$.
 \end{defn}

  We need the final case of Lemma 
\ref{vt-transform-k}. We set
$$\tZ_{\vt, \Ga}:=\tZ_{\vt_{[\up]},\Ga},
 \;\; \tZ^\dagger_{\vt, \Ga}:=\tZ^\dagger_{\vt_{[\up]},\Ga}.$$

\subsection{${\wp}$-transforms of  $\Ga$-schemes in   $\tsV_{({\wp}_{(k\tau)}\fr_\mu\fs_h)}$}
\label{subsection:wp-transform-vsk}  $\ $

We now construct the ${\wp}$-transform of $Z_\Ga$ in $\tsV_{({\wp}_{(k\tau)}\fr_\mu\fs_h)}
\subset \tsR_{({\wp}_{(k\tau)}\fr_\mu\fs_h)}$. 


\begin{lemma}\label{wp-transform-ktauh} 
 Fix any subset $\Ga$ of $\rU_\um$.  Assume that $Z_\Ga$ is integral. 

Consider $(k\tau) \mu h \in  \Index_\Phi$.

Then, we have the following:
\begin{itemize}
\item  there exists a closed subscheme $\tZ_{({\wp}_{(k\tau)}\fr_\mu\fs_h),\Ga}$ of
$\tsV_{({\wp}_{(k\tau)}\fr_\mu\fs_h)}$ with an induced morphism  
$\tZ_{({\wp}_{(k\tau)}\fr_\mu\fs_h),\Ga} 
\to Z_\Ga$;
\item   $\tZ_{({\wp}_{(k\tau)}\fr_\mu\fs_h),\Ga}$ comes equipped with an irreducible component  
$\tZ^\dagger_{({\wp}_{(k\tau)}\fr_\mu\fs_h),\Ga}$ with the induced morphism 
$\tZ^\dagger_{({\wp}_{(k\tau)}\fr_\mu\fs_h),\Ga}  
\to Z_\Ga$;
\item  for any standard chart $\fV$ of $\tsR_{({\wp}_{(k\tau)}\fr_\mu\fs_h)}$ such that
$\tZ_{({\wp}_{(k\tau)}\fr_\mu\fs_h),\Ga} \cap \fV \ne \emptyset$, there come equipped with
two subsets, possibly empty,
$$ \tGa^\zero_{\fV} \; \subset \;  \var_\fV, \;\;\;
\tGa^\one_{\fV} \; \subset \;  \var_\fV.$$ 
\end{itemize}

Further,  consider any given chart $\fV$ of $\tsR_{({\wp}_{(k\tau)}\fr_\mu\fs_h)}$ with
$\tZ_{({\wp}_{(k\tau)}\fr_\mu\fs_h),\Ga} \cap \fV \ne \emptyset$. 
Then,  the following hold:
\begin{enumerate}
\item the scheme $\tZ_{({\wp}_{(k\tau)}\fr_\mu\fs_h),\Ga} \cap \fV$,
 as a closed subscheme of the chart $\fV$,
is defined by the following relations
\begin{eqnarray} 
\;\;\;\;\; y , \; \; \; y \in  \tGa^\zero_\fV , \label{Ga-rel-wp-ktauh}\\
\;\;\;  y -1, \; \; \; y \in  \tGa^\one_\fV,  \nonumber \\
\cB_\fV^\mn, \; \cB^\q_\fV, \; L_{\fV, \sF^\star_{\um,\Ga}}; \nonumber
\end{eqnarray}
 further, we take $\tGa^\zero_\fV \subset \var_\fV$
 to be the maximal subset (under inclusion)
among all those subsets that satisfy the above.
\item the induced morphism $\tZ^\dagger_{({\wp}_{(k\tau)}\fr_\mu\fs_h),\Ga} \to Z_\Ga$ is birational; 
\item for any variable $y \in \var_\fV$, $\tZ^\dagger_{({\wp}_{(k\tau)}\fr_\mu\fs_h),\Ga} \cap \fV \subset (y=0)$ if and only 
if  $\tZ_{({\wp}_{(k\tau)}\fr_\mu\fs_h),\Ga} \cap \fV \subset (y=0)$. Consequently, 
 $\tZ^\dagger_{({\wp}_{(k\tau)}\fr_\mu\fs_h),\Ga}\cap \fV \subset \tZ_{\phi_{(k\tau)\mu(h+1)}}\cap \fV$ if and only 
if  $\tZ_{({\wp}_{(k\tau)}\fr_\mu\fs_h),\Ga} \cap \fV \subset \tZ_{\phi_{(k\tau)\mu(h+1)}}\cap \fV$ 
where $\tZ_{\phi_{(k\tau)\mu(h+1)}}$ is the proper transform of the $\wp$-center 
$Z_{\phi_{(k\tau)\mu(h+1)}}$ in  $\tsR_{({\wp}_{(k\tau)}\fr_\mu\fs_h)}$. 
 \end{enumerate}
\end{lemma}  
\begin{proof} We prove by induction on 
 $(k\tau) \mu h \in  \{(11)10)\} \sqcup \Index_{\Phi}$ (cf. \eqref{indexing-Phi}).


The initial case is $(11)10$. In this case, we have
 $$\tsR_{({\wp}_{(11)}\fr_1\fs_0)}:=\tsR_{\vt}, \;\;
 \tsV_{({\wp}_{(11)}\fr_1\fs_0)}:=\tsV_{\vt} , \;\;
  \tZ_{({\wp}_{(11)}\fr_1\fs_0),\Ga}:=\tZ_{\vt,\Ga}, \;\;
   \tZ^\dagger_{({\wp}_{(11)}\fr_1\fs_0),\Ga}:=\tZ^\dagger_{\vt,\Ga}.$$
Then, in this case,  Lemma \ref{wp-transform-ktauh}
 is Lemma \ref{vt-transform-k} for $k=\up$.

We now suppose that Lemma \ref{wp-transform-ktauh} holds for $(k\tau)\mu(h-1)$
for some $(k\tau)\mu h \in \Index_\Phi$.

We then consider the case of $(k\tau)\mu h$.

We let 
\begin{equation}\label{rho-wpktaumuh}
 \rho_{(\wp_{(k\tau)}\fr_\mu\fs_h)}: \tsV_{(\wp_{(k\tau)}\fr_\mu\fs_h)} \lra
\tsV_{(\wp_{(k\tau)}\fr_\mu\fs_{h-1})}\end{equation}
be the morphism induced from
$ \pi_{(\wp_{(k\tau)}\fr_\mu\fs_h)}: \tsR_{(\wp_{(k\tau)}\fr_\mu\fs_h)} \to
\tsR_{(\wp_{(k\tau)}\fr_\mu\fs_{h-1})}$.

Suppose that $\tZ_{({\wp}_{(k\tau)}\fr_\mu\fs_{h-1}),\Ga}$,  or equivalently 
$\tZ^\dagger_{({\wp}_{(k\tau)}\fr_\mu\fs_{h-1}),\Ga}$,
by Lemma \ref{wp-transform-ktauh}  (3) in $({\wp}_{(k\tau)}\fr_\mu\fs_{(h-1)})$,
 is not contained in 
$Z'_{\phi_{(k\tau)\mu h}}$ where $Z'_{\phi_{(k\tau)\mu h}}$ is the proper transform 
in $\tsR_{({\wp}_{(k\tau)}\fr_\mu\fs_{h-1})}$
 of the ${\wp}$-center $Z_{\phi_{(k\tau)\mu h}}$(of $\tsR_{({\wp}_{(k\tau)}\fr_{\mu-1})}$).
We then let $\tZ_{({\wp}_{(k\tau)}\fr_\mu\fs_h),\Ga}$ (resp. $\tZ^\dagger_{({\wp}_{(k\tau)}\fr_\mu\fs_h),\Ga}$) be the
proper transform of $\tZ_{({\wp}_{(k\tau)}\fr_\mu\fs_{h-1}),\Ga}$
 (resp. $\tZ^\dagger_{({\wp}_{(k\tau)}\fr_\mu\fs_{h-1}),\Ga}$)
in $\sV_{({\wp}_{(k\tau)}\fr_\mu\fs_h)}$.
As $\tZ^\dagger_{({\wp}_{(k\tau)}\fr_\mu\fs_h),\Ga}$ is closed in $\tZ_{({\wp}_{(k\tau)}\fr_\mu\fs_h),\Ga}$
and contains a Zariski open subset of $\tZ_{({\wp}_{(k\tau)}\fr_\mu\fs_h),\Ga}$, it is an irreducible
component of $\tZ_{({\wp}_{(k\tau)}\fr_\mu\fs_h),\Ga}$.

Further, consider any standard chart $\fV$ of $\tsR_{({\wp}_{(k\tau)}\fr_\mu\fs_h)}$,
lying over a unique standard chart $\fV'$ of $\tsR_{({\wp}_{(k\tau)}\fr_\mu\fs_{h-1})}$,
such that $\tZ_{({\wp}_{(k\tau)}\fr_\mu\fs_h),\Ga}\cap \fV \ne \emptyset$.
We set 
$$\tGa^\zero_\fV=\{y_\fV \mid y_\fV  
\hbox{ is the proper transform of some $y_{\fV'} \in \tGa^\zero_{\fV'}$}\};$$
$$\tGa^\one_\fV=\{y_\fV \mid y_\fV  
\hbox{ is the proper transform of some $y_{\fV'} \in \tGa^\one_{\fV'}$}\}.$$

We now prove  Lemma \ref{wp-transform-ktauh}  (1), (2) and (3) in $({\wp}_{(k\tau)}\fr_\mu\fs_h)$.

Lemma \ref{wp-transform-ktauh}  (1) in $({\wp}_{(k\tau)}\fr_\mu\fs_h)$ 
follows  from Lemma \ref{wp-transform-ktauh} (1) in $({\wp}_{(k\tau)}\fr_\mu\fs_{h-1})$
because $\tZ_{({\wp}_{(k\tau)}\fr_\mu\fs_h),\Ga}$  is the proper transform of 
$\tZ_{({\wp}_{(k\tau)}\fr_\mu\fs_{h-1}),\Ga}$. 

By construction, we have that 
$\tZ^\dagger_{({\wp}_{(k\tau)}\fr_\mu\fs_h),\Ga} \to \tZ^\dagger_{({\wp}_{(k\tau)}\fr_\mu\fs_{h-1}),\Ga} \to Z_\Ga$
is birational.   
This proves Lemma \ref{wp-transform-ktauh}  (2) in $({\wp}_{(k\tau)}\fr_\mu\fs_h)$.

To show Lemma \ref{wp-transform-ktauh}  (3) in $({\wp}_{(k\tau)}\fr_\mu\fs_h)$, 
we fix any $y \in \var_\fV$. It suffices to show that
if $\tZ^\dagger_{({\wp}_{(k\tau)}\fr_\mu\fs_h),\Ga}\cap \fV
\subset (y=0)$, then $\tZ_{({\wp}_{(k\tau)}\fr_\mu\fs_h),\Ga}\cap \fV
\subset (y=0)$.  By construction, $y \ne \zeta_{\fV,(k\tau)\mu h}$, the exceptional variable
in $\var_\fV$ corresponding to the $\wp$-set $\phi_{(k\tau)\mu h}$. Hence, $y$ is the proper transform
of some $y' \in \var_{\fV'}$. Then, by taking the images under the morphism
$\rho_{(\wp_{(k\tau)}\fr_\mu\fs_h)}$ of \eqref{rho-wpktaumuh}, we obtain
$\tZ^\dagger_{({\wp}_{(k\tau)}\fr_\mu\fs_{h-1}),\Ga}\cap \fV'
\subset (y'=0)$,  hence, $\tZ_{({\wp}_{(k\tau)}\fr_\mu\fs_{h-1}),\Ga}\cap \fV'
\subset (y'=0)$ by the inductive assumption.
 Then,  as $\tZ_{({\wp}_{(k\tau)}\fr_\mu\fs_{h}),\Ga}$ 
 is the proper transform of $\tZ_{({\wp}_{(k\tau)}\fr_\mu\fs_{h-1}),\Ga}$,
 we obtain $\tZ_{({\wp}_{(k\tau)}\fs_{h}),\Ga}\cap \fV \subset (y=0)$.
  The last statement  of Lemma \ref{wp-transform-ktauh}  (3)  follows from the above because
 $\tZ_{\phi_{(k\tau)\mu(h+1)}}\cap \fV=(y_0=y_1=0)$ for some $y_0, y_1 \in \var_\fV$.

\smallskip

We now suppose that $\tZ_{({\wp}_{(k\tau)}\fr_\mu\fs_{h-1}),\Ga}$, or equivalently 
$\tZ^\dagger_{({\wp}_{(k\tau)}\fr_\mu\fs_{h-1}),\Ga}$,
 by  Lemma \ref{wp-transform-ktauh}  (3) in $({\wp}_{(k\tau)}\fr_\mu\fs_{(h-1)})$,
 is contained in the proper transform  $Z'_{\phi_{(k\tau)\mu h}}$ 
 of $Z_{\phi_{(k\tau)\mu h}}$. 

Consider  any standard chart $\fV$ of $\tsR_{({\wp}_{(k\tau)}\fr_\mu\fs_h)}$,
lying over a unique standard chart $\fV'$ of $\tsR_{({\wp}_{(k\tau)}\fr_\mu\fs_{h-1})}$,
such that $\tZ_{({\wp}_{(k\tau)}\fr_\mu\fs_{h-1}),\Ga}\cap \fV' \ne \emptyset$.

We let $\phi'_{(k\tau)\mu h}$
 be the proper transform in  the chart $\fV'$ of the $\wp$-set $\phi_{(k\tau)\mu h}$. Then,  
 $\phi'_{(k\tau)\mu h}$ consists of two  variables such that
 $$\psi'_{(k\tau)\mu h}=\{y'_0,  y'_1\} \subset \var_{\fV'}.$$ 
In addition, we let $\zeta_\fV \in \var_\fV$  be 
 such that  $E_{({\wp}_{(k\tau)}\fr_\mu\fs_{h})} \cap \fV=(\zeta_\fV =0)$ 
 where $E_{({\wp}_{(k\tau)}\fr_\mu\fs_{h})}$
 is the exceptional divisor
of the blowup $\tsR_{({\wp}_{(k\tau)}\fr_\mu\fs_{h})} \to \tsR_{({\wp}_{(k\tau)}\fr_\mu\fs_{h-1})}$. 
  Without loss of generality, we may assume that $y'_0$ 
corresponds to the exceptional variable $\zeta_\fV$ on the chart $\fV$.
 We then let $y_1 \in \var_\fV$ be 
 the proper transform of $y'_1$. 

Now, we observe that 
 $$\phi'_{(k\tau)\mu h}\subset \tGa^\zero_{\fV'}$$
because $\tZ_{({\wp}_{(k\tau)}\fr_\mu\fs_{h-1}),\Ga}$ is contained in the proper transform
$Z'_{\phi_{(k\tau)\mu h}}$.

We set, 
\begin{equation}\label{Gazero-wp-contained-in-bar} 
\overline{\Ga}^\zero_\fV= \{\zeta_\fV, \; y_\fV \mid y_\fV  
\hbox{ is the proper transform of some $ y_{\fV'} \in \tGa^\zero_{\fV'} \-\phi'_{(k\tau)\mu h}$}\},
\end{equation}
\begin{equation}\label{Gaone-wp-contained-in-bar}
 \overline{\Ga}^\one_{\fV}=\{\ y_\fV \mid y_\fV  
\hbox{ is the proper transform of some $y_{\fV'} \in \tGa^\one_{\fV'}$} \}.
\end{equation}

Consider the scheme-theoretic pre-image  
$\rho_{({\wp}_{(k\tau)}\fr_\mu\fs_h)}^{-1}(\tZ_{({\wp}_{(k\tau)}\fr_\mu\fs_{h-1}),\Ga})$.

 Note that scheme-theoretically, we have
$$\rho_{({\wp}_{(k\tau)}\fr_\mu\fs_h)}^{-1}(\tZ_{({\wp}_{(k\tau)}\fr_\mu\fs_{h-1}),\Ga}) \cap \fV=
\pi_{({\wp}_{(k\tau)}\fr_\mu\fs_h)}^{-1}(\tZ_{({\wp}_{(k\tau)}\fr_\mu\fs_{h-1}),\Ga})
 \cap \tsV_{({\wp}_{(k\tau)}\fr_\mu\fs_h)} \cap \fV .$$
Applying Lemma \ref{wp-transform-ktauh} (1)  in $({\wp}_{(k\tau)}\fr_\mu\fs_{h-1})$ to 
$\tZ_{({\wp}_{(k\tau)}\fr_\mu\fs_{h-1}),\Ga}$
and $\rho_{({\wp}_{(k\tau)}\fr_\mu\fs_h)}^{-1}(\tZ_{({\wp}_{(k\tau)}\fr_\mu\fs_{h-1}),\Ga}) $, 
and applying Proposition \ref{equas-vskmuh} to 
$\tsV_{({\wp}_{(k\tau)}\fr_\mu\fs_h)} \cap \fV $, we obtain that  the pre-image 
$\rho_{({\wp}_{(k\tau)}\fr_\mu\fs_h)}^{-1}(\tZ_{({\wp}_{(k\tau)}\fr_\mu\fs_{h-1}),\Ga}) \cap \fV$,
as a closed subscheme of $\fV$,  is defined by 
\begin{equation}\label{wp-pre-image-defined-by}
y_\fV \in \overline{\Ga}^\zero_\fV; \;\;  y_\fV-1, \; y_\fV \in \overline{\Ga}^\one_{\fV};\;\;
\cB^\mn_\fV; \;\; \cB^\q_{\fV}; \;\; L_{\fV,\sF^\star_{\um,\Ga}}.
\end{equation} 
(Observe here that $\zeta_\fV \in  \overline{\Ga}^\zero_\fV$.)


Thus, by setting 
$$\hbox{ $y_\fV=0$ for all $y_\fV \in \overline{\Ga}^\zero_\fV$ and
 $y_\fV=1$ for all $y_\fV \in \overline{\Ga}^\one_{\fV}$} $$ in $\cB^\mn_\fV, 
 \cB^\q_{\fV}, L_{\fV,  \sF^\star_{\um,\Ga}}$ of the above,
 we obtain 
 \begin{equation}\label{vs-lin-xi-pre} \tilde{\cB}^\mn_\fV, 
  \tilde\cB^\q_{\fV}, \tilde{L}_{\fV,  \sF^\star_{\um,\Ga}}.
 \end{equation}
 Note that for any $\bF \in \sF^\star_{\um,\Ga}$,
  if a term of $L_{\fV, F}$ contains $y_1 \in \var_\fV$, 
   then it contains $\zeta_\fV y_1$, hence $\tilde L_{\fV, F}$ does not contain $y_1$.
  We keep those equations of \eqref{vs-lin-xi-pre} such that they contain the variable $y_1 \in \var_\fV$ 
   and obtain
 \begin{equation}\label{vs-lin-xi}
 \hat{\cB}^\mn_\fV, 
 \hat\cB^\q_{\fV}, 
 \end{equation}
 viewed as a  system of equations in $y_1$.
Then, by Proposition \ref{equas-vskmuh}  (1) and (2),
\eqref{vs-lin-xi} is  a  {\it linear} system of equations in $y_1$.
(We point out that $y_1$ here can correspond to either $y_0$ or $y_1$ as in Proposition \ref{equas-vskmuh}.)
Furthermore, one sees that 
 $$(\rho_{({\wp}_{(k\tau)}\fr_\mu\fs_h)}^{-1}(\tZ_{({\wp}_{(k\tau)}\fr_\mu\fs_{h-1}),\Ga})\cap \fV)/
 (\tZ_{({\wp}_{(k\tau)}\fr_\mu\fs_{h-1}),\Ga}\cap \fV')$$
 is defined by the linear system \eqref{vs-lin-xi}.

There are the following two cases for \eqref{vs-lin-xi}:

$(\star a)$  the ranks of the linear system  $\eqref{vs-lin-xi}$ equal one
at  general points of  $\tZ^\dagger_{({\wp}_{(k\tau)}\fr_\mu\fs_{h-1}),\Ga}$.

$(\star b)$  the ranks of the linear system  $\eqref{vs-lin-xi}$ equal zero
at general points of  $\tZ^\dagger_{({\wp}_{(k\tau)}\fr_\mu\fs_{h-1}),\Ga}$, hence at all points
of  $\tZ^\dagger_{({\wp}_{(k\tau)}\fr_\mu\fs_{h-1}),\Ga}$.

\smallskip\noindent
{\it 
Proof of Lemma  \ref{wp-transform-ktauh}   in $({\wp}_{(k\tau)}\fr_\mu \fs_h)$ under the condition $(\star a)$.}
\smallskip

By the condition $(\star a)$,
 there exists a Zariski open subset $\tZ^{\dagger\circ}_{({\wp}_{(k\tau)}\fr_\mu\fs_{h-1}),\Ga}$ 
of $\tZ^\dagger_{({\wp}_{(k\tau)}\fr_\mu\fs_{h-1}),\Ga}$ such that the rank of the linear system
  $\eqref{vs-lin-xi}$ equals one at any point of  $\tZ^{\dagger\circ}_{({\wp}_{(k\tau)}\fr_\mu\fs_{h-1}),\Ga}$. 
By solving  $y_1$ from
 the linear system \eqref{vs-lin-xi} over $\tZ^{\dagger\circ}_{({\wp}_{(k\tau)}\fr_\mu\fs_{h-1}),\Ga}$, we obtain 
that the induced morphism
$$
\rho_{({\wp}_{(k\tau)}\fr_\mu\fs_h)}^{-1}(\tZ^{\dagger\circ}_{({\wp}_{(k\tau)}\fr_\mu\fs_{h-1}),\Ga})  \lra \tZ^\circ_{({\wp}_{(k\tau)}\fr_\mu\fs_{h-1}),\Ga}$$ is an isomorphism. 

Suppose  $y_1$ is identically zero along 
$\rho_{({\wp}_{(k\tau)}\fr_\mu\fs_h)}^{-1}(\tZ^{\dagger\circ}_{({\wp}_{(k\tau)}\fr_\mu\fs_{h-1}),\Ga})$.
We then set,  
\begin{equation}\label{Gazero-ktah-contained-in-a} 
\tGa^\zero_\fV=\{y_1 \} \cup \overline{\Ga}^\zero_\fV
\end{equation}
where $\overline{\Ga}^\zero_\fV$ is as in \eqref{Gazero-wp-contained-in-bar}.
In this case, we let
\begin{equation}\label{ZktahGa-with-xi}
\tZ_{({\wp}_{(k\tau)}\fr_\mu\fs_{h}),\Ga}=\rho_{({\wp}_{(k\tau)}\fr_\mu\fs_h)}^{-1}(\tZ_{({\wp}_{(k\tau)}\fr_\mu\fs_{h-1}),\Ga})\cap D_{y_1}
\end{equation}
scheme-theoretically, where $D_{y_1}$
 is the closure of $(y_1=0)$ in $\tsR_{({\wp}_{(k\tau)}\fr_\mu\fs_h)}$.
We remark here that $D_{y_1}$ does not depend on the choice of the chart $\fV$.

Suppose $y_1$ is not identically zero along 
$\rho_{({\wp}_{(k\tau)}\fr_\mu\fs_h)}^{-1}(\tZ^{\dagger\circ}_{({\wp}_{(k\tau)}\fr_\mu\fs_{h-1}),\Ga})$.
We then set,  
\begin{equation}\label{Gazero-ktah-contained-in-a-2} 
\tGa^\zero_\fV= \overline{\Ga}^\zero_\fV
\end{equation}
where $\overline{\Ga}^\zero_\fV$ is as in \eqref{Gazero-wp-contained-in-bar} .
In this case, we let
\begin{equation}\label{ZktahGa-without-xi}
\tZ_{({\wp}_{(k\tau)}\fr_\mu\fs_{h}),\Ga}=\rho_{({\wp}_{(k\tau)}\fr_\mu\fs_h)}^{-1}(\tZ_{({\wp}_{(k\tau)}\fr_\mu\fs_{h-1}),\Ga}).
\end{equation}
We always set (under the condition $(\star a)$)
\begin{equation}\label{Gaone-ktah-contained-in-a} \tGa^\one_{\fV}= \overline{\Ga}^\one_{\fV}
\end{equation}
where $\overline{\Ga}^\one_{\fV}$ is as in  \eqref{Gaone-wp-contained-in-bar}.

In each case, we have 
$$\rho_{({\wp}_{(k\tau)}\fr_\mu\fs_h)}^{-1}(\tZ^{\dagger\circ}_{({\wp}_{(k\tau)}\fr_\mu\fs_{h-1}),\Ga})  
\subset \tZ_{({\wp}_{(k\tau)}\fr_\mu\fs_h),\Ga},$$ and we let $\tZ^\dagger_{({\wp}_{(k\tau)}\fr_\mu\fs_h),\Ga}$ be the closure of 
$\rho_{({\wp}_{(k\tau)}\fr_\mu\fs_h)}^{-1}(\tZ^{\dagger\circ}_{({\wp}_{(k\tau)}\fr_\mu\fs_{h-1}),\Ga}) $ in $\tZ_{({\wp}_{(k\tau)}\fr_\mu\fs_h),\Ga}$.
It is an irreducible component of $\tZ_{({\wp}_{(k\tau)}\fr_\mu\fs_h),\Ga}$ because
$\tZ^\dagger_{({\wp}_{(k\tau)}\fr_\mu\fs_h),\Ga}$ is closed in $\tZ_{({\wp}_{(k\tau)}\fr_\mu\fs_h),\Ga}$ and contains
the Zariski open subset $\rho_{({\wp}_{(k\tau)}\fr_\mu\fs_h)}^{-1}(\tZ^{\dagger\circ}_{({\wp}_{(k\tau)}\fr_\mu\fs_{h-1}),\Ga}) $
of $\tZ_{({\wp}_{(k\tau)}\fr_\mu\fs_h),\Ga}$.
Then, it follows that the composition
 $$\tZ^\dagger_{({\wp}_{(k\tau)}\fr_\mu\fs_h),\Ga} \lra \tZ^\dagger_{({\wp}_{(k\tau)}\fr_\mu\fs_{h-1}),\Ga}
 \lra Z_\Ga$$ is birational.  
 This proves Lemma \ref{wp-transform-ktauh} (2) in $({\wp}_{(k\tau)}\fs_h)$.


In each case of \eqref{Gazero-ktah-contained-in-a} and \eqref{Gazero-ktah-contained-in-a-2}, 
by the paragraph of \eqref{wp-pre-image-defined-by},
we have that $\tZ_{({\wp}_{(k\tau)}\fr_\mu\fs_h),\Ga}\cap \fV$
 as a closed subscheme of $\fV$ is defined by the equations as stated in the Lemma. 
 This proves Lemma \ref{wp-transform-ktauh} (1) in $({\wp}_{(k\tau)}\fr_\mu\fs_h)$.

It remains to prove Lemma \ref{wp-transform-ktauh} (3) in $({\wp}_{(k\tau)}\fr_\mu\fs_h)$.

Fix any $y \in \var_\fV$, it suffices to show that
if $\tZ^\dagger_{({\wp}_{(k\tau)}\fr_\mu\fs_{h}),\Ga}\cap \fV
\subset (y=0)$, then $\tZ_{({\wp}_{(k\tau)}\fr_\mu\fs_{h}),\Ga}\cap \fV
\subset (y=0)$.  
If $y \ne \zeta_\fV,  y_1$, then $y$ is the proper transform of some variable $y' \in \var_{\fV'}$.
Hence, by taking the images under $\rho_{({\wp}_{(k\tau)}\fr_\mu\fs_h)}$,
 we obtain $\tZ^\dagger_{({\wp}_{(k\tau)}\fr_\mu\fs_{h-1}),\Ga}\cap \fV'
\subset (y'=0)$, and then,  by Lemma \ref{wp-transform-ktauh} (3) in (${\wp}_{(k\tau)}\fr_\mu\fs_{h-1}$),
$\tZ_{({\wp}_{(k\tau)}\fr_\mu\fs_{h-1}),\Ga}\cap \fV'
\subset (y'=0)$,  thus $y' \in \tGa^\zero_{\fV'}$ by the maximality of $\tGa^\zero_{\fV'}$. 
 Therefore, 
$\tZ_{({\wp}_{(k\tau)}\fr_\mu\fs_{h}),\Ga}\cap \fV \subset ( y=0)$,  
by (the already-proved) Lemma \ref{wp-transform-ktauh} (1) in $(\wp_{(k\tau)}\fr_\mu\fs_h)$.
Next, suppose $y = y_1$ (if it occurs). Then, by construction,
 $\tZ_{({\wp}_{(k\tau)}\fs_{h}),\Ga}\cap \fV \subset (y=0)$. 
Finally, we let $y =\zeta_\fV$. 
 Again, by construction, $\tZ_{({\wp}_{(k\tau)}\fr_\mu\fs_{h}),\Ga}\cap \fV \subset (\zeta_\fV=0)$. 
As in the previous case, the last statement  of Lemma \ref{wp-transform-ktauh}  (3)  follows from the above.

\smallskip\noindent
{\it  Proof of Lemma \ref{wp-transform-ktauh} in $({\wp}_{(k\tau)}\fr_\mu\fs_h)$
 under the condition $(\star b)$.}
\smallskip

Under the condition $(\star b)$, we have that $$ 
\rho_{({\wp}_{(k\tau)}\fr_\mu\fs_h)}^{-1}(\tZ^\dagger_{({\wp}_{(k\tau)}\fr_\mu\fs_{h-1}),\Ga})
 \lra \tZ^\dagger_{({\wp}_{(k\tau)}\fr_\mu\fs_{h-1}),\Ga}$$
can be canonically identified with the trivial $\PP_{[\xi_0,\xi_1]}$-bundle:
$$\rho_{({\wp}_{(k\tau)}\fr_\mu\fs_h)}^{-1}(\tZ^\dagger_{({\wp}_{(k\tau)}\fr_\mu\fs_{h-1}),\Ga}) 
= \tZ^\dagger_{({\wp}_{(k\tau)}\fr_\mu\fs_{h-1}),\Ga}
\times \PP_{[\xi_0,\xi_1]}.$$
In this case, we define 
$$\tZ_{({\wp}_{(k\tau)}\fr_\mu\fs_h),\Ga}= \rho_{({\wp}_{(k\tau)}\fr_\mu\fs_h)}^{-1}(\tZ_{({\wp}_{(k\tau)}\fr_\mu\fs_{h-1}),\Ga})
 \cap ((\xi_0,\xi_1)= (1,1)),$$
 $$\tZ^\dagger_{({\wp}_{(k\tau)}\fr_\mu\fs_h),\Ga}:= 
 \rho_{({\wp}_{(k\tau)}\fr_\mu\fs_h)}^{-1}(\tZ^\dagger_{({\wp}_{(k\tau)}\fr_\mu\fs_{h-1}),\Ga})
 \cap ((\xi_0,\xi_1)= (1,1)),$$
 both scheme-theoretically.
 
 The induced morphism $\tZ^\dagger_{({\wp}_{(k\tau)}\fr_\mu\fs_h),\Ga} \lra \tZ^\dagger_{({\wp}_{(k\tau)}\fr_\mu\fs_{h-1}),\Ga}$
 is an isomorphism. 
 Again, one sees that $\tZ^\dagger_{({\wp}_{(k\tau)}\fr_\mu\fs_h),\Ga}$ is 
 an irreducible component of $\tZ_{({\wp}_{(k\tau)}\fr_\mu\fs_h),\Ga}$. Therefore,
 $$\tZ^\dagger_{({\wp}_{(k\tau)}\fr_\mu\fs_h),\Ga} \lra
 \tZ^\dagger_{({\wp}_{(k\tau)}\fr_\mu\fs_{h-1}),\Ga} \lra  Z_\Ga$$ is birational.
 This proves Lemma \ref{wp-transform-ktauh} (2) in $({\wp}_{(k\tau)}\fr_\mu\fs_h)$.

Further, under the condition $(\star b)$, we set 
 \begin{eqnarray}
 \tGa^\zero_\fV= \overline{\Ga}^\zero_\fV,  \;\;\;\;
\tGa^\one_\fV=\{y_1\} \cup \overline{\Ga}^\one_{\fV}. \nonumber \end{eqnarray}

Then, by 
the paragraph of \eqref{wp-pre-image-defined-by},
we have that $\tZ_{({\wp}_{(k\tau)}\fr_\mu\fs_h),\Ga}\cap \fV$,
 as a closed subscheme of $\fV$, is defined by the equations as stated in the Lemma. 
 This proves Lemma \ref{wp-transform-ktauh} (1) in $({\wp}_{(k\tau)}\fr_\mu\fs_h)$.

 It remains to prove Lemma \ref{wp-transform-ktauh} (3) in $({\wp}_{(k\tau)}\fr_\mu\fs_h)$.

Fix any $y \in \var_\fV$, it suffices to show that
if $\tZ^\dagger_{({\wp}_{(k\tau)}\fs_{h}),\Ga}\cap \fV
\subset (y=0)$, then $\tZ_{({\wp}_{(k\tau)}\fs_{h}),\Ga}\cap \fV
\subset (y=0)$.  By construction, $y \ne y_1$.
Then, the corresponding proof of Lemma \ref{wp-transform-ktauh} (3) in $({\wp}_{(k\tau)}\fr_\mu\fs_h)$
under the condition $(\star a)$ goes through here without change. As earlier, 
the last statement  of Lemma \ref{wp-transform-ktauh}  (3)  follows from the above. 
This proves Lemma \ref{wp-transform-ktauh} (3) in $({\wp}_{(k\tau)}\fr_\mu\fs_h)$
 under the condition $(\star b)$.

Putting all together, this completes the proof of Lemma \ref{wp-transform-ktauh}.    
\end{proof}

 \begin{defn}
 We call $\tZ_{({\wp}_{(k\tau)}\fr_\mu\fs_{h}),\Ga}$ the $\wp$-transform of $Z_\Ga$ 
 in $\tsV_{({\wp}_{(k\tau)}\fr_\mu\fs_{h})}$ for any $(k\tau)\mu h \in \Index_\Phi$.
 \end{defn}

  We need the final case of Lemma 
\ref{wp-transform-ktauh}. We set
$$\tZ_{\wp, \Ga}:=\tZ_{({\wp}_{(k\tau)}\fr_\mu\fs_h),\Ga}, \;\; \tZ^\dagger_{\wp, \Ga}:=\tZ^\dagger_{({\wp}_{(k\tau)}\fr_\mu\fs_h),\Ga}$$
where $k=\up, \tau=\ft_{F_\up},\mu=\rho_{\up\ft_{F_\up}},
h=\si_{(\up\ft_{F_\up})\rho_{\up\ft_{F_\up}} }.$ That is,
$$\tZ_{\wp,\Ga}= \tZ_{({\wp}_{(\up\ft_{F_\up})}\fr_{\rho_{\up\ft_{F_\up}}}
\fs_{\si_{(\up\ft_{F_\up})\rho_{\up\ft_{F_\up}} }}),\Ga},\;\;
\tZ^\dagger_{\wp,\Ga}= \tZ^\dagger_{({\wp}_{(\up\ft_{F_\up})}\fr_{\rho_{\up\ft_{F_\up}}}
\fs_{\si_{(\up\ft_{F_\up})\rho_{\up\ft_{F_\up}} }}),\Ga}.$$

\subsection{$\eth$-transforms of $\Ga$-schemes in  $ \tsV_{(\eth_{(k\tau)}\fr_\mu\fs_h)}$}
\label{subsection:wp-transform-ktauh}  $\ $

We now construct the $\eth$-transform of $Z_\Ga$ in $\tsV_{(\eth_{(k\tau)}\fr_\mu\fs_h)}
\subset \tsR_{(\eth_{(k\tau)}\fr_\mu\fs_h)}$. 
(This subsection may be combined with the previous subsection
to save some space, using the notation of \S \ref{combine-all}. But, we choose
to exhibit it separately for clarity.)


\begin{lemma}\label{vr-transform-ktauh} 
 Fix any  subset $\Ga$ of $\rU_\um$.  Assume that $Z_\Ga$ is integral. 

Consider $(k\tau) \mu h \in  \Index_\Psi$.

Then, we have the following:
\begin{itemize}
\item  there exists a closed subscheme $\tZ_{(\eth_{(k\tau)}\fr_\mu\fs_h),\Ga}$ of
$\tsV_{(\eth_{(k\tau)}\fr_\mu\fs_h)}$ with an induced morphism  
$\tZ_{(\eth_{(k\tau)}\fr_\mu\fs_h),\Ga} 
\to Z_\Ga$;
\item   $\tZ_{(\eth_{(k\tau)}\fr_\mu\fs_h),\Ga}$ comes equipped with an irreducible component  
$\tZ^\dagger_{(\eth_{(k\tau)}\fr_\mu\fs_h),\Ga}$ with the induced morphism 
$\tZ^\dagger_{(\eth_{(k\tau)}\fr_\mu\fs_h),\Ga}  
\to Z_\Ga$;
\item  for any standard chart $\fV$ of $\tsR_{(\eth_{(k\tau)}\fr_\mu\fs_h)}$ such that
$\tZ_{(\eth_{(k\tau)}\fr_\mu\fs_h),\Ga} \cap \fV \ne \emptyset$, there are two subsets, possibly empty,
$$ \tGa^\zero_{\fV} \; \subset \;  \var_\fV, \;\;\;
\tGa^\one_{\fV} \; \subset \;  \var_\fV.$$ 
\end{itemize}

Further,  consider any given chart $\fV$ of $\tsR_{({\eth}_{(k\tau)}\fr_\mu\fs_h)}$ with
$\tZ_{({\eth}_{(k\tau)}\fr_\mu\fs_h),\Ga} \cap \fV \ne \emptyset$. 
Then,  the following hold:
\begin{enumerate}
\item the scheme $\tZ_{(\eth_{(k\tau)}\fr_\mu\fs_h),\Ga} \cap \fV$,
 as a closed subscheme of the chart $\fV$,
is defined by the following relations
\begin{eqnarray} 
\;\;\;\;\; y , \; \; \; y \in  \tGa^\zero_\fV , \label{Ga-rel-wp-ktauh-eth}\\
\;\;\;  y -1, \; \; \; y \in  \tGa^\one_\fV,  \nonumber \\
\cB_\fV^\mn, \; \cB^\res_\fV, \;  \cB^\q_\fV, \; L_{\fV, \sF^\star_{\um,\Ga}};  \nonumber
\end{eqnarray}
 further, we take $\tGa^\zero_\fV \subset \var_\fV$
 to be the maximal subset (under inclusion)
among all those subsets that satisfy the above.
\item the induced morphism $\tZ^\dagger_{(\eth_{(k\tau)}\fr_\mu\fs_h),\Ga} 
\to Z_\Ga$ is birational;
\item for any variable $y \in \var_\fV$, $\tZ^\dagger_{(\eth_{(k\tau)}\fr_\mu\fs_h),\Ga} \cap \fV \subset (y=0)$ if and only 
if  $\tZ_{(\eth_{(k\tau)}\fr_\mu\fs_h),\Ga} \cap \fV \subset (y=0)$. Consequently, 
 $\tZ^\dagger_{(\eth_{(k\tau)}\fr_\mu\fs_h),\Ga}\cap \fV \subset\tZ_{\psi_{(k\tau)\mu(h+1)}}\cap \fV$ if and only 
if  $\tZ_{(\eth_{(k\tau)}\fr_\mu\fs_h),\Ga} \cap \fV \subset \tZ_{\psi_{(k\tau)\mu(h+1)}}\cap \fV$ 
where $\tZ_{\psi_{(k\tau)\mu(h+1)}}$ is the proper transform in  
$\tsR_{(\eth_{(k\tau)}\fr_\mu\fs_h)}$ of 
 the $\eth$-center $Z_{\psi_{(k\tau)\mu(h+1)}}$. 
\end{enumerate}
\end{lemma}  
\begin{proof} We prove by induction on 
$ (k\tau)\mu h \in \{(11)10\} \sqcup \Index_\Psi$ (cf. \eqref{indexing-Psi}).
(The proof is  analogous to that of Lemma \ref{wp-transform-ktauh}.
We provide details for completeness.)


The initial case is $(11)10$. In this case, we have
 $$\tsR_{(\eth_{11}\fr_1\fs_0)}:=\tsR_{\wp}, \;\;
 \tsV_{(\eth_{11}\fr_1\fs_0)}:=\tsV_{\wp} , \;\;
  \tZ_{(\eth_{11}\fr_1\fs_0),\Ga}:=\tZ_{\wp,\Ga} , \;\;
  \tZ^\dagger_{(\eth_{11}\fr_1\fs_0),\Ga}:=\tZ^\dagger_{\wp,\Ga}.$$
Then, in this case,  Lemma \ref{vr-transform-ktauh}
 is Lemma \ref{wp-transform-ktauh} for $\tsR_\wp$.


We now suppose that Lemma \ref{vr-transform-ktauh} holds for $(k\tau)\mu(h-1)$
for some $(k\tau)\mu h \in \Index_\Psi$.

We then consider the case of $(k\tau)\mu h \in \Index_\Psi$.

We let 
\begin{equation}\label{rho-ethktaumuh}
\rho_{(\eth_{(k\tau)}\fr_\mu\fs_h)}: \tsV_{(\eth_{(k\tau)}\fr_\mu\fs_h)} \lra
\tsV_{(\eth_{(k\tau)}\fr_\mu\fs_{h-1})} \end{equation}
be the morphism induced from
$\pi_{(\eth_{(k\tau)}\fr_\mu\fs_h)}: \tsR_{(\eth_{(k\tau)}\fr_\mu\fs_h)} \to
\tsR_{(\eth_{(k\tau)}\fr_\mu\fs_{h-1})}$.

Suppose that $\tZ_{(\eth_{(k\tau)}\fr_\mu\fs_{h-1}),\Ga}$,  or equivalently $\tZ^\dagger_{(\eth_{(k\tau)}\fr_\mu\fs_{h-1}),\Ga}$, by Lemma \ref{vr-transform-ktauh} (3) in $(\eth_{(k\tau)}\fr_\mu\fs_{(h-1)})$,
 is not contained in  $Z'_{\psi_{(k\tau)h}}$ where $Z'_{\psi_{(k\tau)\mu h}}$ is the proper transform in 
 $\tsR_{(\eth_{(k\tau)}\fr_\mu\fs_{h-1})}$
 of the $\eth$-center $Z_{\psi_{(k\tau)\mu h}}$ (of $\tsR_{(\eth_{(k\tau)}\fr_{\mu-1})}$).

We then let $\tZ_{(\eth_{(k\tau)}\fr_\mu\fs_h),\Ga}$ 
(resp. $\tZ^\dagger_{(\eth_{(k\tau)}\fr_\mu\fs_h),\Ga}$) be the
proper transform of $\tZ_{(\eth_{(k\tau)}\fr_\mu\fs_{h-1}),\Ga}$ (resp. $\tZ^\dagger_{(\eth_{(k\tau)}\fr_\mu\fs_{h-1}),\Ga}$)
in $\tsV_{(\eth_{(k\tau)}\fr_\mu\fs_h)}$.
As $\tZ^\dagger_{(\eth_{(k\tau)}\fr_\mu\fs_h),\Ga}$ is closed in $\tZ_{(\eth_{(k\tau)}\fr_\mu\fs_h),\Ga}$
and contains a Zariski open subset of $\tZ_{(\eth_{(k\tau)}\fr_\mu\fs_h),\Ga}$, it is an irreducible
component of $\tZ_{(\eth_{(k\tau)}\fr_\mu\fs_h),\Ga}$.

Further, consider any standard chart $\fV$ of $\tsR_{(\eth_{(k\tau)}\fr_\mu\fs_h)}$,
lying over a unique standard chart $\fV'$ of $\tsR_{(\eth_{(k\tau)}\fr_\mu\fs_{h-1})}$,
such that $\tZ_{(\eth_{(k\tau)}\fr_\mu\fs_h),\Ga}\cap \fV \ne \emptyset$.
We set 
$$\tGa^\zero_\fV=\{y_\fV \mid y_\fV  \hbox{ is the proper transform of some $y_{\fV'} \in \tGa^\zero_{\fV'}$}\};$$
$$\tGa^\one_\fV=\{y_\fV \mid y_\fV  \hbox{ is the proper transform of some $y_{\fV'} \in \tGa^\one_{\fV'}$}\}.$$

We now prove Lemma  \ref{vr-transform-ktauh} (1), (2) and (3) in $(\eth_{(k\tau)}\fr_\mu\fs_h)$.

Lemma  \ref{vr-transform-ktauh} (1) in $(\eth_{(k\tau)}\fr_\mu\fs_h)$ 
follows  from Lemma \ref{vr-transform-ktauh} (1) in $(\eth_{(k\tau)}\fr_\mu\fs_{h-1})$
because $\tZ_{(\eth_{(k\tau)}\fr_\mu\fs_h),\Ga}$ 
 is the proper transform of $\tZ_{(\eth_{(k\tau)}\fr_\mu\fs_{h-1}),\Ga}$. 

By construction, we have that the composition
$\tZ^\dagger_{(\eth_{(k\tau)}\fr_\mu\fs_h),\Ga} \to \tZ^\dagger_{(\eth_{(k\tau)}\fr_\mu\fs_{h-1}),\Ga} \to Z_\Ga$
is birational. 
This proves Lemma  \ref{vr-transform-ktauh} (2) in $(\eth_{(k\tau)}\fr_\mu\fs_h)$.


It remains to prove Lemma \ref{vr-transform-ktauh} (3) in $(\eth_{(k\tau)}\fr_\mu\fs_h)$.
But, it follows from the identical lines of the corresponding proof
of  Lemma \ref{wp-transform-ktauh} (3). We avoid the repetition.

\smallskip

We now suppose that $\tZ_{(\eth_{(k\tau)}\fr_\mu\fs_{h-1}),\Ga}$, or equivalently $\tZ^\dagger_{(\eth_{(k\tau)}\fr_\mu\fs_{h-1}),\Ga}$,
 by  Lemma \ref{vr-transform-ktauh} (3) in $(\eth_{(k\tau)}\fr_\mu\fs_{(h-1)})$,
 is contained in the proper transform  $Z'_{\psi_{(k\tau)\mu h}}$ of 
 $Z_{\psi_{(k\tau)\mu h}}$.


Consider  any standard chart $\fV$ of $\tsR_{(\eth_{(k\tau)}\fr_\mu\fs_h)}$,
lying over a unique standard chart $\fV'$ of $\tsR_{(\eth_{(k\tau)}\fr_\mu\fs_{h-1})}$,
such that $\tZ_{(\eth_{(k\tau)}\fr_\mu\fs_{h-1}),\Ga}\cap \fV' \ne \emptyset$.


 We let $\psi'_{(k\tau)\mu h}$
 be the proper transform in  the chart $\fV'$ of the $\eth$-set $\psi_{(k\tau)\mu h}$.
Then, 
we can express
 $$\psi'_{(k\tau)\mu h}=\{y'_0, y'_1\} \subset \var_{\fV'}.$$ 
 We let $\zeta_\fV \in \var_\fV$  be such that 
$E_{(\eth_{(k\tau)}\fr_\mu\fs_{h})} \cap \fV=(\zeta_\fV =0)$ where $E_{(\eth_{(k\tau)}\fr_\mu\fs_{h})}$
 is the exceptional divisor
of the blowup $\tsR_{(\eth_{(k\tau)}\fr_\mu\fs_{h})} \to \tsR_{(\eth_{(k\tau)}\fr_\mu\fs_{h-1})}$. 
Without loss of generality, we assume that $y'_0$ 
corresponds to the exceptional $\zeta_\fV$ on the chart $\fV$.
 We then let $y_1 \in \var_\fV$ be 
 the proper transform of $y'_1$ on the chart $\fV$.

Here, we observe that 
 $$\psi'_{(k\tau)\mu h}=\{y'_0,  y'_1\} \subset \tGa^\zero_{\fV'}$$ 
because $\tZ_{(\eth_{(k\tau)}\fr_\mu\fs_{h-1}),\Ga}$ is contained in the proper transform
$Z'_{\psi_{(k\tau)\mu h}}$.

We set, 
\begin{equation}\label{Gazero-eth}
\overline{\Ga}^\zero_\fV= \{\zeta_\fV, \; y_\fV \mid y_\fV  
\hbox{ is the proper transform of some $ y_{\fV'} \in \tGa^\zero_{\fV'} \- \psi'_{(k\tau)\mu h}$}\},
\end{equation}
\begin{equation}\label{Gaone-eth}
 \overline{\Ga}^\one_{\fV}=\{\ y_\fV \mid y_\fV  
\hbox{ is the proper transform of some $y_{\fV'} \in \tGa^\one_{\fV'}$} \}.
\end{equation}

Consider the scheme-theoretic pre-image  
$\rho_{(\eth_{(k\tau)}\fr_\mu\fs_h)}^{-1}(\tZ_{(\eth_{(k\tau)}\fr_\mu\fs_{h-1}),\Ga})$.

Note that scheme-theoretically, we have
$$\rho_{(\eth_{(k\tau)}\fr_\mu\fs_h)}^{-1}(\tZ_{(\eth_{(k\tau)}\fr_\mu\fs_{h-1}),\Ga}) \cap \fV=
\pi_{(\eth_{(k\tau)}\fr_\mu\fs_h)}^{-1}(\tZ_{(\eth_{(k\tau)}\fr_\mu\fs_{h-1}),\Ga})
 \cap \tsV_{(\eth_{(k\tau)}\fr_\mu\fs_h)} \cap \fV .$$
Applying Lemma \ref{vr-transform-ktauh} (1)  in $(\eth_{(k\tau)}\fr_\mu\fs_{h-1})$ to 
$\tZ_{(\eth_{(k\tau)}\fr_\mu\fs_{h-1}),\Ga}$
and $\rho_{(\eth_{(k\tau)}\fr_\mu\fs_h)}^{-1}(\tZ_{(\eth_{(k\tau)}\fr_\mu\fs_{h-1}),\Ga}) $, 
and applying Proposition \ref{equas-p-k} to 
$\tsV_{(\eth_{(k\tau)}\fr_\mu\fs_h)} \cap \fV $, 
we obtain that the pre-image
$\rho_{(\eth_{(k\tau)}\fr_\mu\fs_h)}^{-1}(\tZ_{(\eth_{(k\tau)}\fr_\mu\fs_{h-1}),\Ga}) \cap \fV$
as a closed subscheme of $\fV$,  is defined by 
\begin{equation}\label{pre-image-defined-by}
 y_\fV \in \overline{\Ga}^\zero_\fV; \;\;  y_\fV-1, \; y_\fV \in \overline{\Ga}^\one_{\fV};\;\;
\cB^\mn_\fV; \;\; 
\cB^\q_{\fV}; \;\; L_{\fV, \sF^\star_{\um,\Ga}}.
\end{equation}
(Note that we have $\zeta_\fV \in \overline{\Ga}^\zero_\fV$.)

Thus, by setting 
$$\hbox{ $y_\fV=0$ for all $y_\fV \in \tGa^\zero_\fV$ and
 $y_\fV=1$ for all $y_\fV \in \overline{\Ga}^\one_{\fV}$} $$ 
 in $\cB^\mn_\fV, 
 \cB^\q_{\fV}, L_{\fV, \sF^\star_{\um,\Ga}}$ of the above,
 we obtain 
 \begin{equation}\label{lin-xi-pre} \tilde{\cB}^\mn_\fV, 
  \tilde\cB^\q_{\fV}, \tilde{L}_{\fV, \sF^\star_{\um,\Ga}}.
 \end{equation}
Again,  for any $\bF \in \sF^\star_{\fV,\Ga}$,
 if a term of $L_{\fV, F}$ contains $y_1 \in \var_\fV$, 
 then it contains $\zeta_\fV y_1$, hence $\tilde L_{\fV, F}$ does not contain $y_1$.
  We keep those equations of \eqref{lin-xi-pre} such that they contain the variable
 $y_1$ and obtain
 \begin{equation}\label{vr-lin-xi}
 \hat{\cB}^\mn_\fV, 
 \hat\cB^\q_{\fV}, 
 \end{equation}
 viewed as a  system of equations in $y_1$.
 Then, Proposition \ref{equas-p-k} (1) and (2),
\eqref{vr-lin-xi} is  a  {\it linear} system of equations in $y_1$.
(We point out that $y_1$ here can correspond to either $y_0$ or $y_1$ as in Proposition \ref{equas-p-k}.)
Furthermore, one sees that 
 $$(\rho_{(\eth_{(k\tau)}\fr_\mu\fs_h)}^{-1}(\tZ_{(\eth_{(k\tau)}\fr_\mu\fs_{h-1}),\Ga})\cap \fV)/
 (\tZ_{(\eth_{(k\tau)}\fr_\mu\fs_{h-1}),\Ga}\cap \fV')$$
 is defined by the linear system \eqref{vr-lin-xi}.

There are the following two cases for \eqref{vr-lin-xi}:

$(\star a)$  the ranks of the linear system  $\eqref{vr-lin-xi}$ equal one 
at general points of $\tZ^\dagger_{(\eth_{(k\tau)}\fr_\mu\fs_{h-1}),\Ga}$.

$(\star b)$  the ranks of the linear system  $\eqref{vr-lin-xi}$ equal zero
at general points of  $\tZ^\dagger_{(\eth_{(k\tau)}\fr_\mu\fs_{h-1}),\Ga}$, hence at all points
of $\tZ^\dagger_{(\eth_{(k\tau)}\fr_\mu\fs_{h-1}),\Ga}$.

\smallskip\noindent
{\it  Proof of Lemma \ref{vr-transform-ktauh} in $(\eth_{(k\tau)}\fr_\mu\fs_h)$ under the condition $(\star a)$.}
\smallskip

By the condition $(\star a)$,
 there exists a Zariski open subset $\tZ^{\dagger\circ}_{(\eth_{(k\tau)}\fr_\mu\fs_{h-1}),\Ga}$ 
of $\tZ^\dagger_{(\eth_{(k\tau)}\fr_\mu\fs_{h-1}),\Ga}$ such that the rank of the linear system
  $\eqref{vr-lin-xi}$ equals one at any point of  $\tZ^{\dagger\circ}_{(\eth_{(k\tau)}\fr_\mu\fs_{h-1}),\Ga}$. 
By solving $y_1$ from
 the linear system \eqref{vr-lin-xi} over $\tZ^{\dagger\circ}_{(\eth_{(k\tau)}\fr_\mu\fs_{h-1}),\Ga}$, we obtain 
that the induced morphism
$$
\rho_{(\eth_{(k\tau)}\fr_\mu\fs_h)}^{-1}(\tZ^{\dagger\circ}_{(\eth_{(k\tau)}\fr_\mu\fs_{h-1}),\Ga})  \lra \tZ^\circ_{(\eth_{(k\tau)}\fr_\mu\fs_{h-1}),\Ga}$$ is an isomorphism. 

Suppose  $y_1$ is identically zero along 
$ \rho_{(\eth_{(k\tau)}\fr_\mu\fs_h)}^{-1}(\tZ^{\dagger\circ}_{(\eth_{(k\tau)}\fr_\mu\fs_{h-1}),\Ga})$.
We then set,  
\begin{equation}\label{tGa-eth} 
\tGa^\zero_\fV=\{y_1 \} \cup \overline{\Ga}^\zero_\fV
\end{equation}
where $\overline{\Ga}^\zero_\fV$ is as in  \eqref{Gazero-eth}.
In this case, we let
\begin{equation}\label{ZktahGa-with-xi-2}
\tZ_{(\eth_{(k\tau)}\fs_{h}),\Ga}=\rho_{(\eth_{(k\tau)}\fr_\mu\fs_h)}^{-1}(\tZ_{(\eth_{(k\tau)}\fr_\mu\fs_{h-1}),\Ga})\cap D_{y_1}
\end{equation}
scheme-theoretically, where $D_{y_1}$ is the closure of $(y_1=0)$ in $\tsR_{(\eth_{(k\tau)}\fr_\mu\fs_h)}$.
We remark here that $D_{y_1}$ does not depend on the choice of the chart $\fV$.

Suppose  $y_1$ is not identically zero along 
$\rho_{(\eth_{(k\tau)}\fr_\mu\fs_h)}^{-1}(\tZ^{\dagger\circ}_{(\eth_{(k\tau)}\fr_\mu\fs_{h-1}),\Ga})$.
We then set,  
\begin{equation}\label{Gazero-ktah-contained-in-a-3} 
\tGa^\zero_\fV= \overline{\Ga}^\zero_\fV
\end{equation}
where $\overline{\Ga}^\zero_\fV$ is as in \eqref{Gazero-eth}.
In this case, we let
\begin{equation}\label{ZktahGa-without-xi-2}
\tZ_{(\eth_{(k\tau)}\fs_{h}),\Ga}=\rho_{(\eth_{(k\tau)}\fr_\mu\fs_h)}^{-1}(\tZ_{(\eth_{(k\tau)}\fr_\mu\fs_{h-1}),\Ga}).
\end{equation}
We always set (under the condition $(\star a)$)
\begin{equation}\label{Gaone-ktah-contained-in-a-2} \tGa^\one_{\fV}= \overline{\Ga}^\one_{\fV}
\end{equation}
where $\overline{\Ga}^\one_{\fV}$ is as in  \eqref{Gaone-eth}.

In each case,  we have 
$$\rho_{(\eth_{(k\tau)}\fr_\mu\fs_h)}^{-1}(\tZ^{\dagger\circ}_{(\eth_{(k\tau)}\fr_\mu\fs_{h-1}),\Ga})  
\subset \tZ_{(\eth_{(k\tau)}\fr_\mu\fs_h),\Ga},$$ and we let $\tZ^\dagger_{(\eth_{(k\tau)}\fr_\mu\fs_h),\Ga}$ be the closure of 
$\rho_{(\eth_{(k\tau)}\fr_\mu\fs_h)}^{-1}(\tZ^{\dagger\circ}_{(\eth_{(k\tau)}\fr_\mu\fs_{h-1}),\Ga}) $ in $\tZ_{(\eth_{(k\tau)}\fr_\mu\fs_h),\Ga}$.
It is an irreducible component of $\tZ_{(\eth_{(k\tau)}\fr_\mu\fs_h),\Ga}$ because
$\tZ^\dagger_{(\eth_{(k\tau)}\fr_\mu\fs_h),\Ga}$ is closed in $\tZ_{(\eth_{(k\tau)}\fr_\mu\fs_h),\Ga}$ and contains
the Zariski open subset $\rho_{(\eth_{(k\tau)}\fr_\mu\fs_h)}^{-1}(\tZ^{\dagger\circ}_{(\eth_{(k\tau)}\fr_\mu\fs_{h-1}),\Ga}) $
of $\tZ_{(\eth_{(k\tau)}\fr_\mu\fs_h),\Ga}$.
Then, we have that the composition
 $$\tZ^\dagger_{(\eth_{(k\tau)}\fr_\mu\fs_h),\Ga} \lra \tZ^\dagger_{(\eth_{(k\tau)}\fr_\mu\fs_{h-1}),\Ga}
 \lra Z_\Ga$$ is birational. 
 This proves Lemma \ref{vr-transform-ktauh} (2) in $(\eth_{(k\tau)}\fr_\mu\fs_h)$.


In each case of \eqref{tGa-eth} and \eqref{Gazero-ktah-contained-in-a-3}, 
by the paragraph of \eqref{pre-image-defined-by},
we have that $\tZ_{(\eth_{(k\tau)}\fr_\mu\fs_h),\Ga}\cap \fV$
 as a closed subscheme of $\fV$ is defined by the equations as stated in the Lemma. 
 This proves Lemma \ref{vr-transform-ktauh} (1) in $(\eth_{(k\tau)}\fr_\mu\fs_h)$.

It remains to prove Lemma \ref{vr-transform-ktauh} (3) in $(\eth_{(k\tau)}\fr_\mu\fs_h)$.
But, it follows from the identical lines of the corresponding proof
of  Lemma \ref{wp-transform-ktauh} (3). We avoid the repetition.

\smallskip\noindent
{\it  Proof of Lemma \ref{vr-transform-ktauh} in $(\eth_{(k\tau)}\fr_\mu\fs_h)$ under the condition $(\star b)$.}
\smallskip

Under the condition $(\star b)$, we have that $$ 
\rho_{(\eth_{(k\tau)}\fr_\mu\fs_h)}^{-1}(\tZ^\dagger_{(\eth_{(k\tau)}\fr_\mu\fs_{h-1}),\Ga})
 \lra \tZ^\dagger_{(\eth_{(k\tau)}\fr_\mu\fs_{h-1}),\Ga}$$
can be canonically identified with the trivial $\PP_{[\xi_0,\xi_1]}$-bundle:
$$\rho_{(\eth_{(k\tau)}\fr_\mu\fs_h)}^{-1}(\tZ^\dagger_{(\eth_{(k\tau)}\fr_\mu\fs_{h-1}),\Ga}) 
= \tZ^\dagger_{(\eth_{(k\tau)}\fr_\mu\fs_{h-1}),\Ga}
\times \PP_{[\xi_0,\xi_1]}.$$
In this case, we define 
$$\tZ_{(\eth_{(k\tau)}\fr_\mu\fs_h),\Ga}= \rho_{(\eth_{(k\tau)}\fr_\mu\fs_h)}^{-1}(\tZ_{(\eth_{(k\tau)}\fr_\mu\fs_{h-1}),\Ga})
 \cap ((\xi_0,\xi_1)= (1,1)),$$
 $$\tZ^\dagger_{(\eth_{(k\tau)}\fr_\mu\fs_h),\Ga}:= 
 \rho_{(\eth_{(k\tau)}\fr_\mu\fs_h)}^{-1}(\tZ^\dagger_{(\eth_{(k\tau)}\fr_\mu\fs_{h-1}),\Ga})
 \cap ((\xi_0,\xi_1)= (1,1)),$$
 both scheme-theoretically.
 
 The induced morphism $\tZ^\dagger_{(\eth_{(k\tau)}\fr_\mu\fs_h),\Ga} \to \tZ^\dagger_{(\eth_{(k\tau)}\fr_\mu\fs_{h-1}),\Ga}$
 is an isomorphism.  Therefore,
 $\tZ^\dagger_{(\eth_{(k\tau)}\fr_\mu\fs_h),\Ga} \to
 \tZ^\dagger_{(\eth_{(k\tau)}\fr_\mu\fs_{h-1}),\Ga} \to  Z_\Ga$ is birational.
 Again, one sees that $\tZ^\dagger_{(\eth_{(k\tau)}\fr_\mu\fs_h),\Ga}$ is 
 an irreducible component of $\tZ_{(\eth_{(k\tau)}\fr_\mu\fs_h),\Ga}$.
 This proves Lemma \ref{vr-transform-ktauh} (2) in $(\eth_{(k\tau)}\fr_\mu\fs_h)$.

Further, under the condition $(\star b)$, we set 
 \begin{eqnarray}
 \tGa^\zero_\fV= \overline{\Ga}^\zero_\fV,  \;\;\;\;
\tGa^\one_\fV=\{y_1\}\cup \overline{\Ga}^\one_{\fV}. \nonumber \end{eqnarray}

Then, by the paragraph of \eqref{pre-image-defined-by},
we have that $\tZ_{(\eth_{(k\tau)}\fr_\mu\fs_h),\Ga}\cap \fV$,
 as a closed subscheme of $\fV$, is defined by the equations as stated in the Lemma. 
 This proves Lemma \ref{vr-transform-ktauh} (1) in $(\eth_{(k\tau)}\fr_\mu\fs_h)$.

 It remains to prove Lemma \ref{vr-transform-ktauh} (3) in $(\eth_{(k\tau)}\fr_\mu\fs_h)$.
But, again, it follows from the identical lines of the corresponding proof
of  Lemma \ref{wp-transform-ktauh} (3). We avoid the repetition.


This completes the proof of Lemma \ref{vr-transform-ktauh}.
\end{proof}

 \begin{defn}
 We call $\tZ_{(\eth_{(k\tau)}\fr_\mu\fs_{h}),\Ga}$ the $\eth$-transform of $Z_\Ga$ 
 in $\tsV_{(\eth_{(k\tau)}\fr_\mu\fs_{h})}$ for any $(k\tau)\mu h \in \Index_\Psi$.
 \end{defn}

   We need the final case of Lemma 
\ref{vr-transform-ktauh}. We set
$$\tZ_{\eth, \Ga}:=\tZ_{(\eth_{(k\tau)}\fr_\mu\fs_h),\Ga}, \;\; \tZ^\dagger_{\vr, \Ga}:=\tZ^\dagger_{(\eth_{(k\tau)}\fr_\mu\fs_h),\Ga}$$
where $k=\up, \tau=\ft_{F_\up},\mu=\vk_{\up\ft_{F_\up}},
h=\vs_{(\up\ft_{F_\up})\vk_{\up\ft_{F_\up}} }.$ That is,
$$\tZ_{\vr,\Ga}= \tZ_{(\eth_{(\up\ft_{F_\up})}\fr_{\vk_{\up\ft_{F_\up}}}
\fs_{\vs_{(\up\ft_{F_\up})\vk_{\up\ft_{F_\up}} }}),\Ga},\;\;
\tZ^\dagger_{\vr,\Ga}= \tZ^\dagger_{(\eth_{(\up\ft_{F_\up})}\fr_{\vk_{\up\ft_{F_\up}}}
\fs_{\vs_{(\up\ft_{F_\up})\vk_{\up\ft_{F_\up}} }}),\Ga}.$$

\begin{cor} \label{eth-transform-up} 
Fix any standard chart $\fV$ of $\tsR_\eth$ such that
$\tZ_{\eth,\Ga} \cap \fV \ne \emptyset.$
Then, $\tZ_{\eth,\Ga} \cap \fV$,  as a closed subscheme of $\fV$,
is defined by the following relations
\begin{eqnarray} 
\;\;\;\;\; y , \; \; \forall \;\;  y \in   \tGa^\zero_\fV; 
\;\;\;  y -1, \; \; \forall \;\;  y \in  \tGa^\one_\fV; \nonumber \\
\cB_\fV^\mn, \; 
\cB^\q_\fV, \; L_{\fV, \sF^\star_{\um,\Ga}}.
\end{eqnarray}
Furthermore, the induced morphism $\tZ^\dagger_{\eth,\Ga} \to Z_\Ga$ is birational.
\end{cor}

\section{Main Statements on the Final Scheme $\tsV_\eth$}\label{main-statement}


{ \it  Let $p$ be an arbitrarily fixed prime number.
 Let $\mathbb F$ be either $\QQ$ or a finite field with $p$ elements.
In this entire section, every scheme is defined over $\ZZ$, consequently,
 is defined over $\mathbb F$, and is considered as a scheme over 
 the perfect field $\mathbb F$.}

Let $\Ga$ be any subset $\var_{\rU_\um}$.
Assume that $Z_\Ga$ is integral.
Let $\tZ_{\eth,\Ga}$ be  the $\eth$-transform of $Z_\Ga$ in $ \tsV_{\eth}$.
We now investigate local properties of  the  $\eth$-transform $\tZ_{\eth, \Ga}$, using standard charts.
(As just mentioned, $Z_\Ga$ and $\tZ_{\eth,\Ga}$ are $\FF$-schemes.)
We remind the reader that  $\tZ_{\eth,\emptyset}=\tsV_\eth$ when $\Ga=\emptyset$.

Fix any standard chart $\fV$ of  $\tsR_{\eth}$.
By Corollary \ref{eth-transform-up},
the scheme $\tZ_{\eth, \Ga} \cap \fV$, if nonempty, 
 as a closed subscheme of
the chart $\fV$ of $\tsR_\eth$,  is defined by 
\begin{eqnarray} 
y, \; y \in  \tGa^\zero_\fV; \;\;\;\; y-1, \; y \in \tGa^\one_\fV; \;\;\;\; 
 \cB^q_\fV; \label{Bres-final} \\
\;\;\;\; \;\;\;\;\;\;\; B_{\fV, (k\tau)}, \;\; \forall \;\; s=(k\tau) \in S_{F_k} \- s_{F_k}, \; F_k \in \sfm, 
 \label{barB-ktau-final} \\ 
L_{\fV,F_k}: \;\; 
\forall \;\;  F_k \in   \sF^\star_{\fV,\Ga}.
 \label{barL-ktau-final}
\end{eqnarray}
 We  remind the reader that any main binomial is indexed by some $(k\tau) \in \Index_{\cB^\mn}$;
it corresponds to the $k$-th $\pl$ equation $\bF_k \in \sfm$
and a non-leading term of $\bF_k$ indexed by  some $s \in S_{F_k} \- s_{F_k}$. 
 Via this correspondence, we sometime write $s=(k\tau)$.
Consequently, we may also write $B_s =B_{(k\tau)}$.

Fix any $k \in [\up]$. We let $\fG_k$ be the set of all equations in 
\eqref{barB-ktau-final} and \eqref{barL-ktau-final},
 called the  block $k$ of the defining equations of $\tZ_{\eth, \Ga}$. 
We let $\fG=\bigsqcup_{k=1}^\up \fG_k$. 

{\it Throughout the remainder of this section, 
we let  $\bz$ be a fixed closed point of $\tsV_\eth\cap \fV$.}

Observe that all the binomial equations $B_{(k\tau)}$ terminate on the chart $\tsR_\eth$,
according to Corollary \ref{all-terminate}. In particular, 
both terms of  $B_{ \fV, (k\tau)}$ do not vanish at $\bz$.

The standard chart $\fV$ must be lying over a unique standard chart $ \fV_{[0]}$ of 
$\tsR_{\vt_{[0]}}=\sR_\sF$.
By Definition \ref{fv-k=0},  the chart $\fV_{[0]}$ is
   uniquely indexed by a set 
\begin{equation}\label{fixed-la-o}
\La_\sfm^o=\{(\uv_{s_{F,o}},\uv_{s_{F,o}}) \mid \bF \in \sfm \}.
\end{equation}
And, given $\La_\sfm^o$, we have  the set $\La_\sfm^\star=\La_\sfm \- \La_\sfm^o$.

We let  $\pi:=\pi_{\fV, \fV_{[0]}}:  \fV \lra  \fV_{[0]}$ be the (induced) projection.
Consider any $k \in [\up]$. 
Recall that we have the set $\cB^\mn_{F_k}=\{B_{(k\tau)} \mid  1 \le \tau \le \ft_{F_k} \}$.
In addition, the notion of termination of $B_{ \fV_{[0]}, (k\tau)}$ 
at the point $\pi_{\fV, \fV_{[0]}} (\bz)$ have been introduced
in Definition \ref{general-termination} 
for all $(k\tau) \in \Index_{\cB^\mn}$.

\begin{defn}\label{defn:termination} Consider the point $\bz$ of $\fV$. Consider $(k\tau) \in \Index_{\cB^\mn}$.
 We say $B_{ \fV, (k\tau)}$ is original at $\bz$ if
$B_{\fV_{[0]}, (k\tau)}$ terminates at the point $\pi(\bz)$; we let
$\cB^\ori_{F_k}$ be the set of all original $B_{ \fV, (k\tau)}$ at $\bz$.
We say  $B_{ \fV, (k\tau)}$ is $\hs$-intrinsic (or simply, intrinsic) at $\bz$ if
$B_{\fV_{[0]}, (k\tau)}$ is not original; we let
$\cB^{\inc}_{F_k}$ be the set of all  $\hs$-intrinsic $B_{ \fV, (k\tau)}$ at $\bz$. 
 \end{defn}
 
 Thus, we have $$\cB^\mn_{F_k}=\cB^\ori_{F_k} \sqcup \cB^\inc_{F_k}.$$
 We point out that this decomposition depends on the fixed point $\bz \in \tsV_\eth\cap \fV$.
 
 When both terms of $B_{ \fV_{[0]}, (k\tau)}$ do not vanish at $\pi (\bz)$,
the form of $B_{\fV_{[0]}, (k\tau)}$ remains unchanged throughout the $\hs$-blowups, except the changes
of the names of the variables, whence the name of original. 
On the other hand, 
if $B_{ \fV, (k\tau)}$ is $\hs$-intrinsic  (e.g., when $F_k$ is $\Ga$-irrelevant)
and $\pi (\bz) \in \sV$, 
then both terms of $B_{\fV_{[0]}, (k\tau)}$ vanishes at $\pi (\bz)$. Thus, some $\hs$-blowup
with $\hs \in \Game$
has to affect a neighborhood of $\pi_{\eth, \hs} (\bz)$ 
where  $\pi_{\eth, \hs}: \tsR_\eth\lra \tsR_{\hs}$ is the induced blowup morphism, 
and, the form of $B_{\fV_{[0]}, (k\tau)}$
 must have positively changed when reaching at the final destination on $\fV$,
around the point $\bz$,
whence the name of $\hs$-intrinsic. 


By Proposition \ref{equas-p-k=0}, on the chart $ \fV_{[0]}$, we have that 
 $\tsV_{\vt_{[0]}} \cap  \fV_{[0]}$, as a closed subscheme of $ \fV_{[0]}$, is defined by
\begin{eqnarray}  
\;\;\;\;\;\;\;\;  
\cB^\res_{ \fV_{[0]}}; \;\;\; \cB^\q_{ \fV_{[0]}}; \nonumber \\
x_{ \fV_{[0]}, (\uu_s, \uv_s)}x_{ \fV_{[0]}, \uu_F} - x_{ \fV_{[0]}, (\um,\uu_F)}   
x_{ \fV_{[0]}, \uu_s} x_{ \fV_{[0]}, \uv_s}, \;\; \forall \;\; s \in S_F \- s_F,  \label{eq-B-k=0''} \\ 
L_{ \fV_{[0]}, F}: \;\; \sum_{s \in S_F} \vsgn (s) x_{ \fV_{[0]},(\uu_s,\uv_s)},\;\;
\forall \;\; \bF \in \sfm.
\label{linear-pl-k=0''}
\end{eqnarray}
 Here, recall that $\uu_F:=\uu_{s_F}$ and $s_F \in  S_F$ is the index for the
leading term of $\bF$.
Also recall from the proof of Proposition \ref{meaning-of-var-p-k=0} that on the chart 
$ \fV_{[0]}$, $x_{ \fV_{[0]}, \uu}=x_{\uu}$, and 
$x_{ \fV_{[0]}, (\uu, \uv)}$ is the de-homogenization of $x_{(\uu, \uv)}$.

\begin{lemma}\label{2-cases} Consider the fixed point $\bz \in \tsV_\eth\cap \fV$.
Fix any  $\bF \in \sfm$.
Assume that $\cB^\ori_F \ne \emptyset$, equivalently, $\cB^\inc_F \ne  \cB^\mn_F$.
Then, $x_{\fV_{[0]}, \uu_F} (\pi(\bz))\ne 0$, hence, $x_{\fV, \uu_F}$ exists as a coordinate variable
in  $\var_\fV$
and  $x_{\fV, \uu_F} (\bz) \ne 0$.
Equivalently, if  $x_{\fV_{[0]}, \uu_F} (\pi(\bz)) = 0$, 
then $\cB^\ori_F = \emptyset$
and $\cB^\inc_F = \cB^\mn_F$.
\end{lemma}
\begin{proof} This follows from definition.
\end{proof}

Fix $\bF \in \sfm$. To continue, we let
$$S^\ori_F=\{s  \in S_F \- s_F \mid B_s \in \cB_F^\ori\}, \;\; 
S^\inc_F=\{s  \in S_F \- s_F \mid B_s \in \cB_F^\inc \}.$$

There are two extreme cases for the decomposition 
$\cB^\mn_{F}=\cB^\ori_F \sqcup \cB^\inc_F$: $\cB^\mn_{F}=\cB^\ori_F$
and $\cB^\mn_{F}= \cB^\inc_F$, where the latter must be the case 
when $x_{\fV_{[0]}, \uu_F} (\pi(\bz)) = 0$.

To analyze all possibilities, we first divide it into two classess:
$$s_{F,o} = s_F\; \;\hbox{and}\;\; s_{F,o} \ne s_F.$$
 (See \eqref{fixed-la-o} for the explanation of $s_{F,o}$.)


\medskip\noindent
{\it Suppose  $s_{F,o}=s_F$ and $\cB^\ori_F \ne \emptyset$.}
 Then, we can write $S^\ori_F=\{s_1, \cdots,  s_\ell\}$ for some $1 \le  \ell \le \ft_F$
 and $S_F \- s_F = S_F^\ori \sqcup S^\inc_F$.
In this case, $x_{\fV,(\um,\uu_F)}  =x_{\fV, (\uu_{s_{F,o}}, \uv_{s_{F,o}})}=1$. Hence, the binomial
equations of  $\cB^\ori_F$ on the chart $\fV$ are
\begin{equation}\label{barBF=}
B_{\fV, s_i}: \;\; x_{\fV, (\uu_{s_i}, \uv_{s_i})}x_{\uu_F}-  x_{\fV, \uu_{s_i}} x_{\fV, \uv_{s_i}}, \;\; 1 \le i \le \ell.
\end{equation}

\medskip\noindent
{\it Suppose $s_{F,o} \ne s_F$, and $\cB^\ori_F \ne \emptyset$.}  
Then, in this case,  as $x_{\fV, (\uu_{s_{F,o}}, \uv_{s_{F,o}})}=1$ and  $x_{\fV, \uu_F} (\bz) \ne 0$ 
(by Lemma \ref{2-cases}), we  must have
$s_{F,o} \in S^{\rm ori}_F$.
Hence, we can write $S^\ori_F=\{s_{F,o}, s_1, \cdots, s_\ell\}$ for some $1 \le \ell \le \ft_F$.
Therefore, the binomial
equations of  $\cB^\ori_F$ on the chart $\fV$ are
\begin{eqnarray}\label{barBFnot=} 
B_{\fV, s_o}: \;\;  x_{\fV,\uu_F}- x_{\fV,(\um,\uu_F)}  x_{\fV,\uu_{s_{F,o}}} x_{\fV,\uv_{s_{F,o}}}, \\
B_{\fV, s_i}: \;\;  x_{\fV,(\uu_{s_i}, \uv_{s_i})}x_{\fV,\uu_F}-
 x_{\fV,(\um,\uu_F)}  x_{\fV,\uu_{s_i}} x_{\fV,\uv_{s_i}} =0, \;\; 1 \le i \le \ell.
\end{eqnarray}
Indeed, we must have  $\ell \ge 1$ in the above.
To see this, consider $\pi(\bz) \in\fV_{[0]} \cap \sV_{[0]}$,
then, $L_{\fV_{[0]}, F}$ is satisfied by $\pi(\bz)$. Thus, one sees that  there must be
an index $s_1\ne s_{F,o}$ such that $x_{ \fV_{[0]}, (\uu_{s_1}, \uv_{s_1})}(\pi(\bz)) \ne 0$.
 Hence, $B_{\fV, s_1}$ is original 
because $x_{ \fV_{[0]}, \uu_F} (\pi(\bz))\ne 0$ as well.

Next, we consider $S^\inc_F$. We suppose $S^\inc_F \ne \emptyset$. 
We write $$S^\inc_F =\{t_1 < \cdots < t_q\}$$
for some $1\le q \le \ft_F$.

By Definition \ref{termi-v0},  we let $y_{\fV, t_i}$ be the  (unique) terminating central variable
for $B_{t_i} \in S^\inc_F$ for all $1 \le i \le q$. We write
\begin{equation}\label{list-y}
{\bf y}^\inc_{\fV,F}=\{y_{\fV, t_1}, \cdots, y_{\fV, t_q}\}.
\end{equation}
We then set
\begin{equation}\label{partial-y}
J^*(\cB^\inc_{\fV,F}):={ {\partial (\cB^\inc_{\fV,F})} \over
{\partial{(\bf y}^\inc_{\fV,F}) } }:= { {\partial (B_{\fV, t_1},\cdots,  B_{\fV, t_q})} \over
{\partial (y_{\fV,t_1}, \cdots, y_{\fV, t_q})} }.\end{equation}

 \begin{lemma}\label{pre-tri} Suppose $\bz \in \tsV_\eth\cap \fV$. We have the following.
 \begin{enumerate}
 \item The matrix $J^*(\cB^\inc_{\fV,F})$ is invertible  lower-triangular at $\bz$;
 \item Consider any $\bG \in \sfm$ such that $\bG < \bF$. Then, the variable $y_{\fV,t_i}$ does not
 appear in the main binomial equation $B_{\fV, s}$, either original or $\hs$-intrinsic at $\bz$,
   for all $s \in S_G \- s_G$ and all $1 \le i \le q$.
\item  Further, for any $\bH \in \sfm$, 
if $y_{\fV,t_i}$ appears in $L_{\fV, H}$, then it appears in a term that vanishes
 at $\bz$, for all $1 \le i \le q$.
 \end{enumerate}
 \end{lemma}
  \begin{proof}
  (1). Note that $B_{\fV, t_i}$,  $1 \le i \le q$, terminates at $\bz$. 
  As the term of $B_{\fV, t_i}$ that contains $y_{\fV,t_i}$ is linear in $y_{\fV,t_i}$
 according to Corollary \ref{e2},  we have
  ${ {\partial B_{\fV, t_i}} \over
{\partial y_{\fV,t_i}}}(\bz) \ne 0$. By Proposition \ref{equas-p-k} (3), when $j<i$,
at the point $\bz$, the $\hs$-intrinsic binomial $B_{\fV, t_{j}}$ must terminate 
earlier than the $\hs$-intrinsic binomial $B_{\fV, t_i}$  does.
 Hence, one sees that  the variable $y_{\fV,t_i}$, as the terminating central variable of
$B_{\fV, t_i}$ at $\bz$,
 can not appear in  $B_{\fV, t_{j}}$ since $B_{\fV, t_{j}}$ already terminates.
Hence, $ {{\partial B_{\fV, t_{j}}} \over {\partial y_{\fV,t_i}}}(\bz) =0$ when $j<i$. This implies that 
the matrix $J^*(\cB^\inc_{\fV,F})$ is  lower-triangular with nonzero entries along the diagonal at $\bz$,
hence also invertible at $\bz$.

(2). For any $\bG \in \sfm$ such that $\bG < \bF$, we have that $B_{\fV, s}$
with $s \in S_G \- s_G$, either original or $\hs$-intrinsic at $\bz$,
must terminate earlier than $B_{\fV, t_i}$ does at $\bz$, for all $1 \le i \le q$.
Hence, by the identical arguments as in (1) (for the case $j<i$), we conclude that the variable
 $y_{\fV,t_i}$ does not appear in $B_{\fV, s}$.
 
 (3). Now, if $y_{\fV,t_i}$ appears in  $L_{\fV, H}$ for some $\bH \in \sfm$, so does $\zeta_\fV y_{\fV,t_i}$
 where $\zeta_\fV \in \var_\fV$ is the terminating exceptional parameter
(see Definition \ref{termi-v0} for the explanation of $\zeta_\fV$). Since $\zeta_\fV (\bz) =0$,
 the statement follows.
   \end{proof}

 \begin{defn}\label{pleasant-v} Fix and consider
  any variable $y \in \var_\fV$ that appears in some equations of the block $\fG_k$
 for some $k \in [\up]$ such that $y(\bz) \ne 0$.
 (See \eqref{barB-ktau-final}, \eqref{barL-ktau-final}, and the subsequent paragraph 
  for the explanation of $\fG_k$).
 We say that $y$ is pleasant if either $y$ does not appear in any equation of the block $\fG_j$
for all $1 \le j <k$, or else, when it does, it appears in a term that vanishes at $\bz$.
 \end{defn}
 By this definition,  the terminating central variables $y_{\fV, t_i}$ in Lemma \ref{pre-tri}
 are pleasant due to (2) and (3) of that lemma.

\begin{defn}\label{disjoint-smooth} 
A scheme $X$ is smooth if it is a disjoint union of connected smooth schemes of
possibly  various dimensions.
\end{defn}

\begin{thm} \label{main-thm}
 Let $\Ga$ be any subset $\var_{\rU_\um}$.
Assume that $Z_\Ga$ is integral.
Let $\tZ_{\eth,\Ga}$ be  the $\eth$-transform of $Z_\Ga$ in $ \tsV_{\eth}$.
Then,  $\tZ_{\eth,\Ga}$ is smooth over  $\Spec \mathbb F$.
Consequently,  $\tZ^\dagger_{\eth,\Ga}$ is smooth over  $\Spec \mathbb F$.
 
 In particular,  when $\Ga=\emptyset$, we obtain that
$\tsV_\eth$ is smooth  over $\Spec \mathbb F$.
\end{thm}
\begin{proof} Fix any closed point $\bz \in \tZ_{\eth,\Ga}$ and 
let $\fV$ be a standard chart of $\tsR_\eth$ such that $\bz \in \fV$. 
Let $\tGa^\zero_\fV $ and  $\tGa^\one_\fV$ be as in Corollary \ref{eth-transform-up}. 
We let $$\tGa_\fV=\tGa^\zero_\fV \sqcup \tGa^\one_\fV.$$
By setting $y=0$ for all $y \in \tGa^\zero_\fV $ and 
$y=1$ for all $y \in \tGa^\one_\fV $, we obtain an affine subspace $\fV_\Ga$ of $\fV$.
That is,
 $$\fV_\Ga =\{ y =0, \; y \in  \tGa^\zero_\fV;  \;  y=1, \; y \in \tGa^\one_\fV\} \subset \fV.$$

As an affine space,
 $\fV_\Ga$ comes equipped with the set of coordinate variables
 $$\{ y \mid y \in \var_\fV \-  \tGa_\fV\}.$$     
 For any polynomial $f \in \kk[y]_{y \in \var_\fV}$, we let $f|_{\tGa_\fV}$ be obtained from
 $f$ by setting all variables in  $\tGa^\zero_\fV$ to be  0
 and setting  all variables in  $\tGa^\zero_\fV$ to be 1.
 This way,  $f|_{\tGa_\fV}$
 becomes a polynomial over $\fV_\Ga$.

  Then, $\tZ_{\eth, \Ga}\cap \fV$ 
  can be identified with the closed subscheme of $\fV_\Ga$ defined by
 \begin{eqnarray}
 \;\;\;\;\;  
  B^\q_{\fV}|_{\tGa_\fV}, \;\;  \forall \;\; B^\q \in \cB^q, \label{Bres-final-ga} \\
 \;\;\;\;\;\;\; \;\;\;\;\;  \;\;\;\;\;\; B_{ \fV, (k\tau)}|_{\tGa_\fV}, \;\; \forall \;\; (k\tau) \in \Index_{\cB^\mn},
 \label{barB-ktau-final-ga}\\ 
L_{\fV,F_k}|_{\tGa_\fV}, \;\; \forall \;\;  \bF_k \in  \sF^\star_{\fV,\Ga}
\label{barL-ktau-final-ga}
\end{eqnarray}

For any subset $P$ of polynomials over $\fV$, we let
$P|_{\tGa_\fV}=\{f|_{\tGa_\fV} \mid f \in P\}.$ This way, we have
$\cB^\mn_\fV|_{\tGa_\fV}, \cB^\q_\fV|_{\tGa_\fV}$,  etc..

In what follows, we focus on the polynomial equations in \eqref{barB-ktau-final-ga} 
and \eqref{barL-ktau-final-ga}, treated as polynomials in 
 $\kk[y]_{y \in \var_\fV \- \tGa_\fV}$. We will analyze the Jacobian of these polynomials.

 Fix any $\bF =\bF_k \in \sfm$ with $k \in [\up]$. We divide the analysis into three cases.

  \medskip\noindent
 {\sl Case $\alpha$.} {\it First, we suppose $\cB^\ori_{\fV,F}= \emptyset$.}
  Hence, in this case,  $\cB^\mn_{\fV,F}=\cB^\inc_{\fV,F}.$
   (Note that this must be the case when $F$ is $\Ga$-irrelevant.)
  
Assume first that  $\bF \notin  \sF^\star_{\fV,\Ga}$
  (thus, $F$ must be $\Ga$-irrelevant. cf. \eqref{Lrel-up}).
  In this case, by Corollary \ref{eth-transform-up}, $L_{\fV, F}|_{\tGa_\fV}$ is not
  one of the defining equations. 
  {\it Solely for the purpose of uniform writing to be 
  conveniently used later,
  we set $``L_{\fV, F}|_{\tGa_\fV} \equiv 0"$ when  $\bF \notin  \sF^\star_{\fV,\Ga}$.}
 In such a case, we only need to consider  $\cB^\inc_{\fV,F}$.
 Note that there may exist $B \in \cB^\inc_{\fV,F}$
 such that $B_\fV|_{\tGa_\fV}$ is not one of the defining equations
 for $\tZ_{\eth, \Ga}\cap \fV$. This happens if when the variables of $\tGa_\fV$ are plugged into
 $B_\fV$, it is automatically satisfied.
  {\it Again, solely for the purpose of uniform writing to be 
  conveniently used later,
  we set $``B_\fV|_{\tGa_\fV} \equiv 0"$ when  such a case occurs.}
 We let ${\bf y}^\inc_{\fV,F}$ be the set of
the terminating central variables for the equations $\cB^\inc_{\fV,F}$, determined 
and listed as in 
  \eqref{list-y} and \eqref{partial-y} (cf. Lemma \ref{pre-tri}).
 Then,   we set 
 $$J^*(L_{\fV,F}|_{\tGa_\fV}, \cB^\mn_{\fV,F}|_{\tGa_\fV} )
 = J^*(L_{\fV,F}|_{\tGa_\fV}, \cB^\inc_{\fV,F}|_{\tGa_\fV})
 := {{\partial (\cB^\inc_{\fV,F}|_{\tGa_\fV} ) } 
 \over {\partial  ({\bf y}^\inc_{\fV,F}) }} .$$
 We remark here again that $``L_{\fV, F}|_{\tGa_\fV} \equiv 0"$,
 and possibly $``B_\fV|_{\tGa_\fV} \equiv 0"$ for some
 $B \in \cB^\inc_{\fV,F}$, may not be part of the defining equations.

Further, one checks directly  from construction that any of the terminating central variables 
in ${\bf y}^\inc_{\fV,F}$ does not belong to $\tGa_\fV=\tGa^\zero_\fV \cup \tGa^\one_\fV$.

  Next we assume that  $\bF \in \sF^\star_{\fV,\um}$
   (this is automatic when $F$ is $\Ga$-relevant).
  As $L_F$  terminates on $\sR_\sF$, there exists $s \in S_F$ such that
  $x_{\fV, (\uu_s,\uv_s)}$ exists in $\var_\fV$ and the term 
  $\vsgn (s) x_{\fV, (\uu_s,\uv_s)}$ of  $L_{\fV,F}$
  does not vanish at $\bz$. Hence,
 $ {{\partial L_{\fV, F}|_{\tGa_\fV}} \over {\partial x_{\fV, (\uu_s,\uv_s)} }}(\bz) =\vsgn (s) \ne 0$.
 Now, we consider $\cB^\inc_{\fV,F}$.
 We let ${\bf y}^\inc_{\fV,F}$ be the set of
the terminating central variables for the equations $\cB^\inc_{\fV,F}$, 
determined and listed as in  \eqref{list-y} and \eqref{partial-y}.
 Then,   we set 
 $$J^*(L_{\fV,F}|_{\tGa_\fV}, \cB^\mn_{\fV,F}|_{\tGa_\fV})= 
 J^*(L_{\fV,F}|_{\tGa_\fV}, \cB^\inc_{\fV,F}|_{\tGa_\fV})
 := {{\partial (L_{\fV, F|_{\tGa_\fV}},\cB^\inc_{\fV,F}|_{\tGa_\fV}) } 
 \over {\partial (x_{\fV, (\uu_s,\uv_s)}, {\bf y}^\inc_{\fV,F}) }} .$$
Then, one sees that at the point $\bz$,  $J^*(L_{\fV,F}|_{\tGa_\fV}, \cB_{\fV,F})$ takes form
\begin{equation}\nonumber  
J^*(L_{\fV,F}|_{\tGa_\fV}, \cB^\mn_{\fV,F}|_{\tGa_\fV})(\bz)= \left(
\begin{array}{cccccccccc}
  {{\partial L_{\fV, F}|_{\tGa_\fV}} \over {\partial x_{\fV, (\uu_s,\uv_s)} }} (\bz)& 0  \\
*   & J^*(\cB^\inc_{\fV,F})(\bz)
\end{array}
\right),
\end{equation}
where the zero entries in the upper-right corner are due to Lemma \ref{pre-tri} (3).

Then,  in either case of $\bF \notin  \sF^\star_{\fV,\Ga}$ and $\bF \in  \sF^\star_{\fV,\Ga}$,
 by Lemma \ref{pre-tri} (1),  
 $J^*(\cB^\inc_{\fV,F}|_{\tGa_\fV})$  is an invertible lower-triangular matrix
  at the point $\bz$, implying that in either case,
  $J^*(L_{\fV,F}|_{\tGa_\fV}, \cB_{\fV,F}|_{\tGa_\fV})$ is of full rank at $\bz$. 

Further, when $\bF \in  \sF^\star_{\fV,\Ga}$,
 observe that the variable $x_{\fV, (\uu_s,\uv_s)}$ uniquely appears in the equations in
 the block $\fG_k$ with $F=F_k$.
 Hence, it is pleasant (see Definition \ref{pleasant-v}).
 In addition, by Lemma \ref{pre-tri} (2) and (3), all the terminating central variables used to compute
 $J^*(\cB^\inc_{\fV,F}|_{\tGa_\fV})$ are  pleasant, too.
 Thus, we conclude that in this case (i.e., {\sl Case $\alpha$}),
 in either case of
  $\bF \notin  \sF^\star_{\fV,\Ga}$ and $\bF \in  \sF^\star_{\fV,\Ga}$,
all the variables that are used to compute
 the (partial) Jacobian $J^*(L_{\fV,F}|_{\tGa_\fV}, \cB^\mn_{\fV,F}|_{\tGa_\fV})$ are pleasant.
 
\smallskip

Next, we assume $\cB^{\rm ori}_F \ne \emptyset$, that is,
$\cB^\inc_F \ne \cB^\mn_F$.  (In such a case, $F$ must be $\Ga$-relevant.)
As earlier, we let $ \fV_{[0]}$ be the unique chart
of $\sR_\sF$ such that $\fV$ lies over $ \fV_{[0]}$
and $\pi=\pi_{\fV, \fV_{[0]}}: \fV \lra  \fV_{[0]}$ be the projection.
Then,  $x_{\fV_{[0]}, \uu_F} (\pi(\bz))\ne 0$,
 hence, $x_{\fV, \uu_F} (\bz)\ne 0$, by Lemma \ref{2-cases}.
As discussed in the paragraphs subsequent to Lemma \ref{2-cases},
 we have the following two situations to consider:  $s_{F,o}=s_F$ and $s_{F,o}\ne s_F$.
 
\medskip\noindent
 {\sl Case $\beta$.} {\it Suppose  $\cB^\ori_F \ne \emptyset$ and $s_{F,o}=s_F$.}
 Then, we  write $S^\ori_F=\{s_1, \cdots, s_\ell\}$ for some $1 \le \ell \le \ft_F$.
And, we have  that the binomial equations of  $\cB^\ori_F$ take the following (original) forms:
\begin{equation}\label{barBF='}
B_{\fV, s_i}: \;\; x_{\fV, (\uu_{s_i}, \uv_{s_i})}x_{\fV, \uu_F}-  x_{\fV, \uu_{s_i}} x_{\fV, \uv_{s_i}}, \;\; 1 \le i \le \ell.
\end{equation}
 Suppose that $\cB^\inc_F \ne \emptyset$. As before, we write
$S^\inc_F =\{t_1 < \cdots <t_q\}$ for some $1\le q \le \ft_F$.  Then, on the chart $ \fV_{[0]}$, we have
$$B_{ \fV_{[0]}, t_i}: \;\; x_{ \fV_{[0]}, (\uu_{t_i}, \uv_{t_i})}x_{\fV_{(10),\uu_F}}-
  x_{\fV_{[0]}, \uu_{t_i}} x_{\fV_{[0]},\uv_{t_i}}, 1 \le i \le q.$$
As $B_{\fV, t_i}$ is $\hs$-intrinsic at $\bz$ and $x_{\fV_{[0]},\uu_F} (\pi(\bz))
\ne 0$, one finds $ x_{ \fV_{[0]}, (\uu_{t_i}, \uv_{t_i})} (\pi(\bz))=0$.
Then, $L_{\fV, F}$ admits the following form
\begin{equation}\label{tildex=0'}
L_{\fV, F}= \vsgn (s_F) + \sum_{i=1}^\ell \vsgn (s_i) x_{\fV, (\uu_{s_i}, \uv_{s_i})} 
+ \sum_{i=1}^q \vsgn (t_i) \tilde{x}_{\fV, (\uu_{t_i}, \uv_{t_i})},\end{equation} 
 where $\tilde{x}_{\fV, (\uu_{t_i}, \uv_{t_i})}=\pi^*x_{\fV_{[0]}, (\uu_{t_i}, \uv_{t_i})}$ is  the pullback
of $x_{\fV_{[0]}, (\uu_{t_i}, \uv_{t_i})}$. In particular, we have
\begin{equation}\label{tildex=0}
\tilde{x}_{\fV, (\uu_{t_i}, \uv_{t_i})}(\bz) = 0 ,\;\; \forall \;\; 1\le i \le q.
 \end{equation}

{\it As the chart $\fV$ is fixed and is clear from the context, in what follows, for simplicity of writing,
we will selectively drop some subindex $``\ \fV \ "$. For instance, we may write
$x_{\uu_F}$ for $x_{\fV,\uu_F}$, $x_{(\uu_{s_1},\uv_{s_1})}$ for $x_{\fV, (\uu_{s_1},\uv_{s_1})}$, etc.
A confusion is unlikely.
}

We introduce the following (partial) Jacobian matrix
$$J^*(\cB^\ori_{\fV,F}|_{\tGa_\fV}, L_{\fV,F}|_{\tGa_\fV})= {{\partial(B_{\fV, s_1}|_{\tGa_\fV} \cdots B_{\fV, s_\ell}|_{\tGa_\fV},L_{\fV,F}|_{\tGa_\fV})} \over {{\partial(x_{\uu_F},
x_{(\uu_{s_1},\uv_{s_1})} \cdots x_{(\uu_{s_\ell},\uv_{s_\ell})})}}} .$$
Then, one calculates and finds
\begin{eqnarray} \nonumber
 J^*(\cB^\ori_{\fV,F}, L_{\fV,F}|_{\tGa_\fV}) =   
\left(
\begin{array}{cccccccccc}
x_{(\uu_{s_1}, \uv_{s_1})} & x_{\uu_F} & 0  & \cdots & 0 \\
x_{(\uu_{s_2}, \uv_{s_2})} & 0 & x_{\uu_F}  & \cdots & 0 \\
\vdots \\
x_{(\uu_{s_\ell}, \uv_{s_\ell})} & 0 & 0&  \cdots & x_{\uu_F} \\
0 & \vsgn (s_1)&  \vsgn (s_2) & \cdots & \vsgn (s_\ell)
\end{array}
\right).
\end{eqnarray}

Multiplying the first column by $-x_{\uu_F}$ ($\ne 0$ at $\bz$), we obtained 
\begin{eqnarray}\nonumber
\left(
\begin{array}{cccccccccc}
-x_{(\uu_{s_1}, \uv_{s_1})}x_{\uu_F}  & x_{\uu_F} & 0  & \cdots & 0 \\
-x_{(\uu_{s_2}, \uv_{s_2})}x_{\uu_F}  & 0 & x_{\uu_F}  & \cdots & 0 \\
\vdots \\
-x_{(\uu_{s_\ell}, \uv_{s_\ell})} x_{\uu_F} & 0 & 0&  \cdots & x_{\uu_F} \\
0 & \vsgn (s_1)&  \vsgn (s_2) & \cdots & \vsgn (s_\ell)
\end{array}
\right).
\end{eqnarray}
Multiplying the $(i+1)$-th column by $x_{(\uu_{s_i}, \uv_{s_i})}$ and added to the first for all $1\le i \le h$, 
we obtain
\begin{eqnarray}\nonumber
\left(
\begin{array}{cccccccccc}
0  & x_{\uu_F} & 0  & \cdots & 0 \\
0 & 0 & x_{\uu_F}  & \cdots & 0 \\
\vdots & \vdots  & \vdots & \cdots & \vdots\\
0 & x_{\uu_F} & 0 &  \cdots & x_{\uu_F} \\
\sum_{i=1}^\ell \vsgn (s_i) x_{(\uu_{s_i}, \uv_{s_i})} & \vsgn (s_1)&  \vsgn (s_2) & \cdots & \vsgn (s_\ell)
\end{array}
\right).
\end{eqnarray}
But, at the point $\bz$, we have
$$\sum_{i=1}^\ell \vsgn (s_i) x_{(\uu_{s_i}, \uv_{s_i})} (\bz)=
 - \vsgn (s_F)  - 
 \sum_{t \in S^\inc_F} \vsgn (t) {\tilde x}_{(\uu_{t}, \uv_{t})} (\bz)= - \vsgn (s_F) \ne 0,$$
 because of \eqref{tildex=0'} and \eqref{tildex=0}.
Thus, we conclude that 
$J^*(\cB^\ori_{\fV,F}|_{\tGa_\fV}, L_{\fV,F}|_{\tGa_\fV})$ is a square matrix of full rank at $\bz$.

We now consider $\cB^\inc_{\fV,F}$. 
 If $\cB_F^\inc =\emptyset$, there is nothing to consider, and we move on.
 Suppose $\cB_F^\inc \ne \emptyset$.
We let ${\bf y}^\inc_{\fV,F}$ be the set of
the terminating central variables for the equations $\cB^\inc_{\fV,F}$, determined and listed as
in \eqref{list-y} and \eqref{partial-y}.  We then set
\begin{equation}\nonumber 
J^*(\cB^\ori_{\fV,F}|_{\tGa_\fV}, L_{\fV,F}|_{\tGa_\fV}, \cB^\inc_{\fV,F}|_{\tGa_\fV})= 
{{\partial(B_{\fV, s_1}|_{\tGa_\fV} \cdots B_{\fV, s_\ell}|_{\tGa_\fV},L_{\fV,F}|_{\tGa_\fV}, \cB^\inc_{\fV,F}|_{\tGa_\fV})} \over {{\partial(x_{\uu_F},
 x_{(\uu_{s_1},\uv_{s_1})}\cdots x_{(\uu_{s_\ell},\uv_{s_\ell})},  {\bf y}^\inc_{\fV,F})}}}.
 \end{equation}
Then, by Lemma \ref{pre-tri} (2) and (3),
one sees that $J^*(\cB^\ori_{\fV,F}|_{\tGa_\fV}, L_{\fV,F}|_{\tGa_\fV}, \cB^\inc_{\fV,F}|_{\tGa_\fV})$
 at the point $\bz$
takes form
\begin{equation}\nonumber
 J^*(\cB^\ori_{\fV,F}|_{\tGa_\fV}, L_{\fV,F}|_{\tGa_\fV}, \cB^\inc_{\fV,F}|_{\tGa_\fV})(\bz)= \left(
\begin{array}{cccccccccc}
 J^*(\cB^\ori_{\fV,F}|_{\tGa_\fV}, L_{\fV,F}|_{\tGa_\fV}) (\bz) & 0  \\
*   &  {{\partial (\cB^\inc_F|_{\tGa_\fV})} \over {{\partial ({\bf y}^\inc_{\fV,F})}}}(\bz)
\end{array}
\right).
\end{equation}
Then, applying Lemma \ref{pre-tri} (1) to ${{\partial (\cB^\inc_F|_{\tGa_\fV})} \over {{\partial ({\bf y}^\inc_{\fV,F})}}}$
and the previous discussion on the matrix $J^*(\cB^\ori_{\fV,F}|_{\tGa_\fV}, L_{\fV,F}|_{\tGa_\fV})$, we conclude that
the (partial) Jacobian $J^*(\cB^\ori_{\fV,F}|_{\tGa_\fV}, L_{\fV,F}|_{\tGa_\fV}, \cB^\inc_{\fV,F}|_{\tGa_\fV})$ for the block of
equations in $\fG_k$,  where $F=F_k$,  
achieves its full rank at $\bz$.

Further, observe here that the variables 
$x_{(\uu_{s_1},\uv_{s_1})}\cdots x_{(\uu_{s_\ell},\uv_{s_\ell})}$ uniquely appear in 
the block $\fG_k$.
Also, $x_{\fV, \uu_F}$, as the leading variable of $F$, does not appear in any 
equation of $\fG_j$ with $j < k$ (see Proposition \ref{leadingTerm}).
 In addition, by Lemma \ref{pre-tri} (2) and (3), all the terminating central variables used to compute
 $J^*(\cB^\inc_{\fV,F}|_{\tGa_\fV})$ are pleasant.
 Thus, we conclude that  in this case (i.e., {\sl Case $\beta$}),
 all the variables that are used to compute
 the (partial) Jacobian $J^*(\cB^\ori_{\fV,F}|_{\tGa_\fV}, L_{\fV,F}|_{\tGa_\fV}, \cB^\inc_{\fV,F}|_{\tGa_\fV})$ are pleasant.  

\medskip\noindent
 {\sl Case $\gamma$.}
{\it Now, we suppose  $\cB^\ori_F \ne \emptyset$ and $s_{F,o} \ne s_F$.}  
Then, as earlier, we can write $S^\ori_F=\{s_{F,o}, s_1, \cdots, s_\ell\}$ for some $1 \le \ell \le \ft_F$.
And, we have  the binomial equations of  $\cB^\ori_F$ as
\begin{eqnarray}
B_{\fV, s_o}: \;\;  x_{\fV,\uu_F}- x_{\fV,(\um,\uu_F)}  x_{\fV,\uu_{s_{F,o}}} x_{\fV,\uv_{s_{F,o}}},
\label{barBFnot='}  \\
B_{\fV, s_i}: \;\;  x_{\fV,(\uu_{s_i}, \uv_{s_i})}x_{\fV, \uu_F}-
 x_{\fV,(\um,\uu_F)}  x_{\fV,\uu_{s_i}} x_{\fV,\uv_{s_i}} =0, \;\; 1 \le i \le \ell.
\end{eqnarray}
Again, we selectively drop some subindex $``\ \fV \ "$ below. Then,
in this case, one calculates and finds  
$$J^*(\cB^\ori_{\fV,F}|_{\tGa_\fV}, L_{\fV,F}|_{\tGa_\fV}):=
{{\partial(B_{\fV, s_{F,o}}|_{\tGa_\fV}, B_{\fV, s_1}|_{\tGa_\fV} \cdots B_{\fV, s_\ell}|_{\tGa_\fV},L_{\fV,F}|_{\tGa_\fV})} \over {{\partial(x_{\uu_F},
x_{(\um,\uu_F)}, x_{(\uu_{s_1},\uv_{s_1})}\cdots x_{(\uu_{s_\ell},\uv_{s_\ell})})}}}$$ is equal to
\begin{eqnarray} \nonumber
\left(
\begin{array}{cccccccccc}
1 & -x_{\uu_{s_{F,o}}} x_{\uv_{s_{F,o}}} & 0  & \cdots & 0 \\
x_{(\uu_{s_1}, \uv_{s_1})} & -x_{\uu_{s_1}} x_{\uv_{s_1}} & x_{\uu_F}  & \cdots & 0 \\
\vdots \\
x_{(\uu_{s_\ell}, \uv_{s_\ell})} & -x_{\uu_{s_\ell}} x_{\uv_{s_\ell}} & 0&  \cdots & x_{\uu_F} \\
0 & \vsgn (s_F)&  \vsgn (s_1) & \cdots & \vsgn (s_\ell)
\end{array}
\right).
\end{eqnarray}

Multiplying the first column by $x_{\uu_F}$ ($\ne 0$ at $\bz$), we obtained 
\begin{eqnarray}\nonumber
\left(
\begin{array}{cccccccccc}
x_{\uu_F} & -x_{\uu_{s_{F,o}}} x_{\uv_{s_{F,o}}} & 0  & \cdots & 0 \\
x_{\uu_F}x_{(\uu_{s_1}, \uv_{s_1})} & -x_{\uu_{s_1}} x_{\uv_{s_1}} & x_{\uu_F}  & \cdots & 0 \\
\vdots \\
x_{\uu_F}x_{(\uu_{s_\ell}, \uv_{s_\ell})} & -x_{\uu_{s_\ell}} x_{\uv_{s_\ell}} & 0&  \cdots & x_{\uu_F} \\
0 & \vsgn (s_F)&  \vsgn (s_1) & \cdots & \vsgn (s_\ell)
\end{array}
\right).
\end{eqnarray}
Multiplying the second column by $x_{(\um, \uu_F)}$ and added to the first, we obtain
\begin{eqnarray}\nonumber
\left(
\begin{array}{cccccccccc}
0 & -x_{\uu_{s_{F,o}}} x_{\uv_{s_{F,o}}} & 0  & \cdots & 0 \\
0& -x_{\uu_{s_1}} x_{\uv_{s_1}} & x_{\uu_F}  & \cdots & 0 \\
\vdots \\
0& -x_{\uu_{s_\ell}} x_{\uv_{s_\ell}} & 0&  \cdots & x_{\uu_F} \\
\vsgn (s_F) x_{(\um, \uu_F)} & \vsgn (s_F)&  \vsgn (s_1) & \cdots & \vsgn (s_\ell)
\end{array}
\right).
\end{eqnarray}
As $\vsgn (s_F) x_{(\um, \uu_F)}$,  $x_{\uu_{s_{F,o}}} x_{\uv_{s_{F,o}}} $ and $x_{\uu_F}$ are all not
vanishing at $\bz$  by \eqref{barBFnot='}, one see that the above matrix, hence
$J^*(\cB^\ori_{\fV,F}|_{\tGa_\fV}, L_{\fV,F}|_{\tGa_\fV})$,
 is of full rank at $\bz$.

 We now consider $\cB^\inc_{\fV,F}$. If $\cB_F^\inc =\emptyset$, there is nothing to consider.
 Suppose $\cB_F^\inc \ne \emptyset$.
 As in {\sl Case $\beta$}, we let ${\bf y}^\inc_{\fV,F}$ be the set of
the terminating central variables for the equations $\cB^\inc_{\fV,F}$, determined and listed as in
\eqref{list-y} and \eqref{partial-y}.
We then set
\begin{equation} \nonumber 
J^*(\cB^\ori_{\fV,F}|_{\tGa_\fV}, L_{\fV,F}|_{\tGa_\fV}, \cB^\inc_{\fV,F}|_{\tGa_\fV})= 
{{\partial(B_{\fV, s_{F,o}}|_{\tGa_\fV}, B_{\fV, s_1}|_{\tGa_\fV} \cdots B_{\fV, s_\ell}|_{\tGa_\fV},L_{\fV,F}|_{\tGa_\fV}, \cB^\inc_{\fV,F})} \over {{\partial(x_{\uu_F},
x_{(\um,\uu_F)}, x_{(\uu_{s_1},\uv_{s_1})}\cdots x_{(\uu_{s_\ell},\uv_{s_\ell})},  {\bf y}^\inc_{\fV,F})}}}.
 \end{equation}
Then, by Lemma \ref{pre-tri} (2) and (3),
one sees that $J^*(\cB^\ori_{\fV,F}, L_{\fV,F}|_{\tGa_\fV}, \cB^\inc_{\fV,F})$ takes 
the following form at the point $\bz$
\begin{equation}\nonumber
 J^*(\cB^\ori_{\fV,F}|_{\tGa_\fV}, L_{\fV,F}|_{\tGa_\fV}, \cB^\inc_{\fV,F}|_{\tGa_\fV})(\bz)= \left(
\begin{array}{cccccccccc}
 J^*(\cB^\ori_{\fV,F}|_{\tGa_\fV}, L_{\fV,F}|_{\tGa_\fV})(\bz) & 0  \\
*   &  {{\partial (\cB^\inc_F|_{\tGa_\fV})} \over {{\partial ({\bf y}^\inc_{\fV,F})}}}(\bz)
\end{array}
\right).
\end{equation}
Then, applying Lemma \ref{pre-tri} (1) to 
${{\partial (\cB^\inc_F|_{\tGa_\fV})} \over {{\partial ({\bf y}^\inc_{\fV,F})}}}$
and the previous discussion on the matrix $J^*(\cB^\ori_{\fV,F}|_{\tGa_\fV}, L_{\fV,F}|_{\tGa_\fV})$, we conclude that
the (partial) Jacobian $J^*(\cB^\ori_{\fV,F}|_{\tGa_\fV}, L_{\fV,F}|_{\tGa_\fV}, \cB^\inc_{\fV,F}|_{\tGa_\fV})$ for the block of
equations in $\fG_k$,  where $F=F_k$, 
achieves its full rank at $\bz$.

Further, observe here that the variables 
$x_{(\um,\uu_F)}, x_{(\uu_{s_1},\uv_{s_1})}\cdots x_{(\uu_{s_\ell},\uv_{s_\ell})}$ uniquely appear in 
the block $\fG_k$.
Also, $x_{\uu_F}$, as the leading variable of $F$, does not appear in any 
equation of $\fG_j$ with $j < k$.
 In addition, by Lemma \ref{pre-tri} (2) and (3), all the terminating central variables used to compute
 $J^*(\cB^\inc_{\fV,F})$ are pleasant as well.
 Thus, we conclude that all the variables that are used to compute
 the (partial) Jacobian $J^*(\cB^\ori_{\fV,F}|_{\tGa_\fV}, L_{\fV,F}|_{\tGa_\fV}, \cB^\inc_{\fV,F}|_{\tGa_\fV})$ are pleasant,
  for all $F$, 
 in this case (i.e., {\sl Case $\gamma$}).

\smallskip
 Now, we return to the general case of the decomposition
 $\cB_F=\cB_F^\ori \sqcup \cB_F^\inc$,
 regardless weather  $\cB_F^\ori =\emptyset$ or not,
 i.e., we consider all {\sl $``$Cases $\alpha, \beta$, and $\gamma$$"$} together.
 In all the cases, we conclude that the (partial) Jacobian
 $J^*(\cB^\ori_{\fV,F}|_{\tGa_\fV}, L_{\fV,F}|_{\tGa_\fV}, \cB^\inc_{\fV,F}|_{\tGa_\fV})$,
 where we allow either $\cB_F^\ori =\emptyset$ or $\cB_F^\inc =\emptyset$,
 and also $L_{\fV,F}|_{\tGa_\fV} \equiv 0$ is not one of the defining equation when $\bF \in \sfmgir$,
  is a square matrix of full rank at the point $\bz$
 and all the variables that are used to compute it
  are pleasant, for all $F$.

Thus, at the point $\bz$, by combining 
the (partial) Jacobian $J^*(\cB^\ori_{\fV,F}|_{\tGa_\fV}, L_{\fV,F}|_{\tGa_\fV}, 
\cB^\inc_{\fV,F}|_{\tGa_\fV})$ together for all
$\bF \in \sfm$, we obtain a maximal minor, simply denoted by $J^*$,
 of the complete Jacobian of all the equations of
$$\{B_{\fV, s}|_{\tGa_\fV}, L_{\fV, F}|_{\tGa_\fV} \mid \bF \in \sfm, \; s \in S_F \- s_F\}$$
as follows (here again, we drop some subindex $``\ \fV \ "$)
\begin{equation}\label{the-matrix} 
{\scriptsize
 \left(
\begin{array}{cccccccccc}
 J^*(\cB^\ori_{F_1}|_{\tGa_\fV}, L_{F_1}|_{\tGa_\fV}, \cB^\inc_{F_1}|_{\tGa_\fV})(\bz) & 0 & \cdots & 0 \\
* &  J^*(\cB^\ori_{F_2}|_{\tGa_\fV}, L_{F_2}|_{\tGa_\fV},\cB^\inc_{F_2}|_{\tGa_\fV} )(\bz) & \cdots & 0 \\
\vdots & \vdots & \vdots &\vdots \\
*  & * & *&   J^*(\cB^\ori_{F_\up}|_{\tGa_\fV}, L_{F_\up}|_{\tGa_\fV}, \cB^\inc_{F_\up}|_{\tGa_\fV})(\bz) \end{array}
\right).
}
\end{equation}
All the zero blocks in the upper-right corner are because for every $F_k$ with $1\le k \le \up$,
the variables that are used to compute $J^*(\cB^\ori_{\fV,F}|_{\tGa_\fV}, L_{\fV,F}|_{\tGa_\fV}, 
\cB^\inc_{\fV,F}|_{\tGa_\fV})$ 
are pleasant. Hence, by our previous discussions, 
\eqref{the-matrix} is a square matrix of full rank at $\bz$.

\smallskip
 Now, we consider the case when $\Ga=\emptyset$. Thus, we have
  $Z_\emptyset=\rU_\um \cap \Gr^{d,E}$ and
$\tZ_{\eth,\emptyset }=\tsV_\eth$. 
In this case, we let $J:=J(\cB^\q, \cB^\mn_\fV, L_{\fV,\sfm})$ be the full Jacobian
 of all the defining equations of 
$\cB^\q_\fV, \cB^\mn_\fV$, and $L_{ \fV, \sfm}:=\{L_{\fV, F} \mid \bF \in \sfm \}$;
 we let  $J^*:=J^*(\cB^\mn_\fV,L_{\fV, \sfm})$ 
be the matrix of \eqref{the-matrix} in the case of $\Ga=\emptyset$.
Take any $\bz \in \tsV_\eth$,
 let $T_\bz (\tsV_\eth)$ be the Zariski tangent space of
$\tsV_\eth$ at $\bz$. Then,
we have
$$\dim T_\bz (\tsV_\eth)= \dim \tsR_\eth- \rk J(\bz) \le \dim \tsR_\eth- \rk J^*(\bz)$$
$$= \dim \rU_\um + |\cB^\mn|  - (|\cB^\mn| + \up)
= \dim \rU_\um  - \up = \dim \tsV_\eth,$$
where  $\dim \tsR_\eth=\dim \rU_\um + |\cB^\mn|$ by \eqref{dim} and
$\rk J^*(\bz) =|\cB^\mn| + \up$ by \eqref{the-matrix}.
Hence, $\dim T_\bz (\tsV_\eth)  = \dim \tsV_\eth,$
thus, $\tsV_\eth$ is smooth at $\bz$. Therefore, $\tsV_\eth$ is smooth.

Consequently,  one sees that on any standard chart $\fV$ of
the {\it final} scheme $\tsR_\eth$, all the relations of 
$\cB^q_\fV$
 must lie in the ideal generated by relations of $\cB^\mn_\fV$ and $L_{\fV, \sfm}$,  
 thus,   can be discarded from the chart $\fV$.

 Now, we return to a general  subset $\Ga$ of $\var_{\rU_\um}$
 as stated in the theorem.
  By the previous paragraph, over any standard chart $\fV$ of $\tsR_\eth$
  with $\tZ_{\eth,\Ga} \cap \fV \ne \emptyset$, we can discard 
   $\cB^\q_\fV|_{\tGa_\fV}$
   from the defining equations
 of $\tZ_{\eth,\Ga} \cap \fV$ and focus only on
 the equations of $\cB^\mn_\fV|_{\tGa_\fV}$ and 
 $L_{\fV, \sF^\star_{\um,\Ga}}|_{\tGa_\fV}$.
 In other words, 
  $\tZ_{\eth,\Ga} \cap \fV$, if nonempty, as a closed subcheme of $\fV_\Ga$, is defined by the equations in 
 $\cB^\mn_\fV|_{\tGa_\fV}$ and $L_{\fV, \sF^\star_{\um,\Ga}}|_{\tGa_\fV}$.
 Then, by  \eqref{the-matrix}, the rank of the full Jacobian of 
 $\cB^\mn_\fV|_{\tGa_\fV}$ and $L_{\fV, \sF^\star_{\um,\Ga}}|_{\tGa_\fV}$
 equals to the number of the above defining equations at any closed point $\bz$ 
 of $\tZ_{\eth,\Ga} \cap \fV$. Hence,
 $$\dim T_\bz (\tZ_{\eth,\Ga} \cap \fV)=
  \dim \fV_\Ga-  \rk J^*(\cB^\ori_{\fV,F}|_{\tGa_\fV}, L_{\fV,F}|_{\tGa_\fV}, 
\cB^\inc_{\fV,F}|_{\tGa_\fV}) (\bz)$$
$$= \dim \fV_\Ga - (|\cB^\mn| + \up)
\le \dim \tZ_{\eth,\Ga} \cap \fV.$$ 
 Hence, $\tZ_{\eth,\Ga} \cap \fV$ is smooth, thus,  so is $\tZ_{\eth,\Ga}$.

This proves the theorem.
\end{proof}

Let $X$ be an integral scheme.  We say $X$ admits a resolution  if there exists  a smooth
scheme $\tX$ and a projective birational morphism
from $\tX$ onto $X$.

\begin{thm}\label{cor:main} 
Let $\Ga$ be any subset $\var_{\rU_\um}$.
Assume that $Z_\Ga$ is integral. Then, the morphism
 $\tZ^\dagger_{\eth, \Ga} \to Z_{\Ga}$ can be decomposed as
$$\tZ^\dagger_{\vr, \Ga} \to \cdots 
\to \tZ^\dagger_{\hs,\Ga}  \to \tZ^\dagger_{\hs',\Ga} \to \cdots \to
Z^\dagger_{\sF_{[j]},\Ga}  \to Z^\dagger_{\sF_{ [j-1]},\Ga} \to \cdots \to Z_\Ga$$
such that every morphism $\tZ^\dagger_{\hs,\Ga}  \to \tZ^\dagger_{\hs',\Ga}$
in the sequence is 
$ \tZ^\dagger_{(\eth_{(k\tau)}\fr_\mu\fs_{h}),\Ga}
 \to \tZ^\dagger_{(\eth_{(k\tau)}\fr_\mu \fs_{h-1}),\Ga}$ for some $(k\tau) \mu h \in \Index_\Psi$, or
$ \tZ^\dagger_{(\eth_{(k\tau)}\fr_\mu\fs_{h}),\Ga}
 \to \tZ^\dagger_{(\eth_{(k\tau)}\fr_\mu\fs_{h-1}),\Ga} $ for some $(k\tau)\mu h \in \Index_\Phi$, or
$\tZ^\dagger_{\vt_{[k]},\Ga}  \to \tZ^\dagger_{\vt_{ [k-1]},\Ga}$ for some $k \in [\up]$.
Further, every morphism in the sequence is surjective, projective, and  birational.  
In particular, $\tZ^\dagger_{\eth, \Ga} \to Z_\Ga$ 
is a resolution if $Z_\Ga$ is singular.
\end{thm}
\begin{proof} The smoothness of $\tZ^\dagger_{\eth, \Ga}$ follows from Theorem \ref{main-thm};
the decomposition of $\tZ^\dagger_{\eth, \Ga} \to Z_{\Ga}$ follows from
Lemmas \ref{wp-transform-sVk-Ga},  \ref{vt-transform-k},
 \ref{wp-transform-ktauh}, and Lemma \ref{vr-transform-ktauh}.
\end{proof}

\section{Local Resolution of Singularity}\label{local-resolution} 

We follow Lafforgue's presentation of \cite{La03} on Mn\"ev's universality theorem.

As before, suppose we have a set of vector spaces, 
$E_1, \cdots, E_n$ such that 
$E_\alpha$ is of dimension 1 (or, a free module of rank 1 over $\ZZ$).
 We let 
$$E_I = \bigoplus_{\alpha \in I} E_\alpha, \;\; \forall \; I \subset [n],$$
$$E:=E_{[n]}=E_1 \oplus \ldots \oplus E_n.$$ 
(Lafforgues \cite{La03} considers the  more general case by allowing $E_\alpha$ to be
of any  finite dimension.)

For any fixed  integer $1 \le d <n$, the Grassmannian
$$\Gr^{d,E}=\{ F \hookrightarrow E \mid \dim F=d\}$$
 decomposes into a disjoint union of locally closed strata
$$\Gr^{d,E}_\ud=\{ F \hookrightarrow E \mid \dim (F\cap E_I)=d_I,  \;\; \forall \; I \subset [n] \}$$
indexed by the family  $\ud=(d_I)_{I \subset [n]}$ of nonnegative integers $d_I \in \NN$ verifying

$\bullet$ $d_\emptyset=0, d_{[n]}=d$,

$\bullet$ $d_I +d_J \le d_{I\cup J} + d_{I \cap J}$, for all $I, J \subset [n]$.

The family $\ud$ is called a matroid of rank $d$ on the set $[n]$.
The stratum $\Gr^{d,E}_\ud$ is called a thin Schubert cell.

The Grassmannian $\Gr^{d,E}$ comes equipped with the (lattice) polytope
$$\Delta^{d,n} =\{ (x_1, \cdots, x_n) \in {\mathbb R}^n \mid 0  \le x_\alpha \le 1, \;
\forall \; \alpha; \; x_1 +\cdots + x_n = d\}.$$
For any $\ui=(i_1,\cdots,i_d) \in \II_{d,n}$, we let $\bx_\ui = (x_1, \cdots, x_n)$ be defined by
\begin{equation}\label{eta-L-2}
\left\{ 
\begin{array}{lcr}
x_i=1, &  \hbox{if $i \in \ui$,} \\
x_i=0, & \hbox{otherwise}.
\end{array} \right.
\end{equation}
It is known that  $\Delta^{d,n} \cap \NN^n =\{\bx_\ui \mid \ui \in \II_{d,n}\}$ and
it consists of precisely the vertices of the polytope $\Delta^{d,n}$.

Then, the matroid $\ud=(d_I)_{I \subset [n]}$ above defines the 
following subpolytope of $\Delta^{d,n}$
$$\Delta^{d,n}_\ud =\{ (x_1, \cdots, x_n) \in \Delta^{d,n} \mid  \sum_{\alpha \in I} x_\alpha \ge d_I, \; \forall \; I \subset [n]\}.$$
This is called the matroid subpolytope of $\Delta^{d,n}$ corresponding to $\ud$.

Recall that we have a canonical decomposition
$$\wedge^d E=\bigoplus_{\ui \in \II_{d,n}} E_{i_1}\otimes \cdots \otimes E_{i_d}$$
and it gives rise to the $\pl$ embedding of the Grassmannian
$$\Gr^{d,E} \hookrightarrow \PP(\wedge^d E)=\{(p_\ui)_{\ui \in \II_{d,n}} \in \GG_m 
\backslash (\wedge^d E  \- \{0\} )\}.$$

\begin{prop}\label{to-Ga} {\rm (Proposition, p4, \cite{La03})} 
Let $\ud$ be any matroid of rank $d$ on the set $[n]$ as considered above.
Then, in the Grassmannian 
$$\Gr^{d,E} \hookrightarrow \PP(\wedge^d E)=\{(p_\ui)_{\ui \in \II_{d,n}} \in \GG_m 
\backslash (\wedge^d E \- \{0\} )\},$$
the thin Schubert cell $\Gr^{d,E}_\ud$, as a locally closed subscheme, is defined by
$$p_\ui = 0, \;\;\; \forall \; \bx_\ui \notin \Delta^{d,n}_\ud ,$$ 
$$p_\ui \ne 0, \;\;\; \forall \; \bx_\ui \in \Delta^{d,n}_\ud . $$   
\end{prop}

Let $\ud=(d_I)_{I \subset [n]}$ be a matroid of rank $d$ on the set $[n]$ as above. 
Assume that $\ud_{[n]\setminus \{\alpha\}} =d-1$ for all $1 \le \alpha \le n$.
Then, the configuration space $C^{d,n}_\ud$ defined by the matroid $\ud$ is the classifying scheme
of families of $n$ points
$$P_1, \cdots, P_n$$
on the projective space $\PP^{d-1}$ such that for any nonempty subset $I \subset [n]$,
the projective subspace $P_I$ of $\PP^{d-1}$ generated by the points $P_\alpha, \alpha \in I$, is
of dimension 
$$\dim P_I = d-1 -\ud_I.$$

\begin{thm}\label{Mn-La} {\rm (Mn\"ev, Theorem I. 14, \cite{La03})}
Let $X$ be an affine scheme of finite type over $\Spec \ZZ$.
Then, there exists a matroid $\ud$ of rank 3 on
the set $[n]$ such that $\PGL_3$ acts freely on the configuration space $C^{3,n}_\ud$.
Further, there exists a positive integer $r$ and
 an open subset $U \subset X \times \AA^r$ projecting surjectively
onto $X$ such that $U$ is isomorphic to the quotient space
${\underline C}^{3,n}_\ud :=C^{3,n}_\ud/\PGL_3$.
\end{thm}

\begin{thm}\label{GM} {\rm (Gelfand, MacPherson, Theorem I. 11, \cite{La03})}
Let $\ud$ be any matroid of rank $d$ on the set $[n]$ as considered above.
Then, the action of $\PGL_{d-1}$ on $C^{d,n}_\ud$ is free if and only if
$\dim_{\mathbb R} \Delta^{d,n}_\ud =n-1$. Similarly, 
 The action of $\GG_m^n/\GG_m$ on $\Gr^{d,n}_\ud$ is free if and only if
$\dim_{\mathbb R} \Delta^{d,n}_\ud =n-1$. 
Further, when $\dim_{\mathbb R} \Delta^{d,n}_\ud =n-1$,  the quotient
$C^{d,n}_\ud/\PGL_{d-1}$ can be canonically identified with  the quotient
$\Gr^{d,E}_\ud/(\GG_m^n/\GG_m )$.
\end{thm}


By the above correspondence, we have the following equivalent version of Theorem \ref{Mn-La}.

\begin{thm}\label{Mn-La-Gr} {\rm (Mn\"ev, Theorem I. 14, \cite{La03})}
Let $X$ be an affine scheme of finite type over $\Spec \ZZ$.
Then, there exists a matroid $\ud$ of rank 3 on
the set $[n]$ such that $(\GG_m^n/\GG_m )$ acts freely on the thin Schubert cell $\Gr^{3,E}_\ud$.
Further, there exists a positive integer $r$ and an open subset $U \subset X \times \AA^r$ projecting surjectively
onto $X$ such that $U$ is isomorphic to the quotient space
$\bGr^{3,E}_\ud:=\Gr^{3,E}_\ud/(\GG_m^n/\GG_m)$.
\end{thm}


\begin{thm}\label{main-resolusion-0} {\rm (Local Resulution)}
Let $X$ be 
a scheme of finite presentation over a 
 perfect field $\kk$. 
Assume further that $X$ is integral and 
singular. Take any singular closed point $q \in X$. Then, there exists an open subset
$V$ of $X$ containing the point $q$ such that $V$ admits a resolution. 
\end{thm}
\begin{proof}
The question is local. Hence, we can assume that $X$ is affine.

 First, we assume that $X$ is defined over $\Spec \ZZ$.

 We apply Theorem \ref{Mn-La-Gr} to $X$ and follow  the notations in Theorem \ref{Mn-La-Gr}.

We identify $U \subset X \times \AA^r$ with the quotient space
$\bGr^{3,E}_\ud=\Gr^{3,E}_\ud/(\GG_m^n/\GG_m)$.

Consider the quotient map 
$$\pi: \Gr^{3,E}_\ud \lra \bGr^{3,E}_\ud=\Gr^{3,E}_\ud/(\GG_m^n/\GG_m).$$
We have the diagram
$$ \xymatrix{
\Gr^{3,E}_\ud  \ar^{\pi \;\;\;\;\;\;\;\;\;\;\;\;\;\;\;\;\;\;\;\;\;\;\;\;\;\;}[r] 
& \bGr^{3,E}_\ud=\Gr^{3,E}_\ud/(\GG_m^n/\GG_m) 
\cong U  \ar @{^{(}->}[r]  & X \times \AA^r  \ar[d]   \\
& & X.
} $$

We can apply Proposition \ref{to-Ga} to the  thin Schubert cell 
$\Gr^{3,E}_\ud$.  

Since $\Delta^{3,n}_\ud \ne \emptyset$, there exists $\um \in \II_{3,n}$ such that
$\bx_\um \in \Delta^{3,n}_\ud$.
We define
\begin{equation}\label{ud=Ga}
\Ga:=\Ga_\ud =\{ \ui \in \II_{3,n} \mid \bx_{\ui} \notin \Delta^{3,n}_\ud \}.
\end{equation}

Then, we have that 
$$\xymatrix{ 
\Gr^{3,E}_\ud \ar @{^{(}->}[r]  &  Z_\Ga \ar @{^{(}->}[r]  & \rU_\um } $$
and $\Gr^{3,E}_\ud$ is an open subset of the $\Ga$-scheme $Z_\Ga \subset \rU_\um$. 
As $X$ is integral (by assumption), one sees that $Z_\Ga$ is integral.


We then let $$\vp_\Ga: \tZ^\dagger_{\eth, \Ga} \lra Z_\Ga$$ be as in
 Theorem \ref{cor:main}. This is a resolution.
We set $\tGr^{3,E}_\ud=\vp_\Ga^{-1}(\Gr^{3,E}_\ud)$, scheme-theoretically.
Then $$\varpi|_{\tGr^{3,E}_\ud}:  \tGr^{3,E}_\ud \lra \Gr^{3,E}_\ud$$ is a resolution.

 Take any general closed point $a \in \AA^r$, such that 
$$(q,a) \in  (X \times a) \cap U \subset  X \times \AA^r.$$
Also take  a split of $(\GG_m^n/\GG_m)$ such that 
$$(\GG_m^n/\GG_m) \cong \GG_m^{n-1}.$$

The quotient map $$\pi: 
\Gr^{3,E}_\ud \lra \bGr^{3,E}_\ud (\cong U)$$ is  a principal $(\GG_m^n/\GG_m)$-bundle, and is
\'etale  locally trivial.  As  any \'etale  locally trivial principal $(\GG_m^n/\GG_m)$-bundle
is Zariski locally trivial (that is, $(\GG_m^n/\GG_m)$ is special in the sense of Serre),
there exist  a suitably small open subset $O$ of $\bGr^{3,E}_\ud$ and
a suitably small open subset
$V$ of $X$ containing $ q$ such that the following hold:

 $\bullet$ $\Gr^{3,E}_\ud|_{O} \cong O \times (\GG_m^n/\GG_m);$

$\bullet$ $\bW_a:=(V \times a) \cap U  \stackrel{\cong}{\lra} V;$

$\bullet$    $\bW_a$, considered as a locally closed subset of $\bGr^{3,E}_\ud (\cong U)$,
admits a lift $W_a  \subset  \Gr^{3,E}_\ud$ such that we have 
 $$ W_a \subset O \times b$$
where $b=(b_1, \cdots, b_{n-1}) \in (\GG_m^n/\GG_m) \cong \GG_m^{n-1}$, and
$$W_a  \stackrel{\cong}{\lra}  \bW_a=(V \times a) \cap U   \stackrel{\cong}{\lra} V.$$

  Let ${\overline W}_a$ be the closure of $W_a$ in $Z_\Ga$ ($\supset \Gr^{3,E}_\ud$).
  For generic choice of $a$, 
  we let $\widetilde{{\overline W}}_a$ be the proper transform of ${\overline W}_a$ in 
$\tZ^\dagger_{\eth, \Ga}$  ($\supset \tGr^{3,E}_\ud$). 
  It is immediate that
the induced morphism 
$\rho: \widetilde{{\overline W}}_a \lra {\overline W}_a$ is surjective, projective and birational with
 $a \in \AA^r$ being generic. Let $\tW_a=\rho^{-1}(W_a)$. Then, the induced morphism
 $$\rho|_{\tW_a}: \tW_a \lra W_a$$ is surjective, projective and birational.

It remains to show that $\tW_a$ is smooth. 

We take any closed point $\bz \in \tW_a$.

We can depict the above situation as the following diagram.
$$ \xymatrix{ 
\bz \in \tW_a \ar[d]^{\rho|_{\tW_a}} \ar @{^{(}->}[r]  &   \ar[d]^{\varpi_\Ga} \ar @{^{(}->}[r] \tGr^{3,E}_\ud \subset {\tZ^\dagger_{\eth, \Ga}} & \tsV_\eth\ar[d] \\ 
W_a           \ar[d]^{\cong} \ar @{^{(}->}[r]  &\Gr^{3,E}_\ud   \ar[d]^{\pi} \subset Z_\Ga \ar @{^{(}->}[r]  
& \rU_\um \\ 
\bW_a  \ar[d]^{\cong} \ar @{^{(}->}[r]  &  \bGr^{3,E}_\ud \cong U \ar[d] \ar @{^{(}->} [r]  & X \times \AA^r  \\
 q \in V \ar @{^{(}->} [r] & X 
} $$

Now  we let $\fV \subset \tsR_\eth$ be a standard chart of $\tsR_\eth$ such that $\bz \in \fV$.
We let $I$ be the ideal of  $\FF[x]_{x \in \var_\fV}$ corresponding to $\tZ^\dagger_{\eth, \Ga} \cap \fV$ 
and $I_a$ be the ideal of  $\FF[x]_{x \in \var_\fV}$ corresponding to
$\widetilde{{\overline W}}_a \cap \fV$.
Further, we let  $\fm$ be the maximal ideal of $\FF[x]_{x \in \var_\fV}$ corresponding to
$\bz \in \fV$, $\fn={\mathfrak m} / I$ be the maximal ideal of
$\sO_{{\tZ^\dagger_{\eth, \Ga}} \cap \fV}({\tZ^\dagger_{\eth, \Ga}} \cap \fV)$ corresponding to
$\bz \in {\tZ^\dagger_{\eth, \Ga}}\cap \fV$, and $\fn_a={\mathfrak m} / I_a$ be the maximal ideal of
$\sO_{\widetilde{{\overline W}}_a \cap \fV}(\widetilde{{\overline W}}_a \cap \fV)$ 
corresponding to
$\bz \in  \widetilde{{\overline W}}_a \cap \fV$.

Then, we have the following digram of  exact sequences of vector spaces over $\FF(\bz)\cong \FF$ 
(see, equation (2.1), page 128, \cite{LiuQing})
$$ \xymatrix{
0  \ar[r] &  I/(I \cap \fm^2) \ar[d] \ar[r] &  \fm/\fm^2  \ar[d]^{=} \ar[r] &
\fn/\fn^2 \ar[d] \ar[r] &  0, \\
0  \ar[r] & I_a/(I_a \cap \fm^2)  \ar[r] &  \fm/\fm^2 \ar[r] & 
\fn_a/{\mathfrak n}^2_a  \ar[r] &  0.
}
$$
Hence, the left downarrow is injective. 

Let $\AA^r =\Spec \FF[t_1,\cdots, t_r]$. 
For every $i \in [r]$, we first pullback $t_i-  a_i$ via $\pi$ 
 to obtain an element of the coordinate ring $Z_\Ga$ and then lift it
 to an element $t_i'$ of the coordinate ring $\rU_\um$.
We then let $\tt_i$ be the proper transform of $t_i'$ over
 the standard chart $\fV$.
This way, we obtain $\tt_1,\cdots, \tt_r \in \FF[x]_{x \in \var_\fV}$.
 Then, $(\tt_1, \cdots, \tt_r)$ belong to the ideal $I_a$.
 
  Further, we let
 $$(\GG_m^n/\GG_m) \cong \GG_m^{n-1}=\Spec [t_{r+1}^\pm, \cdots, t_{r+n-1}^\pm].$$
We  lift $t_{r+i}-b_i$ to an element $t_{r+i}'$ of the coordinate ring $\rU_\um$.
 We then let $\tt_{r+i}$ be the proper transform of $t'_{r+i}$ over
 the standard chart $\fV$ for all $i \in [n-1]$. Then, 
 $(\tt_{r+1}, \cdots, \tt_{r+n-1})$ also belong to the ideal $I_a$.
 
 Observe now that  $(\tt_1, \cdots, \tt_{r+n-1})$ form a (part of) system of local parameters
 at the point $\bz \in \tW_a$.
 Hence, one concludes that   the images of  $\tt_1, \cdots, \tt_{r+n-1}$  
  in $I_a/(I_a \cap {\mathfrak m}^2)$  
 are linearly independent over $\FF(\bz)\cong \FF$.
 Since $\tW_a$ is birational to $W_a$ with 
 $W_a \cong \bW_a \cong V,$
 we have that the vector subspace spanned 
 by the images of $(\tt_1, \cdots, \tt_{r+n-1})$   intersects trivially with
  $I/(I \cap {\mathfrak m}^2)$, both regarded as vector subspaces of 
$I_a/(I_a \cap {\mathfrak m}^2)$.

Then, the above  implies that we have
$$\dim_{\FF(\bz)} {\mathfrak n}_a/{\mathfrak n}_a^2 = \dim_{\FF(\bz)} {\mathfrak m}/{\mathfrak m}^2 
- \dim_{\FF(\bz)} I_a/(I_a \cap {\mathfrak m}^2)$$
$$ \le \dim_{\FF(\bz)} {\mathfrak m}/{\mathfrak m}^2  -  \dim_{\FF(\bz)} I/(I \cap {\mathfrak m}^2) - (r+n-1) $$
$$ = \dim_{\FF(\bz)} {\mathfrak n}/{\mathfrak n}^2 - (r+n-1) 
= \dim {\tZ^\dagger_{\eth, \Ga}} - (r+n-1)=\dim \tW_a. $$
Here, the equality $\dim_{\FF(\bz)} {\mathfrak n}/{\mathfrak n}^2 = \dim \tZ_{\eth, \Ga}$ 
holds because 
$\tZ^\dagger_{\eth, \Ga}$ is smooth at $\bz$.  
The last equality follows from a simple dimension count.

Hence, $\tW_a$ is smooth for general closed point $a \in \AA^r$.  
Thus, 
 $$\rho|_{\tW_a}: \tW_a \lra W_a \cong \bW_a \cong V$$ is a resolution.
    Let $\tV= \tW_a$.
   Then, $\tV \lra V$ provides a resolution of $V$ as in the statement of the theorem,
   when $X$ is defined over $\Spec \ZZ$..

   Now, consider the general case when $X$ is of finite presentation over a   perfect field $\kk$.
   The field $\kk$ is an extension of its unique prime (minimal) subfield ${\mathbb F}'$. This 
   unique prime  subfield ${\mathbb F}'$ is isomorphic to 
   $\QQ$ when $\kk$ has characteristic zero or isomorphic to 
   ${\mathbb F}_p$ when $\kk$ has the characteristic $p>0$,
   that is, ${\mathbb F}' \cong \FF$.

  We let $S$ be an integral affine scheme such that $\kk$ is 
  identified with the function field of $S$.
 Then,  according to \cite{EGAIV}, Theorems 8.8.2 and 8.10.5, 
 there exist an affine open subset $B$ of $S$, 
  an integral affine scheme $Y/\ZZ$,
 and a dominant morphism $f: Y \to B$ of finite presentation over $B$,
  such that $X/\kk$ is isomorphic to the generic fiber of the morphism $f$.

   Take a singular point $x$ of $X$, lying over a point $y$ of $Y$. 
   Replacing $X$ and $Y$ by suitable open subsets containing $x$ and $y$, respectively, if necessary,
   we may assume that $Y$ admits a resolution over $\FF'$, as  $Y$ is defined over $\FF'$.
   
 Take a resolution $\tY \to Y$ over $\FF'$. And, consider the induced dominant morphism
 $\tilde f: \tY \to B$. 
 Since $\tY$ is  smooth,  we have that the generic fiber $\tX/\kk$ of $\tilde f$ is regular as well.
 Because $\kk$ is perfect, we have that $\tX/\kk$  is smooth.  
 Clearly, the induced morphism $\tX/\kk \to  X/\kk$ is proper and birational, hence is a resolution,
 as desired.
\end{proof}

When the base field $\kk$ has characteristic zero,
the above theorem is  well known from Hironaka's resolution \cite{Hironaka64}. 
When the base field $\kk$ has positive characteristic, Abhyankar \cite{Abh} proved resolution of singularities for algebraic threefolds in characteristic greater than 6.
(One may consult \cite{Zariski2} for the case when $\kk$ is not perfect.)


In general, consider any fixed (singular) integral scheme  $X$.  Let $K(X)$ be the function field of $X$.
Because the Riemann-Zariski space of $K(X)/\kk$ is quasi-compact (\cite{Zariski}), it remains
to glue finitely many local resolutions of $X$, as obtained in this article, 
to obtain  a global resolution of $X$.





 In this article, we approach resolution of singularity 
by performing  blowups of  (a chart of)  $\Gr^{3,E}$. 
It is convincible that certain parallel blowups exist for $(\PP^2)^n$ that can also lead to achieve
local resolution of singularity  (\cite{Hu15b}). 
(Indeed, when the author began to work on resolution of singularity, he tried both approaches and switched between the two for quite a while 
before settling down on the current approach via  Grassmannians.)

\section{ Geometric Resolution}

Prior to de Jong's geometric approach \cite{deJong96}, resolutions of varieties in general dimensions
are essentially done by finding good algorithms.
In such an approach, one isolates a set of bounded invariants and prove that after certain finite steps,
such invariants improve strictly. As the invariants are bounded, the algorithm terminates.
These approaches are nicely presented in Kollar's book \cite{Kollar}.

According to \cite{Vakil}, many moduli spaces or deformation spaces exhibit arbitrary singularities.
In other words, all singularities exist {\it geometrically}.
Since singularities exist for geometric reasons, one would wonder whether
there should be {\it geometric ways} to resolve them, avoiding pure algorithms on polynomials. Being  philosophically optimistic,
the author believes that every singular moduli admits a resolution, in a specific relative sense,
 such that the resolution itself is also a moduli. 
 
In other words, it would be desirable if the following problem can be answered in some positive ways.

\begin{problem}\label{program}
For any singular moduli space $\mathfrak M$, find another moduli space $\widetilde{\mathfrak M}$
that only modifies the boundary objects  of  $\mathfrak M$ 
 such that every irreducible component of $\widetilde{\mathfrak M}$, endowed with the reduced stack structure,
  is smooth, and
all such irreducible components meet transversally.
\end{problem}

Here, an object of $\widetilde{\mathfrak M}$ should be obtained from the corresponding object of
$\mathfrak M$ by adding certain extra data. The extra data should reduce the automorphisms
of the original object, and ideally, should remove  all removable obstructions.

See Conjectures 5.4 and 5.5 of \cite{Hu17} for  somewhat more precise formulations.

\end{document}